\documentclass[12pt,reqno]{amsart}
\usepackage{amsfonts,amsrefs,latexsym,amsmath, amssymb, mathrsfs, verbatim}
\usepackage{url,color}
\usepackage{pifont}
\usepackage{upgreek}
\usepackage{fancyhdr}
\usepackage{hyperref}
\usepackage{calligra}
\usepackage{marvosym}
\usepackage[percent]{overpic}
\usepackage{pict2e}

\usepackage[hypcap=false]{caption}
\usepackage{xcolor}
\usepackage{wasysym}
\usepackage{accents}

\topmargin - .23 in

\newtheorem{theorem}{Theorem}[section]
\newtheorem{proposition}{Proposition}[section]
\newtheorem{lemma}[proposition]{Lemma}

\theoremstyle{definition}
\newtheorem{definition}{Definition}[section]
\newtheorem{remark}{Remark}[section]

\theoremstyle{plain}

\DeclareMathAlphabet{\mathcalligra}{T1}{calligra}{m}{n}
\DeclareFontShape{T1}{calligra}{m}{n}{<->s*[2.2]callig15}{}

\newcommand{\Ent}{s}
\newcommand{\GradEnt}{S}

\newcommand{\Enth}{H}
\newcommand{\Lnenth}{h}

\newcommand{\Temp}{\theta}

\newcommand{\TempoverEnth}{q}

\newcommand{\vort}{\varpi}

\newcommand{\speed}{c}

\newcommand{\uperpvort}{\mbox{\upshape vort}}

\begin{document}
\title{The relativistic Euler equations: Remarkable null structures and regularity
properties}
\author[MD,JS]{Marcelo M. Disconzi$^{* \#}$ and Jared Speck$^{** \dagger}$}
	
\thanks{$^{\#}$MMD gratefully acknowledges support from NSF grant \# 1812826,
from a Sloan Research Fellowship provided by the Alfred P. Sloan foundation,
and from a Discovery grant administered by Vanderbilt University.
}

\thanks{$^{\dagger}$JS gratefully acknowledges support from NSF grant \# 1162211,
from NSF CAREER grant \# 1454419,
from a Sloan Research Fellowship provided by the Alfred P. Sloan foundation,
and from a Solomon Buchsbaum grant administered by the Massachusetts Institute of Technology.
}

\thanks{$^{*}$Vanderbilt University, Nashville, TN, USA.
\texttt{marcelo.disconzi@vanderbilt.edu}}

\thanks{$^{**}$Vanderbilt University, Nashville, TN, USA 
	\\
\& Massachusetts Institute of Technology, Cambridge, MA, USA.
\texttt{jared.speck@vanderbilt.edu}}

\newcommand{\norm}[1]{\Vert#1\Vert}
\newcommand{\abs}[1]{\vert#1\right}
\newcommand{\nnorm}[1]{\left\Vert#1\right\Vert}
\newcommand{\aabs}[1]{\left\vert#1\right\vert}
\newcommand{\ErrorTerm}{\mathrm{f}}
\newcommand{\divthreed}{{}^{(3)}\mbox{\upshape div}}
\newcommand{\curlthreed}{{}^{(3)}\mbox{\upshape curl}}
\newcommand{\td}{t}

\newtheorem{notation}{Notation}[section]

\newcommand{\Charhypersurface}{\mathscr{H}}

\begin{abstract}
We derive a new formulation of the relativistic Euler equations that 
exhibits remarkable properties. This new formulation consists of a coupled system of geometric wave,
transport, and $\mbox{\upshape transport}$-$\mbox{\upshape div}$-$\mbox{\upshape curl}$ equations, sourced by nonlinearities 
that are null forms relative to the acoustical metric. Our new formulation is
well-suited for various applications, in particular for the study of 
stable shock formation, as it is surveyed in the paper. Moreover, 
using the new formulation presented here, 
we establish a local well-posedness result showing 
that the vorticity and the entropy of the fluid
are one degree more differentiable compared to the regularity guaranteed by standard estimates
(assuming that the initial data enjoy the extra differentiability).
This gain in regularity is essential for the study of shock formation without symmetry assumptions.
Our results hold for an arbitrary equation of state, not necessarily of barotropic type.

\bigskip


\bigskip


\end{abstract}

\maketitle

\centerline{\today}

\tableofcontents
\setcounter{tocdepth}{1}

\newpage

\section{Introduction}
\label{S:INTRO}
The relativistic Euler equations are the most well-studied PDE system in relativistic fluid mechanics.
In particular, they play a prominent role in cosmology, 
where they are often used to model the evolution of the average matter-energy content of the universe;
see, for example, Weinberg's well-known monograph \cite{sW2008} for an account
of the role that the relativistic Euler equations play in the standard model of cosmology.
The equations are also widely used in astrophysics and high-energy nuclear physics, as is 
described, for example, in \cite{lRoZ2013}.
Our main result in this article is our derivation of a new formulation of the 
relativistic Euler equations that reveals remarkable new regularity and null structures that are not visible relative
to standard first-order formulations. 
The new formulation is available for an arbitrary equation of state, 
not necessarily of barotropic\footnote{Barotropic equations of state are such that the pressure is a function of the proper energy density $\uprho$ alone.} 
type. Below we will describe potential applications that we anticipate will be the subject of future works.
We mention already that our new formulation of the equations provides a viable framework for the rigorous mathematical
study of stable shock formation without symmetry assumptions in solutions to the
relativistic Euler equations; for reasons to be explained, standard first-order formulations
are not adequate for tracking the behavior of solutions (without symmetry assumptions) 
all the way to the formation of a shock or for extending the solution
(uniquely, in a weak sense tied to suitable selection criteria) 
past the first singularity. 

We derive the new formulation by differentiating 
a standard first-order formulation with various geometric
differential operators and observing remarkable cancellations.\footnote{In observing many of the cancellations,
the precise numerical coefficients in the equations are important; roughly, these cancellations lead to the presence of the 
null form structures described below. However, for most applications, the overall coefficient of the null forms is not important;
what matters is that the cancellations lead to null forms.}
The calculations are rather involved and make up the bulk of the article.
We have carefully divided them into manageable pieces; see Sects.\,\ref{S:PRELIMINID}-\ref{S:TRANSPORTDIVCURLFORVORTICITY}.
Readers can jump ahead to Theorem~\ref{T:MAINTHEROEMSCHEMATIC} for a rough statement
of the equations and Theorem~\ref{T:NEWFORMULATIONRELEULER} for the precise version.

As we alluded to above, 
the relativistic Euler equations 
are typically formulated as a first-order quasilinear hyperbolic PDE system. 
In our new formulation, the equations take the form
of a system of covariant wave equations coupled to transport equations and to two $\mbox{\upshape transport}$-$\mbox{\upshape div}$-$\mbox{\upshape curl}$ systems.
The new formulation is well-suited for various applications
in ways that first-order formulations are not. 
In particular, the equations of Theorem~\ref{T:NEWFORMULATIONRELEULER} 
can be used to prove that \emph{the vorticity and entropy are one degree more differentiable than one might naively expect}
(assuming that the gain in differentiability is present in the initial data).
This gain in differentiability is crucial for the rigorous mathematical study of some fundamental phenomena that occur 
in fluid dynamics. In particular, this gain, as well as other structural aspects of the new formulation,
is essential for the study of shock waves (without symmetry assumptions) in relativistic fluid mechanics;
see Subsect.\,\ref{SS:SHOCKWAVES} for further discussion.
Although the gain in differentiability for the vorticity had previously been observed relative to Lagrangian coordinates
\cites{mSsSjS2015,jJpLnM2016},
Lagrangian coordinates are inadequate, for example, for the study of the formation of shock singularities 
because they are not adapted to the acoustic
characteristics, whose intersection corresponds to a shock. Hence,
it is of fundamental importance that our new formulation allows one to prove the gain in differentiability relative
to \emph{arbitrary vectorfield differential operators} (with suitably regular coefficients).
In this vein, we also mention the works \cites{dChLsS2010,dCsS2011,dCsS2012} on the non-relativistic compressible Euler equations, 
in which a gain in differentiability for the vorticity was shown relative to Lagrangian coordinates,
and the first author's joint work \cite{mDdE2017}, in which elliptic estimates were used to show that for the non-relativistic 
barotropic compressible Euler equations, it is possible to gain one derivative on the density relative to the velocity
(again, assuming that the gain is present in the initial data).

We also highlight the following key advantage of our new formulation: 
\begin{quote}
It dramatically enlarges the set of energy estimate techniques 
that can be applied to the study of the relativistic Euler equations. More precisely, 
the new formulation partially decouples the ``wave parts'' and ``transport parts'' of the system
and unlocks our ability to apply the full power of the commutator and multiplier vectorfield methods to the study of the
wave part; see Subsect.\,\ref{SS:ENERGIESFORWAVE} for further discussion.
\end{quote}
For applications to shock waves, 
it is fundamentally important that one is able to use the full scope of the vectorfield
method on the wave part of the system; see the introduction of \cite{jLjS2018}
for a discussion of this issue in the related context of the non-relativistic barotropic compressible Euler equations with vorticity. 
In particular, our new formulation of the equations allows one to derive a coercive energy estimate
for the wave part of the system for any multiplier vectorfield that is causal relative to the
acoustical metric $g$ of Def.\,\ref{D:ACOUSTICALMETRICANDINVERSE} and on any hypersurface that is 
null or spacelike relative to $g$; see Subsubsect.\,\ref{SSS:ESTIMATESFORCOVARIANTWAVEEQUATIONS}  for further discussion.
In contrast, for first-order hyperbolic systems 
(a special case of which is the relativistic Euler equations)
without additional structure, there is, up to scalar function multiple, 
only one\footnote{Here we further explain how standard first-order formulations of 
the relativistic Euler equations limit the available energy estimates. 
In deriving energy estimates for
the relativistic Euler equations in their standard first-order form, 
one is effectively controlling the wave and transport parts of the system at the same time,
and, up to a scalar function multiple, there is only one energy estimate available for transport equations.
To see this limitation in a more concrete fashion, one can rewrite the relativistic Euler equations 
in first-order symmetric hyperbolic form
as $A^{\alpha}(\mathbf{V}) \partial_{\alpha} \mathbf{V} = 0$,
where $\mathbf{V}$ is the array of solution variables and the $A^{\alpha}$ are symmetric matrices
with $A^0$ positive definite; see, for example,
\cite{aR1992a} for a symmetric hyperbolic formulation of the
general relativistic Euler equations in the barotropic case. 
The standard energy estimate for symmetric hyperbolic systems is obtained by
taking the Euclidean dot product of both sides of the equation with $\mathbf{V}$ and then integrating by parts
over an appropriate spacetime domain foliated by spacelike hypersurfaces. 
The key point is that for systems without additional structure, no other energy estimate is known, 
aside from rescaling the standard one by a scalar function.} 
available energy estimate on each causal or spacelike hypersurface.
 
Our second result in this article is that we provide a proof of local well-posedness 
for the relativistic Euler equations that relies on the new formulation; 
see Theorem~\ref{T:UPGRADEDLOCALWELLPOSEDNESS}. The new feature of Theorem~\ref{T:UPGRADEDLOCALWELLPOSEDNESS} 
compared to standard proofs of local well-posedness for the relativistic Euler equations
is that it provides the aforementioned gain in differentiability for the vorticity and entropy.
Although many aspects of the proof of the theorem are standard, 
we also rely on some geometric and analytic insights that are tied 
to the special structure of our new formulation of the equations
and thus are likely not known to the broader PDE research community; 
see the end of Subsubsect.\,\ref{SSS:REGULARITYANDSINGULARENERGYESTIMATES} for further discussion of this point.

For convenience, throughout the article, we restrict our attention to the 
special relativistic Euler equations, that is, 
the relativistic Euler equations on the Minkowski spacetime background
$(\mathbb{R}^{1+3},\upeta)$,
where $\upeta$ is the Minkowski metric. However, using arguments similar
to the ones given in the present article, our results could be extended 
to apply to the relativistic Euler equations on a general Lorentzian manifold;
such an extension could be useful, for example, in applications to
fluid mechanics in the setting of general relativity. 
For use throughout the article, 
we fix a standard rectangular coordinate system
$\lbrace x^{\alpha} \rbrace_{\alpha = 0,1,2,3}$, 
relative to which $\upeta_{\alpha \beta} := \mbox{\upshape diag}(-1,1,1,1)$.
See Subsect.\,\ref{SS:NOTATIONANDCONVENTIONS} for our index conventions.
We clarify that in Sect.\,\ref{S:WELLPOSEDNESS}, we prove local well-posedness
for the relativistic Euler equations (including the aforementioned gain in regularity for the vorticity and entropy) 
on the flat spacetime background $(\mathbb{R} \times \mathbb{T}^3,\upeta)$,
where the ``spatial manifold'' $\mathbb{T}^3$ is the three-dimensional torus and
we recycle the notation in the sense that 
$\lbrace x^{\alpha} \rbrace_{\alpha = 0,1,2,3}$ denotes standard coordinates on $\mathbb{R} \times \mathbb{T}^3$
(see Subsubsect.\,\ref{SSS:NOTATION} for further discussion)
and $\upeta$ again denotes the Minkowski metric;
the compactness of $\mathbb{T}^3$ allows for a simplified approach to some technical aspects
of the argument while allowing us to illustrate the ideas needed to exhibit the gain in regularity
for the vorticity and entropy.

Our work here can be viewed as extensions of the second author's previous joint work
\cite{jLjS2016a}, in which the authors derived a similar formulation of the
non-relativistic compressible Euler equations under an arbitrary barotropic equation of state,
as well as the second author's work \cite{jS2017a}, 
which extended the results of \cite{jLjS2016a}
to a general equation of state. However, since the geo-analytic structures
revealed by \cites{jLjS2016a,jS2017a} are rather delicate 
(that is, quite unstable under perturbations of the equations), it is 
far from obvious that similar results hold in the relativistic case.
We also stress that compared to the non-relativistic case, 
our work here is substantially more intricate in that it extensively 
relies on decompositions of various spacetime tensors into tensors that are
parallel to the four-velocity $u$ and tensors that are 
$\upeta$-orthogonal to $u$. In particular, we heavily exploit that
many of the tensorfields appearing in our analysis
exhibit improved regularity under $u$-directional differentiation
\emph{or} contraction against $u$.

\subsection{Rough statement of the new formulation}
\label{SS:ROUGHSTATEMENT}
In this subsection, we provide a schematic version of our
new formulation of the equations; 
in Subsect.\,\ref{SS:SHOCKWAVES}, we will refer
to the schematic version when describing potential applications.
In any formulation of the relativistic Euler equations, 
there is great freedom in choosing state-space variables 
(i.e., the fundamental unknowns in the system).
In this article, as state space variables, we use
the logarithmic enthalpy $\Lnenth$,
the entropy $s$, and the four-velocity $u$, 
which is a future-directed timelike vectorfield normalized by 
$\upeta_{\alpha \beta} u^{\alpha} u^{\beta} = -1$. 
Other fluid quantities such as the proper energy density $\uprho$, the pressure $p$, etc.\
will also play a role in our discussion, but these quantities
can be viewed as functions of the state space variables;
see see Sect.\,\ref{S:SOMEBACKGROUND} for detailed descriptions 
of all of these variables as well as the first-order formulation of the equations that forms
the starting point for our ensuing analysis.

As we mentioned earlier, our new formulation comprises a system 
of covariant wave equations coupled to transport equations and to two $\mbox{\upshape transport}$-$\mbox{\upshape div}$-$\mbox{\upshape curl}$ systems.
Roughly, the wave equations correspond to the propagation of sound waves, while the transport
equations correspond to the transporting of vorticity and entropy along the integral curves of $u$.
The $\mbox{\upshape transport}$-$\mbox{\upshape div}$-$\mbox{\upshape curl}$ systems are needed to control the top-order derivatives of the vorticity and the entropy
and to exhibit the aforementioned gain in differentiability.
In addition to the state space variables $\Lnenth$, $\Ent$, and $u$,
our formulation also involves a collection of auxiliary\footnote{By ``auxiliary,'' we mean that they are determined by 
$\Lnenth$, $\Ent$, and $u$.} fluid variables, 
including the entropy gradient one-form $\GradEnt_{\alpha} := \partial_{\alpha} \Ent$
and the vorticity $\vort^{\alpha}$, which is a vectorfield that is $\upeta$-orthogonal to $u$
(see Def.\,\ref{D:VORTICITYDEF}).
Among these auxiliary variables, of crucial importance
for our work is that we have identified new combinations 
of fluid variables that solve transport equations
with unexpectedly good structure. These structures can be used to show that the 
combinations exhibit a gain in regularity compared
to what can be inferred from a standard first-order formulation of the equations.
We refer to these special combinations as ``modified variables,'' and throughout,
we denote them by $\mathcal{C}^{\alpha}$ and $\mathcal{D}$; 
see Def.\,\ref{D:MODIFIEDVARIABLES}.

The remaining discussion in this subsection relies on some schematic notation and refers to some geometric objects
that are not precisely defined until later in the article:
\begin{itemize}
	\item The notation ``$\sim$'' below means that we are only highlighting the
		maximum number of derivatives of the state-space variables that the auxiliary variables
		depend on. We note, however, that in practice, the precise structure of many of 
		the terms that we encounter is important for observing the cancellations that lie behind our main results.
	\item ``$\partial$'' schematically denotes the spacetime gradient with respect to the rectangular coordinates,
		and ``$\partial^2$'' schematically denotes two differentiations with respect to the rectangular coordinates.
	\item $g = g(\Lnenth,\Ent,u)$ denotes the acoustical metric, which is Lorentzian 
		(see Def.\,\ref{D:ACOUSTICALMETRICANDINVERSE}).
	\item $\vort \sim \partial u + \partial \Lnenth$ is the vorticity vectorfield
		(see Def.\,\ref{D:VORTICITYDEF}).
	\item $\GradEnt_{\alpha} := \partial_{\alpha} \Ent$ is the entropy gradient one-form.
		\item $\mathcal{C}^{\alpha} \sim \partial^2 u + \partial^2 \Lnenth$ is a modified version of the vorticity of $\vort$,
		that is, the vorticity of the vorticity (see Def.\,\ref{D:MODIFIEDVARIABLES}).
	\item $\mathcal{D} \sim \partial^2 \Ent$ is a modified version of $\partial_{\alpha} \GradEnt^{\alpha}$
		(see Def.\,\ref{D:MODIFIEDVARIABLES}).
	\item $\mathfrak{Q}(\partial T_1,\cdots,\partial T_m)$ 
		denotes \emph{special} terms that are quadratic in the
		tensorfields $\partial T_1,\cdots,\partial T_m$.
		More precisely, the $\mathfrak{Q}(\partial T_1,\cdots,\partial T_m)$ are
		linear combinations of the standard \textbf{null forms relative to $g$};
 		see Def.\,\ref{D:STANDARDNULLFORMS}
		for the definitions of the standard null forms relative to $g$
		and Subsubsect.\,\ref{SSS:NONLINEARNULL} for a discussion of the significance 
		that the special structure of these null forms plays in the context of the study of shock waves.
	\item $\mathfrak{L}(\partial T_1,\cdots,\partial T_m)$ denotes
		linear combinations of terms that are at most \textbf{linear} in $\partial T_1,\cdots,\partial T_m$;
		see Subsubsect.\,\ref{SSS:NONLINEARNULL} for a discussion of 
		the significance of the linear dependence in the context of the study of shock waves.
\end{itemize}

Before schematically stating our main theorem, we first provide the definitions of the
standard null forms relative to $g$.
\begin{definition}[Standard null forms relative to $g$]
\label{D:STANDARDNULLFORMS}
We define the standard null forms relative to $g$ (which we refer to as ``standard $g$-null forms'' for short)
as follows, where $\phi$ and $\psi$ are scalar functions and $0 \leq \mu < \nu \leq 3$:
\begin{align} \label{E:STANDARDNULLFORMS}
	\mathfrak{Q}^{(g)}(\partial \phi, \partial \psi)
	& := (g^{-1})^{\alpha \beta} (\partial_{\alpha} \phi) (\partial_{\beta} \psi),
	\\
		\mathfrak{Q}_{\mu \nu}(\partial \phi, \partial \psi)
	& := (\partial_{\mu} \phi) (\partial_{\nu} \psi) 
			- 
			(\partial_{\nu} \phi) (\partial_{\mu} \psi).
	\notag
\end{align}
\end{definition}

We now present the schematic version of our main theorem; see Theorem~\ref{T:NEWFORMULATIONRELEULER}
for the precise statements.

\begin{theorem}[New formulation of the relativistic Euler equation (schematic version)]
	\label{T:MAINTHEROEMSCHEMATIC}
	Assume that $(\Lnenth,\Ent,u^{\alpha})$ is
	a $C^3$ solution to the (first-order) relativistic Euler equations 
	\eqref{E:ENTHALPYEVOLUTION}-\eqref{E:ENTROPYEVOLUTION} + \eqref{E:UISUNITLENGTH}.
	Then $\Lnenth$, $u^{\alpha}$, and $\Ent$ also verify the following covariant\footnote{Relative to arbitrary coordinates,
	for scalar functions $f$, we have
	\begin{align}
	\square_g f= 
\frac{1}{\mbox{$\sqrt{|\mbox{\upshape det} g|}$}}
\partial_{\alpha}\left(\sqrt{|\mbox{\upshape det} g|} (g^{-1})^{\alpha \beta} \partial_{\beta} f \right).
\notag
	\end{align}
\label{FN:COVWAVEOPARBITRARYCOORDS}}  
	wave equations,
	where the schematic notation ``$\simeq$'' below means
	that we have ignored the coefficients of the inhomogeneous terms
	and also harmless (from the point of view of applications to shock waves) 
	lower-order terms, which are allowed to depend on 
	$\Lnenth$, $\Ent$, $u$, $\GradEnt$, and $\vort$
	(but not their derivatives):
	\begin{subequations}
	\begin{align}
	\square_g \Lnenth
	& \simeq 
		\mathcal{D} 
		+
		\mathfrak{Q}(\partial \Lnenth,\partial u)
		+
		\mathfrak{L}(\partial \Lnenth),
			\label{E:INTROCOVWAVELNENTHALPY} \\
	\square_g u^{\alpha}
	& 
		\simeq
		\mathcal{C}^{\alpha}
		+
		\mathfrak{Q}(\partial \Lnenth,\partial u)
		+
		\mathfrak{L}(\partial \Lnenth,\partial u),
		\label{E:INTROCOVWAVEVELOCITY}
			\\
	\square_g \Ent
	& 
		\simeq
		\mathcal{D}
		+
		\mathfrak{L}(\partial \Lnenth).
		\label{E:INTROCOVWAVEENT}
	\end{align}
	\end{subequations}
	
In addition, $\Ent$, $\GradEnt^{\alpha}$, and $\vort^{\alpha}$ verify the
following transport equations:
	\begin{subequations}
	\begin{align}
		u^{\kappa} \partial_{\kappa} \Ent 
		& = 0,
			\label{E:INTROENTTRANSPORT} \\
		u^{\kappa} \partial_{\kappa} \GradEnt^{\alpha}
		& \simeq \mathfrak{L}(\partial u),
			\label{E:INTROGRADENTTRANSPORT} \\
		u^{\kappa} \partial_{\kappa} \vort^{\alpha} 
		& \simeq \mathfrak{L}(\partial \Lnenth,\partial u).
		\label{E:INTROVORTICITYTRANSPORT}
	\end{align}
	\end{subequations}
	
	Moreover, $\GradEnt^{\alpha}$ verifies the following $\mbox{\upshape transport}$-$\mbox{\upshape div}$-$\mbox{\upshape curl}$ system:
	\begin{subequations}
	\begin{align}
		u^{\kappa} \partial_{\kappa} \mathcal{D}
		& \simeq 
				\mathcal{C}
				+
				\mathfrak{Q}(\partial \GradEnt,\partial \Lnenth,\partial u) 
				+
				\mathfrak{L}(\partial \Lnenth,\partial u),
			\label{E:INTROTRANSPORTOFMODIFIEDENTROPYDIVERGENCE} \\
		\uperpvort^{\alpha}(\GradEnt)
		& = 0,
		\label{E:INTROVORTICITYOFENTROPYGRADIENT}
	\end{align}
	\end{subequations}
	where the vorticity operator $\uperpvort$ is defined in Def.\,\ref{D:VORTICITYOFAONEFORM}.
	
	Finally, $\vort^{\alpha}$ verifies the following $\mbox{\upshape transport}$-$\mbox{\upshape div}$-$\mbox{\upshape curl}$ system:
	\begin{subequations}
	\begin{align}
		\partial_{\kappa} \vort^{\kappa}
		& \simeq \mathfrak{L}(\partial \Lnenth),
			\label{E:INTRODIVERGENCEOFVORTICITY} \\
		u^{\kappa} \partial_{\kappa} \mathcal{C}^{\alpha}
		& \simeq \mathcal{C}
				+
				\mathcal{D}
				+
				\mathfrak{Q}(\partial \GradEnt,\partial \vort,\partial \Lnenth,\partial u) 
				+
				\mathfrak{L}(\partial \GradEnt,\partial \vort,\partial \Lnenth,\partial u).
			\label{E:INTROTRANSPORTVORTICITYOFVORTICITY} 
	\end{align}
	\end{subequations}
\end{theorem}

\subsection{Connections to the study of shock waves}
\label{SS:SHOCKWAVES}
As we have mentioned, the relativistic Euler equations are an example of a quasilinear hyperbolic PDE system.
A central feature of the study of such systems is that initially smooth solutions can form shock singularities in finite time. 
By a ``shock,'' we roughly mean that one of the solution's partial derivatives with respect to the standard coordinates
blows up in finite time while the solution itself remains bounded.
In the last decade, for interesting classes of quasilinear hyperbolic PDEs in multiple spatial dimensions,
there has been dramatic progress \cites{dC2007,jS2016b,dCsM2014,jSgHjLwW2016,sM2018,sMpY2017,jLjS2018,jS2017c,js2018b}
on our understanding of the formation of shocks 
as well as our understanding of
the subsequent behavior of solutions past their singularities
\cites{dCaL2016,dC2017}
(where the equations are verified in a weak sense past singularities). 

The works cited above have roots in the work of John \cite{fJ1974} on singularity formation for quasilinear
wave equations in one spatial dimension
as well as Alinhac's foundational works \cites{sA1999a,sA2001b}, which were the first
to provide a constructive description of shock formation for quasilinear wave equations
in more than one spatial dimension without symmetry assumptions.
More precisely, Alinhac's approach allowed him to follow the solution precisely
to the time of first blowup, but not further. His work yielded sharp information
about the first singularity, but only for a subset of ``non-degenerate'' initial data such that 
the solution's first singularity is isolated in the constant-time hypersurface of first blowup;
in particular, his proof did not apply to spherically symmetric initial data, where
the ``first'' singularity typically corresponds to blowup on a sphere.

Subsequently, Christodoulou \cite{dC2007} proved a breakthrough result
on the formation of shocks for solutions to the relativistic Euler equations
in irrotational (that is, vorticity-free) and isentropic regions of spacetime.
More precisely, for the family of quasilinear wave equations that arise in the study 
of the irrotational and isentropic 
relativistic Euler equations,\footnote{For solutions with vanishing vorticity and constant entropy,
one can introduce a potential function $\Phi$ and reformulate the relativistic Euler equations
as a quasilinear wave equation in $\Phi$.}
Christodoulou gave a complete description of the maximal development 
of an open set (without symmetry assumptions) of initial data
and showed in particular that an open subset of these data lead to shock-forming\footnote{One of the key results of 
\cite{dC2007} is conditional: for small data, the only possible singularities that can form are
shocks driven by the intersection of the acoustic characteristics. Here ``small'' means a small perturbation of the 
data of a non-vacuum constant fluid state, where the size of the perturbation is measured relative to a high-order Sobolev norm. 
Another result of \cite{dC2007}
is that there is an open subset of small data, perhaps strictly contained in the aforementioned set of data, 
such that the acoustic characteristics do in fact intersect in 
finite time. The results of \cite{dC2007} leave open the possibility 
that there might exist some non-trivial small global solutions.} 
solutions.
Moreover, he gave a precise geometric description of the set of 
spacetime points where blowup occurs by showing that
the singularity formation is exactly characterized by the
intersection of the acoustic characteristics. In practice, 
he accomplished this by constructing an acoustical \emph{eikonal function} $U$,
whose level sets are acoustic characteristics (see Subsubsect.\,\ref{SSS:NONLINEARGEOMETRICOPTICS} for further discussion),
and then constructing an initially positive geometric scalar function $\upmu \sim 1/\partial U$
known as the \emph{inverse foliation density} of the characteristics,
such that $\upmu \to 0$ corresponds to the intersection of the characteristics
and the blowup of $\partial U$ and of the fluid solution's derivatives too. 
Analytically, $\upmu$ plays the role of a weight that appears throughout the work \cite{dC2007}, and the main theme of
the proof is to control the solution all the way up to the region where $\upmu = 0$.
We stress that \cite{dC2007} was the first work that provided sharp information about 
the boundary of the maximal development in more than one spatial dimension
in the context of shock formation. Roughly, the maximal development
is the largest possible classical solution that is uniquely determined by the initial data;
see \cites{jSb2016,wW2013} for further discussion. 

To prove his results, Christodoulou relied on a novel formulation of the 
relativistic Euler equations. However, since he studied the shock formation 
only in irrotational and isentropic  
regions, he was able to introduce a potential function $\Phi$, 
and his new formulation of the equations was drastically simpler
than the equations of Theorem~\ref{T:MAINTHEROEMSCHEMATIC}. In fact,
the equations are exactly the covariant wave equation system 
$\square_{\widetilde{g}} \partial_{\alpha} \Phi = 0$ (with $\alpha = 0,1,2,3$),
where $\widetilde{g}$ is an appropriate scalar function multiple of the acoustical metric $g$ and
$\widetilde{g} = \widetilde{g}(\partial \Phi)$. 
In particular, Christodoulou was able to avoid deriving/relying on the $\mbox{\upshape transport}$-$\mbox{\upshape div}$-$\mbox{\upshape curl}$ equations
from Theorem~\ref{T:MAINTHEROEMSCHEMATIC}, and he therefore did not need to derive elliptic
estimates for the fluid variables. 
In total, the potential formulation leads to dramatic simplifications compared to the
equations of Theorem~\ref{T:MAINTHEROEMSCHEMATIC},
especially in the context of the study of shock waves;
it seems quite miraculous that the equations of
Theorem~\ref{T:MAINTHEROEMSCHEMATIC} have structures that are compatible
with extending Christodoulou's results away from the irrotational and isentropic case
(see below for further discussion).

Although the sharp information that Christodoulou derived about the maximal
development is of interest in itself, it is also an essential ingredient for setting up the
shock development problem. The shock development problem, which was recently partially\footnote{In \cite{dC2017},
Christodoulou solved the ``restricted'' shock development problem, in which he ignored the jump in entropy and 
vorticity across the shock hypersurface.} 
solved in the breakthrough work \cite{dC2017} 
(see also the precursor work \cite{dCaL2016} in spherical symmetry),
is the problem of constructing the shock hypersurface of discontinuity
(across which the solution jumps) 
as well as constructing a unique weak solution 
in a neighborhood of the shock hypersurface
(uniqueness is enforced by selection criteria that are equivalent to the well-known Rankine--Hugoniot conditions).
Christodoulou's description of the maximal development provided substantial new information 
that was not available under Alinhac's approach;
as we mentioned above,
due to some technical limitations tied to his reliance on Nash--Moser estimates,
Alinhac was able to follow the solution only to the constant-time hypersurface of first blowup.
In contrast, by exploiting some delicate tensorial regularity properties of
eikonal functions for wave equations (see below for more details),
Christodoulou was able to avoid Nash--Moser estimates; this was a key ingredient in
his following the solution to the boundary of the maximal development. 
Readers can consult \cite{gHsKjSwW2016} for a survey of some of these works, 
with a focus on the geometric and analytic
techniques that lie behind the proofs.

We now aim to connect the works mentioned above to the new formulation of the relativistic Euler equations
that we provide in this paper. To this end, for the equations in the works mentioned above,
we first highlight the main structural features that allowed the proofs to go through.
Specifically, the works \cites{dC2007,jS2016b,dCsM2014,jSgHjLwW2016,sM2018,sMpY2017,jLjS2018,jS2017c,js2018b} 
crucially relied on the following ingredients:
\begin{enumerate}
	\item (\textbf{Nonlinear geometric optics}). The authors relied on 
		geometric decompositions adapted to the characteristic hypersurfaces
		(also known as ``characteristics'' or ``null hypersurfaces'' in the context of wave equations)
		corresponding to the solution variable whose derivatives blow up. 
		This was implemented with the help of an 
		\emph{eikonal function} $U$, whose level sets are characteristics.
		The eikonal function is a solution to the \emph{eikonal equation}, which
		is a fully nonlinear transport equation that is coupled to the solution
		in the sense that the coefficients of the eikonal equation depend on the solution.
		Moreover, the authors showed that the intersection of the characteristics
		corresponds to the formation of a singularity in the derivatives of the eikonal function
		and in the derivatives of the solution.
	\item (\textbf{Quasilinear null structure}). The authors found a formulation of the equations
		exhibiting remarkable null structures, where the notion of ``null'' is tied to the true characteristics,
		which are solution-dependent in view of the quasilinear nature of the equations.
		These structures allow one to derive sharp, fully nonlinear decompositions
		along characteristic hypersurfaces that reveal exactly which directional derivatives blow up
		and that precisely identify the terms driving the blowup 
		(which are typically of Riccati-type, i.e., in analogy with the nonlinearities in the ODE $\dot{y} = y^2$).
	\item (\textbf{Regularity properties and singular high-order energy estimates}). The authors' formulation allows one to derive
		sufficient $L^2$-type Sobolev regularity for all unknowns in the problem, including the eikonal function,
		\emph{whose regularity properties are tied to the regularity of the solution through the dependence of 
		the coefficients of the eikonal equation on the solution}.
		In particular, to close these estimates, 
		\emph{the authors had to show that various solution variables are one degree more differentiable
			compared to the degree of differentiability guaranteed by standard energy estimates.}
	\item (\textbf{Structures amenable to commutations with geometric vectorfields}). 
		The authors' formulation is such that one can commute \emph{all} of the equations with geometric vectorfields
		constructed out of the eikonal function $U$, generating only controllable commutator error terms.
		By  ``controllable,'' we mean both from the point of view of regularity and from the point of view of 
		the strength of their singular nature. 
		In the works \cites{jLjS2018,jS2017c,js2018b} that treat systems with multiple characteristic speeds,
		these are particularly delicate tasks that are quite sensitive to the structure of the equations;
		one key reason behind their delicate nature is that the eikonal function 
		(and thus the geometric vectorfields constructed from it)
		can be fully adapted only to ``one speed,'' that is, to the characteristics whose intersection
		correspond to the singularity.
	\end{enumerate}

In the remainder of this subsection, 
we explain why our new formulation of the relativistic Euler equations
has all four of the features listed above
and is therefore well-suited for studying shocks without symmetry assumptions.
Readers can consult the works \cites{jLjS2016a,jS2017a,jS2018c} for 
related but extended discussion in the case of the non-relativistic compressible Euler equations.

\subsubsection{Nonlinear geometric optics and geometric coordinates}
\label{SSS:NONLINEARGEOMETRICOPTICS}
First, to implement nonlinear geometric optics, one can construct an eikonal function. 
In the context of the relativistic Euler equations, 
one would construct an eikonal function $U$ adapted to the acoustic characteristics, 
that is, a solution to the eikonal equation 
\begin{align} \label{E:EIKONAL}
	(g^{-1})^{\alpha \beta} \partial_{\alpha} U \partial_{\beta} U = 0,
\end{align}
supplemented by appropriate initial conditions,
where $g= g(\Lnenth,\Ent,u)$ is the acoustical metric (see Def.\,\ref{D:ACOUSTICALMETRICANDINVERSE}).
Note that $U$ is adapted to the ``wave part'' of the system and not the transport part.
In the context of the relativistic Euler equations, this is reasonable in the sense that the 
transport part corresponds to the evolution of vorticity and entropy, and there are no known blowup
results for these quantities, even in one spatial dimension.\footnote{In one spatial dimension, the
vorticity must vanish, but the entropy can be dynamic.}
Put differently, $U$ is adapted to the ``portion'' of the relativistic Euler flow that
is expected to develop singularities. 
More generally, eikonal functions are a natural tool 
for the study of wave-like systems, regardless of whether or not one is studying shocks.
We also stress that introducing an eikonal function
is essentially the same as relying on
the method of characteristics. However, in more than one spatial dimension,
the method of characteristics must be supplemented with an exceptionally technical ingredient that we further describe below: 
energy estimates that hold all the way up to the shock.

The first instance of an eikonal function being used to study the global properties of solutions 
to a quasilinear hyperbolic PDE occurred not in the context of singularity formation, but rather in a celebrated
global existence result: the Christodoulou--Klainerman \cite{dCsK1993} proof
of the stability of the Minkowski spacetime as a solution to the Einstein-vacuum equations.
Alinhac's aforementioned works \cites{sA1999a,sA2001b} 
were the first instances in which an eikonal function
was used to study a non-trivial set 
of solutions (without symmetry assumptions)
to a quasilinear wave equation all the way up to the first singularity. 
Eikonal functions also played a fundamental role in all of the other shock formation results mentioned above.
They have also played a role in other contexts, such as low-regularity local
well-posedness for quasilinear wave equations \cites{sKiR2003,sKiRjS2015,hSdT2005,qW2017}.
In all of these works, the eikonal equation is a fully nonlinear hyperbolic PDE 
that is coupled to the PDE system of interest (here the relativistic Euler equations) 
through its coefficients (here through the acoustical metric, since $g= g(\Lnenth,\Ent,u)$). 
As we mentioned above,
in the case of the relativistic Euler equations, 
the level sets of $U$ are characteristics for the ``wave part'' of the system.
Following Alinhac \cites{sA1999a,sA2001b} and Christodoulou \cite{dC2007}, 
in order to study the formation of shocks in relativistic Euler solutions, 
one completes $U$ to a \emph{geometric coordinate system} 
\begin{align} \label{E:GEOMETRICCOORDINATES}
	(t,U,\vartheta^1,\vartheta^2) 
\end{align}
on spacetime,
where $t = x^0$ is the Minkowski time coordinate and the $\vartheta^A$ are solutions to the transport equation
$(g^{-1})^{\alpha \beta} \partial_{\alpha} U \partial_{\beta} \vartheta^A = 0$
supplemented by appropriate initial conditions on the initial constant-time hypersurface $\Sigma_0$. 
Note that $(t,\vartheta^1,\vartheta^2)$ can be viewed as a coordinate system
along each characteristic 
hypersurface $\lbrace U = \mbox{\upshape const} \rbrace$.

\subsubsection{Nonlinear null structure}
\label{SSS:NONLINEARNULL}
We now aim to explain the role that the nonlinear null structure of the equations
played in the works \cites{dC2007,jS2016b,dCsM2014,jSgHjLwW2016,sM2018,sMpY2017,jLjS2018,jS2017c,js2018b}
and to explain why the equations of Theorem~\ref{T:MAINTHEROEMSCHEMATIC} enjoy the same good structures.
In total, one could say that the equations of Theorem~\ref{T:MAINTHEROEMSCHEMATIC}
have been geometrically decomposed into terms
that are capable of generating shocks and ``harmless'' terms,
whose nonlinear structure is such that they do not interfere with the shock formation mechanisms.
To flesh out these notions, we first provide some background material. In the works cited above,
the main idea behind proving shock formation
is to study the solution relative to the geometric coordinates \eqref{E:GEOMETRICCOORDINATES}
and to show that in fact, the solution remains rather smooth in these coordinates,
all the way up to the shock. This approach allows one to transform the problem of shock formation
into a more traditional one in which one tries to derive long-time estimates for the solution
relative to the geometric coordinates. 
One then recovers the blowup of the solution's derivatives with respect to the \emph{original} coordinates 
by showing that the geometric coordinates
degenerate in a precise fashion relative to the standard rectangular coordinates as the shock forms;
the degeneration is exactly tied to the vanishing of the inverse foliation density $\upmu$ that we mentioned earlier.
Although the above description might seem compellingly simple,
as we explain in Subsubsect.\,\ref{SSS:REGULARITYANDSINGULARENERGYESTIMATES},
in implementing this approach, one encounters severe analytical difficulties.

We now highlight another key aspect of the proofs in the works cited above: showing that Euclidean-unit-length 
derivatives of the solution in directions \emph{tangent} to the characteristics remain bounded all the way up 
to the shock. It turns out that in terms of the geometric coordinates \eqref{E:GEOMETRICCOORDINATES},
		this is equivalent to showing that the
		$
		\frac{\partial}{\partial t}
		$
		and	
		$
		\frac{\partial}{\partial \vartheta^A }
		$ 
		derivatives of the solution remain bounded all the way up to the shock.
	Put differently, the following holds: 
	\begin{quote}
		The singularity occurs only for derivatives of the solution with respect to vectorfields
		that are transversal to the characteristics
		and non-degenerate\footnote{In all known shock formation results, 
		at the location of shock singularities,
		the geometric partial derivative vectorfield $\frac{\partial}{\partial U}$ 
		has vanishing
		Euclidean length 
		(i.e., $\updelta_{ab} \left(\frac{\partial}{\partial U}\right)^a \left(\frac{\partial}{\partial U}\right)^b = 0$, where
		$\left\lbrace \left(\frac{\partial}{\partial U}\right)^a \right\rbrace_{a=1,2,3}$ denotes the rectangular spatial components of 
		$\frac{\partial}{\partial U}$ and $\updelta_{ab}$ is the Kronecker delta). 
		That is, at the shock singularities,
		$\frac{\partial}{\partial U}$
		degenerates with respect to the rectangular coordinates.
		Due to this degeneracy, the solution's $\frac{\partial}{\partial U}$ derivatives can remain bounded all the way up to the shock, 
		even though 
		$\frac{\partial}{\partial U}$
		is transversal to the characteristics.}
		with respect to the rectangular coordinates.
	\end{quote}
In the works cited above, to prove all of these facts, 
the authors had to control various inhomogeneous error terms 
by showing that they enjoy a good nonlinear
null structure \emph{relative to the wave characteristics}. 
A key conclusion of the present article is that the derivative-quadratic inhomogeneous terms 
in the equations of Theorem~\ref{T:MAINTHEROEMSCHEMATIC}
enjoy the same good structure (which we further describe just below).
In fact, \emph{all terms} on the RHSs of \emph{all equations} 
of Theorem~\ref{T:MAINTHEROEMSCHEMATIC} are harmless in that 
they do not drive the Riccati-type blowup that lies behind shock formation.
Consequently, the equations of Theorem~\ref{T:MAINTHEROEMSCHEMATIC} pinpoint the
dangerous nonlinear terms in the relativistic Euler equations:
 
\begin{quote}
	The terms capable of driving shock formation are of Riccati-type and are
	hidden in the covariant wave operator terms 
	on LHSs~\linebreak\eqref{E:INTROCOVWAVELNENTHALPY}-\eqref{E:INTROCOVWAVEVELOCITY}.
	These terms become visible only when the covariant wave operator terms 
	are expanded relative to the standard coordinates.
\end{quote}

In view of the above remarks, one might wonder why it is important to ``hide'' the
dangerous terms in the covariant wave operator. The answer is that there is an advanced
framework for constructing geometric vectorfields adapted to wave equations,
and \emph{the framework is tailored to covariant wave operators}.\footnote{Roughly, these covariant wave operators
are equivalent to divergence-form wave operators. In this way, one could say that a better theory is available for
divergence-form wave operators than for non-divergence-form wave operators. This reminds one of the
situation in elliptic PDE theory, where better results are known
for elliptic PDEs in divergence form compared to ones in non-divergence form.} 
As we explain later in this subsection, this geometric
framework seems to be essential in more than one spatial dimension,\footnote{In one spatial dimension, one can rely
exclusively on the method of characteristics and thus avoid energy estimates.} 
when one is forced to commute the wave equations with suitable vectorfields and to derive energy estimates.

We now further describe the good structure found
in the terms on the RHSs of the equations of Theorem~\ref{T:MAINTHEROEMSCHEMATIC}.
The good nonlinear ``null structure'' is found precisely in the (quadratic) 
null form terms $\mathfrak{Q}$ appearing on the RHSs of the equations of Theorem~\ref{T:MAINTHEROEMSCHEMATIC}.
More precisely, these $\mathfrak{Q}$ are \emph{null forms relative to the acoustical metric $g$}, 
which means that they are linear combinations 
(with coefficients that are allowed to depend on the solution variables -- but not their derivatives) 
of the standard null forms relative to $g$ (see Def.\,\ref{D:STANDARDNULLFORMS}).
The key property of null forms relative to $g$ is that given \emph{any} hypersurface $\Charhypersurface$ 
that is characteristic relative to $g$ 
(e.g., any level set of any eikonal function $U$ that solves equation \eqref{E:EIKONAL}), 
we have the following well-known schematic decomposition:
\begin{align} \label{E:NULLFORMSCHEMATICDECOMP}
	\mathfrak{Q}(\partial \phi, \partial \psi) 
	& = \mathscr{T} \phi \cdot \partial \psi + \mathscr{T} \psi \cdot \partial \phi,
\end{align}
where $\mathscr{T}$ denotes a differentiation in a direction \emph{tangent} to $\Charhypersurface$
and $\partial$ denotes a generic directional derivative;
see, for example, \cite{jLjS2016a} for a standard proof of \eqref{E:NULLFORMSCHEMATICDECOMP}.
Equation \eqref{E:NULLFORMSCHEMATICDECOMP} implies that even though $\mathfrak{Q}$ is quadratic, 
it \emph{never involves two differentiations in directions transversal to any characteristic}.
Since, in all known proofs, it is precisely the transversal derivatives that blow up when a shock forms
(since the Riccati-type terms that drive the blowup are precisely quadratic in the transversal derivatives),
we see that $g$-null forms are linear in the tensorial component of the solution that blows up.
This can be viewed as \emph{the absence of the worst possible combinations of terms in} $\mathfrak{Q}$.
In terms of the geometric coordinates \eqref{E:GEOMETRICCOORDINATES},
null forms do not contain any ``dangerous'' terms proportional to
$	
	\frac{\partial}{\partial U} \phi \cdot \frac{\partial}{\partial U} \psi
$.
We also note that, obviously, the terms $\mathfrak{L}$ from Theorem~\ref{T:MAINTHEROEMSCHEMATIC}
cannot contain any dangerous quadratic terms since they are linear in the solution's derivatives.
In contrast, upon expanding the covariant wave operator terms on 
LHSs~\eqref{E:INTROCOVWAVELNENTHALPY}-\eqref{E:INTROCOVWAVEVELOCITY}
relative to the standard coordinates, one typically encounters terms that
are quadratic in derivatives of $\Lnenth$ and $u$ that are transversal to the characteristics;
as we highlighted above, it is precisely such ``Riccati-type'' terms that can drive the formation of a shock.
We stress that near a shock, such transversal-derivative-quadratic terms are much larger
than the null form terms.
We also stress that for the relativistic Euler equations,
one encounters such transversal-derivative-quadratic terms
on LHSs \eqref{E:INTROCOVWAVELNENTHALPY}-\eqref{E:INTROCOVWAVEVELOCITY}
under \emph{any} equation of state 
aside from a single exceptional one. In the irrotational and isentropic case 
(in which case the relativistic Euler equations reduce to a quasilinear wave equation satisfied by a potential function), 
this exceptional equation of state was identified in \cite{dC2007}; 
it corresponds to the quasilinear wave equation satisfied by a timelike minimal surface graph in an ambient Minkowski spacetime,
which can be expressed as follows:
$
\displaystyle
\partial_{\alpha} 
\left\lbrace
	\frac{(\upeta^{-1})^{\alpha \beta} \partial_{\beta} \Phi}{
	\sqrt{1 + (\upeta^{-1})^{\kappa \lambda}(\partial_{\kappa} \Phi)(\partial_{\lambda} \Phi)}}
\right\rbrace
= 0$.

In view of the previous paragraph, we would like to highlight the following point:
\begin{quote}
	Proofs of shock formation 
	are unstable under typical perturbations of the equations by 
	nonlinear terms that are of quadratic or higher-order in derivatives.
	However, proofs of shock formation for wave equations 
	typically \emph{are} stable under perturbations
	of the equations by null forms that are adapted to the metric of the shock-forming wave. 
	By ``stable,'' we
	mean in the following sense: as the shock forms, 
	null form terms become ``asymptotically negligible''
	compared to the shock-driving terms (for the reasons described above).
\end{quote}
The reason that the precise structure of the nonlinearities is so important for the proofs 
is that the known framework is designed precisely
to handle specific kinds of singularity-driving derivative-quadratic
terms: the kind that are hidden in the covariant wave operator terms 
on LHSs~\eqref{E:INTROCOVWAVELNENTHALPY}-\eqref{E:INTROCOVWAVEVELOCITY}.
In the context of the relativistic Euler equations,
this means that if \emph{any} of the equations of 
Theorem~\ref{T:MAINTHEROEMSCHEMATIC} had contained, on the right-hand side,
an inhomogeneous non-$g$-null-form quadratic term of type 
$(\partial \Lnenth)^2$, $\partial u \cdot \partial \Lnenth$, $(\partial u)^2$, etc., 
or a term of type $(\partial \Lnenth)^3$, $(\partial \Lnenth)^4$, etc., then the 
\emph{only known framework for proving
shock formation would not work.} The difficulty is that adding 
such terms to the equation could in principle radically alter the expected blowup-rate 
or even altogether prevent the formation of a singularity; either way,
this would invalidate\footnote{As is explained in \cite{jLjS2016a}, in the known framework for 
proving shock formation, one crucially relies on the fact that the derivatives of the solution blow up at a linear 
rate, that is like $\frac{C}{T_{(Lifespan)} - t}$,
where $C$ is a constant and $T_{(Lifespan)} > 0$ is the (future) classical lifespan of the solution; if one perturbs the equation
by adding terms that are expected to alter this blowup-rate, then one should expect that
the known approach for proving shock formation will not work (at least in its current form).} 
the known approach for proving shock formation.
One might draw an analogy with the Riccati ODE
$\dot{y} = y^2$, which we suggest as a caricature model for the formation of shocks
(in the case of the relativistic Euler equations, $y$ should be identified with $\partial \Lnenth$ and/or $\partial u$).
Note that for all data $y(0) = y_0$ with $y_0 > 0$,
the solution to the Riccati ODE blows up in finite time.
Now if one perturbs the Riccati ODE to obtain the perturbed equation
$\dot{y} = y^2 \pm \epsilon y^3$, with $\epsilon$ a small positive number, 
then depending on the sign of $\pm$, the perturbed solutions with $y_0 > 0$ will either
exist for all time or will blow up at a quite different rate 
compared to the blowup-rate 
for the unperturbed equation.

\subsubsection{Regularity properties and singular high-order energy estimates}
\label{SSS:REGULARITYANDSINGULARENERGYESTIMATES}
In the rigorous mathematical study of quasilinear hyperbolic PDEs in more than one spatial dimension,
one is forced to derive energy estimates for the solution's higher derivatives
by commuting the equations with appropriate differential operators.
Indeed, all known approaches to studying even the basic local well-posedness theory for such equations
rely on deriving estimates in $L^2$-based Sobolev spaces. In the works 
\cites{dC2007,jS2016b,dCsM2014,jSgHjLwW2016,sM2018,sMpY2017,jLjS2018,jS2017c,js2018b}
on shock formation in multiple spatial dimensions,
the authors controlled the solutions' higher geometric derivatives by 
differentiating the equations with geometric ``commutator vectorfields''
$Z$ that are adapted to the characteristics, more precisely to the characteristics 
corresponding to the variables that form a shock singularity.
As we mentioned earlier, the $Z$ are designed to avoid generating uncontrollable commutator error terms. It turns out that
all $Z$ that have been successfully used to study shock formation
have the schematic structure $Z^{\alpha} \sim \partial U$, where 
$Z^{\alpha}$ denotes a rectangular component of $Z$ and $U$ is the eikonal function.

Although the geometric vectorfields $Z$ exhibit good commutation properties 
with the differential operators corresponding to the characteristics to which they are adapted, 
the regularity theory of the vectorfields themselves is very delicate
and is intimately tied to that of the solution. We now further explain this fact
in the context of wave equations whose principal operator is
$(g^{-1})^{\alpha \beta} \partial_{\alpha} \partial_{\beta}$.
The corresponding eikonal equation is 
the nonlinear transport equation $(g^{-1})^{\alpha \beta} \partial_{\alpha} U \partial_{\beta} U = 0$.
The key point is that the standard regularity theory
of transport equations yields only that $U$ is as regular as its coefficients,
that is, as regular as $g_{\alpha \beta}$.
In the context of the relativistic Euler equations
(where the formation of a shock corresponds to the intersection of the wave characteristics and $g = g(\Lnenth,\Ent,u)$),
this suggests that one might expect $U$ to be only as regular as $\Lnenth$, $\Ent$, and $u$. 
Since, as we mentioned in the previous paragraph, we have $Z^{\alpha} \sim \partial U$,
this leads to the following severe difficulty: in commuting equation 
the wave equation \eqref{E:INTROCOVWAVELNENTHALPY}
with $Z$, one obtains the wave equation
$\square_g (Z \Lnenth) = \square_g Z^{\alpha} \cdot \partial_{\alpha} \Lnenth + \cdots \sim \partial^3 U \cdot \partial \Lnenth + \cdots$
(one would obtain similar wave equations for $Z \Ent$ and $Z u^{\alpha}$ upon commuting
equations \eqref{E:INTROCOVWAVEVELOCITY} and \eqref{E:INTROCOVWAVEENT} with $Z$).
The difficulty is that the above discussion suggests that the factor $\partial^3 U$ can be controlled only in terms of 
\emph{three} derivatives of $\Lnenth$, $\Ent$, and $u$, while standard energy estimates for the wave equations
$\square_g (Z \Lnenth) = \cdots$, 
$\square_g (Z \Ent) = \cdots$,
and $\square_g (Z u^{\alpha}) = \cdots$ yield control of only \emph{two} derivatives of
$\Lnenth$, $\Ent$, and $u$. 
This suggests that there is a loss of regularity and in fact, this is the reason that Alinhac 
used Nash--Moser estimates in his works \cites{sA1999a,sA2001b}.
However, for wave equations, one can in fact overcome this loss of regularity by
exploiting some delicate tensorial properties of the eikonal equation
$(g^{-1})^{\alpha \beta} \partial_{\alpha} U \partial_{\beta} U = 0$
and of the wave equation itself relative to geometric coordinates, 
which together can be used to show that in directions \emph{tangent to the characteristics}, 
some geometric tensors constructed out of the derivatives of
$U$ are \emph{one degree more differentiable than one might naively expect}. 
In particular, the factor $\partial^3 U$ in the aforementioned product $\partial^3 U \cdot \partial \Lnenth$
has special structure and enjoys this gain in regularity.
These crucial structures were first observed by Christodoulou--Klainerman in their proof \cite{dCsK1993} 
of the stability of Minkowski spacetime as a solution to Einstein's equations, 
and later by Klainerman--Rodnianski in their proof of improved-regularity local well-posedness \cite{sKiR2003} 
for a general class of scalar quasilinear wave equations.
In total, using this gain in regularity along the characteristics and carefully 
accounting for the precise tensorial structure of the product 
$\partial^3 U \cdot \partial \Lnenth$ highlighted above, one can 
avoid the loss of derivatives tied to the product $\partial^3 U \cdot \partial \Lnenth$.

Despite the fact that the procedure described above allows one to avoid losing derivatives, 
at least in the context of wave equations,\footnote{Actually, it is not known whether or not the 
derivative-loss-avoiding procedure can be implemented for general systems of wave equations featuring
more than one distinct wave operator. From this perspective, we find it fortunate that the equations of Theorem~\ref{T:MAINTHEROEMSCHEMATIC}
feature only one wave operator.} 
one pays a steep price: it turns out that upon implementing this procedure, 
one introduces a dangerous factor into the wave equation energy identities, one that in fact blows up as the shock forms.
More precisely, the singular factor is $1/\upmu$, where $\upmu$ is the inverse foliation density mentioned earlier,
with $\upmu \to 0$ signifying the formation of a shock.
This leads to singular top-order a priori energy estimates for the wave equation solutions relative to the geometric coordinates. 
At first glance, these singular geometric energy estimates
might seem to obstruct the philosophy of obtaining regular estimates relative to the geometric coordinates.
However, below the top derivative level, one can allow the loss of a derivative,
and it turns out that this allows one to derive improved (i.e., less singular) energy estimates below the top derivative level.
In fact, by an induction-from-the-top-down argument, one can show that the mid-derivative-level and below geometric energies 
remain bounded up to the shock. This allows one to show that indeed, the solution remains rather smooth relative
to the geometric coordinates, which in practice is a crucial ingredient that is needed to close 
the proof. It also turns out that many steps are needed to descend to the level of a non-singular energy, 
which in practice means that one must assume that the data have a lot of Sobolev regularity to close the proof;
see \cite{gHsKjSwW2016} for an in-depth overview of these issues in the context of quasilinear wave equations.

The structures described above, which allow one to avoid the loss of derivatives
in eikonal functions for quasilinear wave equations, are rather delicate. 
Thus, it is not a priori clear that one can also avoid the loss of derivatives
in eikonal functions for the relativistic Euler equations.
A key advantage of our new formulation of the relativistic Euler equations
is that it can be used to prove that \emph{one can still avoid the loss of derivatives, 
even though there is deep coupling between the wave and transport equations in the new formulation}.
That is, one can show that the acoustic eikonal function $U$
(see \eqref{E:EIKONAL}, where $g= g(\Lnenth,\Ent,u)$ is the acoustical metric from Def.\,\ref{D:ACOUSTICALMETRICANDINVERSE})
for the relativistic Euler equations has enough regularity to be used in the study of shock formation;
see three paragraphs below for further discussion.
However, this requires one to first prove that the fluid variables have a consistent
amount of regularity among themselves. 
At first thought, the desired consistency of regularity might seem to follow
from standard local well-posedness.
However, all standard local well-posedness results for 
the relativistic Euler equations are based on first-order formulations,
which are not known to be sufficient for avoiding a loss of derivatives in the eikonal function $U$;
the above outline for how to avoid derivative loss in $U$ implicitly relied on the assumption that 
$\Lnenth$,
$\Ent$,
and $u^{\alpha}$  
solve wave equations whose source terms have an allowable amount of regularity,
which, as we will explain, for the relativistic Euler equations is a true -- but deep -- fact.
Moreover, the first-order formulations do not seem to be sufficient for
studying solutions all the way up to a shock; 
as we have mentioned, the known framework for studying shocks crucially relies on
the special null structures exhibited by the equations of Theorem~\ref{T:MAINTHEROEMSCHEMATIC}.

In view of the regularity concerns raised in the previous paragraph,
one must carefully check that (under suitable assumptions on the initial data), 
all terms in the equations of Theorem~\ref{T:MAINTHEROEMSCHEMATIC} 
have a \emph{consistent amount of regularity}. We stress that this is not obvious, as we now illustrate
by counting derivatives. For example, to control $\partial u^{\alpha}$ in $L^2$ using standard energy estimates
for the wave equation \eqref{E:INTROCOVWAVEVELOCITY}, one must control, also in $L^2$, 
the source term $\mathcal{C}^{\alpha}$ on RHS~\eqref{E:INTROCOVWAVEVELOCITY}.
Note that from the point of 
view of regularity, we have the schematic relationship (see \eqref{E:MODIFIEDVORTICITYOFVORTICITY} for the definition of $\mathcal{C}^{\alpha}$)
$\mathcal{C}^{\alpha} \sim \uperpvort^{\alpha}(\vort) \sim \partial \vort$.
Moreover, since $\vort$ solves the transport equation \eqref{E:INTROVORTICITYTRANSPORT},
whose source term depends on $\partial u$ and $\partial \Lnenth$, this suggests 
that $\partial \vort$ should be no more regular\footnote{In the absence of special structures,
solutions to transport equations are not more regular than their source terms.} 
than $(\partial^2 u, \partial^2 \Lnenth)$ and thus $\mathcal{C}^{\alpha}$ should be no more regular than
$(\partial^2 u, \partial^2 \Lnenth)$. In total, this discussion suggests that the wave equation for $u$ 
has the following schematic structure from the point of view of regularity:
$\square_g u^{\alpha} = \partial^2 u + \cdots$. That is, this discussion \emph{suggests}
that in order to control $\partial u$ in $L^2$ using standard energy estimates
for wave equations, we must control $\partial^2 u$ in $L^2$. This approach therefore \emph{seems} to lead
to a loss in derivatives, which is a serious obstacle to using the equations of
Theorem~\ref{T:MAINTHEROEMSCHEMATIC} to prove any rigorous result. Similar difficulties arise
in the study of $\Lnenth$ and $\Ent$, due to the source term $\mathcal{D}$ 
in the wave equations \eqref{E:INTROCOVWAVELNENTHALPY} and \eqref{E:INTROCOVWAVEENT}.

A crucial feature of the equations of Theorem~\ref{T:MAINTHEROEMSCHEMATIC} is that 
\emph{one can in fact overcome the loss-of-derivative-difficulty for the fluid variables described in the previous paragraph}. 
To this end, one must rely on the $\mbox{\upshape transport}$-$\mbox{\upshape div}$-$\mbox{\upshape curl}$ equations
for $\vort$ and $\GradEnt$; see Subsect.\,\ref{SS:ENERGYNORMTIEDTOELLITPICESTIMATES}
and the proofs of Prop.\,\ref{P:APRIORIESTIMATES} and Theorem~\ref{T:UPGRADEDLOCALWELLPOSEDNESS}
for the details on how one can use these equations and elliptic estimates
to avoid the loss of derivatives. Equally important for applications to shock waves 
is the fact that the elliptic div-curl estimates, which occur across space, 
are compatible with the proof of the formation of a spatially localized shock singularity
and with the singular high-order geometric energy estimates described earlier in this subsubsection.
These are delicate issues, especially since the elliptic estimates involve derivatives in directions
transversal to the characteristics, i.e., in the singular directions;
see \cite{jLjS2016a} for an overview of how to derive the relevant
elliptic estimates in the context of shock-forming solutions to the non-relativistic compressible Euler equations.

We now return to the issue of the regularity of the acoustic eikonal function $U$
for the relativistic Euler equations
(see \eqref{E:EIKONAL}, where $g= g(\Lnenth,\Ent,u)$ is the acoustical metric from Def.\,\ref{D:ACOUSTICALMETRICANDINVERSE}).
As we explained above, in order to avoid a loss of regularity in $U$,
one needs to show that its regularity theory is compatible with the regularity of the fluid
variables. It turns out that this requires proving, in particular, 
that $\square_g \Lnenth$, 
$\square_g \Ent$,
and
$\square_g u^{\alpha}$
have the same regularity as 
$\partial \Lnenth$, 
$\partial \Ent$, 
and $\partial u^{\alpha}$.
The connection between
$\square_g \Lnenth$, 
$\square_g \Ent$,
and
$\square_g u^{\alpha}$
and the regularity theory of $U$ is through the null mean curvature
of the level sets of $U$, a critically important geometric quantity
whose evolution equation\footnote{The evolution equation is in fact the famous \emph{Raychaudhuri equation}, which plays
an important role in general relativity.} 
depends on a certain component of the
Ricci curvature tensor of the Lorentzian metric $g(\Lnenth,\Ent,u)$,
whose rectangular components can be shown to depend on
$\square_g \Lnenth$, 
$\square_g \Ent$,
and
$\square_g u^{\alpha}$.
We will not further discuss this crucial technical issue here;
we instead refer readers to \cite{gHsKjSwW2016}*{Section~3.4} for further discussion of the regularity theory of
eikonal functions in the context of shock formation for quasilinear wave equations.
In view of the wave equations \eqref{E:INTROCOVWAVELNENTHALPY}-\eqref{E:INTROCOVWAVEENT},
we see that 
obtaining the desired regularity 
for
$\square_g \Lnenth$, 
$\square_g \Ent$,
and
$\square_g u^{\alpha}$
requires, in particular, 
establishing that the source terms $\mathcal{C}^{\alpha}$ and $\mathcal{D}$ on RHSs~\eqref{E:INTROCOVWAVELNENTHALPY}-\eqref{E:INTROCOVWAVEENT}
have the same regularity as $\partial \Lnenth$, $\partial \Ent$, and $\partial u^{\alpha}$.
This is again tantamount to showing that 
the vorticity and entropy are one degree more differentiable compared
to the regularity guaranteed by deriving standard energy estimates 
for first-order formulations of the equations;
to obtain the desired extra regularity for $\mathcal{C}^{\alpha}$ and $\mathcal{D}$, 
one can again rely on the $\mbox{\upshape transport}$-$\mbox{\upshape div}$-$\mbox{\upshape curl}$ equations mentioned in the previous paragraph.
We prove a rigorous version of this gain in regularity in Theorem \ref{T:UPGRADEDLOCALWELLPOSEDNESS}, 
in which we use the new formulation of the relativistic Euler equations
to prove a local well-posedness result that in particular yields
the desired extra differentiability (assuming that it is present in the initial data).

Although one might view the results of Theorem \ref{T:UPGRADEDLOCALWELLPOSEDNESS} 
as expected consequences of our new formulation of the relativistic Euler equations,
we highlight that its proof relies on a few ingredients that 
are not entirely straightforward:
\begin{enumerate}
	\renewcommand{\labelenumi}{\textbf{\roman{enumi})}}
	\item
	\emph{Time-continuity} for the $L^2$ norms of the 
		vorticity and entropy at top-order, i.e, including the 
		extra differentiability of these variables, is 
		non-standard in view of the necessity of invoking elliptic-hyperbolic estimates.
	\item The $\mbox{\upshape transport}$-$\mbox{\upshape div}$-$\mbox{\upshape curl}$ systems featured in the new formulation of the equations
		involve \emph{spacetime} divergence and curl operators, but we need 
to extract $L^2$ regularity along the constant-time hypersurfaces. This requires connecting
the spacetime divergence and curl to spatial elliptic estimates, which in turn requires some geometric and technical insights.
\item For the wave equation energy estimates, one cannot use the multiplier\footnote{See Subsubsect.\,\ref{SSS:ESTIMATESFORCOVARIANTWAVEEQUATIONS} for
additional details regarding the multiplier method in the context of wave equations.}
$\partial_t$  
when the three-velocity is large, since the corresponding energy will not
necessarily be coercive\footnote{Equations 
\eqref{E:SPEEDOFSOUNDBOUNDS},
\eqref{E:UISUNITLENGTH},
and
\eqref{E:ACOUSTICALMETRIC}
collectively imply that when $\sum_{a=1}^3|u^a|$ 
is large, $g(\partial_t,\partial_t) = g_{00} = - 1 + (\speed^{-2} - 1) u_a u^a$
can be positive, i.e., $\partial_t$ can be spacelike with respect to the acoustical metric $g$;
it is well-known that this can lead to indefinite energies if the standard partial time derivative vectorfield $\partial_t$ 
is used as a multiplier in the wave equation energy estimates.
\label{F:PARTIALTNOTTIMELIKE}} in this case. 
Consequently, one has to use the four-velocity as a multiplier.\footnote{The use of $u$ as a multiplier 
is likely familiar to researchers who have previously studied the 
relativistic Euler equations, but it might be unknown to the broader PDE community. 
We also remark that in searching the literature, we were unable to find results that,
given our new formulation of the relativistic Euler equations, 
could be directly applied to establish points \textbf{i)} and \textbf{ii)} above.
Moreover, we were not able to locate a local well-posedness result for elliptic-hyperbolic systems 
that can be directly applied to our new formulation of the equations.
In particular, we could not locate a result that would directly
imply continuous dependence of solutions on the initial data up to top order, 
i.e., a result that applies in the case when the vorticity and entropy 
enjoy the aforementioned extra regularity.
}
\end{enumerate}

\subsubsection{Structures amenable to commutations with geometric vectorfields}
\label{SSS:COMMUTINGWITHGEOMETRICVECTORFIELDS}
A key point is that the geometric vectorfields $Z$ described in 
Subsubsect.\,\ref{SSS:REGULARITYANDSINGULARENERGYESTIMATES}
are adapted only to the principal part of the shock-forming solution variables,
e.g., the operator $\square_g$ in the case that a wave equation solution
is the shock-forming variable.
However, to close the proof of shock formation for a system in which
wave equations of the type $\square_g \cdot = \cdots$ are \emph{coupled} 
to other equations, one must commute that $Z$ through \emph{all} of the equations
in the system. One then has to handle the commutator terms generated
by commuting the $Z$ through the other equations.
It turns out, perhaps not surprisingly, that commuting $Z$ through a 
generic second-order differential operator $\partial^2$ leads
to uncontrollable error terms, from the point of view of regularity and from
the point of view of the singular nature of the commutator error terms; see 
the work \cite{jLjS2016a} on the non-relativistic compressible barotropic Euler equations 
for further discussion on this point. However, as was first shown in \cite{jLjS2016a},
it is possible to commute the $Z$ through an arbitrary first-order
differential operator $\partial$ by first weighting it by $\upmu$ 
(where $\upmu$ is the inverse foliation density mentioned above); 
it can be shown that this leads to commutator error terms that are controllable under the scope of the approach.
\textbf{It is for this reason that we have formulated Theorem~\ref{T:MAINTHEROEMSCHEMATIC}
in such a way that all of the equations are of the type $\square_g \cdot = \cdots$
or are first-order}; i.e., the equations of
Theorem~\ref{T:MAINTHEROEMSCHEMATIC} are such that the approach 
described in \cites{jLjS2016a} can be applied. Put differently,
the geometric vectorfields $Z$ that are of essential importance for 
commuting the wave equations of
Theorem~\ref{T:MAINTHEROEMSCHEMATIC}
can also be commuted through all of the remaining equations,
generating only controllable error terms.

\section{A first-order formulation of the relativistic Euler equations, geometric tensorfields, and the modified fluid variables}
\label{S:SOMEBACKGROUND}
In this section, we introduce some notation, 
define the fluid variables that play a role in the subsequent discussion,
introduce some geometric tensorfields associated to the flow,
and provide the standard first-order formulation of the relativistic Euler equations 
that will serve as a starting point for our main results.
Most of the discussion here is standard and therefore, we are somewhat terse;
we refer readers to \cite{dC2007}*{Chapter~1} for a detailed introduction to the relativistic Euler equations.
Subsubsect.\,\ref{SSS:MODIFIEDFLUIDVARIABLES}, however, is not standard. 
In that subsubsection, we define \emph{modified fluid variables}, which are special combinations
of the derivatives of the vorticity and entropy. 
The structures revealed by Theorem~\ref{T:NEWFORMULATIONRELEULER}
imply (see the proof of Theorem~\ref{T:UPGRADEDLOCALWELLPOSEDNESS} for additional details)
that these special combinations enjoy a gain of one derivative
compared to the regularity afforded by standard estimates. As we mentioned
in the introduction, this gain is crucial for applications to shock waves.

\subsection{Notation and conventions}
\label{SS:NOTATIONANDCONVENTIONS}
We somewhat follow the setup of \cite{dC2007}, but there are some differences, including sign differences
and notational differences.

Greek ``spacetime'' indices $\alpha, \beta, \cdots$ take on the values $0,1,2,3$, while Latin ``spatial'' indices $a,b,\cdots$ 
take on the values $1,2,3$. Repeated indices are summed over (from $0$ to $3$ if they are Greek, and from $1$ to $3$ if they are Latin).
Greek and Latin indices are lowered and raised with the Minkowski metric $\upeta$ and its inverse $\upeta^{-1}$,
and \textbf{not with the acoustical metric $g$ of Def.\,\ref{D:ACOUSTICALMETRICANDINVERSE}}.
Moreover, 
$\upepsilon_{\alpha \beta \gamma \delta}$ denotes the fully antisymmetric symbol normalized by
$\upepsilon_{0123} = 1$. Note that $\upepsilon^{0123} = - 1$.

If $X^{\alpha}$ is a vectorfield and $\xi_{\beta_1 \cdots \beta_m}^{\alpha_1 \cdots \alpha_l}$
is a type $\binom{l}{m}$ tensorfield, then
\begin{align} \label{E:LIEDERIVATIVE}
	(\mathcal{L}_X \xi)_{\beta_1 \cdots \beta_m}^{\alpha_1 \cdots \alpha_l}
	& = X^{\kappa} \partial_{\kappa} \xi_{\beta_1 \cdots \beta_m}^{\alpha_1 \cdots \alpha_l}
			-
			\sum_{a=1}^l
			(\partial_{\kappa} X^{\alpha_a}) \xi_{\beta_1 \cdots \beta_m}^{\alpha_1 \cdots \alpha_{a-1} \kappa \alpha_{a+1} \cdots \alpha_l}
			\\
			& \ \
			+
			\sum_{b=1}^m
			(\partial_{\beta_b} X^{\kappa}) \xi_{\beta_1 \cdots \beta_{b-1} \kappa \beta_{b+1} \cdots \beta_m}^{\alpha_1 \cdots \alpha_l}  
			\notag
\end{align}
denotes the Lie derivative of $\xi$ with respect to $X$.

We derive all of our results relative to a Minkowski-rectangular coordinate system $\lbrace x^{\alpha} \rbrace_{\alpha=0,1,2,3}$,
that is, a coordinate system on $\mathbb{R}^{1+3}$ in which the Minkowski metric $\upeta$ takes the form
$\upeta_{\alpha \beta} := \mbox{\upshape diag}(-1,1,1,1)$.
$\lbrace \partial_{\alpha} \rbrace_{\alpha=0,1,2,3}$ denotes the corresponding rectangular
coordinate partial derivative vectorfields. We sometimes use the alternate notation $x^0 := t$ and $\partial_t := \partial_0$.

Throughout, $d$ denotes the exterior derivative operator. In particular, if $f$ is a scalar function, 
then $(df)_{\alpha} := \partial_{\alpha} f$, 
and if $V$ is a one-form, then 
$(d V)_{\alpha \beta} := \partial_{\alpha} V_{\beta} - \partial_{\beta} V_{\alpha}$.
We use the notation $V_{\flat}$ to denote the one-form that is $\upeta$-dual to the vectorfield $V$, i.e.,
$(V_{\flat})_{\alpha} := \upeta_{\alpha \kappa} V^{\kappa}$.

\subsection{Definitions of the fluid variables and related geometric quantities}
\label{SS:DEFSOFFLUIDVARIABLESANDRELATEDGEOMETRICQUANTITIES}
In this subsection, we define the fluid variables and geometric quantities that
play a role in the subsequent discussion.

\subsubsection{The basic fluid variables}
\label{SSS:FLUIDVARIABLES}
The fluid four velocity $u^{\alpha}$ is future-directed and normalized by $u_{\alpha} u^{\alpha} = - 1$.
$p$ denotes the pressure, $\uprho$ denotes the proper energy density, $n$ denotes the proper number density,
$\Ent$ denotes the entropy per particle, $\Temp$ denotes the temperature,
and 
\begin{align} \label{E:ENTHALPYDEFINITION}
\Enth = (\uprho + p)/n 
\end{align}
is the enthalpy per particle. 
Thermodynamics supplies the following laws:
\begin{align} \label{E:LAWSOFTHERMODYNAMICS}
\Enth
& =
\frac{\partial \uprho}{\partial n}\left|_{\Ent}\right.,
&
\Temp 
& = \frac{1}{n} \frac{\partial \uprho}{\partial \Ent} \left|\right._n,
&
d \Enth 
& = \frac{dp}{n} + \Temp d \Ent,
\end{align}
where 
$\frac{\partial}{\partial n}\left|_{\Ent}\right.$
denotes partial differentiation with respect to $n$ at fixed $\Ent$
and
$\frac{\partial}{\partial \Ent} \left|\right._n$ denotes partial differentiation with
respect to $\Ent$ at fixed $n$. Below we employ similar partial differentiation notation,
and in Def.\,\ref{D:ALTERNATEPARTIALDERIVATIVENOTATION}, we introduce alternate
partial differentiation notation, which we use throughout the remainder of the article.

\subsubsection{The $u$-orthogonal vorticity of a one-form and auxiliary fluid variables}
\label{SSS:AUXFLUIDVARIABLES}
In this subsubsection, we define some auxiliary fluid variables that will play a role
throughout the paper. By ``auxiliary,'' we mean that they are determined
by the variables introduced in Subsubsect.\,\ref{SSS:FLUIDVARIABLES}.

We start by defining the $u$-orthogonal vorticity of a one-form.
\begin{definition}[The $u$-orthogonal vorticity of a one-form]
	\label{D:VORTICITYOFAONEFORM}
	Given a one-form $V$, we define the corresponding $u$-orthogonal vorticity
	vectorfield as follows:
	\begin{align} \label{E:UORTHGONALVORTICITYOFONEFORM} 
		\uperpvort^{\alpha}(V)
		& := - \upepsilon^{\alpha \beta \gamma \delta} u_{\beta} \partial_{\gamma} V_{\delta}.
	\end{align}
\end{definition}

\begin{definition}[Vorticity vectorfield]
\label{D:VORTICITYDEF}
We define the vorticity vectorfield $\vort^{\alpha}$ as follows:
\begin{align} \label{E:VORTICITYDEF}
		\vort^{\alpha}
		& 
		:=
		\uperpvort^{\alpha}(\Enth u)
		= - \upepsilon^{\alpha \beta \gamma \delta} u_{\beta} \partial_{\gamma} (\Enth u_{\delta}).
\end{align}
\end{definition}

We find it convenient to work with the natural log of the enthalpy.
\begin{definition}[Logarithmic enthalpy]
	Let $\overline{\Enth} > 0$ be a fixed constant value of the enthalpy. We define
	the (dimensionless) logarithmic enthalpy $\Lnenth$ as follows:
	\begin{align} \label{E:LOGARITHMICENTHALPY}
		\Lnenth
		&:= \ln \left(\Enth/\overline{\Enth} \right).
	\end{align}
\end{definition}

\begin{definition}[The quantity $\TempoverEnth$]
	\label{D:TEMPOVERENTH}
	We define the quantity $\TempoverEnth$ as follows:
	\begin{align} \label{E:TEMPOVERENTH}
		\TempoverEnth
		&:= \frac{\Temp}{\Enth}.
	\end{align}
\end{definition}

\begin{definition} [Entropy gradient one-form]
	\label{D:ENTROPYGRADIENT}
	We define the entropy gradient one-form $\GradEnt_{\alpha}$ as follows:
	\begin{align} \label{E:ENTROPYGRADIENT}
		\GradEnt_{\alpha}
		& := \partial_{\alpha} \Ent.
	\end{align}
\end{definition}

\subsubsection{Equation of state and speed of sound}
To obtain a closed system of equations,
we assume an \emph{equation of state} of the form $p = p(\uprho,\Ent)$.
The \emph{speed of sound} is defined by
\begin{align} \label{E:SPEEDOFSOUND}
	\speed 
	& := 
	\sqrt{\frac{\partial p}{\partial \uprho} \left|\right._{\Ent}}.
\end{align}
For reasons that will become clear in Subsect.\,\ref{SS:FIRSTIRDERFORMULATION}, 
in the rest of the article, we view the speed of sound
to be a function of $\Lnenth$ and $\Ent$:
\begin{align}
	\speed = \speed(\Lnenth,\Ent).
	\label{E:SOUNDSPEEDDEFENTHALTPYENTROPY}
\end{align}

In this article, we will confine our study to equations of state and solutions that verify
\begin{align} \label{E:SPEEDOFSOUNDBOUNDS}
	0 < \speed \leq 1.
\end{align}	
The upper bound in \eqref{E:SPEEDOFSOUNDBOUNDS} signifies that the speed of sound is no
bigger than the speed of light. In this article, 
we exploit both inequalities in \eqref{E:SPEEDOFSOUNDBOUNDS}.
We use the bound $\speed \leq 1$ to ensure that
we can always solve for time derivatives of the solution in terms of 
spatial derivatives; see the discussion surrounding equation \eqref{E:A0DET}.
The bound $\speed > 0$ is important because some 
of the equations featured in Theorem~\ref{T:NEWFORMULATIONRELEULER} contain factors of $\speed^{-1}$.

\subsubsection{Projection onto the Minkowski-orthogonal complement of the four-velocity  and the acoustical metric}
\label{SS:GEOMETRICTENSORS}
 We start by introducing the tensorfield $\Pi^{\alpha \beta}$, defined by
\begin{align} \label{E:UPERPPROJECTION}
\Pi^{\alpha \beta}
	& := (\upeta^{-1})^{\alpha \beta} + u^{\alpha} u^{\beta}.
\end{align}
It is straightforward to see $\Pi$ is the projection onto the $\upeta$-orthogonal complement of $u$.
In particular, $\Pi^{\alpha \kappa} u_{\kappa} = 0$.
	
We now introduce the acoustical metric $g$. It is a Lorentzian\footnote{That is, the signature of the $4 \times 4$ matrix $g_{\alpha \beta}$, viewed as a quadratic form,
is $(-,+,+,+)$.} metric that drives the propagation of sound waves.

\begin{definition}[Acoustical metric and its inverse]
\label{D:ACOUSTICALMETRICANDINVERSE}
\begin{subequations}
We define the acoustical metric
$g_{\alpha \beta}$
and its inverse\footnote{It is straightforward to check that
$(g^{-1})^{\alpha \kappa} g_{\kappa \beta} = \updelta_{\beta}^{\alpha}$,
where $\updelta_{\beta}^{\alpha}$ is the Kronecker delta. That is, $g^{-1}$ is indeed the inverse of $g$.}
$(g^{-1})^{\alpha \beta}$
as follows:
\begin{align}
	g_{\alpha \beta}
	& := \speed^{-2} \upeta_{\alpha \beta} 
		+ (\speed^{-2} - 1) u_{\alpha} u_{\beta},
			\label{E:ACOUSTICALMETRIC} \\
(g^{-1})^{\alpha \beta}
	& := \speed^2 \Pi^{\alpha \beta} - u^{\alpha} u^{\beta}
		= \speed^2 (\upeta^{-1})^{\alpha \beta} 
			+ (\speed^2 - 1)
				u^{\alpha} u^{\beta}.
				\label{E:INVERSEACOUSTICALMETRIC}
\end{align}
\end{subequations}
\end{definition}

It is straightforward to compute that relative to the rectangular 
coordinates, we have
\begin{subequations}
\begin{align}
\mbox{\upshape det} g 
& = - \speed^{-6},
	\label{E:DETG} \\
|\mbox{\upshape det} g|^{1/2} 
(g^{-1})^{\alpha \beta}
& = 	\speed^{-1} (\upeta^{-1})^{\alpha \beta} 
			+ 
			(\speed^{-1} - \speed^{-3})
			u^{\alpha} u^{\beta}.
			\label{E:DETGGINVERSE}
\end{align}
\end{subequations}

The notation featured in the next definition will allow for a simplified presentation
of various equations.
\begin{definition}[Partial derivatives with respect to $\Lnenth$ and $\Ent$]
	\label{D:ALTERNATEPARTIALDERIVATIVENOTATION}
	If $Q$ is a quantity that can be expressed as a function of
	$(\Lnenth,\Ent)$, then
	\begin{subequations}
	\begin{align} \label{E:ALTERNATEPARTIALDERIVATIVENOTATION}
	Q_{;\Lnenth}
	= Q_{;\Lnenth}(\Lnenth,\Ent)
	& := \frac{\partial Q}{\partial \Lnenth}\left|_{\Ent}\right.,
		\\
	Q_{;\Ent}
	= Q_{;\Ent}(\Lnenth,\Ent)
	& := \frac{\partial Q}{\partial \Ent}\left|_{\Lnenth}\right.,
\end{align}
\end{subequations}
where 
$\frac{\partial}{\partial \Lnenth}\left|_{\Ent}\right.$ denotes partial differentiation
with respect to $\Lnenth$ at fixed $\Ent$ and
$\frac{\partial}{\partial \Ent}\left|_{\Lnenth}\right.$ denotes partial differentiation
with respect to $\Ent$ at fixed $\Lnenth$.
\end{definition}

\subsubsection{Modified fluid variables}
\label{SSS:MODIFIEDFLUIDVARIABLES}
In our analysis, we will have to control the vorticity of the vorticity, that is,
$ \uperpvort^{\alpha}(\vort)$.
The following modified version of $\uperpvort^{\alpha}(\vort)$, denoted by $\mathcal{C}^{\alpha}$ obeys a transport equation
(see \eqref{E:MAINTHMTRANSPORTFORMODIFIEDVORTICITYOFVORTICITY})
with a better structure 
(from the point of view of the regularity of the RHS and also the null structure of the RHS)
than the one satisfied by $\uperpvort^{\alpha}(\vort)$.
Similar remarks apply to the modified version of the divergence of entropy
gradient, which we denote by $\mathcal{D}$
(see equation \eqref{E:MAINTHMSMODIFIEDDIVERGENCEENTROPYGRADIENTTRANSPORT}
for the transport equation verified by $\mathcal{D}$).

\begin{definition}[Modified fluid variables]
	\label{D:MODIFIEDVARIABLES}
	\begin{subequations}
	\begin{align} 
			\mathcal{C}^{\alpha}
		& := \uperpvort^{\alpha}(\vort)
			+
		\speed^{-2}
		\upepsilon^{\alpha \beta \gamma \delta} 
		u_{\beta}
		(\partial_{\gamma} \Lnenth) \vort_{\delta}
			\label{E:MODIFIEDVORTICITYOFVORTICITY} \\
		& \ \
			+
			(\Temp - \Temp_{;\Lnenth}) 
			\GradEnt^{\alpha} 
			(\partial_{\kappa} u^{\kappa})
		+
		(\Temp - \Temp_{;\Lnenth}) u^{\alpha} (\GradEnt^{\kappa} \partial_{\kappa} \Lnenth)
		\notag
		\\& \ \
		+
		(\Temp_{;\Lnenth} - \Temp) \GradEnt^{\kappa} ((\upeta^{-1})^{\alpha \lambda} \partial_{\lambda} u_{\kappa}),
				\notag \\
		\mathcal{D}
		& := \frac{1}{n} (\partial_{\kappa} \GradEnt^{\kappa})
			+
			\frac{1}{n} (\GradEnt^{\kappa} \partial_{\kappa} \Lnenth)
			-
			\frac{1}{n} \speed^{-2} (\GradEnt^{\kappa} \partial_{\kappa} \Lnenth).
			 \label{E:MODIFIEDDIVERGENCEOFENTROPY}
\end{align}
\end{subequations}

	\end{definition}

\subsection{A standard first-order formulation of the relativistic Euler equations}
\label{SS:FIRSTIRDERFORMULATION}
In formulating the relativistic Euler equations as a first-order hyperbolic system,
we will consider $\Lnenth$, $\Ent$, and $\lbrace u^{\alpha} \rbrace_{\alpha=0,1,2,3}$ 
to be the fundamental unknowns.\footnote{On might argue that it is more accurate to think of $u^0$ as being ``redundant'' in the sense that 
it is algebraically determined in terms of $\lbrace u^a \rbrace_{a=1,2,3}$ via the condition $u^0 > 0$ and the
normalization condition \eqref{E:UISUNITLENGTH}. 
In fact, in most of Sect.\,\ref{S:WELLPOSEDNESS},
we adopt this point of view. However, prior to Sect.\,\ref{S:WELLPOSEDNESS}, we do not adopt this
point of view.}
In terms of these variables
and the quantities defined in \eqref{E:SPEEDOFSOUND},
\eqref{E:UPERPPROJECTION},
and \eqref{E:TEMPOVERENTH},
the relativistic Euler equations are
\begin{align}
	u^{\kappa} \partial_{\kappa} \Lnenth
	+
	\speed^2 \partial_{\kappa} u^{\kappa}
	& = 0,
		\label{E:ENTHALPYEVOLUTION} \\
	u^{\kappa} \partial_{\kappa} u^{\alpha}
	+ 
	\Pi^{\alpha \kappa} \partial_{\kappa} \Lnenth
	-
	\TempoverEnth
	(\upeta^{-1})^{\alpha \kappa} \partial_{\kappa} s
	& = 0,
	\label{E:VELOCITYEVOLUTIONWITHPROJECTION}
		\\
	u^{\kappa} \partial_{\kappa} \Ent
	& = 0.
	\label{E:ENTROPYEVOLUTION}
\end{align}
It is straightforward to see that the following constraint is preserved by the flow of 
equations \eqref{E:VELOCITYEVOLUTIONWITHPROJECTION}-\eqref{E:ENTROPYEVOLUTION}.
\begin{align} \label{E:UISUNITLENGTH}
	u_{\kappa} u^{\kappa}
	& = - 1.
\end{align}

\begin{remark}[More common first-order formulations]
Many authors define the relativistic Euler equations to be the system 
comprising \eqref{E:UISUNITLENGTH}, 
\eqref{E:PARTIACLECURRENTCONSERVATION},
and the four equations $\partial_{\kappa} T^{\alpha \kappa} = 0$,
where $T^{\alpha \beta} := (\uprho + p)u^{\alpha} u^{\beta} + p (\upeta^{-1})^{\alpha \beta}$
is the fluid's energy-momentum tensor.
These equations are in fact equivalent (at least in the case of $C^1$ solutions with $\uprho > 0$) 
to equations \eqref{E:ENTHALPYEVOLUTION}-\eqref{E:UISUNITLENGTH}.
We refer readers to \cite{dC2007}*{Chapter~1} for background material 
that is sufficient for understanding the equivalence.
\end{remark}

Note that \eqref{E:ENTROPYEVOLUTION} is equivalent to
\begin{align} \label{E:VELOCITYANDENTGRADIENTAREMINKOWSKIPERP}
	u^{\kappa} \GradEnt_{\kappa}
	& = 0.
\end{align}

Equation \eqref{E:VELOCITYEVOLUTIONWITHPROJECTION}
can be written more explicitly as
\begin{align} \label{E:VELOCITYEVOLUTION} 
	u^{\kappa} \partial_{\kappa} u_{\alpha}
	+
	\partial_{\alpha} \Lnenth
	+
	u_{\alpha} u^{\kappa} \partial_{\kappa} \Lnenth
	-
	\TempoverEnth \GradEnt_{\alpha}
	& = 0.
\end{align}

Also, from \eqref{E:VELOCITYEVOLUTION},
we easily derive
\begin{align} \label{E:EVOLUTIONFORVELOCITYTIMESENTHALPHY}
	u^{\kappa} \partial_{\kappa} (\Enth u_{\alpha})
	+
	\partial_{\alpha} \Enth
	-
	\Temp \GradEnt_{\alpha}
	& = 0.
\end{align}	

Moreover, differentiating \eqref{E:ENTROPYEVOLUTION} with a rectangular coordinate partial
derivative, we deduce
\begin{align} \label{E:ENTROPYGRADIENTEVOLUTION}
	u^{\kappa} \partial_{\kappa} \GradEnt_{\alpha}
	& = - \GradEnt_{\kappa} (\partial_{\alpha} u^{\kappa}).
\end{align}

In our analysis, we will also use the following evolution equation for $n$:
\begin{align} \label{E:PARTIACLECURRENTCONSERVATION}
	u^{\kappa} \partial_{\kappa} n
	+
	n \partial_{\kappa} u^{\kappa}
	& = 0.
\end{align}
To obtain \eqref{E:PARTIACLECURRENTCONSERVATION}, we first use
equations \eqref{E:ENTHALPYEVOLUTION} and \eqref{E:ENTROPYEVOLUTION},
the thermodynamic relation
$d \Enth  = dp/n + \Temp d \Ent$,
and the relation $\Enth = (\uprho + p)/n$
to deduce
$
u^{\kappa} \partial_{\kappa} p
+
\speed^2 (\uprho + p) \partial_{\kappa} u^{\kappa}
=
0
$.
We then use this equation, 
\eqref{E:SPEEDOFSOUND},
and equation \eqref{E:ENTROPYEVOLUTION}
to deduce
$
u^{\kappa} \partial_{\kappa} \uprho
+
(\uprho + p) \partial_{\kappa} u^{\kappa}
=
0
$.
Next, using this equation
and equation \eqref{E:ENTROPYEVOLUTION},
we deduce
$
\frac{\partial \uprho(n,s)}{\partial n}\left|_{\Ent}\right.
u^{\kappa} \partial_{\kappa} n
+
(\uprho + p) \partial_{\kappa} u^{\kappa}
=
0
$.
Finally, from this equation and the thermodynamic relation
$
\uprho + p = n \frac{\partial \uprho(n,s)}{\partial n}\left|_{\Ent}\right.
$,
we conclude \eqref{E:PARTIACLECURRENTCONSERVATION}.

For future use, we also note that equations
\eqref{E:ENTHALPYEVOLUTION}-\eqref{E:ENTROPYEVOLUTION}
can be written (using \eqref{E:UISUNITLENGTH}) in the form
\begin{align} \label{E:MATRIXVECTORFORMFIRSTORDEREULER}
A^{\alpha} \partial_{\alpha} 
\begin{pmatrix}
	\Lnenth  \\	
	u^0 \\
	u^1  \\
	u^2 \\
	u^3 \\
	\Ent
\end{pmatrix}
	& = 0,
\end{align}
where for $\alpha = 0,1,2,3$, $A^{\alpha}$ is a $6 \times 6$ matrix that is a smooth function
of the solution array $(\Lnenth,u^0,u^1,u^2,u^3,\Ent)$.
In particular, we compute that
\begin{align} \label{E:A0EXPLICIT}
	A^0
	& = 
\begin{pmatrix}
	u^0 & \speed^2 & 0 & 0 & 0 & 0 \\
	u_a u^a & u^0 & 0 & 0 & 0 & \TempoverEnth \\
	u^0 u^1 & 0 & u^0 & 0 & 0 & 0 \\
	u^0 u^2 & 0 & 0 & u^0 & 0 & 0 \\
	u^0 u^3 & 0 & 0 & 0 & u^0 & 0 \\
	0 & 0 & 0 & 0 & 0 & u^0
\end{pmatrix},
\end{align}
and we compute that
\begin{align} \label{E:A0DET}
	\mbox{\upshape det} A^0
	& = (u^0)^6 - \speed^2 (u^0)^4 u_a u^a
		= (1 + u_a u^a)^4
			\left\lbrace
				1  
				+ (1 - \speed^2) u_b u^b
			\right\rbrace.
\end{align}
In particular, in view of \eqref{E:SPEEDOFSOUNDBOUNDS},
we deduce from \eqref{E:A0DET} that $A^0$ is invertible.

\section{The new formulation of the relativistic Euler equations}
\label{S:NEWFORMULATIONRELATIVISTICEULER}
In the next theorem, we provide the main result of the article:
the new formulation of the relativistic Euler equations.

\begin{theorem}[New formulation of the relativistic Euler equations]
\label{T:NEWFORMULATIONRELEULER}
For $C^3$ solutions $(\Lnenth,\Ent,u^{\alpha})$
to the relativistic Euler equations \eqref{E:ENTHALPYEVOLUTION}-\eqref{E:ENTROPYEVOLUTION} + \eqref{E:UISUNITLENGTH}, 
the following equations hold,
where the phrase ``$g$-\textbf{null form}''
refers to a linear combination of the standard $g$-null forms of
Def.\,\ref{D:STANDARDNULLFORMS} with coefficients that are allowed to depend on the quantities
$(\Lnenth,\Ent,u^{\alpha},\GradEnt^{\alpha},\vort^{\alpha})$
(but \textbf{not} their derivatives).

\medskip

\noindent \underline{\textbf{Wave equations}}.
The logarithmic enthalpy $\Lnenth$ verifies the following covariant wave equation
(see Footnote~\ref{FN:COVWAVEOPARBITRARYCOORDS} on pg.\,\pageref{FN:COVWAVEOPARBITRARYCOORDS}
for a formula for the covariant wave operator):
\begin{align}  \label{E:MAINTHMCOVARIANTWAVEENTHALPY}
		\square_g \Lnenth
	&  = 
			n \speed^2 \TempoverEnth \mathcal{D}
			+
			\mathfrak{Q}_{(\Lnenth)}
			+
			\mathfrak{L}_{(\Lnenth)},
	\end{align}
where $\mathfrak{Q}_{(\Lnenth)}$ is the $g$-\textbf{null form} defined by 
\begin{subequations}
\begin{align} \label{E:MAINTHMENTHALPHYWAVEEQNNULLFORM}
	\mathfrak{Q}_{(\Lnenth)}
	& :=
		-
			\speed^{-1}
			\speed_{;\Lnenth}
			(g^{-1})^{\kappa \lambda}
			(\partial_{\kappa} \Lnenth)
			(\partial_{\lambda} \Lnenth)
			\\
			& \ \
		+
		\speed^2 
		\left\lbrace
			(\partial_{\kappa} u^{\kappa}) 
			(\partial_{\lambda} u^{\lambda})
			-
			(\partial_{\lambda} u^{\kappa}) 
			(\partial_{\kappa} u^{\lambda})
		\right\rbrace,
		\notag
\end{align}
and $\mathfrak{L}_{(\Lnenth)}$, which is at most linear in the derivatives of $(\Lnenth,\Ent,u^{\alpha},\GradEnt^{\alpha},\vort^{\alpha})$, is defined by 
\begin{align} \label{E:MAINTHMENTHALPHYWAVEEQNLINEARTERMS}
	\mathfrak{L}_{(\Lnenth)}
	& :=
		\left\lbrace
			(1 - \speed^2) \TempoverEnth 
			+
			\speed^2 \TempoverEnth_{;\Lnenth} 
			-
			\speed \speed_{;\Ent}
		\right\rbrace
		(\GradEnt^{\kappa} \partial_{\kappa} \Lnenth)
		+
		\speed^2 \TempoverEnth_{;\Ent} 
		\GradEnt_{\kappa} \GradEnt^{\kappa}.
\end{align}
\end{subequations}

Moreover, the rectangular four-velocity components\footnote{We stress that on LHS~\eqref{E:MAINTHMCOVARIANTWAVEVELOCITY},
the components $u^{\alpha}$ are treated as scalar functions under the action of the covariant wave operator $\square_g$.} 
$u^{\alpha}$ verify the following covariant wave equations:
\begin{align}  \label{E:MAINTHMCOVARIANTWAVEVELOCITY}
		\square_g u^{\alpha}
		& 
		= 
			-
			\frac{\speed^2}{\Enth} 
			\mathcal{C}^{\alpha}
			+ 
			\mathfrak{Q}_{(u^{\alpha})}
			+
			\mathfrak{L}_{(u^{\alpha})},
\end{align}
where $\mathfrak{Q}_{(u^{\alpha})}$ is the $g$-\textbf{null form} defined by
\begin{subequations}
\begin{align} \label{E:MAINTHMVELOCITYEQNNULLFORM}
	\mathfrak{Q}_{(u^{\alpha})}
	& :=
		(\upeta^{-1})^{\alpha \lambda}
		\left\lbrace
			(\partial_{\kappa} u^{\kappa}) (\partial_{\lambda} \Lnenth)
			- 
			(\partial_{\lambda} u^{\kappa}) (\partial_{\kappa} \Lnenth)
		\right\rbrace
			\\
			& \ \
			+\speed^2 
			u^{\alpha} 
			\left\lbrace
				(\partial_{\kappa} u^{\lambda}) (\partial_{\lambda} u^{\kappa})
				-
				(\partial_{\lambda} u^{\lambda}) (\partial_{\kappa} u^{\kappa}) 
			\right\rbrace
			\notag
		\\
& \ \
			- 
			\left\lbrace
				1 + \speed^{-1} \speed_{;\Lnenth} 
			\right\rbrace
			(g^{-1})^{\kappa \lambda} 
			(\partial_{\kappa} \Lnenth) 
			(\partial_{\lambda} u^{\alpha}),
			\notag
\end{align}
and $\mathfrak{L}_{(u^{\alpha})}$, which is at most linear in the derivatives of $(\Lnenth,\Ent,u^{\alpha},\GradEnt^{\alpha},\vort^{\alpha})$, is defined by 
\begin{align} \label{E:MAINTHMVELOCITYWAVEEQNLINEARTERMS}
	\mathfrak{L}_{(u^{\alpha})}
	& :=
			-
			\frac{\speed^2}{\Enth} 
			\upepsilon^{\alpha \beta \gamma \delta} 
			(\partial_{\beta} u_{\gamma}) 
			\vort_{\delta}
			+
			\frac{(1 - \speed^2)}{\Enth}
			\upepsilon^{\alpha \beta \gamma \delta} 
			u_{\beta}
			(\partial_{\gamma} \Lnenth) \vort_{\delta}
			\\
			& \ \
			+
			\frac{(1 - \speed^2) \TempoverEnth }{\Enth} 
			\upepsilon^{\alpha \beta \gamma \delta} 
			\GradEnt_{\beta} 
			u_{\gamma} 
			\vort_{\delta}	
			\notag
				\\
		& \ \
		+
		\left\lbrace
			\TempoverEnth
			-
			\speed \speed_{;\Ent}
		\right\rbrace
		(\GradEnt^{\kappa} \partial_{\kappa} u^{\alpha})
		+
		\TempoverEnth
		(\speed^2 - 1) 
		u^{\alpha}
		\GradEnt^{\kappa} 
		(u^{\lambda} \partial_{\lambda} u_{\kappa})
			\notag \\
		& \ \
				+
				\GradEnt^{\kappa}
				\left\lbrace
					\speed^2
					\TempoverEnth 
					+
				\frac{(\Temp - \Temp_{;\Lnenth}) \speed^2}{\Enth}
				\right\rbrace
				((\upeta^{-1})^{\alpha \lambda} \partial_{\lambda} u_{\kappa})  
				\notag
					\\
		& \ \
		+
		\left\lbrace
		2 \speed^{-1} \speed_{;\Lnenth} \TempoverEnth 
		\GradEnt^{\alpha}
		+ 
		2 \speed^{-1} \speed_{;\Ent} 
		\GradEnt^{\alpha}
		-
		\TempoverEnth_{;\Lnenth}
		\GradEnt^{\alpha} 
		\right\rbrace
		(u^{\kappa} \partial_{\kappa} \Lnenth) 
		\notag
			\\
	& \ \
		+
		\GradEnt^{\alpha}
		\left\lbrace
			\frac{(\Temp - \Temp_{;\Lnenth}) \speed^2}{\Enth}
			-
			\TempoverEnth 
		\right\rbrace
		(\partial_{\kappa} u^{\kappa})
		+
		\frac{(\Temp - \Temp_{;\Lnenth}) \speed^2}{\Enth}
		u^{\alpha}
		(\GradEnt^{\kappa} \partial_{\kappa} \Lnenth).
		\notag
\end{align}
\end{subequations}

\medskip

\noindent \underline{\textbf{Auxiliary wave equation for $\Ent$}}.
The entropy $\Ent$ verifies the following covariant wave equation:\footnote{The wave equation \eqref{E:MAINTHMENTAUXILIARYWAVEQUATION} is auxiliary 
in the sense that we do not use it in our proof of
Theorem~\ref{T:UPGRADEDLOCALWELLPOSEDNESS}. However, in applications (for example, in the study of shock formation), 
one has to compute $\square_g$ applied to the scalar component functions $g_{\alpha \beta}$,
and, by virtue of the chain rule, the quantity $\square_g \Ent$ arises in such computations.
It is for this reason that we have included equation \eqref{E:MAINTHMENTAUXILIARYWAVEQUATION} in this paper.
\label{FN:JUSTIFICATIONENTROPYWAVEEQUATION} }
\begin{align} \label{E:MAINTHMENTAUXILIARYWAVEQUATION}
	\square_g \Ent
	& = 	
		\speed^2 n \mathcal{D}
		+
		\mathfrak{L}_{(\Ent)},
\end{align}
where $\mathfrak{L}_{(\Ent)}$, which is at most linear in the derivatives of $(\Lnenth,\Ent,u^{\alpha},\GradEnt^{\alpha},\vort^{\alpha})$, is defined by 
\begin{align} \label{E:MAINTHMENTAUXILIARYWAVEQUATIONLINEARTERMS}
	\mathfrak{L}_{(\Ent)}
	& :=
		\left\lbrace
		1
		- 
		\speed^2
		- 
		\speed \speed_{;\Lnenth} 
		\right\rbrace
		(\GradEnt^{\kappa} \partial_{\kappa} \Lnenth)
		-
		\speed \speed_{;\Ent}
		\GradEnt_{\kappa} \GradEnt^{\kappa}.
\end{align}

\medskip

\noindent \underline{\textbf{Transport equations}}.
The rectangular components of the entropy gradient vectorfield
$\GradEnt^{\alpha}$, whose $\upeta$-dual is defined in \eqref{E:ENTROPYGRADIENT},
verify the following transport equations:
\begin{align} \label{E:MAINTHMENTROPYGRADIENTEVOLUTION}
	u^{\kappa} \partial_{\kappa} \GradEnt^{\alpha}
	& = - 
	\GradEnt_{\kappa}
	((\upeta^{-1})^{\alpha \lambda} \partial_{\lambda} u^{\kappa}).
\end{align}

Moreover, the rectangular components of the vorticity vectorfield
$\vort^{\alpha}$, which is defined in \eqref{E:VORTICITYDEF},
verify the following transport equations:
\begin{align}	\label{E:MAINTHMVORTICITYTRANSPORT}
		u^{\kappa} \partial_{\kappa} \vort^{\alpha} 
		& = 
			-	
			u^{\alpha} (\vort^{\kappa} \partial_{\kappa} \Lnenth) 
			+
			\vort^{\kappa} \partial_{\kappa} u^{\alpha}
			-
			\vort^{\alpha} (\partial_{\kappa} u^{\kappa}) 
			\\
			& \ \ 
			+
			(\Temp - \Temp_{;\Lnenth}) 
			\upepsilon^{\alpha \beta \gamma \delta} 
			u_{\beta}
			(\partial_{\gamma} \Lnenth) 
			\GradEnt_{\delta}
			+
			\TempoverEnth 
			u^{\alpha}
		 \vort^{\kappa} \GradEnt_{\kappa}.
		 \notag
	\end{align}

\medskip

\noindent \underline{\textbf{$\mbox{\upshape Transport}$-$\mbox{\upshape div}$-$\mbox{\upshape curl}$ systems}}.
The modified divergence of the entropy gradient 
$\mathcal{D}$ (which is defined in \eqref{E:MODIFIEDDIVERGENCEOFENTROPY})
and the rectangular components
$\uperpvort^{\alpha}(\GradEnt)$
of the $u$-orthogonal vorticity
of the entropy gradient vectorfield
(see definition \eqref{E:UORTHGONALVORTICITYOFONEFORM})
verify the following $\mbox{\upshape transport}$-$\mbox{\upshape div}$-$\mbox{\upshape curl}$ system:
	\begin{subequations}
	\begin{align}	\label{E:MAINTHMSMODIFIEDDIVERGENCEENTROPYGRADIENTTRANSPORT}
		u^{\kappa} \partial_{\kappa} \mathcal{D}
		& = 
			\frac{2}{n}
			\left\lbrace
				(\partial_{\kappa} \GradEnt^{\kappa}) (\partial_{\lambda} u^{\lambda})
				-
				(\partial_{\lambda} \GradEnt^{\kappa}) (\partial_{\kappa} u^{\lambda})
			\right\rbrace
			\\
			& \ \
			+
				\frac{1}{n}
				\speed^{-2}
				u^{\kappa}
				\left\lbrace
					(\partial_{\kappa} \Lnenth) 
					(\partial_{\lambda} \GradEnt^{\lambda})
					-
					(\partial_{\lambda} \Lnenth)
					(\partial_{\kappa} \GradEnt^{\lambda})
				\right\rbrace	
				\notag
					\\
		& \ \
			+
			\frac{\GradEnt_{\kappa} \mathcal{C}^{\kappa}}{n \Enth}
			+
			\mathfrak{Q}_{(\mathcal{D})}
			+
			\mathfrak{L}_{(\mathcal{D})},
			\notag
				\\
		\uperpvort^{\alpha}(\GradEnt)
		& = 0,
		\label{E:MAINTHMUPERPCURLOFENTROPYGRADIENTISZERO}
	\end{align}
	\end{subequations}
	where $\mathfrak{Q}_{(\mathcal{D})}$ is the $g$-\textbf{null form} defined by
	\begin{subequations}
	\begin{align} \label{E:MAINTHMENTROPYGRADIENTDIVERGENCENULLFORM}
		\mathfrak{Q}_{(\mathcal{D})}
		& := 
		\frac{1}{n}
					\speed^{-2} 
					\GradEnt^{\kappa}
					\left\lbrace
						(\partial_{\kappa} u^{\lambda}) (\partial_{\lambda} \Lnenth)
						-
						(\partial_{\lambda} u^{\lambda}) (\partial_{\kappa} \Lnenth)
					\right\rbrace,
	\end{align}
	and $\mathfrak{L}_{(\mathcal{D})}$, which is linear in the derivatives of $(\Lnenth,\Ent,u^{\alpha},\GradEnt^{\alpha},\vort^{\alpha})$, is defined by 
	\begin{align} \label{E:MAINTHMENTROPYGRADIENTDIVERGENCELINEARTERMS}
		\mathfrak{L}_{(\mathcal{D})}
		& :=
			\frac{(1 - \speed^{-2})}{n \Enth}
			\upepsilon^{\alpha \beta \gamma \delta} 
			\GradEnt_{\alpha}
			u_{\beta}
			(\partial_{\gamma} \Lnenth) \vort_{\delta}
			+
			\frac{1}{n \Enth} 
			\upepsilon^{\alpha \beta \gamma \delta} 
			\GradEnt_{\alpha} (\partial_{\beta} u_{\gamma}) \vort_{\delta}
			\\
		& \ \
			+
			\frac{\GradEnt^{\kappa} \GradEnt^{\lambda}}{n}
			\left\lbrace
				\frac{(\Temp - \Temp_{;\Lnenth})}{\Enth}
				-
				2 \TempoverEnth
			\right\rbrace
			(\partial_{\kappa} u_{\lambda})
			\notag \\
		& \ \
			+
			\frac{\GradEnt_{\kappa} \GradEnt^{\kappa}}{n}
			\left\lbrace
				\frac{(\Temp_{;\Lnenth} - \Temp) }{\Enth}
				+
				2 \speed^{-1} \speed_{;\Ent}
				-
				\speed^2
				\TempoverEnth_{;\Lnenth} 
				+
				\TempoverEnth
			\right\rbrace
			(\partial_{\lambda} u^{\lambda}).
			\notag
	\end{align}
	\end{subequations}
	
Finally, the divergence of the vorticity vectorfield $\vort^{\alpha}$ 
(which is defined in \eqref{E:VORTICITYDEF})
and the rectangular components $\mathcal{C}^{\alpha}$ of the modified vorticity of the vorticity
(which is defined in \eqref{E:MODIFIEDVORTICITYOFVORTICITY})
verify the following equations:
\begin{subequations}
\begin{align}  \label{E:MAINTHMDIVOFVORTICITY}
	\partial_{\alpha} \vort^{\alpha}
	& = 
		- 
		\vort^{\kappa} \partial_{\kappa} \Lnenth
		+
		2 \TempoverEnth \vort^{\kappa} \GradEnt_{\kappa},
			\\
	u^{\kappa} \partial_{\kappa} \mathcal{C}^{\alpha}
	& = 
		\mathcal{C}^{\kappa} \partial_{\kappa} u^{\alpha}
		- 
		2 \mathcal{C}^{\alpha} (\partial_{\kappa} u^{\kappa}) 
		+
		u^{\alpha} 
		\mathcal{C}^{\kappa}
		(u^{\lambda} \partial_{\lambda} u_{\kappa}) 
		\label{E:MAINTHMTRANSPORTFORMODIFIEDVORTICITYOFVORTICITY}		\\
		& \ \
			-
			2 \upepsilon^{\alpha \beta \gamma \delta} 
			u_{\beta}
			(\partial_{\gamma} \vort^{\kappa}) (\partial_{\delta} u_{\kappa})
		\notag \\
		& \ \
			+
		(\Temp_{;\Lnenth} - \Temp) 
		\left\lbrace
			(\upeta^{-1})^{\alpha \kappa}
			+
			2 u^{\alpha} u^{\kappa}
		\right\rbrace
			\left\lbrace
				(\partial_{\kappa} \Lnenth) (\partial_{\lambda} \GradEnt^{\lambda})
				-
				(\partial_{\lambda} \Lnenth) (\partial_{\kappa} \GradEnt^{\lambda})
			\right\rbrace
			\notag \\
		& \ \
		+
		(\Temp - \Temp_{;\Lnenth}) n u^{\alpha} (u^{\kappa} \partial_{\kappa} \Lnenth) \mathcal{D} 
			\notag
			\\
		& \ \
		+
		(\Temp - \Temp_{;\Lnenth}) 
		\TempoverEnth 
		\GradEnt^{\alpha} 
		(\partial_{\kappa} \GradEnt^{\kappa})
		+
		(\Temp_{;\Lnenth} - \Temp) 
		\TempoverEnth 
		\GradEnt_{\kappa}
		((\upeta^{-1})^{\alpha \lambda} \partial_{\lambda} \GradEnt^{\kappa}) 
		\notag \\
	& \ \
		+ 
		\mathfrak{Q}_{(\mathcal{C}^{\alpha})}
		+
		\mathfrak{L}_{(\mathcal{C}^{\alpha})},
		\notag
	\end{align}
	\end{subequations}
where $\mathfrak{Q}_{(\mathcal{C}^{\alpha})}$ is the $g$-\textbf{null form} defined by 
\begin{subequations}
\begin{align} \label{E:MAINTHMVORTICITYOFVORTICITYEASYNULLFORMS}
	\mathfrak{Q}_{(\mathcal{C}^{\alpha})}
	& := 
			-
			\speed^{-2}
			\upepsilon^{\kappa \beta \gamma \delta} 
			(\partial_{\kappa} u^{\alpha})
			u_{\beta}
			(\partial_{\gamma} \Lnenth) \vort_{\delta}
		\\
		& \ \
		+
		(\speed^{-2} + 2)
		\upepsilon^{\alpha \beta \gamma \delta} 
		u_{\beta}
		(\partial_{\gamma} \Lnenth) \vort^{\kappa} (\partial_{\delta} u_{\kappa})	
		\notag
			\\
	& \ \
	+
	\speed^{-2}
	\upepsilon^{\alpha \beta \gamma \delta} 
	u_{\beta}
	\vort_{\delta}
	\left\lbrace
		(\partial_{\kappa} u^{\kappa}) (\partial_{\gamma} \Lnenth)
		- 
		(\partial_{\gamma} u^{\kappa}) (\partial_{\kappa} \Lnenth)
	\right\rbrace
	\notag
			\\
	& \ \
		+
	\left\lbrace	
		(\Temp_{;\Lnenth;\Lnenth} - \Temp_{\Lnenth}) 
		+
		\speed^{-2} 
		(\Temp - \Temp_{;\Lnenth})
	\right\rbrace
	u^{\kappa}
	(\upeta^{-1})^{\alpha \lambda}
	\GradEnt^{\beta}
		\notag \\
	&
	\ \ \ \ \
	\times
	\lbrace
		(\partial_{\kappa} \Lnenth)
		(\partial_{\lambda} u_{\beta})
		-
		(\partial_{\lambda} \Lnenth)
		(\partial_{\kappa} u_{\beta})
	\rbrace
	\notag
		\\
	& \ \
		+
		(\Temp_{;\Lnenth} - \Temp) 
		\GradEnt^{\kappa}
		u^{\lambda}
		\left\lbrace
			(\partial_{\kappa} u^{\alpha})
			(\partial_{\lambda} \Lnenth) 
			-
			(\partial_{\lambda} u^{\alpha})
			(\partial_{\kappa} \Lnenth)
		\right\rbrace
		\notag
		\\
	& \ \
		+
	(\Temp_{;\Lnenth} - \Temp) 
	\left\lbrace
		(\upeta^{-1})^{\alpha \kappa} 
		+
		u^{\alpha} u^{\kappa}
	\right\rbrace
	\GradEnt^{\beta}
	\left\lbrace
		(\partial_{\kappa} u_{\beta})
		(\partial_{\lambda} u^{\lambda})
		- 
		(\partial_{\lambda} u_{\beta}) 
		(\partial_{\kappa} u^{\lambda})
	\right\rbrace
	\notag
		\\
	& \ \
		+
		(\Temp_{;\Lnenth} - \Temp) 
		\GradEnt^{\alpha} 
		\left\lbrace
			(\partial_{\kappa} u^{\lambda}) \partial_{\lambda} u^{\kappa} 
				-
			(\partial_{\lambda} u^{\lambda})
			(\partial_{\kappa} u^{\kappa})
		\right\rbrace
				\notag \\
	& \ \
		+
		(\Temp_{;\Lnenth} - \Temp)
			\GradEnt^{\kappa}
			\left\lbrace
				(\partial_{\kappa} u^{\alpha})
				(\partial_{\lambda} u^{\lambda})
				-	 
				 (\partial_{\lambda} u^{\alpha})
				 (\partial_{\kappa} u^{\lambda})
			\right\rbrace
			\notag \\
& \ \
	+
	\GradEnt^{\alpha}
		\left\lbrace
			\speed^{-2}
			(\Temp_{\Lnenth} - \Temp_{;\Lnenth;\Lnenth})
			+
			\speed^{-4} 
			(\Temp_{;\Lnenth} - \Temp)
		\right\rbrace
		(g^{-1})^{\kappa \lambda} (\partial_{\kappa} \Lnenth) (\partial_{\lambda} \Lnenth),
		\notag 
\end{align}
and $\mathfrak{L}_{(\mathcal{C}^{\alpha})}$, which is linear in the derivatives of $(\Lnenth,\Ent,u^{\alpha},\GradEnt^{\alpha},\vort^{\alpha})$, is defined by 
\begin{align} \label{E:MAINTHMVORTICITYOFVORTICITYLINEARTERMS}
	\mathfrak{L}_{(\mathcal{C}^{\alpha})}
	& := 
		\frac{2 \TempoverEnth}{\Enth}
				(\vort^{\kappa} \GradEnt_{\kappa} \vort^{\alpha})
		-
		\frac{2}{\Enth} \vort^{\alpha} (\vort^{\kappa} \partial_{\kappa} \Lnenth) 
			\\
	& \ \
		+
		2 \speed^{-3} \speed_{;\Ent} 
	\upepsilon^{\alpha \beta \gamma \delta} 
	u_{\beta}
	\GradEnt_{\gamma}
	\vort_{\delta}
	(u^{\kappa} \partial_{\kappa} \Lnenth)
	\notag
		\\
	& \ \
			-
			2 \TempoverEnth
			\upepsilon^{\alpha \beta \gamma \delta}
			u_{\beta}
			\GradEnt_{\gamma} 
			\vort^{\kappa} (\partial_{\delta} u_{\kappa})
			-
		\TempoverEnth
		\upepsilon^{\alpha \beta \gamma \delta} 
		\GradEnt_{\beta}
		u_{\gamma}
		\vort_{\delta} 
		(\partial_{\kappa} u^{\kappa})
			\notag
			\\
	& \ \
	+
	\frac{1}{\Enth} 
	(\Temp - \Temp_{;\Lnenth}) 
	\upepsilon^{\kappa \beta \gamma \delta} 
	(\partial_{\kappa} u^{\alpha})
	\GradEnt_{\beta} u_{\gamma} \vort_{\delta}
	+
		\speed^{-2}
		\TempoverEnth
		\upepsilon^{\alpha \beta \gamma \delta} 
		\GradEnt_{\beta}
		(\partial_{\gamma} \Lnenth) \vort_{\delta}
		\notag
		\\
		& \ \
	-
			\speed^{-2}
			\TempoverEnth
			u^{\alpha}
			\upepsilon^{\kappa \beta \gamma \delta} 
			\GradEnt_{\kappa}
			u_{\beta}
			(\partial_{\gamma} \Lnenth) \vort_{\delta}
		\notag
		\\
& \ \
	+ 
	(\Temp_{;\Lnenth} - \Temp) 
	\TempoverEnth
	\GradEnt_{\kappa} \GradEnt^{\kappa}
	(u^{\lambda} \partial_{\lambda} u^{\alpha})
	\notag
		\\
& \ \
	+
	u^{\alpha}
	\GradEnt_{\kappa} \GradEnt^{\kappa}
	\left\lbrace
		(\Temp_{;\Lnenth} - \Temp) 
		\TempoverEnth
		+
		(\Temp_{;\Lnenth;\Ent} - \Temp_{;\Ent})
	\right\rbrace
	(u^{\lambda} \partial_{\lambda} \Lnenth) 
		\notag
		\\
	& \ \
			+
		\GradEnt^{\alpha} 
			\left\lbrace
				(\Temp_{;\Ent} - \Temp_{;\Lnenth;\Ent})
				+
				(\Temp - \Temp_{;\Lnenth}) 
				\TempoverEnth_{;\Lnenth} 
			\right\rbrace
			(\GradEnt^{\kappa} \partial_{\kappa} \Lnenth)
							\notag
				\\
	& \ \
		+
	\GradEnt_{\kappa} \GradEnt^{\kappa}
	\lbrace
		(\Temp_{;\Lnenth;\Lnenth} - \Temp_{\Lnenth})
		\TempoverEnth
		+
		(\Temp_{;\Lnenth;\Ent} - \Temp_{;\Ent})
		+
		(\Temp - \Temp_{;\Lnenth})
		\TempoverEnth 
		\speed^{-2} 
		+
		(\Temp_{;\Lnenth} - \Temp) \TempoverEnth_{;\Lnenth} 
	\rbrace
		\notag \\
	& 
		\ \ \ \ \ \ \ \ \ \ \ 
		\times
		((\upeta^{-1})^{\alpha \lambda} \partial_{\lambda} \Lnenth).
		\notag 
\end{align}
\end{subequations}
	
\end{theorem}

\begin{remark}[Special structure of the inhomogeneous terms]
	\label{R:SPECIALSTRUCTUREINHOMOGENEOUSTERMS}
	We emphasize the following two points, which are of crucial importance for applications to shock waves
	(see Subsubsect.\,\ref{SSS:NONLINEARNULL} for further discussion):
	\textbf{i)} all inhomogeneous terms on the RHSs of the equations of the theorem
	are at most quadratic in the derivatives of 
	$(\Lnenth,\Ent,u^{\alpha},\GradEnt^{\alpha},\vort^{\alpha})$
	and \textbf{ii)}
	\textbf{all derivative-quadratic terms on the RHSs of the equations of the theorem
	are linear combinations of standard $g$-null forms}. 
	In particular, the following are linear combinations of standard $g$-null forms, even though we did not explicitly state
	so in the theorem: the terms on the first and second lines
	of RHS~\eqref{E:MAINTHMSMODIFIEDDIVERGENCEENTROPYGRADIENTTRANSPORT}
	and the terms on the second and third lines of RHS~\eqref{E:MAINTHMTRANSPORTFORMODIFIEDVORTICITYOFVORTICITY}.
	We have separated these null forms, which involve the derivatives of $\vort$ and $\GradEnt$, 
	because they need to be handled with elliptic estimates, at least at the top derivative level
	(see the proof of Theorem~\ref{T:UPGRADEDLOCALWELLPOSEDNESS}).
	This is different compared to the terms $\mathfrak{Q}_{(\Lnenth)}$, $\mathfrak{Q}_{(u^{\alpha})}$, $\mathfrak{Q}_{(\mathcal{D})}$, 
	and $\mathfrak{Q}_{(\mathcal{C}^{\alpha})}$,
	which can be handled with standard energy estimates at all derivative levels.
\end{remark}

\begin{proof}
	Theorem~\ref{T:NEWFORMULATIONRELEULER} follows from a lengthy series of calculations, 
	most of which we derive later in the paper,
	except that we have somewhat reorganized 
	(using only simple algebra)
	the terms on the right-hand sides of the equations of the theorem.
	More precisely, we prove \eqref{E:MAINTHMCOVARIANTWAVEENTHALPY}-\eqref{E:MAINTHMENTHALPHYWAVEEQNLINEARTERMS} in
	Prop.\,\ref{P:COVARIANTWAVEENTHALPY}.
	
	We prove \eqref{E:MAINTHMCOVARIANTWAVEVELOCITY}-\eqref{E:MAINTHMVELOCITYWAVEEQNLINEARTERMS}
	in Prop.\,\ref{P:COVARIANTWAVEVELOCITY}.
	
	We prove \eqref{E:MAINTHMENTAUXILIARYWAVEQUATION}-\eqref{E:MAINTHMENTAUXILIARYWAVEQUATIONLINEARTERMS}
	in Prop.\,\ref{P:COVARIANTWAVEENT}.
	
	Equation \eqref{E:MAINTHMENTROPYGRADIENTEVOLUTION} follows from raising the indices of \eqref{E:ENTROPYGRADIENTEVOLUTION}
	with the inverse Minkowski metric.
	
	We prove \eqref{E:MAINTHMVORTICITYTRANSPORT}
	in Prop.\,\ref{P:VORTICITYTRANSPORT}.
	
	Except for \eqref{E:MAINTHMUPERPCURLOFENTROPYGRADIENTISZERO},
	\eqref{E:MAINTHMSMODIFIEDDIVERGENCEENTROPYGRADIENTTRANSPORT}-\eqref{E:MAINTHMENTROPYGRADIENTDIVERGENCELINEARTERMS}
	follow from Prop.\,\ref{P:DIVERGENCEOFENTROPYGRADIENTTRANSPORT}.
	\eqref{E:MAINTHMUPERPCURLOFENTROPYGRADIENTISZERO} is a simple consequence of definition
	\eqref{E:UORTHGONALVORTICITYOFONEFORM}
	and the symmetry property 
	$\partial_{\alpha} \GradEnt_{\beta} = \partial_{\beta} \GradEnt_{\alpha}$
	(see \eqref{E:ENTSYMMETRYOFMIXEDPARTIALS}).
	
	Finally, we prove
	\eqref{E:MAINTHMDIVOFVORTICITY}-\eqref{E:MAINTHMVORTICITYOFVORTICITYLINEARTERMS}
	in Prop.\,\ref{P:EQUATIONSFORMODIFIEDVORTICITYVORTICITY}.
\end{proof}

\section{Preliminary identities}
\label{S:PRELIMINID}
In the next lemma, we derive some preliminary identities that we will later use when
deriving the equations stated in Theorem~\ref{T:NEWFORMULATIONRELEULER}.

\begin{lemma}[Some useful identities]
	\label{L:IDENTITIES}
	Assume that $(\Lnenth,\Ent,u^{\alpha})$ is a $C^2$ solution 
	to \eqref{E:ENTHALPYEVOLUTION}-\eqref{E:ENTROPYEVOLUTION} + \eqref{E:UISUNITLENGTH},
	and let $V_{\alpha}$ be any $C^1$ one-form.
	Then the following identities hold:
	\begin{align} \label{E:ENTSYMMETRYOFMIXEDPARTIALS}
		\partial_{\alpha} \GradEnt_{\beta}
		& = \partial_{\beta} \GradEnt_{\alpha},
	\end{align}
	
	\begin{align} \label{E:VORTISORTHGONALTOU}
		\vort^{\kappa} u_{\kappa} = 0,
	\end{align}
	
	\begin{align} \label{E:DERIVATIVEOFVELOCITYCONTRACTEDWITHVELOCITYISZERO}
		u^{\kappa} \partial_{\alpha} u_{\kappa}
		& = 0,
	\end{align}
	
	\begin{align} \label{E:TRANSFERDERIVATIVESFROMENTROPYGRADIENTTOVORTICITY}
		u^{\kappa} \partial_{\alpha} \GradEnt_{\kappa}
		& =
		-
		\GradEnt^{\kappa} \partial_{\alpha} u_{\kappa},
	\end{align}
	
		\begin{align} \label{E:TRANSFERDERIVATIVESFROMVORTICITYTOVELOCITY}
			u^{\kappa} \partial_{\alpha} \vort_{\kappa}
			& =
			-
			\vort^{\kappa} \partial_{\alpha} u_{\kappa},
		\end{align}

	\begin{align} \label{E:COORDINATEDERIVATIVEINTERMSOFUDERIVAIVESANDORTHOGONALDERIVATIVES}
		\partial_{\alpha}
		& =
			-
			u_{\alpha} u^{\kappa} \partial_{\kappa} 
			+
			\Pi_{\alpha}^{\ \kappa} \partial_{\kappa},
	\end{align}
	
	\begin{align} \label{E:DIVOFVINTERMSOFUDERIVAIVESANDORTHOGONALDIVERGENCE}
		\partial_{\kappa} V^{\kappa}
		& = 
		-
		u_{\kappa} u^{\lambda} \partial_{\lambda} V^{\kappa}
		+
		\Pi^{\kappa \lambda} \partial_{\kappa} V_{\lambda},
	\end{align}
	
	\begin{align} \label{E:ANTISYMMETRICPARTOFGRADIENTVINTERMSOFVORTICITYANDUPARTOFV}
			\partial_{\alpha} V_{\beta}
			 - 
			\partial_{\beta} V_{\alpha}
				& = 
				\upepsilon_{\alpha \beta \gamma \delta} u^{\gamma} \uperpvort^{\delta}(V)
				+
				u_{\alpha} u^{\kappa} \partial_{\beta} V_{\kappa}
				-
				u_{\beta} u^{\kappa} \partial_{\alpha} V_{\kappa}
				\\
				& \ \ +
				u_{\beta} u^{\kappa} \partial_{\kappa} V_{\alpha}
				-
				u_{\alpha} u^{\kappa} \partial_{\kappa} V_{\beta},
				\notag
	\end{align}
	
	\begin{align} \label{E:UPERPNORMSQUAREDOFANTISYMMETRICPARTINTERMSOFSQUARENORMOFVORTICITY}
		\Pi^{\alpha \beta} \Pi^{\gamma \delta} 
			(\partial_{\alpha} V_{\gamma} - \partial _{\gamma} V_{\alpha}) 
			(\partial_{\beta} V_{\delta} - \partial_{\delta} V_{\beta})
		& = 
		2 \Pi^{\alpha \beta} \uperpvort_{\alpha}(V) \uperpvort_{\beta}(V).
	\end{align}
	
	Moreover, if $u^{\kappa} V_{\kappa} = 0$, then
	\begin{align} \label{E:VANDUAREORTHOGONALANTISYMMETRICPARTOFGRADIENTVINTERMSOFVORTICITYANDUPARTOFV}
			\partial_{\alpha} V_{\beta}
			 - 
			\partial_{\beta} V_{\alpha}
				& = 
				\upepsilon_{\alpha \beta \gamma \delta} u^{\gamma} \uperpvort^{\delta}(V)
				-
				u_{\alpha} V_{\kappa} \partial_{\beta} u^{\kappa}
				+
				u_{\beta} V_{\kappa} \partial_{\alpha} u^{\kappa} 
				\\
				& \ \
				+
				u_{\beta} u^{\kappa} \partial_{\kappa} V_{\alpha}
				-
				u_{\alpha} u^{\kappa} \partial_{\kappa} V_{\beta}.
				\notag
	\end{align}
	
	In addition, the following identity holds,
	where the indices for $\epsilon$ on LHS~\eqref{E:LIEUVOLUMEFORMINDICESUP} are raised before Lie differentiation:
	\begin{align}
		\mathcal{L}_u (\upepsilon^{\alpha \beta \gamma \delta})
		& = (-\partial_{\kappa} u^{\kappa}) \upepsilon^{\alpha \beta \gamma \delta}.
		 \label{E:LIEUVOLUMEFORMINDICESUP}	
	\end{align}

	Furthermore, the following identities hold:
	\begin{align}
		\mathcal{L}_u  (u_{\flat})_{\alpha}
		& = u^{\kappa} \partial_{\kappa} u_{\alpha}
			= 
			- 
			\partial_{\alpha} \Lnenth
			-
			u_{\alpha} u^{\kappa} \partial_{\kappa} \Lnenth
			+
			\TempoverEnth \GradEnt_{\alpha},
				\label{E:LIEUVELOCITYONEFORM}
	\end{align}
	
	\begin{align} \label{E:LIEUOFEXTERIORDERIVATIVEOFENTHALPHYTIMESVELOCITYONEFORMID}
		\mathcal{L}_u d (\Enth u_{\flat})
		& = d \mathcal{L}_u (\Enth u_{\flat}),
	\end{align}
	
	\begin{align} \label{E:LIEUOFEXTERIORDERIVATIVEOFENTHALPHYTIMESVELOCITYONEFORMEVOLUTION}
		[\mathcal{L}_u d (\Enth u_{\flat})]_{\alpha \beta}
		& =
		\Temp_{;\Lnenth}(\partial_{\alpha} \Lnenth) \partial_{\beta} \Ent
		-
		\Temp_{;\Lnenth} (\partial_{\alpha} \Ent) \partial_{\beta} \Lnenth,
	\end{align}
	
	\begin{align} \label{E:IDANTISYMMETRICPARTOFGRADIENTOFENTHALPHYTIMESVELOCITY}
		\partial_{\alpha} (\Enth u_{\beta})
		-
		\partial_{\beta} (\Enth u_{\alpha})
		& 
		= 
		\upepsilon_{\alpha \beta \gamma \delta} u^{\gamma} \vort^{\delta}
		+
		\Temp 
		\left\lbrace
			\GradEnt_{\alpha} u_{\beta}  
			- 
			\GradEnt_{\beta} u_{\alpha} 
		\right\rbrace,
	\end{align}
		
	\begin{align} \label{E:CONTRACTIONVOLUMEFORMANDEXTERIORDERIVATIVEOFENTHALPHYTIMESVELOCITY}
	\upepsilon^{\alpha \beta \gamma \delta}
	\partial_{\gamma} (\Enth u_{\delta})
	& 
	= \vort^{\alpha} u^{\beta} - u^{\alpha} \vort^{\beta}
	+
	\Temp
	\upepsilon^{\alpha \beta \gamma \delta}
	\GradEnt_{\gamma} u_{\delta}, 
\end{align}

	\begin{align} \label{E:PARITALALPHAUBETAMINUSPARTIALBETAUALPHA}
		\partial_{\alpha} u_{\beta}
		-
		\partial_{\beta} u_{\alpha}
		& 
		= \frac{1}{\Enth} \upepsilon_{\alpha \beta \gamma \delta} u^{\gamma} \vort^{\delta}
			- 
			(\partial_{\alpha} \Lnenth) u_{\beta}
			+ 
			(\partial_{\beta} \Lnenth) u_{\alpha}
		\\
		& \ \
			+
		\TempoverEnth 
		\left\lbrace
			\GradEnt_{\alpha} u_{\beta}  
			- 
			\GradEnt_{\beta} u_{\alpha} 
		\right\rbrace,
		\notag
	\end{align}

	\begin{align} \label{E:IDVELOCITYDERIVATIVEOFVELOCITYCONTRACTEDWITHENTROPYGRADIENT}
		(u^{\kappa} \partial_{\kappa} u_{\lambda}) \GradEnt^{\lambda}
		& = - \GradEnt^{\kappa} \partial_{\kappa} \Lnenth
			+
			\TempoverEnth \GradEnt_{\kappa} \GradEnt^{\kappa},
	\end{align}
	
	\begin{align} \label{E:IDVELOCITYDERIVATIVEOFENTROPYGRADIENTCONTRACTEDWITHVELOCITY}
		(u^{\kappa} \partial_{\kappa} \GradEnt_{\lambda}) u^{\lambda}
		& = \GradEnt^{\kappa} \partial_{\kappa} \Lnenth
			-
			\TempoverEnth \GradEnt_{\kappa} \GradEnt^{\kappa},
	\end{align}

	\begin{align} \label{E:ENTROPYGRADIENTCONTRACTEDAGAINSTSECONDSLOTPARITALALPHAUBETAMINUSPARTIALBETAUALPHA}
		\GradEnt^{\kappa} \partial_{\alpha} u_{\kappa}
		& 
		= 
			\GradEnt^{\kappa} \partial_{\kappa} u_{\alpha}
			+ 
			(\GradEnt^{\kappa} \partial_{\kappa} \Lnenth) u_{\alpha}
			-
			\TempoverEnth 
			\GradEnt^{\kappa} \GradEnt_{\kappa} u_{\alpha} 
			+
			\frac{1}{\Enth} \upepsilon_{\alpha \kappa \gamma \delta} \GradEnt^{\kappa} u^{\gamma} \vort^{\delta}
				\\
		& =
			\GradEnt^{\kappa} \partial_{\kappa} u_{\alpha}
			- 
			(\GradEnt^{\kappa} u^{\lambda} \partial_{\lambda} u_{\kappa}) u_{\alpha}
			+
			\frac{1}{\Enth} \upepsilon_{\alpha \beta \gamma \delta} \GradEnt^{\beta} u^{\gamma} \vort^{\delta},
			\notag
	\end{align}

	\begin{align} \label{E:VORTICITYCONTRACTEDPARITALALPHAUBETAMINUSPARTIALBETAUALPHA}
			\vort^{\kappa} \partial_{\kappa} u_{\alpha}
			& 
		= 
			\vort^{\kappa} \partial_{\alpha} u_{\kappa}
			- 
			(\vort^{\kappa} \partial_{\kappa} \Lnenth) u_{\alpha}
			+
			\TempoverEnth
			\vort^{\kappa} \GradEnt_{\kappa} u_{\alpha},
	\end{align}

	\begin{align} \label{E:GRADIENTOFUCONTRACTEDAGAINSTVOLUMEFORM}
		\upepsilon^{\alpha \beta \gamma \delta} \partial_{\gamma} u_{\delta}
		& = 
			\frac{1}{\Enth}
			\vort^{\alpha} u^{\beta} 
			- 
			\frac{1}{\Enth}
			u^{\alpha} \vort^{\beta}
			-
			\upepsilon^{\alpha \beta \gamma \delta} (\partial_{\gamma} \Lnenth) u_{\delta}
			+
			\TempoverEnth
			\upepsilon^{\alpha \beta \gamma \delta}
			\GradEnt_{\gamma} u_{\delta},
	\end{align}
	
	\begin{align} \label{E:GRADIENTOFUCONTRACTEDAGAINSTVOLUMEFORMANDVELOCITY}
		\upepsilon^{\alpha \beta \gamma \delta} u_{\beta} \partial_{\gamma} u_{\delta}
		& = 
			-
			\frac{1}{\Enth}
			\vort^{\alpha},
	\end{align}
	
	\begin{align} \label{E:IDANTISYMMETRICPARTOFGRADIENTOFVORTICITY}
	\partial_{\gamma} \vort_{\delta}
	-
	\partial_{\delta} \vort_{\gamma}
	& = 
	\upepsilon_{\gamma \delta \kappa \lambda} u^{\kappa} \uperpvort^{\lambda}(\vort)
	-
	(u^{\kappa} \partial_{\kappa} \vort_{\delta}) u_{\gamma}
	+
	u^{\kappa} (\partial_{\delta} \vort_{\kappa}) u_{\gamma}
	\\
	& \ \
	+
	(u^{\kappa} \partial_{\kappa} \vort_{\gamma}) u_{\delta}
	-
	u^{\kappa} (\partial_{\gamma} \vort_{\kappa}) u_{\delta},
	\notag
\end{align}

\begin{align} \label{E:CONTRACTIONOFEXTERIORDERIVAITVEOFVORTICITYOFVORTICITYAGAINSTVOLUMEFORM}
	\upepsilon^{\alpha \beta \gamma \delta} 
	\partial_{\gamma} \vort_{\delta}
	& = 
	\uperpvort^{\alpha}(\vort) u^{\beta} 
	-
	u^{\alpha} \uperpvort^{\beta}(\vort)
	+
	\upepsilon^{\alpha \beta \gamma \delta}
	(u^{\kappa} \partial_{\kappa} \vort_{\gamma}) u_{\delta}
		\\
	& \ \
	-
	\upepsilon^{\alpha \beta \gamma \delta}
	u^{\kappa} (\partial_{\gamma} \vort_{\kappa}) u_{\delta}.
	\notag
\end{align}
	
\end{lemma}

\begin{proof}
		\eqref{E:ENTSYMMETRYOFMIXEDPARTIALS} follows from definition \eqref{E:ENTROPYGRADIENT} and the symmetry property
		$\partial_{\alpha} \partial_{\beta} \Ent = \partial_{\beta} \partial_{\alpha} \Ent$.
		
		\eqref{E:VORTISORTHGONALTOU} is a simple consequence of definition \eqref{D:VORTICITYDEF}.
		
		\eqref{E:DERIVATIVEOFVELOCITYCONTRACTEDWITHVELOCITYISZERO} follows from differentiating
		\eqref{E:UISUNITLENGTH} with $\partial_{\alpha}$.
		
		\eqref{E:TRANSFERDERIVATIVESFROMENTROPYGRADIENTTOVORTICITY} follows from differentiating
		\eqref{E:VELOCITYANDENTGRADIENTAREMINKOWSKIPERP} with $\partial_{\alpha}$.
		
		\eqref{E:TRANSFERDERIVATIVESFROMVORTICITYTOVELOCITY} follows from differentiating
		\eqref{E:VORTISORTHGONALTOU} with $\partial_{\alpha}$.
		
		\eqref{E:COORDINATEDERIVATIVEINTERMSOFUDERIVAIVESANDORTHOGONALDERIVATIVES}
		follows directly from definition \eqref{E:UPERPPROJECTION}.
		
		\eqref{E:DIVOFVINTERMSOFUDERIVAIVESANDORTHOGONALDIVERGENCE} then
		follows from \eqref{E:COORDINATEDERIVATIVEINTERMSOFUDERIVAIVESANDORTHOGONALDERIVATIVES}.
		
		To prove \eqref{E:ANTISYMMETRICPARTOFGRADIENTVINTERMSOFVORTICITYANDUPARTOFV},
	we first use definition \eqref{E:UORTHGONALVORTICITYOFONEFORM}
	to express the first product on RHS~\eqref{E:ANTISYMMETRICPARTOFGRADIENTVINTERMSOFVORTICITYANDUPARTOFV} as follows:
	\begin{align} \label{E:CONTRACTIONVOLUMEFORMVELOCITYVORTICITY}
		\upepsilon_{\alpha \beta \gamma \delta} u^{\gamma} \uperpvort^{\delta}(V)
		& =
		- 
		\upepsilon_{\alpha \beta \gamma \delta} 
		\upepsilon^{\delta \theta \kappa \lambda}
		u^{\gamma} u_{\theta} \partial_{\kappa} V_{\lambda}.
	\end{align}
	Next, we observe the following identity for the first two factors on RHS~\eqref{E:CONTRACTIONVOLUMEFORMVELOCITYVORTICITY}:
	\begin{align} \label{E:VOLUMEFORMONECONTRACTIONIDENTITY}
		- 
		\upepsilon_{\alpha \beta \gamma \delta} 
		\upepsilon^{\delta \theta \kappa \lambda}
		& 
		=
		\upepsilon_{\alpha \beta \gamma \delta} 
		\upepsilon^{\theta \kappa \lambda \delta}
			\\
		& = \updelta_{\alpha}^{\theta} \updelta_{\gamma}^{\kappa} \updelta_{\beta}^{\lambda}
				-
				\updelta_{\alpha}^{\theta} \updelta_{\gamma}^{\lambda} \updelta_{\beta}^{\kappa}
				+
				\updelta_{\alpha}^{\lambda} \updelta_{\gamma}^{\theta} \updelta_{\beta}^{\kappa}
				-
				\updelta_{\alpha}^{\lambda} \updelta_{\gamma}^{\kappa} \updelta_{\beta}^{\theta}
				+
				\updelta_{\alpha}^{\kappa} \updelta_{\gamma}^{\lambda} \updelta_{\beta}^{\theta}
				-
				\updelta_{\alpha}^{\kappa} \updelta_{\gamma}^{\theta} \updelta_{\beta}^{\lambda}.
			\notag
	\end{align}
	Using \eqref{E:VOLUMEFORMONECONTRACTIONIDENTITY} to substitute
	on RHS~\eqref{E:CONTRACTIONVOLUMEFORMVELOCITYVORTICITY},
	we deduce, in view of \eqref{E:UISUNITLENGTH},
	the following identity:
	\begin{align} \label{E:SECONDIDCONTRACTIONVOLUMEFORMVELOCITYVORTICITY}
		- 
		\upepsilon_{\alpha \beta \gamma \delta} 
		\upepsilon^{\delta \theta \kappa \lambda}
		u_{\theta} \partial_{\kappa} V_{\lambda}
		u^{\gamma} 
		& = 
			u_{\alpha} u^{\kappa} \partial_{\kappa} V_{\beta} 
			-
			u_{\alpha} u^{\kappa} \partial_{\beta} V_{\kappa}
			- 
			\partial_{\beta} V_{\alpha}
			\\
			& \ \ 
			-
			u_{\beta} u^{\kappa} \partial_{\kappa} V_{\alpha}
			+
			u_{\beta} u^{\kappa} \partial_{\alpha} V_{\kappa}
			+ 
			\partial_{\alpha} V_{\beta}.
			\notag
	\end{align}
	Combining \eqref{E:CONTRACTIONVOLUMEFORMVELOCITYVORTICITY}
	and \eqref{E:SECONDIDCONTRACTIONVOLUMEFORMVELOCITYVORTICITY} and
	rearranging the terms, we arrive at the desired identity
	\eqref{E:ANTISYMMETRICPARTOFGRADIENTVINTERMSOFVORTICITYANDUPARTOFV}.
	
	\eqref{E:VANDUAREORTHOGONALANTISYMMETRICPARTOFGRADIENTVINTERMSOFVORTICITYANDUPARTOFV}
	then follows from \eqref{E:ANTISYMMETRICPARTOFGRADIENTVINTERMSOFVORTICITYANDUPARTOFV}
	and the relation
	$u^{\kappa} \partial_{\alpha} V_{\kappa} = - V_{\kappa} \partial_{\alpha} u^{\kappa}$,
	which follows from differentiating the assumed identity
	$u^{\kappa} V_{\kappa} = 0$ with $\partial_{\alpha}$.
	
	To prove \eqref{E:UPERPNORMSQUAREDOFANTISYMMETRICPARTINTERMSOFSQUARENORMOFVORTICITY}, we first use
	\eqref{E:ANTISYMMETRICPARTOFGRADIENTVINTERMSOFVORTICITYANDUPARTOFV}
	to deduce
	\begin{align}
		\Pi^{\alpha \beta} & \Pi^{\gamma \delta} 
		(\partial_{\alpha} V_{\gamma} - \partial _{\gamma} V_{\alpha}) 
		(\partial_{\beta} V_{\delta} - \partial_{\delta} V_{\beta})
		 \\
		 & 
			=\Pi^{\alpha \beta} \Pi^{\gamma \delta}
				\upepsilon_{\alpha \gamma \kappa \lambda} 
				\upepsilon_{\beta \delta \mu \nu}
				u^{\kappa} \uperpvort^{\lambda}(V)
				u^{\mu} \uperpvort^{\nu}(V).
				\notag
	\end{align}
	Next, we note the following identity,
	which follows easily from definition \eqref{E:UPERPPROJECTION}:
	\begin{align} \label{E:PROJECTIONDOESNOTAFFECTEPSILONUCONTRACTION}
		\Pi^{\alpha \beta} & \Pi^{\gamma \delta}
				\upepsilon_{\alpha \gamma \kappa \lambda} 
				\upepsilon_{\beta \delta \mu \nu}
				u^{\kappa} \uperpvort^{\lambda}(V)
				u^{\mu} \uperpvort^{\nu}(V)
				\\
		& = (\upeta^{-1})^{\alpha \beta} (\upeta^{-1})^{\gamma \delta}
				\upepsilon_{\alpha \gamma \kappa \lambda} 
				\upepsilon_{\beta \delta \mu \nu}
				u^{\kappa} \uperpvort^{\lambda}(V)
				u^{\mu} \uperpvort^{\nu}(V)
				\notag
					\\
	 &  = \upepsilon^{\alpha \beta \kappa \lambda} 
				\upepsilon_{\alpha \beta \mu \nu}
				u_{\kappa} \uperpvort_{\lambda}(V)
				u^{\mu} \uperpvort^{\nu}(V).
				\notag
	\end{align}
	From \eqref{E:PROJECTIONDOESNOTAFFECTEPSILONUCONTRACTION},
	the identity
	$
	\upepsilon^{\alpha \beta \kappa \lambda} 
	\upepsilon_{\alpha \beta \mu \nu}
	= 
	2 \updelta_{\mu}^{\lambda} \updelta_{\nu}^{\kappa}
	- 
	2 \updelta_{\mu}^{\kappa} \updelta_{\nu}^{\lambda}
	$,
	\eqref{E:UISUNITLENGTH},
	and the simple identity
	$u_{\alpha} \uperpvort^{\alpha}(V) = 0$ (which follows easily from definition \eqref{E:UORTHGONALVORTICITYOFONEFORM}),
	we find that 
	$\mbox{RHS~\eqref{E:PROJECTIONDOESNOTAFFECTEPSILONUCONTRACTION}} = 2 \uperpvort_{\alpha}(V) \uperpvort^{\alpha}(V)$.
	Again using that $u_{\alpha} \uperpvort^{\alpha}(V) = 0$, we conclude, in view of definition \eqref{E:UPERPPROJECTION},
	the desired identity \eqref{E:UPERPNORMSQUAREDOFANTISYMMETRICPARTINTERMSOFSQUARENORMOFVORTICITY}.
	
	\eqref{E:LIEUVOLUMEFORMINDICESUP} is a standard geometric identity,
		as is \eqref{E:LIEUOFEXTERIORDERIVATIVEOFENTHALPHYTIMESVELOCITYONEFORMID}.

		To prove \eqref{E:LIEUVELOCITYONEFORM},
		we first note the Lie differentiation identity
		$
		\mathcal{L}_u  (u_{\flat})_{\alpha}
		= u^{\kappa} \partial_{\kappa} u_{\alpha}
			+
			u_{\kappa} \partial_{\alpha} u^{\kappa}
		$,
		which follows from \eqref{E:LIEDERIVATIVE}.
		\eqref{E:LIEUVELOCITYONEFORM} follows from this identity,
	  \eqref{E:VELOCITYEVOLUTION}, and \eqref{E:DERIVATIVEOFVELOCITYCONTRACTEDWITHVELOCITYISZERO}.
		
		To prove \eqref{E:LIEUOFEXTERIORDERIVATIVEOFENTHALPHYTIMESVELOCITYONEFORMEVOLUTION},
	we first use
	\eqref{E:EVOLUTIONFORVELOCITYTIMESENTHALPHY}
	and the Lie derivative formula \eqref{E:LIEDERIVATIVE}
	to deduce that
	$
	\mathcal{L}_u (\Enth u_{\flat})_{\alpha}
	= 
	u^{\kappa} \partial_{\kappa} (\Enth u_{\alpha})
	+ 
	\Enth u_{\kappa} \partial_{\alpha} u^{\kappa}
	= 
	- \partial_{\alpha} \Enth
	+
	\Temp \partial_{\alpha} \Ent
	+
	\Enth u_{\kappa} \partial_{\alpha} u^{\kappa}
	$.
	From \eqref{E:DERIVATIVEOFVELOCITYCONTRACTEDWITHVELOCITYISZERO}, 
	we see that the last product on the RHS of this identity vanishes.
	Hence, taking the exterior derivative of the identity, we obtain
	$
	 [d \mathcal{L}_u (\Enth u_{\flat})]_{\alpha \beta}
	= \Temp_{;\Lnenth}(\partial_{\alpha} \Lnenth) \partial_{\beta} \Ent
		-
		\Temp_{;\Lnenth} (\partial_{\alpha} \Ent) \partial_{\beta} \Lnenth
	$.
	The desired identity
	\eqref{E:LIEUOFEXTERIORDERIVATIVEOFENTHALPHYTIMESVELOCITYONEFORMEVOLUTION}
	now follows from this identity and
	\eqref{E:LIEUOFEXTERIORDERIVATIVEOFENTHALPHYTIMESVELOCITYONEFORMID}.
		
		To prove \eqref{E:IDANTISYMMETRICPARTOFGRADIENTOFENTHALPHYTIMESVELOCITY},
		we first use definition \eqref{E:VORTICITYDEF}
		to compute that
		\begin{align}
				\upepsilon_{\alpha \beta \gamma \delta} u^{\gamma} \vort^{\delta}
			=
			-
			\upepsilon_{\alpha \beta \gamma \delta} 
			\upepsilon^{\delta \kappa \theta \lambda} 
			u^{\gamma} 
			u_{\kappa} \partial_{\theta} (\Enth u_{\lambda}).
			\notag
		\end{align}
	Using the identity
	$\upepsilon_{\alpha \beta \gamma \delta} 
	\upepsilon^{\delta \kappa \theta \lambda}
	=
	-
	\upepsilon_{\alpha \beta \gamma \delta} 
	\upepsilon^{\lambda \kappa \theta \delta}
	= 
		\updelta_{\alpha}^{\lambda} \updelta_{\beta}^{\kappa} \updelta_{\gamma}^{\theta}
		-
		\updelta_{\alpha}^{\lambda} \updelta_{\beta}^{\theta} \updelta_{\gamma}^{\kappa}
		+
		\updelta_{\alpha}^{\kappa} \updelta_{\beta}^{\theta} \updelta_{\gamma}^{\lambda}
		-
		\updelta_{\alpha}^{\kappa} \updelta_{\beta}^{\lambda} \updelta_{\gamma}^{\theta}
		+
		\updelta_{\alpha}^{\theta} \updelta_{\beta}^{\lambda} \updelta_{\gamma}^{\kappa}
		-
		\updelta_{\alpha}^{\theta} \updelta_{\beta}^{\kappa} \updelta_{\gamma}^{\lambda}
	$,
	we deduce, in view of \eqref{E:UISUNITLENGTH}, 
	that
	\begin{align} \label{E:FIRSTSTEPIDANTISYMMETRICPARTOFGRADIENTOFENTHALPHYTIMESVELOCITY}
			-
			\upepsilon_{\alpha \beta \gamma \delta} 
			\upepsilon^{\delta \kappa \theta \lambda} 
			u^{\gamma} 
			u_{\kappa} \partial_{\theta} (\Enth u_{\lambda})
		& = 
			\partial_{\alpha} (\Enth u_{\beta})
			-
			\partial_{\beta} (\Enth u_{\alpha})
			\\
	& \ \
			- 
			u_{\beta} u^{\kappa} \partial_{\kappa} (\Enth u_{\alpha})
			-
			u_{\alpha} u^{\kappa} \partial_{\beta} (\Enth u_{\kappa})
			\notag
			\\
			& \ \
			+
			u_{\alpha} u^{\kappa} \partial_{\kappa} (\Enth u_{\beta})
			+
			u_{\beta} u^{\kappa} \partial_{\alpha} (\Enth u_{\kappa}).
			\notag
	\end{align}
	Using \eqref{E:UISUNITLENGTH},
	\eqref{E:EVOLUTIONFORVELOCITYTIMESENTHALPHY},
	and \eqref{E:DERIVATIVEOFVELOCITYCONTRACTEDWITHVELOCITYISZERO},
	we compute that the last four products on RHS~\eqref{E:FIRSTSTEPIDANTISYMMETRICPARTOFGRADIENTOFENTHALPHYTIMESVELOCITY}
	sum to $\Temp (u_{\alpha} \GradEnt_{\beta} - u_{\beta} \GradEnt_{\alpha})$, 
	which yields the desired identity \eqref{E:IDANTISYMMETRICPARTOFGRADIENTOFENTHALPHYTIMESVELOCITY}.
	
To prove \eqref{E:CONTRACTIONVOLUMEFORMANDEXTERIORDERIVATIVEOFENTHALPHYTIMESVELOCITY},
we first contract 
$\frac{1}{2} \upepsilon^{\alpha \beta \gamma \delta}$
against \eqref{E:IDANTISYMMETRICPARTOFGRADIENTOFENTHALPHYTIMESVELOCITY}
to obtain the identity
\begin{align} \label{E:FIRSTSTEPCONTRACTIONVOLUMEFORMANDEXTERIORDERIVATIVEOFENTHALPHYTIMESVELOCITY}
	\upepsilon^{\alpha \beta \gamma \delta}
	\partial_{\gamma} (\Enth u_{\delta})
	& =
	\frac{1}{2} \upepsilon^{\alpha \beta \gamma \delta}
	\upepsilon_{\gamma \delta \kappa \lambda} 
	u^{\kappa} \vort^{\lambda}
	+
	\Temp
	\upepsilon^{\alpha \beta \gamma \delta}
	\GradEnt_{\gamma} u_{\delta}.
\end{align}
\eqref{E:CONTRACTIONVOLUMEFORMANDEXTERIORDERIVATIVEOFENTHALPHYTIMESVELOCITY}
now follows from		
using the identity
$
\frac{1}{2} \upepsilon^{\alpha \beta \gamma \delta}
	\upepsilon_{\gamma \delta \kappa \lambda} 
	= \updelta_{\kappa}^{\beta} \updelta_{\lambda}^{\alpha} 
		-
		\updelta_{\kappa}^{\alpha} \updelta_{\lambda}^{\beta} 
$
to substitute for the factor
$
\frac{1}{2} \upepsilon^{\alpha \beta \gamma \delta}
	\upepsilon_{\gamma \delta \kappa \lambda} 
$
on RHS~\eqref{E:FIRSTSTEPCONTRACTIONVOLUMEFORMANDEXTERIORDERIVATIVEOFENTHALPHYTIMESVELOCITY}.
	
	\eqref{E:PARITALALPHAUBETAMINUSPARTIALBETAUALPHA}
	follows from \eqref{E:IDANTISYMMETRICPARTOFGRADIENTOFENTHALPHYTIMESVELOCITY} and simple computations.
	
	To prove \eqref{E:IDVELOCITYDERIVATIVEOFVELOCITYCONTRACTEDWITHENTROPYGRADIENT},
	we contract $\GradEnt^{\alpha}$
	against equation \eqref{E:VELOCITYEVOLUTION} 
	and use equation \eqref{E:VELOCITYANDENTGRADIENTAREMINKOWSKIPERP}.
	
	\eqref{E:IDVELOCITYDERIVATIVEOFENTROPYGRADIENTCONTRACTEDWITHVELOCITY}
	then follows from \eqref{E:TRANSFERDERIVATIVESFROMENTROPYGRADIENTTOVORTICITY} 
	and 
	\eqref{E:IDVELOCITYDERIVATIVEOFVELOCITYCONTRACTEDWITHENTROPYGRADIENT}.
	
	To prove the first equality in \eqref{E:ENTROPYGRADIENTCONTRACTEDAGAINSTSECONDSLOTPARITALALPHAUBETAMINUSPARTIALBETAUALPHA},
	we contract $\GradEnt^{\beta}$ against \eqref{E:PARITALALPHAUBETAMINUSPARTIALBETAUALPHA}
	and use equation \eqref{E:VELOCITYANDENTGRADIENTAREMINKOWSKIPERP}. To obtain the second
	equality in \eqref{E:ENTROPYGRADIENTCONTRACTEDAGAINSTSECONDSLOTPARITALALPHAUBETAMINUSPARTIALBETAUALPHA},
	we use the first equality and the identity \eqref{E:IDVELOCITYDERIVATIVEOFVELOCITYCONTRACTEDWITHENTROPYGRADIENT}.

	\eqref{E:VORTICITYCONTRACTEDPARITALALPHAUBETAMINUSPARTIALBETAUALPHA}
	follows from contracting \eqref{E:PARITALALPHAUBETAMINUSPARTIALBETAUALPHA}
	against $\vort^{\alpha}$
	and using \eqref{E:VORTISORTHGONALTOU}.
	
	To prove \eqref{E:GRADIENTOFUCONTRACTEDAGAINSTVOLUMEFORM},
	we first use \eqref{E:PARITALALPHAUBETAMINUSPARTIALBETAUALPHA}
	to deduce that
	\begin{align} \label{E:FIRSTSTEPGRADIENTOFUCONTRACTEDAGAINSTVOLUMEFORM}
		\upepsilon^{\alpha \beta \gamma \delta} \partial_{\gamma} u_{\delta}
		& = 
			\frac{1}{2}
			\frac{1}{\Enth} 
			\upepsilon^{\alpha \beta \gamma \delta}
			\upepsilon_{\gamma \delta \kappa \lambda} u^{\kappa} \vort^{\lambda}
			-
			\upepsilon^{\alpha \beta \gamma \delta} (\partial_{\gamma} \Lnenth) u_{\delta}
			+
			\TempoverEnth
			\upepsilon^{\alpha \beta \gamma \delta}
			\GradEnt_{\gamma} u_{\delta}.
	\end{align}
	\eqref{E:GRADIENTOFUCONTRACTEDAGAINSTVOLUMEFORM}
	now follows from using the identity
	$
	\frac{1}{2}
		\upepsilon^{\alpha \beta \gamma \delta}
		\upepsilon_{\gamma \delta \kappa \lambda}
	= \updelta_{\kappa}^{\beta} \updelta_{\lambda}^{\alpha} 
		-
		\updelta_{\kappa}^{\alpha} \updelta_{\lambda}^{\beta} 
	$
	to substitute for the product
	$
	\frac{1}{2}
	\upepsilon^{\alpha \beta \gamma \delta}
	\upepsilon_{\gamma \delta \kappa \lambda}
	$
	on RHS~\eqref{E:FIRSTSTEPGRADIENTOFUCONTRACTEDAGAINSTVOLUMEFORM}.
	
	To prove \eqref{E:GRADIENTOFUCONTRACTEDAGAINSTVOLUMEFORMANDVELOCITY},
	we contract \eqref{E:GRADIENTOFUCONTRACTEDAGAINSTVOLUMEFORM} against $u_{\beta}$
	and use
	\eqref{E:UISUNITLENGTH} and \eqref{E:VORTISORTHGONALTOU}.
	
	To prove \eqref{E:IDANTISYMMETRICPARTOFGRADIENTOFVORTICITY},
	we first use definition \eqref{E:UORTHGONALVORTICITYOFONEFORM}
	to express the first product on RHS~\eqref{E:IDANTISYMMETRICPARTOFGRADIENTOFVORTICITY}
	as follows:
	\begin{align} \label{E:FIRSTSTEPIDANTISYMMETRICPARTOFGRADIENTOFVORTICITY}
		\upepsilon_{\gamma \delta \kappa \lambda} u^{\kappa} \uperpvort^{\lambda}(\vort)
		& 
		= 
		-
		\upepsilon_{\gamma \delta \kappa \lambda} 
		\upepsilon^{\lambda \theta \alpha \beta}
		u^{\kappa} 
		u_{\theta} \partial_{\alpha} \vort_{\beta}.
	\end{align}
	Next, we use the identity
	$
		-
		\upepsilon_{\gamma \delta \kappa \lambda} 
		\upepsilon^{\lambda \theta \alpha \beta}
		=
		\upepsilon_{\gamma \delta \kappa \lambda} 
		\upepsilon^{\theta \alpha \beta \lambda}
		=
		\updelta_{\gamma}^{\theta} \updelta_{\delta}^{\beta} \updelta_{\kappa}^{\alpha}
		-
		\updelta_{\gamma}^{\theta} \updelta_{\delta}^{\alpha} \updelta_{\kappa}^{\beta}
		+
		\updelta_{\gamma}^{\alpha} \updelta_{\delta}^{\theta} \updelta_{\kappa}^{\beta}
		-
		\updelta_{\gamma}^{\alpha} \updelta_{\delta}^{\beta} \updelta_{\kappa}^{\theta}
		+
		\updelta_{\gamma}^{\beta} \updelta_{\delta}^{\alpha} \updelta_{\kappa}^{\theta}
		-
		\updelta_{\gamma}^{\beta} \updelta_{\delta}^{\theta} \updelta_{\kappa}^{\alpha}
 	$
	to substitute on RHS~\eqref{E:FIRSTSTEPIDANTISYMMETRICPARTOFGRADIENTOFVORTICITY}, thereby obtaining,
	in view of \eqref{E:UISUNITLENGTH},
	the following identity:
	\begin{align} \label{E:SECONDSTEPIDANTISYMMETRICPARTOFGRADIENTOFVORTICITY}
		\upepsilon_{\gamma \delta \kappa \lambda} u^{\kappa} \uperpvort^{\lambda}(\vort)
		& 
		= 
		u_{\gamma} u^{\kappa} \partial_{\kappa} \vort_{\delta}
		-
		u_{\gamma} u^{\kappa} \partial_{\delta} \vort_{\kappa}
		+
		u_{\delta} u^{\kappa} \partial_{\gamma} \vort_{\kappa}
		+ 
		\partial_{\gamma} \vort_{\delta}
		\\
		& \ \
		- 
		\partial_{\delta} \vort_{\gamma}
		-
		u_{\delta} u^{\kappa} \partial_{\kappa} \vort_{\gamma}.
		\notag
	\end{align}
	Finally, we note that
	it is straightforward to see that
	\eqref{E:SECONDSTEPIDANTISYMMETRICPARTOFGRADIENTOFVORTICITY}
	is equivalent to the desired identity
	\eqref{E:IDANTISYMMETRICPARTOFGRADIENTOFVORTICITY}.

	To prove \eqref{E:CONTRACTIONOFEXTERIORDERIVAITVEOFVORTICITYOFVORTICITYAGAINSTVOLUMEFORM},
	we first contract \eqref{E:IDANTISYMMETRICPARTOFGRADIENTOFVORTICITY}
	against
	$\frac{1}{2} \upepsilon^{\alpha \beta \gamma \delta}$
	to deduce
	\begin{align} \label{E:FIRSTSTEPCONTRACTIONOFEXTERIORDERIVAITVEOFVORTICITYOFVORTICITYAGAINSTVOLUMEFORM}
		\upepsilon^{\alpha \beta \gamma \delta} 
		\partial_{\gamma} \vort_{\delta}
		& =
			\frac{1}{2} \upepsilon^{\alpha \beta \gamma \delta}
			\upepsilon_{\gamma \delta \kappa \lambda} 
			u^{\kappa} \uperpvort^{\lambda}(\vort)
	+
	\upepsilon^{\alpha \beta \gamma \delta}
	(u^{\kappa} \partial_{\kappa} \vort_{\gamma}) u_{\delta}
	\\
	& \ \ 
	-
	\upepsilon^{\alpha \beta \gamma \delta}
	u^{\kappa} (\partial_{\gamma} \vort_{\kappa}) u_{\delta}.
	\notag
	\end{align}
	Using the identity 
$
\frac{1}{2} 
	\upepsilon^{\alpha \beta \gamma \delta}
	\upepsilon_{\gamma \delta \kappa \lambda}
	=
	\frac{1}{2} 
	\upepsilon^{\gamma \delta \alpha \beta}
	\upepsilon_{\gamma \delta \kappa \lambda}
	= \updelta_{\kappa}^{\beta} \updelta_{\lambda}^{\alpha}
	-
	\updelta_{\kappa}^{\alpha} \updelta_{\lambda}^{\beta}
$
to substitute in the first product on
RHS~\eqref{E:FIRSTSTEPCONTRACTIONOFEXTERIORDERIVAITVEOFVORTICITYOFVORTICITYAGAINSTVOLUMEFORM},
we arrive at the desired identity
\eqref{E:CONTRACTIONOFEXTERIORDERIVAITVEOFVORTICITYOFVORTICITYAGAINSTVOLUMEFORM}.
	
\end{proof}

\section{Wave equations}
\label{S:WAVEEQUATIONS}
In this section, with the help of the preliminary identities of Lemma~\ref{L:IDENTITIES},
we derive the covariant wave equations 
\eqref{E:MAINTHMCOVARIANTWAVEENTHALPY},
\eqref{E:MAINTHMCOVARIANTWAVEVELOCITY},
and \eqref{E:ENTAUXILIARYWAVEQUATION}.

\subsection{Covariant wave operator}
\label{SS:COVARIANTWAVEOPERATOR}
We start by establishing a formula for the covariant wave operator 
of the acoustical metric acting on a scalar function.

\begin{lemma}[Covariant wave operator of $g$]
	\label{L:COVARIANTWAVEOPERATORAPPLIEDTOSCALAR}
	Assume that $(\Lnenth,\Ent,u^{\alpha})$ is a $C^2$ solution 
	to \eqref{E:ENTHALPYEVOLUTION}-\eqref{E:ENTROPYEVOLUTION} + \eqref{E:UISUNITLENGTH}.
	Then the covariant wave operator of the acoustical metric $g=g(\Lnenth,\Ent,u)$ 
	(see Def.\,\ref{D:ACOUSTICALMETRICANDINVERSE})
	acts on
	scalar functions $\phi$ as follows, where
	RHS~\eqref{E:COVARIANTWAVEOPERATORAPPLIEDTOSCALAR} is expressed relative
	to the rectangular coordinates:
	\begin{align} \label{E:COVARIANTWAVEOPERATORAPPLIEDTOSCALAR}
		\square_g \phi
		& = 
			(\speed^2 - 1)
			u^{\kappa} \partial_{\kappa} (u^{\lambda} \partial_{\lambda} \phi)
			+
			\speed^2 ((\upeta^{-1})^{\kappa \lambda} \partial_{\kappa} \partial_{\lambda} \phi)
				\\
		& \ \
			+
			(\speed^2 - 1)
			(\partial_{\kappa} u^{\kappa}) 
			(u^{\lambda} \partial_{\lambda} \phi)
			+ 
				2 \speed^{-1} \speed_{;\Lnenth}
				(u^{\kappa} \partial_{\kappa} \Lnenth) 
				(u^{\lambda} \partial_{\lambda} \phi)
			\notag \\
		& \ \
			- \speed^{-1} \speed_{;\Lnenth} 
				(g^{-1})^{\kappa \lambda}
				(\partial_{\kappa} \Lnenth)
				(\partial_{\lambda} \phi)
				\notag
					\\
		& \ \
				-
				\speed \speed_{;\Ent}
				(\GradEnt^{\kappa} \partial_{\kappa} \phi).
				\notag
	\end{align}
\end{lemma}

\begin{proof}
	It is a standard fact that relative to arbitrary coordinates
	(and in particular relative to the rectangular coordinates),
	we have
\begin{align}
\square_g \phi & = 
\frac{1}{\mbox{$\sqrt{|\mbox{\upshape det} g|}$}}
\partial_{\kappa}\left(\sqrt{|\mbox{\upshape det} g|} (g^{-1})^{\kappa \lambda} \partial_{\lambda} \phi \right).
\notag
\end{align}	
Using this formula and \eqref{E:DETG}-\eqref{E:DETGGINVERSE}, we compute that
	\begin{align} \label{E:FIRSTCOMPUTATIONCOVARIANTWAVEOPERATORAPPLIEDTOSCALAR}
		\square_g \phi
		& = 
			\speed^3
			\partial_{\kappa}
			\left\lbrace
				-
				(\speed^{-3} - \speed^{-1})
				u^{\kappa} (u^{\lambda} \partial_{\lambda} \phi)
				+ 
				\speed^{-1} ((\upeta^{-1})^{\kappa \lambda} \partial_{\lambda} \phi)
			\right\rbrace
				\\
		& = 
			-
			(1 - \speed^2)
			u^{\kappa} \partial_{\kappa} (u^{\lambda} \partial_{\lambda} \phi)
			-
			(1 - \speed^2)
			(\partial_{\kappa} u^{\kappa}) (u^{\lambda} \partial_{\lambda} \phi)
					\notag \\
		& \ \
			+ (3 \speed^{-1} - \speed) 
				(u^{\kappa} \partial_{\kappa} \speed) 
				(u^{\lambda} \partial_{\lambda} \phi)
				-
				\speed
				(\upeta^{-1})^{\kappa \lambda} (\partial_{\kappa} \speed) (\partial_{\lambda} \phi)
					\notag \\
		&  \ \
			+
			\speed^2 ((\upeta^{-1})^{\kappa \lambda} \partial_{\kappa} \partial_{\lambda} \phi).
				\notag
	\end{align}
	The desired identity \eqref{E:COVARIANTWAVEOPERATORAPPLIEDTOSCALAR} now follows
	from \eqref{E:FIRSTCOMPUTATIONCOVARIANTWAVEOPERATORAPPLIEDTOSCALAR},
	\eqref{E:INVERSEACOUSTICALMETRIC}, 
	the evolution equation \eqref{E:ENTROPYEVOLUTION},
	and straightforward computations.
\end{proof}

\subsection{Covariant wave equation for the logarithmic enthalpy}
\label{E:WAVEEQUATIONFORLOGARITHMICENTHALPY}
We now derive the covariant wave equation \eqref{E:MAINTHMCOVARIANTWAVEENTHALPY}.

\begin{proposition} [Covariant wave equation for the logarithmic enthalpy]
\label{P:COVARIANTWAVEENTHALPY}
Assume that $(\Lnenth,\Ent,u^{\alpha})$ is a $C^2$ solution 
to \eqref{E:ENTHALPYEVOLUTION}-\eqref{E:ENTROPYEVOLUTION} + \eqref{E:UISUNITLENGTH}.
Then the logarithmic enthalpy $\Lnenth$ verifies the following covariant wave equation:
\begin{align}  \label{E:COVARIANTWAVEENTHALPY}
		\square_g \Lnenth
	&  = 
			n \speed^2 \TempoverEnth \mathcal{D}
			-
			\speed^{-1}
			\speed_{;\Lnenth}
			(g^{-1})^{\kappa \lambda}
			(\partial_{\kappa} \Lnenth)
			(\partial_{\lambda} \Lnenth)
			\\
			& \ \ 
			+
		\speed^2 
		\left\lbrace
			(\partial_{\kappa} u^{\kappa}) (\partial_{\lambda} u^{\lambda})
			-
			(\partial_{\kappa} u^{\lambda}) (\partial_{\lambda} u^{\kappa})
		\right\rbrace
		\notag
			\\
	& \ \
		+
		(1 - \speed^2) \TempoverEnth (\GradEnt^{\kappa} \partial_{\kappa} \Lnenth)
		-
		\speed \speed_{;\Ent} (\GradEnt^{\kappa} \partial_{\kappa} \Lnenth)
		+
		\speed^2 \TempoverEnth_{;\Lnenth} (\GradEnt^{\kappa} \partial_{\kappa} \Lnenth)
		+
		\speed^2 \TempoverEnth_{;\Ent} \GradEnt_{\kappa} \GradEnt^{\kappa}.
		\notag
\end{align}

\end{proposition}

\begin{proof}
	From \eqref{E:COVARIANTWAVEOPERATORAPPLIEDTOSCALAR} with $\phi := \Lnenth$,
	we deduce
	\begin{align} \label{E:FIRSTSTEPCOVARIANTWAVEENTHALPY}
		\square_g \Lnenth
		& = 
			(\speed^2 - 1)
			u^{\kappa} \partial_{\kappa} (u^{\lambda} \partial_{\lambda} \Lnenth)
			+
			\speed^2 ((\upeta^{-1})^{\kappa \lambda} \partial_{\kappa} \partial_{\lambda} \Lnenth)
				\\
		& \ \
			+
			(\speed^2 - 1)
			(\partial_{\kappa} u^{\kappa}) (u^{\lambda} \partial_{\lambda} \Lnenth)
			+ 
				2 \speed^{-1} \speed_{;\Lnenth}
				(u^{\kappa} \partial_{\kappa} \Lnenth)
				(u^{\lambda} \partial_{\lambda} \Lnenth)
				\notag \\
		& \ \
			- \speed^{-1} \speed_{;\Lnenth} 
				(g^{-1})^{\kappa \lambda}
				(\partial_{\kappa} \Lnenth)
				(\partial_{\lambda} \Lnenth)
				\notag
					\\
		& \ \	
				-
				\speed \speed_{;\Ent}
				(\GradEnt^{\kappa} \partial_{\kappa} \Lnenth).
				\notag
	\end{align}
	
	Next, we differentiate equation \eqref{E:VELOCITYEVOLUTION} with $\partial_{\beta}$,
	contract against $(\upeta^{-1})^{\alpha \beta}$, and multiply by $\speed^2$
	to obtain the identity
	\begin{align}
		\speed^2 ((\upeta^{-1})^{\kappa \lambda} \partial_{\kappa} \partial_{\lambda} \Lnenth)
		& = 
		- 
		\speed^2 (u^{\kappa} \partial_{\kappa} \partial_{\lambda} u^{\lambda})
		-
		\speed^2 (\partial_{\kappa} u^{\lambda}) (\partial_{\lambda} u^{\kappa})
			\label{E:FLATWAVEOPERATOROFLOGENTHID} \\
	& \ \
		-
		\speed^2 u^{\kappa} \partial_{\kappa} (u^{\lambda} \partial_{\lambda} \Lnenth)
		-
		\speed^2 (\partial_{\kappa} u^{\kappa}) (u^{\lambda} \partial_{\lambda} \Lnenth)
		\notag
			\\
	& \ \
		+
		\speed^2 \TempoverEnth (\partial_{\kappa} \GradEnt^{\kappa})
		+
		\speed^2 \TempoverEnth_{;\Lnenth} (\GradEnt^{\kappa} \partial_{\kappa} \Lnenth)
		+
		\speed^2 \TempoverEnth_{;\Ent} \GradEnt^{\kappa} \GradEnt_{\kappa}.
		\notag
	\end{align}
	
	Next, we use \eqref{E:ENTHALPYEVOLUTION}
	and the evolution equation \eqref{E:ENTROPYEVOLUTION}
	to rewrite the first product on RHS~\eqref{E:FLATWAVEOPERATOROFLOGENTHID}
	as follows:
	\begin{align} \label{E:SPEEDSQUAREDTIMESUDERIVATIVEOFDIVUID}
		- \speed^2 (u^{\kappa} \partial_{\kappa} \partial_{\lambda} u^{\lambda})
		& = \speed^2 u^{\kappa} \partial_{\kappa} (\speed^{-2} u^{\lambda} \partial_{\lambda} \Lnenth)
			\\
		& = 
			u^{\kappa} \partial_{\kappa} (u^{\lambda} \partial_{\lambda} \Lnenth)
			- 
			2 \speed^{-1} \speed_{;\Lnenth} 
			(u^{\kappa} \partial_{\kappa} \Lnenth)
			(u^{\lambda} \partial_{\lambda} \Lnenth).
			\notag
	\end{align}
	
	Using \eqref{E:SPEEDSQUAREDTIMESUDERIVATIVEOFDIVUID} to substitute for the first product
	on RHS~\eqref{E:FLATWAVEOPERATOROFLOGENTHID} and then using the resulting identity
	to substitute for the
	product $\speed^2 (\upeta^{-1})^{\kappa \lambda} \partial_{\kappa} \partial_{\lambda} \Lnenth$
	on RHS~\eqref{E:FIRSTSTEPCOVARIANTWAVEENTHALPY}, we deduce
	\begin{align} \label{E:SECONDSTEPCOVARIANTWAVEENTHALPY}	
		\square_g \Lnenth
		& = 
		-
		\speed^2 (\partial_{\kappa} u^{\lambda}) (\partial_{\lambda} u^{\kappa})
		-
		(\partial_{\kappa} u^{\kappa}) (u^{\lambda} \partial_{\lambda} \Lnenth)	
		\\
		& \ \
			- \speed^{-1} \speed_{;\Lnenth} 
				(g^{-1})^{\kappa \lambda}
				(\partial_{\kappa} \Lnenth)
				(\partial_{\lambda} \Lnenth)
				\notag
					\\
	& \ \
		-
		\speed \speed_{;\Ent} (\GradEnt^{\kappa} \partial_{\kappa} \Lnenth)
		+
		\speed^2 \TempoverEnth (\partial_{\kappa} \GradEnt^{\kappa})
		+
		\speed^2 \TempoverEnth_{;\Lnenth} (\GradEnt^{\kappa} \partial_{\kappa} \Lnenth)
		+
		\speed^2 \TempoverEnth_{;\Ent} \GradEnt^{\kappa} \GradEnt_{\kappa}.
		\notag
	\end{align}
	Finally, we use equation \eqref{E:ENTHALPYEVOLUTION}
	to substitute for the factor
	$u^{\lambda} \partial_{\lambda} \Lnenth$
	in the second product on RHS~\eqref{E:SECONDSTEPCOVARIANTWAVEENTHALPY},
	and we use definition \eqref{E:MODIFIEDDIVERGENCEOFENTROPY}
	to express the product 
	$
	\speed^2 \TempoverEnth (\partial_{\kappa} \GradEnt^{\kappa})
	$
	on RHS~\eqref{E:SECONDSTEPCOVARIANTWAVEENTHALPY}
	as
	$
	n \speed^2 \TempoverEnth \mathcal{D}
	+
	(1 - \speed^2) \TempoverEnth (\GradEnt^{\kappa} \partial_{\kappa} \Lnenth)
	$,
	which in total yields the desired equation \eqref{E:COVARIANTWAVEENTHALPY}.
	
\end{proof}

\subsection{Covariant wave equation for the rectangular components of the four-velocity}
\label{SS:WAVEEQUATIONFORFOURVELOCITYCOMPONENTS}
We now derive the covariant wave equation \eqref{E:MAINTHMCOVARIANTWAVEVELOCITY}.

\begin{proposition} [Covariant wave equation for the rectangular four-velocity components]
\label{P:COVARIANTWAVEVELOCITY}
Assume that $(\Lnenth,\Ent,u^{\alpha})$ is a $C^2$ solution 
to \eqref{E:ENTHALPYEVOLUTION}-\eqref{E:ENTROPYEVOLUTION} + \eqref{E:UISUNITLENGTH}.
Then the rectangular velocity components $u^{\alpha}$ 
verify the following covariant wave equations:
\begingroup
\allowdisplaybreaks
\begin{align}  \label{E:COVARIANTWAVEVELOCITY} 
		\square_g u^{\alpha}
		& = 	
			-
			\frac{\speed^2}{\Enth} 
			\mathcal{C}^{\alpha}
				\\
		& \ \
			-
			\frac{\speed^2}{\Enth} 
			\upepsilon^{\alpha \beta \gamma \delta} 
			(\partial_{\beta} u_{\gamma}) 
			\vort_{\delta}
			+
			\frac{(1 - \speed^2)}{\Enth}
			\upepsilon^{\alpha \beta \gamma \delta} 
			u_{\beta}
			(\partial_{\gamma} \Lnenth) \vort_{\delta}
			\notag
			\\
			& \ \
			+
			\frac{(1 - \speed^2) \TempoverEnth }{\Enth} 
			\upepsilon^{\alpha \beta \gamma \delta} \GradEnt_{\beta} u_{\gamma} \vort_{\delta}
				\notag \\
		& \ \
				- 
				(g^{-1})^{\kappa \lambda} (\partial_{\kappa} \Lnenth) (\partial_{\lambda} u^{\alpha})
				- 
				\speed^{-1} \speed_{;\Lnenth} 
				(g^{-1})^{\kappa \lambda}
				(\partial_{\kappa} \Lnenth)
				(\partial_{\lambda} u^{\alpha})
				\notag \\
		& \ \
			+
				(\upeta^{-1})^{\alpha \lambda}
				\left\lbrace
					(\partial_{\kappa} u^{\kappa}) (\partial_{\lambda} \Lnenth)
					- 
					(\partial_{\lambda} u^{\kappa}) (\partial_{\kappa} \Lnenth)
				\right\rbrace
				\notag
				\\
				& \ \
				+
			\speed^2 
			u^{\alpha} 
			\left\lbrace
				(\partial_{\kappa} u^{\lambda}) (\partial_{\lambda} u^{\kappa}) 
				-
				(\partial_{\lambda} u^{\lambda}) (\partial_{\kappa} u^{\kappa}) 
			\right\rbrace
		\notag
					\\
		& \ \
		-
		\speed \speed_{;\Ent} (\GradEnt^{\kappa} \partial_{\kappa} u^{\alpha})
		+
		\TempoverEnth
		(\GradEnt^{\kappa} \partial_{\kappa} u^{\alpha})
	   \notag
	   \\
	   & \ \		
		+
		(\speed^2 - 1) \TempoverEnth
		u^{\alpha}
		(\GradEnt^{\kappa} u^{\lambda} \partial_{\lambda} u_{\kappa}) 
			+
				\speed^2
				\TempoverEnth 
				\GradEnt^{\kappa}
				((\upeta^{-1})^{\alpha \lambda} \partial_{\lambda} u_{\kappa})  
				\notag
					\\
		& \ \
		+ 
		2 \speed^{-1} \speed_{;\Ent} 
		\GradEnt^{\alpha}
		(u^{\kappa} \partial_{\kappa} \Lnenth) 
		+
		2 \speed^{-1} \speed_{;\Lnenth} \TempoverEnth 
		\GradEnt^{\alpha}
		(u^{\kappa} \partial_{\kappa} \Lnenth)
		\notag
		\\
		& \ \
		-
		\TempoverEnth_{;\Lnenth} \GradEnt^{\alpha} (u^{\kappa} \partial_{\kappa} \Lnenth) 
		-
		\TempoverEnth \GradEnt^{\alpha} (\partial_{\kappa} u^{\kappa}) 
		\notag
			\\
	& \ \
		+
		(\Temp - \Temp_{;\Lnenth})
		\frac{\speed^2}{\Enth}
		\GradEnt^{\alpha} 
		(\partial_{\kappa} u^{\kappa})
		+
		(\Temp - \Temp_{;\Lnenth}) 
		\frac{\speed^2}{\Enth}
		u^{\alpha} 
		(\GradEnt^{\kappa} \partial_{\kappa} \Lnenth)
		\notag
		\\
		& \ \
		+
		(\Temp_{;\Lnenth} - \Temp) 
		\frac{\speed^2}{\Enth}
		\GradEnt^{\kappa} 
		((\upeta^{-1})^{\alpha \lambda} \partial_{\lambda} u_{\kappa}).
		\notag
	\end{align}
\endgroup
\end{proposition}

\begin{proof}
		From \eqref{E:COVARIANTWAVEOPERATORAPPLIEDTOSCALAR}
		with $\phi := u_{\alpha}$,
		we deduce 
		
		\begingroup
		\allowdisplaybreaks
		\begin{align}  \label{E:FIRSTCOVARIANTWAVEVELOCITY}
		\square_g u_{\alpha}
		& = 
			(\speed^2 - 1)
			u^{\kappa} \partial_{\kappa} (u^{\lambda} \partial_{\lambda} u_{\alpha})
			+
			\speed^2 ((\upeta^{-1})^{\kappa \lambda} \partial_{\kappa} \partial_{\lambda} u_{\alpha})
				\\
		& \ \
			+
			(\speed^2 - 1)
			(\partial_{\kappa} u^{\kappa}) (u^{\lambda} \partial_{\lambda} u_{\alpha})
			+ 
				2 \speed^{-1} \speed_{;\Lnenth}
				(u^{\kappa} \partial_{\kappa} \Lnenth) 
				(u^{\lambda} \partial_{\lambda} u_{\alpha})
			\notag \\
		& \ \
			- \speed^{-1} \speed_{;\Lnenth} 
				(g^{-1})^{\kappa \lambda}
				(\partial_{\kappa} \Lnenth)
				(\partial_{\lambda} u_{\alpha})
				-
				\speed \speed_{;\Ent}
				(\GradEnt^{\kappa} \partial_{\kappa} u_{\alpha}).
				\notag
	\end{align}
	\endgroup
		
	Next, we use equations 
	\eqref{E:ENTROPYEVOLUTION},
	\eqref{E:VELOCITYEVOLUTION},
	and
	the second line of 
	\eqref{E:ANOTHERENTROPYGRADIENTEVOLUTION}
	[where below, we derive \eqref{E:ANOTHERENTROPYGRADIENTEVOLUTION}
	 using an independent argument] 
	to rewrite the first product on RHS~\eqref{E:FIRSTCOVARIANTWAVEVELOCITY}
	as follows:
	
	\begingroup
\allowdisplaybreaks
	\begin{align} \label{E:TWOUDERIVATIVESOFUALPHAIDENTITY}
		(\speed^2 - 1) u^{\kappa} \partial_{\kappa} (u^{\lambda} \partial_{\lambda} u_{\alpha})
		& 
		=
		(1 - \speed^2) (u^{\kappa} \partial_{\kappa} \partial_{\alpha} \Lnenth)
		\\
		& \ \ 
		+ 
		(1 - \speed^2)  
		\left\lbrace
			u^{\kappa} \partial_{\kappa} (u^{\lambda} \partial_{\lambda} \Lnenth)
		\right\rbrace
		u_{\alpha}
		\notag
			\\
		& \ \
			+
			(1 - \speed^2) (u^{\kappa} \partial_{\kappa} u_{\alpha}) (u^{\lambda} \partial_{\lambda} \Lnenth)
			+
			(\speed^2 - 1) u^{\kappa} \partial_{\kappa} (\TempoverEnth \GradEnt_{\alpha})
				\notag
				\\
		& 
		=
		(1 - \speed^2) (u^{\kappa} \partial_{\kappa} \partial_{\alpha} \Lnenth)
		+ 
		(1 - \speed^2)  
		\left\lbrace
			u^{\kappa} \partial_{\kappa} (u^{\lambda} \partial_{\lambda} \Lnenth)
		\right\rbrace
		u_{\alpha}
			\notag \\
		& \ \
			+
			(1 - \speed^2) (u^{\kappa} \partial_{\kappa} u_{\alpha}) (u^{\lambda} \partial_{\lambda} \Lnenth)
			+
			(\speed^2 - 1) \TempoverEnth_{;\Lnenth} (u^{\kappa} \partial_{\kappa} \Lnenth) \GradEnt_{\alpha}
			\notag	\\
	& \ \	
			+
			(1 - \speed^2) \TempoverEnth
			(\GradEnt^{\kappa} \partial_{\kappa} u_{\alpha})
			+
			\frac{1}{\Enth} 
			(1 - \speed^2) \TempoverEnth
			\upepsilon_{\alpha \beta \gamma \delta} \GradEnt^{\beta} u^{\gamma} \vort^{\delta}
			\notag
			\\
			& \ \
			+
			(\speed^2 - 1) \TempoverEnth
			\GradEnt^{\kappa} (u^{\lambda} \partial_{\lambda} u_{\kappa}) u_{\alpha}.
			\notag
	\end{align}
	\endgroup
	
	Next, we use 
	definition \eqref{E:MODIFIEDDIVERGENCEOFENTROPY},
	the identity \eqref{E:PARITALALPHAUBETAMINUSPARTIALBETAUALPHA},
	and the evolution equations 
	\eqref{E:ENTHALPYEVOLUTION},
	\eqref{E:ENTROPYEVOLUTION},
	and \eqref{E:ENTROPYGRADIENTEVOLUTION}
	to rewrite the second product on RHS~\eqref{E:FIRSTCOVARIANTWAVEVELOCITY}
	as follows:
	\begin{align} \label{E:TOPDERIVATIVEVELOCITYTERMIDENTITY}
		\speed^2 ((\upeta^{-1})^{\kappa \lambda} \partial_{\kappa} \partial_{\lambda} u_{\alpha})
		& =
			\speed^2 (\partial_{\alpha} \partial_{\kappa} u^{\kappa})
				\\
		& \ \
			+
			\speed^2
			(\upeta^{-1})^{\kappa \lambda} \partial_{\kappa}
			\Big\lbrace
				\frac{1}{\Enth} \upepsilon_{\lambda \alpha \gamma \delta} u^{\gamma} \vort^{\delta}
				- 
				(\partial_{\lambda} \Lnenth) u_{\alpha}
				+
				(\partial_{\alpha} \Lnenth) u_{\lambda}
				\notag
				\\
				& \ \ \ \ \ \ \ \ \ \ \ \ \ \ \ \ \ \ \ \ \ \ \
				+
				\TempoverEnth \GradEnt_{\lambda} u_{\alpha}
				-
				\TempoverEnth \GradEnt_{\alpha} u_{\lambda}
			\Big\rbrace
			\notag	\\
	& = 
			(\speed^2 - 1) (u^{\kappa} \partial_{\kappa} \partial_{\alpha} \Lnenth)
			- 
			(\partial_{\alpha} u^{\kappa}) (\partial_{\kappa} \Lnenth)
			\notag
			\\
			& \ \
			+ 
			2 \speed^{-1} \speed_{;\Lnenth} (\partial_{\alpha} \Lnenth) (u^{\kappa} \partial_{\kappa} \Lnenth)
			+ 
			2 \speed^{-1} \speed_{;\Ent} \GradEnt_{\alpha} (u^{\kappa} \partial_{\kappa} \Lnenth)
			\notag \\
	& \ \
			- 
			\speed^2
			\frac{1}{\Enth} 
			\upepsilon_{\lambda \alpha \gamma \delta} 
			((\upeta^{-1})^{\kappa \lambda} \partial_{\kappa} \Lnenth) 
			u^{\gamma} \vort^{\delta}
			\notag
			\\
			& \ \
			+
			\speed^2
			\frac{1}{\Enth} \upepsilon_{\lambda \alpha \gamma \delta} ((\upeta^{-1})^{\kappa \lambda} \partial_{\kappa} u^{\gamma}) 
			\vort^{\delta}
				\notag \\
	& \ \ +
			\speed^2
			\frac{1}{\Enth} \upepsilon_{\lambda \alpha \gamma \delta} u^{\gamma} ((\upeta^{-1})^{\kappa \lambda} \partial_{\kappa} \vort^{\delta})
		\notag
			\\
	&  \ \
				- 
				\speed^2
				((\upeta^{-1})^{\kappa \lambda} \partial_{\kappa} \partial_{\lambda} \Lnenth)
				u_{\alpha}
				- 
				\speed^2
				(\upeta^{-1})^{\kappa \lambda} 
				(\partial_{\kappa} \Lnenth) (\partial_{\lambda} u_{\alpha})
				\notag
				\\
				& \ \
				+ 
				\speed^2
				(\partial_{\alpha} \Lnenth) (\partial_{\kappa} u^{\kappa})
				\notag 
					\\
		& \ \
			+
			\speed^2
			\TempoverEnth_{;\Lnenth} 
			(\GradEnt^{\kappa} \partial_{\kappa} \Lnenth)
			 u_{\alpha}
			+
			\speed^2
			\TempoverEnth_{;\Ent} 
			\GradEnt_{\kappa} \GradEnt^{\kappa}
			 u_{\alpha}
			\notag
			\\
			& \ \
			+
			\speed^2
			\TempoverEnth 
			(\partial_{\kappa} \GradEnt^{\kappa}) u_{\alpha}
			+
			\speed^2
			\TempoverEnth (\GradEnt^{\kappa} \partial_{\kappa} u_{\alpha})
				\notag \\
		& \ \
				-
				\speed^2
				\TempoverEnth_{;\Lnenth} (u^{\kappa} \partial_{\kappa} \Lnenth) \GradEnt_{\alpha} 
				-
				\speed^2
				\TempoverEnth (u^{\kappa} \partial_{\kappa} \GradEnt_{\alpha})
				-
				\speed^2
				\TempoverEnth (\partial_{\kappa} u^{\kappa}) \GradEnt_{\alpha}
					\notag
					\\
		& = n \speed^2 \TempoverEnth \mathcal{D} u_{\alpha}
				\notag \\
		& \ \
			+
			(\speed^2 - 1) (u^{\kappa} \partial_{\kappa} \partial_{\alpha} \Lnenth)
			- 
			(\partial_{\alpha} u^{\kappa}) (\partial_{\kappa} \Lnenth)
			\notag
			\\
			& \ \
			+ 
			2 \speed^{-1} \speed_{;\Lnenth} (\partial_{\alpha} \Lnenth) (u^{\kappa} \partial_{\kappa} \Lnenth)
			\notag \\
	& \ \
			- 
			\speed^2
			\frac{1}{\Enth} 
			\upepsilon_{\lambda \alpha \gamma \delta} 
			((\upeta^{-1})^{\kappa \lambda} \partial_{\kappa} \Lnenth) 
			u^{\gamma} \vort^{\delta}
			\notag
			\\
			& \ \
			+
			\speed^2
			\frac{1}{\Enth} \upepsilon_{\lambda \alpha \gamma \delta} ((\upeta^{-1})^{\kappa \lambda} \partial_{\kappa} u^{\gamma}) 
			\vort^{\delta}
				\notag \\
	& \ \ +
			\speed^2
			\frac{1}{\Enth} \upepsilon_{\lambda \alpha \gamma \delta} u^{\gamma} ((\upeta^{-1})^{\kappa \lambda} \partial_{\kappa} \vort^{\delta})
		\notag
			\\
	&  \ \
				- 
				\speed^2
				((\upeta^{-1})^{\kappa \lambda} \partial_{\kappa} \partial_{\lambda} \Lnenth)
				u_{\alpha}
				- 
				\speed^2
				(\upeta^{-1})^{\kappa \lambda} 
				(\partial_{\kappa} \Lnenth) (\partial_{\lambda} u_{\alpha})
				\notag 
				\\
				& \ \
				+ 
				\speed^2
				(\partial_{\alpha} \Lnenth) (\partial_{\kappa} u^{\kappa})
				\notag 
					\\
		& \ \
			+
			\speed^2
			\TempoverEnth_{;\Lnenth} 
			(\GradEnt^{\kappa} \partial_{\kappa} \Lnenth)
			 u_{\alpha}
			+
			\speed^2
			\TempoverEnth_{;\Ent} 
			\GradEnt_{\kappa} \GradEnt^{\kappa}
			 u_{\alpha}
			 +
			\speed^2
			\TempoverEnth (\GradEnt^{\kappa} \partial_{\kappa} u_{\alpha})
				\notag \\
		& \ \
				-
				\speed^2
				\TempoverEnth_{;\Lnenth} (u^{\kappa} \partial_{\kappa} \Lnenth) \GradEnt_{\alpha} 
				+
				\speed^2
				\TempoverEnth (\partial_{\alpha} u^{\kappa}) \GradEnt_{\kappa} 
				-
				\speed^2
				\TempoverEnth (\partial_{\kappa} u^{\kappa}) \GradEnt_{\alpha}
					\notag
					 \\
			& \ \
				+ 
				(1 - \speed^2) \TempoverEnth (\GradEnt^{\kappa} \partial_{\kappa} \Lnenth) u_{\alpha}
				+ 
				2 \speed^{-1} \speed_{;\Ent} (u^{\kappa} \partial_{\kappa} \Lnenth) \GradEnt_{\alpha}.
				\notag
	\end{align}

	Next, we use the identity \eqref{E:COVARIANTWAVEOPERATORAPPLIEDTOSCALAR} with $\phi := \Lnenth$
	to substitute for the term $\square_g \Lnenth$ on
	LHS~\eqref{E:COVARIANTWAVEENTHALPY}, 
	which yields the identity
	\begin{align} \label{E:FLATWAVEOPENTHID}
			\speed^2 ((\upeta^{-1})^{\kappa \lambda} \partial_{\kappa} \partial_{\lambda} \Lnenth)
			& = 
			\speed^2 
			\left\lbrace
			(\partial_{\kappa} u^{\kappa}) (\partial_{\lambda} u^{\lambda})
			-
			(\partial_{\lambda} u^{\kappa}) (\partial_{\kappa} u^{\lambda}) 
		\right\rbrace
			\\
		& \ \
			+
			(1 - \speed^2)
			u^{\kappa} \partial_{\kappa} (u^{\lambda} \partial_{\lambda} \Lnenth)
			+
			(1 - \speed^2)
			(\partial_{\kappa} u^{\kappa}) (u^{\lambda} \partial_{\lambda} \Lnenth)
					\notag \\
		& \ \
			- 2 \speed^{-1} \speed_{;\Lnenth} 
				(u^{\kappa} \partial_{\kappa} \Lnenth)
				(u^{\lambda} \partial_{\lambda} \Lnenth)
					\notag 
						\\
		& \ \
			+
			n \speed^2 \TempoverEnth \mathcal{D}
				\notag \\
		& \ \
		+
		(1 - \speed^2) \TempoverEnth (\GradEnt^{\kappa} \partial_{\kappa} \Lnenth)
		+
		\speed^2 \TempoverEnth_{;\Lnenth} (\GradEnt^{\kappa} \partial_{\kappa} \Lnenth)
		+
		\speed^2 \TempoverEnth_{;\Ent} \GradEnt_{\kappa} \GradEnt^{\kappa}.
		\notag
	\end{align}
	
	From \eqref{E:FLATWAVEOPENTHID},
	it follows that the product 
	$
	- \speed^2
		((\upeta^{-1})^{\kappa \lambda} \partial_{\kappa} \partial_{\lambda} \Lnenth)
		u_{\alpha}
	$
	on \linebreak RHS~\eqref{E:TOPDERIVATIVEVELOCITYTERMIDENTITY} 
	can be expressed as
	\begin{align} \label{E:SECONDFLATWAVEOPENTHID}
			- \speed^2
				((\upeta^{-1})^{\kappa \lambda} \partial_{\kappa} \partial_{\lambda} \Lnenth)
				u_{\alpha} 
			& = 
			\speed^2 
			\left\lbrace
				(\partial_{\kappa} u^{\lambda}) (\partial_{\lambda} u^{\kappa}) 
				-
				(\partial_{\lambda} u^{\lambda}) (\partial_{\kappa} u^{\kappa})
			\right\rbrace
			u_{\alpha} 
			\\
		& \ \
			+ 
			(\speed^2 - 1)
			\left\lbrace
				u^{\kappa} \partial_{\kappa} (u^{\lambda} \partial_{\lambda} \Lnenth)
			\right\rbrace
			u_{\alpha} 
			\notag 
			\\
			& \ \
			+
			(\speed^2 - 1)
			(\partial_{\kappa} u^{\kappa}) (u^{\lambda} \partial_{\lambda} \Lnenth)
			u_{\alpha} 
					\notag \\
		& \ \
			+ 2 \speed^{-1} \speed_{;\Lnenth} 
				(u^{\kappa} \partial_{\kappa} \Lnenth)
				(u^{\lambda} \partial_{\lambda} \Lnenth)
				u_{\alpha}
					\notag 
					\\
	& \ \
			-
			n \speed^2 \TempoverEnth \mathcal{D} u_{\alpha}
				\notag \\
		& \ \
		+
		(\speed^2 - 1) \TempoverEnth (\GradEnt^{\kappa} \partial_{\kappa} \Lnenth) u_{\alpha}
		-
		\speed^2 \TempoverEnth_{;\Lnenth} (\GradEnt^{\kappa} \partial_{\kappa} \Lnenth) u_{\alpha}
		\notag 
		\\
		& \ \
		-
		\speed^2 \TempoverEnth_{;\Ent} \GradEnt_{\kappa} \GradEnt^{\kappa} u_{\alpha}.
		\notag
	\end{align}
	
	Using \eqref{E:SECONDFLATWAVEOPENTHID} to substitute for the term
	$
	- 
	\speed^2
	((\upeta^{-1})^{\kappa \lambda} \partial_{\kappa} \partial_{\lambda} \Lnenth)
	u_{\alpha}
	$
	on \linebreak
	RHS~\eqref{E:TOPDERIVATIVEVELOCITYTERMIDENTITY},
	we obtain the identity
	\begin{align} \label{E:REWRITTENTOPDERIVATIVEVELOCITYTERMIDENTITY}
		\speed^2 ((\upeta^{-1})^{\kappa \lambda} \partial_{\kappa} \partial_{\lambda} u_{\alpha})
		& = 
			(\speed^2 - 1) (u^{\kappa} \partial_{\kappa} \partial_{\alpha} \Lnenth)
			- 
			(\partial_{\alpha} u^{\kappa}) (\partial_{\kappa} \Lnenth)
			\\
			& \ \
			+ 
			2 \speed^{-1} \speed_{;\Lnenth} (\partial_{\alpha} \Lnenth) (u^{\kappa} \partial_{\kappa} \Lnenth)
			\notag
			\\
	& \ \
			- 
			\speed^2
			\frac{1}{\Enth} 
			\upepsilon_{\lambda \alpha \gamma \delta} 
			((\upeta^{-1})^{\kappa \lambda} \partial_{\kappa} \Lnenth) 
			u^{\gamma} \vort^{\delta}
			\notag
			\\
			& \ \
			+
			\speed^2
			\frac{1}{\Enth} \upepsilon_{\lambda \alpha \gamma \delta} ((\upeta^{-1})^{\kappa \lambda} \partial_{\kappa} u^{\gamma}) 
			\vort^{\delta}
				\notag \\
	& \ \ +
			\speed^2
			\frac{1}{\Enth} \upepsilon_{\lambda \alpha \gamma \delta} u^{\gamma} ((\upeta^{-1})^{\kappa \lambda} \partial_{\kappa} \vort^{\delta})
		\notag
			\\
	& \ \
			+
			\speed^2 
			\left\lbrace
				(\partial_{\kappa} u^{\lambda}) (\partial_{\lambda} u^{\kappa})
				-
				(\partial_{\lambda} u^{\lambda}) (\partial_{\kappa} u^{\kappa})
			\right\rbrace
			u_{\alpha} 
			\notag \\
		& \ \
			+ 
			(\speed^2 - 1)
			\left\lbrace
				u^{\kappa} \partial_{\kappa} (u^{\lambda} \partial_{\lambda} \Lnenth)
			\right\rbrace
			u_{\alpha} 
			\notag 
			\\
			& \ \
			+
			(\speed^2 - 1)
			(\partial_{\kappa} u^{\kappa}) (u^{\lambda} \partial_{\lambda} \Lnenth)
			u_{\alpha} 
					\notag \\
		& \ \
			+ 2 \speed^{-1} \speed_{;\Lnenth} 
				(u^{\kappa} \partial_{\kappa} \Lnenth)
				(u^{\lambda} \partial_{\lambda} \Lnenth)
				u_{\alpha}
					\notag 
					\\
	&  \ \
				- 
				\speed^2
				(\upeta^{-1})^{\kappa \lambda} 
				(\partial_{\kappa} \Lnenth) (\partial_{\lambda} u_{\alpha})
				+ 
				\speed^2
				(\partial_{\alpha} \Lnenth) (\partial_{\kappa} u^{\kappa})
				\notag 
					\\
		& \ \
			+
			\speed^2
			\TempoverEnth (\GradEnt^{\kappa} \partial_{\kappa} u_{\alpha})
				\notag \\
		& \ \
				-
				\speed^2
				\TempoverEnth_{;\Lnenth} (u^{\kappa} \partial_{\kappa} \Lnenth) \GradEnt_{\alpha} 
				+
				\speed^2
				\TempoverEnth (\partial_{\alpha} u^{\kappa}) \GradEnt_{\kappa} 
				-
				\speed^2
				\TempoverEnth (\partial_{\kappa} u^{\kappa}) \GradEnt_{\alpha}
				\notag 
				\\
				& \ \ 
				+ 
				2 \speed^{-1} \speed_{;\Ent} (u^{\kappa} \partial_{\kappa} \Lnenth) \GradEnt_{\alpha}.
					\notag 
	\end{align}
		
		Using \eqref{E:TWOUDERIVATIVESOFUALPHAIDENTITY}
		and \eqref{E:REWRITTENTOPDERIVATIVEVELOCITYTERMIDENTITY}
		to substitute for the first and second products on RHS~\eqref{E:FIRSTCOVARIANTWAVEVELOCITY},
		and reorganizing the terms,
		we deduce (where we have added and subtracted $(\partial_{\kappa} u^{\kappa}) (\partial_{\alpha} \Lnenth)$
		on the third and fourth lines of 
		RHS~\eqref{E:SECONDCOVARIANTWAVEVELOCITY})
		\begin{align}  \label{E:SECONDCOVARIANTWAVEVELOCITY} 
		\square_g u_{\alpha}
		& = \speed^2 \frac{1}{\Enth} 
				\upepsilon_{\lambda \alpha \gamma \delta} u^{\gamma} ((\upeta^{-1})^{\kappa \lambda} \partial_{\kappa} \vort^{\delta})
				\\
		& \ \
				+
				(1 - \speed^2) (u^{\kappa} \partial_{\kappa} u_{\alpha}) (u^{\lambda} \partial_{\lambda} \Lnenth)
				- 
				\speed^2
				(\upeta^{-1})^{\kappa \lambda} 
				(\partial_{\kappa} \Lnenth) 
				(\partial_{\lambda} u_{\alpha})
				\notag
				\\
				& \ \
				+
				\left\lbrace
					(\partial_{\kappa} u^{\kappa}) (\partial_{\alpha} \Lnenth)
					- 
					(\partial_{\alpha} u^{\kappa}) (\partial_{\kappa} \Lnenth)
				\right\rbrace
				\notag \\
	& \ \
			+ 
			(\speed^2 - 1)
			(\partial_{\kappa} u^{\kappa})
			(\partial_{\alpha} \Lnenth)
			+
			(\speed^2 - 1)
			(\partial_{\kappa} u^{\kappa}) 
			(u^{\lambda} \partial_{\lambda} \Lnenth)
			u_{\alpha}
			\notag
			\\
			& \ \
			+
			(\speed^2 - 1)
			(\partial_{\kappa} u^{\kappa}) 
			(u^{\lambda} \partial_{\lambda} u_{\alpha})
					\notag \\
		& \ \
			+
			2 \speed^{-1} \speed_{;\Lnenth} (\partial_{\alpha} \Lnenth) (u^{\kappa} \partial_{\kappa} \Lnenth)
			+ 
			2 \speed^{-1} \speed_{;\Lnenth} 
			(u^{\kappa} \partial_{\kappa} \Lnenth)
			(u^{\lambda} \partial_{\lambda} \Lnenth) 
			u_{\alpha}
			\notag
			\\
			& \ \
			+ 
				2 \speed^{-1} \speed_{;\Lnenth}
				(u^{\kappa} \partial_{\kappa} \Lnenth) 
				(u^{\lambda} \partial_{\lambda} u_{\alpha})
				\notag \\
		& \ \
			- 
			\speed^2
			\frac{1}{\Enth} 
			\upepsilon_{\lambda \alpha \gamma \delta} 
			((\upeta^{-1})^{\kappa \lambda} \partial_{\kappa} \Lnenth) 
			u^{\gamma} \vort^{\delta}
			+
			\speed^2
			\frac{1}{\Enth} \upepsilon_{\lambda \alpha \gamma \delta} ((\upeta^{-1})^{\kappa \lambda} \partial_{\kappa} u^{\gamma}) 
			\vort^{\delta}
			\notag	\\
		& \ \
			+
			\speed^2 
			\left\lbrace
				(\partial_{\kappa} u^{\lambda}) (\partial_{\lambda} u^{\kappa})
				-
				(\partial_{\lambda} u^{\lambda}) (\partial_{\kappa} u^{\kappa}) 
			\right\rbrace
			u_{\alpha} 
		\notag 
		\\
		& \ \
		- 
		\speed^{-1} \speed_{;\Lnenth} 
				(g^{-1})^{\kappa \lambda}
				(\partial_{\kappa} \Lnenth)
				(\partial_{\lambda} u_{\alpha})
		-
		\speed \speed_{;\Ent} (\GradEnt^{\kappa} \partial_{\kappa} u_{\alpha})
		\notag
			\\
	& \ \
				+ 
				2 \speed^{-1} \speed_{;\Ent} (u^{\kappa} \partial_{\kappa} \Lnenth) \GradEnt_{\alpha}
				-
				\TempoverEnth_{;\Lnenth} (u^{\kappa} \partial_{\kappa} \Lnenth) \GradEnt_{\alpha} 
				+
				\speed^2
				\TempoverEnth (\partial_{\alpha} u^{\kappa}) \GradEnt_{\kappa} 
		       \notag
		       \\
		       & \ \
		       	-
				\speed^2
				\TempoverEnth (\partial_{\kappa} u^{\kappa}) \GradEnt_{\alpha}
			+
			\TempoverEnth
			(\GradEnt^{\kappa} \partial_{\kappa} u_{\alpha})
			\notag 
			\\
			& \ \
			+
			\frac{1}{\Enth} 
			(1 - \speed^2) \TempoverEnth
			\upepsilon_{\alpha \beta \gamma \delta} \GradEnt^{\beta} u^{\gamma} \vort^{\delta}
			+
			(\speed^2 - 1) \TempoverEnth
			(\GradEnt^{\kappa} u^{\lambda} \partial_{\lambda} u_{\kappa}) u_{\alpha}.
			\notag
	\end{align}
	
	Next, using \eqref{E:INVERSEACOUSTICALMETRIC},
	we observe the following identity for the two terms
	on the second line of RHS~\eqref{E:SECONDCOVARIANTWAVEVELOCITY}:
	\begin{align} \label{E:IDENTITYDERIVEDFROMVELOCITYEVOLUTION}
			&
			(1 - \speed^2) (u^{\kappa} \partial_{\kappa} u_{\alpha}) (u^{\lambda} \partial_{\lambda} \Lnenth)
				- 
				\speed^2
				(\upeta^{-1})^{\kappa \lambda} 
				(\partial_{\kappa} \Lnenth) (\partial_{\lambda} u_{\alpha})
				\\
				& =
				- 
				(g^{-1})^{\kappa \lambda} (\partial_{\kappa} \Lnenth) (\partial_{\lambda} u_{\alpha}).
				\notag
	\end{align}
	Moreover, using equation \eqref{E:VELOCITYEVOLUTION},
	we see that the terms on the fourth through seventh lines of RHS~\eqref{E:SECONDCOVARIANTWAVEVELOCITY} sum to 
		$
		(\speed^2 - 1)
		\TempoverEnth (\partial_{\kappa} u^{\kappa}) \GradEnt_{\alpha}
		+
		2 \speed^{-1} \speed_{;\Lnenth} \TempoverEnth 
		(u^{\kappa} \partial_{\kappa} \Lnenth)
		\GradEnt_{\alpha}
	$.
	In addition, appealing to definition \eqref{E:UORTHGONALVORTICITYOFONEFORM} with $V_{\alpha} := \vort_{\alpha}$, 
	we obtain the following identity for the first product on RHS~\eqref{E:SECONDCOVARIANTWAVEVELOCITY}:
	$
	\speed^2
			\frac{1}{\Enth} \upepsilon_{\lambda \alpha \gamma \delta} u^{\gamma} ((\upeta^{-1})^{\kappa \lambda} \partial_{\kappa} \vort^{\delta})
	= -
			\speed^2
			\frac{1}{\Enth} 
			\uperpvort_{\alpha}(\vort)		
	$.
	From these facts, 
	\eqref{E:SECONDCOVARIANTWAVEVELOCITY},
	and
	\eqref{E:IDENTITYDERIVEDFROMVELOCITYEVOLUTION},
	we obtain the following equation:
	\begin{align}  \label{E:THIRDCOVARIANTWAVEVELOCITY} 
		\square_g u_{\alpha}
		& = 	
			-
			\speed^2
			\frac{1}{\Enth} 
			\uperpvort_{\alpha}(\vort)
				\\
		& \ \
			- 
			\speed^2
			\frac{1}{\Enth} 
			\upepsilon_{\lambda \alpha \gamma \delta} 
			((\upeta^{-1})^{\kappa \lambda} \partial_{\kappa} \Lnenth) 
			u^{\gamma} \vort^{\delta}
			+
			\speed^2
			\frac{1}{\Enth} \upepsilon_{\lambda \alpha \gamma \delta} ((\upeta^{-1})^{\kappa \lambda} \partial_{\kappa} u^{\gamma}) 
			\vort^{\delta}
				\notag \\
		& \ \
				+
				\left\lbrace
					(\partial_{\kappa} u^{\kappa}) (\partial_{\alpha} \Lnenth)
					- 
					(\partial_{\alpha} u^{\kappa}) (\partial_{\kappa} \Lnenth)
				\right\rbrace
				- 
				(g^{-1})^{\kappa \lambda} 
				(\partial_{\kappa} \Lnenth) 
				(\partial_{\lambda} u_{\alpha})
				\notag \\
		& \ \
			+
			\speed^2 
			\left\lbrace
				(\partial_{\kappa} u^{\lambda}) (\partial_{\lambda} u^{\kappa}) 
				-
				 (\partial_{\lambda} u^{\lambda}) (\partial_{\kappa} u^{\kappa})
			\right\rbrace
			u_{\alpha} 
			\notag
			\\
			& \ \
		- 
		\speed^{-1} \speed_{;\Lnenth} 
		(g^{-1})^{\kappa \lambda}
		(\partial_{\kappa} \Lnenth)
		(\partial_{\lambda} u_{\alpha})
		-
		\speed \speed_{;\Ent} (\GradEnt^{\kappa} \partial_{\kappa} u_{\alpha})
		\notag \\
		& \ \
				+ 
				2 \speed^{-1} \speed_{;\Ent} (u^{\kappa} \partial_{\kappa} \Lnenth) \GradEnt_{\alpha}
			-
			\TempoverEnth_{;\Lnenth} (u^{\kappa} \partial_{\kappa} \Lnenth) \GradEnt_{\alpha} 
			+
				\speed^2
				\TempoverEnth (\partial_{\alpha} u^{\kappa}) \GradEnt_{\kappa} 
				\notag
					\\
	& \ \	
			+
			\TempoverEnth
			(\GradEnt^{\kappa} \partial_{\kappa} u_{\alpha})
			+
			\frac{1}{\Enth} 
			(1 - \speed^2) \TempoverEnth
			\upepsilon_{\alpha \beta \gamma \delta} \GradEnt^{\beta} u^{\gamma} \vort^{\delta}
			\notag
			\\
			& \ \
			+
			(\speed^2 - 1) \TempoverEnth
			(\GradEnt^{\kappa} u^{\lambda} \partial_{\lambda} u_{\kappa}) u_{\alpha}
			\notag
				\\
	& \ \
		-
		\TempoverEnth (\partial_{\kappa} u^{\kappa}) \GradEnt_{\alpha}
		+
		2 \speed^{-1} \speed_{;\Lnenth} \TempoverEnth 
		(u^{\kappa} \partial_{\kappa} \Lnenth)
		\GradEnt_{\alpha}.
		\notag
	\end{align}
	Using definition \eqref{E:MODIFIEDVORTICITYOFVORTICITY}
	to express the product 
	$-
			\speed^2
			\frac{1}{\Enth} 
			\uperpvort_{\alpha}(\vort)
	$
	on RHS~\eqref{E:THIRDCOVARIANTWAVEVELOCITY} 
	as 
	$
	-
			\speed^2
			\frac{1}{\Enth} 
			\mathcal{C}_{\alpha}
	+ 
	\cdots
	$,
	reorganizing the terms on the RHS of the resulting identity,
	and raising the $\alpha$ index with $\upeta^{-1}$,
	we arrive at the desired identity \eqref{E:COVARIANTWAVEVELOCITY}.
	
\end{proof}

\subsection{Covariant wave equation for the entropy}
\label{SS:WAVEEQUATIONFORENT}
In this subsection, we derive the covariant wave equation \eqref{E:MAINTHMENTAUXILIARYWAVEQUATION}.

\begin{proposition} [Covariant wave equation for $\Ent$]
\label{P:COVARIANTWAVEENT}
Assume that $(\Lnenth,\Ent,u^{\alpha})$ is a $C^2$ solution 
to \eqref{E:ENTHALPYEVOLUTION}-\eqref{E:ENTROPYEVOLUTION} + \eqref{E:UISUNITLENGTH}.
Then the entropy $\Ent$ 
verifies the following covariant wave equation:
\begin{align}  \label{E:ENTAUXILIARYWAVEQUATION} 
		\square_g \Ent
		& = 	
		\speed^2 n \mathcal{D}
		+
		\GradEnt^{\kappa} \partial_{\kappa} \Lnenth
		- 
		\speed^2 \GradEnt^{\kappa} \partial_{\kappa} \Lnenth
		- 
		\speed \speed_{;\Lnenth} 
		\GradEnt^{\kappa} \partial_{\kappa} \Lnenth
		-
		\speed \speed_{;\Ent}
		\GradEnt_{\kappa} \GradEnt^{\kappa}.
\end{align}

\end{proposition}

\begin{proof}
Applying \eqref{E:COVARIANTWAVEOPERATORAPPLIEDTOSCALAR} with $\phi := \Ent$,
using \eqref{E:INVERSEACOUSTICALMETRIC} to algebraically substitute
for the factor of $(g^{-1})^{\kappa \lambda}$ on RHS~\eqref{E:COVARIANTWAVEOPERATORAPPLIEDTOSCALAR},
and using the evolution equation \eqref{E:ENTROPYEVOLUTION}
(which implies that many factors on RHS~\eqref{E:COVARIANTWAVEOPERATORAPPLIEDTOSCALAR} vanish),
we deduce, in view of definition \eqref{E:ENTROPYGRADIENT},
that
\begin{align} \label{E:FIRSTENTAUXILIARYWAVEQUATION}
		\square_g \Ent
		& = 
				\speed^2 \partial_{\kappa} \GradEnt^{\kappa}
				- 
				\speed \speed_{;\Lnenth} 
				\GradEnt^{\kappa} \partial_{\kappa} \Lnenth
				-
				\speed \speed_{;\Ent}
				\GradEnt_{\kappa} \GradEnt^{\kappa}.
\end{align}
We then solve for $\partial_{\kappa} \GradEnt^{\kappa}$ in terms
of the remaining terms in definition \eqref{E:MODIFIEDDIVERGENCEOFENTROPY}
and then use the resulting identity to algebraically substitute for the factor
$\partial_{\kappa} \GradEnt^{\kappa}$ in the first product on
RHS~\eqref{E:FIRSTENTAUXILIARYWAVEQUATION}, 
which in total yields the desired equation \eqref{E:ENTAUXILIARYWAVEQUATION}.

\end{proof}

\section{Transport equations for the entropy gradient and the modified divergence of the entropy}
	\label{S:TRANSPORTFORMODDIVERGENCEOFENTROPY}
		In this section, with the help of the preliminary identities of Lemma~\ref{L:IDENTITIES},
		we derive equations \eqref{E:MAINTHMENTROPYGRADIENTEVOLUTION} and \eqref{E:MAINTHMSMODIFIEDDIVERGENCEENTROPYGRADIENTTRANSPORT}.
		We start by deriving \eqref{E:MAINTHMENTROPYGRADIENTEVOLUTION} (more precisely, its $\upeta$-dual).
	
	\begin{proposition}[Transport equation for the entropy gradient]
	\label{P:ENTROPYGRADIENTTRANSPORT}
	Assume that $(\Lnenth,\Ent,u^{\alpha})$ is a $C^2$ solution 
	to \eqref{E:ENTHALPYEVOLUTION}-\eqref{E:ENTROPYEVOLUTION} + \eqref{E:UISUNITLENGTH}.
	Then the rectangular components the $\GradEnt_{\alpha}$
	of the entropy gradient vectorfield (see Def.\,\ref{D:ENTROPYGRADIENT})
	verify the following transport equations:
	\begin{align}	\label{E:ANOTHERENTROPYGRADIENTEVOLUTION}
		u^{\kappa} \partial_{\kappa} \GradEnt_{\alpha} 
		& = 
			-	 
			\GradEnt^{\kappa} \partial_{\kappa} u_{\alpha}
			-
			\frac{1}{\Enth} \upepsilon_{\alpha \beta \gamma \delta} \GradEnt^{\beta} u^{\gamma} \vort^{\delta}
			- 
			(\GradEnt^{\kappa} \partial_{\kappa} \Lnenth) u_{\alpha}
			+ 
			\TempoverEnth \GradEnt_{\kappa} \GradEnt^{\kappa} u_{\alpha}
				\\
			& = 
			-	 
			\GradEnt^{\kappa} \partial_{\kappa} u_{\alpha}
			-
			\frac{1}{\Enth} \upepsilon_{\alpha \beta \gamma \delta} \GradEnt^{\beta} u^{\gamma} \vort^{\delta}
			+
			\GradEnt^{\kappa} (u^{\lambda} \partial_{\lambda} u_{\kappa}) u_{\alpha}.
			\notag
		\end{align}
	\end{proposition}
	
	\begin{proof}
	From equation \eqref{E:ENTROPYGRADIENTEVOLUTION},
	the identity \eqref{E:PARITALALPHAUBETAMINUSPARTIALBETAUALPHA}, 
	\eqref{E:UISUNITLENGTH},
	and
	\eqref{E:VELOCITYANDENTGRADIENTAREMINKOWSKIPERP},
	we deduce
	\begin{align} \label{E:FIRSTSTEPANOTHERENTROPYGRADIENTEVOLUTION}
		u^{\kappa} \partial_{\kappa} \GradEnt_{\alpha}
		& = - 
				\GradEnt^{\kappa} \partial_{\kappa} u_{\alpha}
				-
			\frac{1}{\Enth} \upepsilon_{\alpha \beta \gamma \delta} \GradEnt^{\beta} u^{\gamma} \vort^{\delta}
			+ 
			(\partial_{\alpha} \Lnenth) \GradEnt^{\kappa} u_{\kappa}
			- 
			(\GradEnt^{\kappa} \partial_{\kappa} \Lnenth) u_{\alpha}
		\\
		& \ \
			-
		\TempoverEnth 
		\left\lbrace
			\GradEnt_{\alpha} \GradEnt^{\kappa} u_{\kappa}  
			- 
			\GradEnt^{\kappa} \GradEnt_{\kappa} u_{\alpha} 
		\right\rbrace
		\notag
				\\
		& = 
			-	 
			\GradEnt^{\kappa} \partial_{\kappa} u_{\alpha}
			-
			\frac{1}{\Enth} \upepsilon_{\alpha \beta \gamma \delta} \GradEnt^{\beta} u^{\gamma} \vort^{\delta}
			- 
			(\GradEnt^{\kappa} \partial_{\kappa} \Lnenth) u_{\alpha}
			+
			\TempoverEnth \GradEnt^{\kappa} \GradEnt_{\kappa} u_{\alpha},
			\notag
	\end{align}
	which yields the first line of \eqref{E:ANOTHERENTROPYGRADIENTEVOLUTION}.
	To obtain the second line of \eqref{E:ANOTHERENTROPYGRADIENTEVOLUTION}
	from the first, we use the identity \eqref{E:IDVELOCITYDERIVATIVEOFVELOCITYCONTRACTEDWITHENTROPYGRADIENT}.
		
	\end{proof}

	We now derive equation \eqref{E:MAINTHMSMODIFIEDDIVERGENCEENTROPYGRADIENTTRANSPORT}.

	\begin{proposition}[Transport equation for the modified divergence of the entropy]
	\label{P:DIVERGENCEOFENTROPYGRADIENTTRANSPORT}
	Assume that $(\Lnenth,\Ent,u^{\alpha})$ is a $C^3$ solution 
	to \eqref{E:ENTHALPYEVOLUTION}-\eqref{E:ENTROPYEVOLUTION} + \eqref{E:UISUNITLENGTH}.
	Then the modified divergence of the entropy gradient 
	$\mathcal{D}$, which is defined in \eqref{E:MODIFIEDDIVERGENCEOFENTROPY}, 
	verifies the following transport equation:
	\begin{align}	\label{E:MODIFIEDDIVERGENCEENTROPYGRADIENTTRANSPORT}
		u^{\kappa} \partial_{\kappa} \mathcal{D}
		& = 
			\frac{2}{n}
			\left\lbrace
				(\partial_{\kappa} \GradEnt^{\kappa}) (\partial_{\lambda} u^{\lambda})
				-
				(\partial_{\lambda} \GradEnt^{\kappa}) (\partial_{\kappa} u^{\lambda})
			\right\rbrace
			\\
			& \ \
			+
				\frac{1}{n}
				\speed^{-2}
				u^{\kappa}
				\left\lbrace
					(\partial_{\kappa} \Lnenth) 
					(\partial_{\lambda} \GradEnt^{\lambda})
					-
					(\partial_{\lambda} \Lnenth)
					(\partial_{\kappa} \GradEnt^{\lambda})
				\right\rbrace	
				\notag
				\\
		& \ \
			+
					\frac{1}{n}
					\speed^{-2} 
					\GradEnt^{\kappa}
					\left\lbrace
						(\partial_{\kappa} u^{\lambda}) (\partial_{\lambda} \Lnenth)
						-
						(\partial_{\lambda} u^{\lambda}) (\partial_{\kappa} \Lnenth)
					\right\rbrace
					\notag \\
		& \ \
			+
			\frac{\GradEnt_{\kappa} \mathcal{C}^{\kappa}}{n \Enth} 
			\notag \\
		& \ \
			+
			\frac{(1 - \speed^{-2})}{n \Enth}
			\upepsilon^{\alpha \beta \gamma \delta} 
			\GradEnt_{\alpha}
			u_{\beta}
			(\partial_{\gamma} \Lnenth) \vort_{\delta}
			+
			\frac{1}{n \Enth} 
			\upepsilon^{\alpha \beta \gamma \delta} 
			\GradEnt_{\alpha} (\partial_{\beta} u_{\gamma}) \vort_{\delta}
			\notag	\\
		& \ \
			+
			\frac{(\Temp - \Temp_{;\Lnenth})}{n \Enth}
			\GradEnt^{\kappa} (\GradEnt^{\lambda} \partial_{\lambda} u_{\kappa})
			-
			\frac{2 \TempoverEnth }{n} \GradEnt^{\kappa} (\GradEnt^{\lambda} \partial_{\lambda} u_{\kappa})
			\notag \\
		& \ \
			+
			\frac{(\Temp_{;\Lnenth} - \Temp) }{n \Enth}
			\GradEnt_{\kappa} \GradEnt^{\kappa} 
			(\partial_{\lambda} u^{\lambda})
			+ 
			\frac{2 \speed^{-1} \speed_{;\Ent} }{n} \GradEnt_{\kappa} \GradEnt^{\kappa} (\partial_{\lambda} u^{\lambda})
			\notag
			\\
			& \ \
			-
			\frac{\speed^2 \TempoverEnth_{;\Lnenth}}{n}
			 \GradEnt_{\kappa} \GradEnt^{\kappa} (\partial_{\lambda} u^{\lambda})
			+
			\frac{\TempoverEnth}{n}
		 \GradEnt_{\kappa} \GradEnt^{\kappa} (\partial_{\lambda} u^{\lambda}).
			\notag
	\end{align}
	
	\end{proposition}
	
	\begin{proof}
		We apply $(\upeta^{-1})^{\alpha \lambda} \partial_{\lambda}$ to equation \eqref{E:ANOTHERENTROPYGRADIENTEVOLUTION}
		(where we use the first equality in \eqref{E:ANOTHERENTROPYGRADIENTEVOLUTION})
		and use the evolution equation \eqref{E:ENTROPYEVOLUTION}
		and the identity
		\eqref{E:IDVELOCITYDERIVATIVEOFVELOCITYCONTRACTEDWITHENTROPYGRADIENT}
		to deduce
		\begin{align}	\label{E:FIRSTSTEPDIVERGENCEENTROPYGRADIENTTRANSPORT}
		u^{\kappa} \partial_{\kappa} \partial_{\lambda} \GradEnt^{\lambda} 
		& = 
			- u^{\kappa} \partial_{\kappa} (\GradEnt^{\lambda} \partial_{\lambda} \Lnenth)
			-
			2	
			(\partial_{\lambda} \GradEnt^{\kappa}) (\partial_{\kappa} u^{\lambda})
				\\
		& \ \
			-	 
			\GradEnt^{\kappa} \partial_{\kappa} \partial_{\lambda} u^{\lambda}
			- 
			(\GradEnt^{\kappa} \partial_{\kappa} \Lnenth) (\partial_{\lambda} u^{\lambda})
				\notag \\
		& \ \
			+
			\frac{1}{\Enth}
			\upepsilon^{\alpha \beta \gamma \delta} 
			(\partial_{\alpha} \Lnenth)
			\GradEnt_{\beta} u_{\gamma} \vort_{\delta}
			-
			\frac{1}{\Enth} 
			\upepsilon^{\alpha \beta \gamma \delta} \GradEnt_{\beta} u_{\gamma} (\partial_{\alpha} \vort_{\delta})
			\notag 
			\\
			& \ \
			-
			\frac{1}{\Enth} 
			\upepsilon^{\alpha \beta \gamma \delta} 
			\GradEnt_{\beta} (\partial_{\alpha} u_{\gamma}) \vort_{\delta}
			\notag \\
		& \ \
			+
			\TempoverEnth_{;\Lnenth} \GradEnt_{\kappa} \GradEnt^{\kappa} (u^{\lambda} \partial_{\lambda} \Lnenth)
			+
			2 \TempoverEnth \GradEnt^{\kappa} (u^{\lambda} \partial_{\lambda} \GradEnt_{\kappa})
			+
			\TempoverEnth \GradEnt_{\kappa} \GradEnt^{\kappa} (\partial_{\lambda} u^{\lambda}).
			\notag
	\end{align}
	Next, we use the evolution equations
	\eqref{E:ENTHALPYEVOLUTION} and \eqref{E:ENTROPYEVOLUTION}
	to rewrite the third product on RHS~\eqref{E:FIRSTSTEPDIVERGENCEENTROPYGRADIENTTRANSPORT} as follows:
	\begin{align} \label{E:IDFORENTROPYDERIVATIVEOFGRADIENTVELOCITY}
		-	 
		\GradEnt^{\kappa} \partial_{\kappa} \partial_{\lambda} u^{\lambda}
		& = \GradEnt^{\kappa} \partial_{\kappa} (\speed^{-2} u^{\lambda} \partial_{\lambda} \Lnenth)
			\\
		& = \speed^{-2} (\GradEnt^{\kappa} u^{\lambda} \partial_{\lambda} \partial_{\kappa} \Lnenth)
				- 
				2 \speed^{-3} \speed_{;\Lnenth} (\GradEnt^{\kappa} \partial_{\kappa} \Lnenth) (u^{\lambda} \partial_{\lambda} \Lnenth)
				\notag
				\\
				& \ \
				- 
				2 \speed^{-3} \speed_{;\Ent} \GradEnt_{\kappa} \GradEnt^{\kappa} (u^{\lambda} \partial_{\lambda} \Lnenth)
				+
				\speed^{-2} (\GradEnt^{\kappa} \partial_{\kappa} u^{\lambda}) (\partial_{\lambda} \Lnenth)
					\notag \\
		& = 
				u^{\kappa} \partial_{\kappa} 
				(\speed^{-2} \GradEnt^{\lambda} \partial_{\lambda} \Lnenth)
				+
				\speed^{-2} (\GradEnt^{\kappa} \partial_{\kappa} u^{\lambda}) (\partial_{\lambda} \Lnenth)
				-
				\speed^{-2} (u^{\kappa} \partial_{\kappa}  \GradEnt^{\lambda}) (\partial_{\lambda} \Lnenth)
				\notag \\
		& \ \ 
				- 
				2 \speed^{-3} \speed_{;\Ent} \GradEnt_{\kappa} \GradEnt^{\kappa} (u^{\lambda} \partial_{\lambda} \Lnenth).
					\notag 
	\end{align}
	Next, with the help of the evolution equation \eqref{E:ENTHALPYEVOLUTION},
	we decompose the second and third products on RHS~\eqref{E:IDFORENTROPYDERIVATIVEOFGRADIENTVELOCITY} as follows:
	\begin{align} \label{E:FIRSTDECOMPFORNULLFORM}
			\speed^{-2} (\GradEnt^{\kappa} \partial_{\kappa} u^{\lambda}) (\partial_{\lambda} \Lnenth)
			& = \speed^{-2} (\GradEnt^{\kappa} \partial_{\kappa} \Lnenth) (\partial_{\lambda} u^{\lambda})
				\\
			& \ \
				+
					\speed^{-2} 
					\left\lbrace
						(\GradEnt^{\kappa} \partial_{\kappa} u^{\lambda}) (\partial_{\lambda} \Lnenth)
						-
						(\partial_{\lambda} u^{\lambda}) (\GradEnt^{\kappa} \partial_{\kappa} \Lnenth)
					\right\rbrace,
				\notag		\\
		- \speed^{-2} (u^{\kappa} \partial_{\kappa} \GradEnt^{\lambda}) (\partial_{\lambda} \Lnenth)
			& = -
			\speed^{-2} (u^{\kappa} \partial_{\kappa} \Lnenth) (\partial_{\lambda} \GradEnt^{\lambda})
					\label{E:SECONDDECOMPFORNULLFORM} \\
			&  \ \
				+
				\speed^{-2}
				\left\lbrace
					(u^{\kappa} \partial_{\kappa} \Lnenth) (\partial_{\lambda} \GradEnt^{\lambda})
					-
					(u^{\kappa} \partial_{\kappa} \GradEnt^{\lambda}) (\partial_{\lambda} \Lnenth)
				\right\rbrace
				\notag
					\\
			& = (\partial_{\kappa} u^{\kappa}) (\partial_{\lambda} \GradEnt^{\lambda})
					\notag \\
			&  \ \
				+
				\speed^{-2}
				\left\lbrace
					(u^{\kappa} \partial_{\kappa} \Lnenth) (\partial_{\lambda} \GradEnt^{\lambda})
					-
					(u^{\kappa} \partial_{\kappa}  \GradEnt^{\lambda}) (\partial_{\lambda} \Lnenth)
				\right\rbrace.
				\notag
	\end{align}
	Using \eqref{E:FIRSTDECOMPFORNULLFORM}-\eqref{E:SECONDDECOMPFORNULLFORM}
	to substitute for the second and third products on RHS~\eqref{E:IDFORENTROPYDERIVATIVEOFGRADIENTVELOCITY}
	and then using the resulting identity to substitute
	for the third product on RHS~\eqref{E:FIRSTSTEPDIVERGENCEENTROPYGRADIENTTRANSPORT},
	we obtain the following equation:
	\begin{align}	\label{E:SECONDSTEPDIVERGENCEENTROPYGRADIENTTRANSPORT}
		&
		u^{\kappa} \partial_{\kappa} 
		\left\lbrace
			\partial_{\lambda} \GradEnt^{\lambda} 
			+
			\GradEnt^{\lambda} \partial_{\lambda} \Lnenth
			-
			\speed^{-2} (\GradEnt^{\lambda} \partial_{\lambda} \Lnenth)
		\right\rbrace
			\\
		& = 
			(\partial_{\kappa} \GradEnt^{\kappa}) (\partial_{\lambda} u^{\lambda})
			-
			2	(\partial_{\lambda} \GradEnt^{\kappa}) (\partial_{\kappa} u^{\lambda})
				\notag \\
		& \ \
				- 
				(\GradEnt^{\kappa} \partial_{\kappa} \Lnenth) (\partial_{\lambda} u^{\lambda})
				+
				\speed^{-2} (\GradEnt^{\kappa} \partial_{\kappa} \Lnenth) (\partial_{\lambda} u^{\lambda})
			\notag \\
		& \ \
			+
					\speed^{-2} 
					\left\lbrace
						(\GradEnt^{\kappa} \partial_{\kappa} u^{\lambda}) (\partial_{\lambda} \Lnenth)
						-
						(\partial_{\lambda} u^{\lambda}) (\GradEnt^{\kappa} \partial_{\kappa} \Lnenth)
					\right\rbrace
					\notag \\
		& \ \
				+
				\speed^{-2}
				\left\lbrace
					(u^{\kappa} \partial_{\kappa} \Lnenth) (\partial_{\lambda} \GradEnt^{\lambda})
					-
					(u^{\kappa} \partial_{\kappa}  \GradEnt^{\lambda}) (\partial_{\lambda} \Lnenth)
				\right\rbrace
					\notag \\
		& \ \
			+
			\frac{1}{\Enth}
			\upepsilon^{\alpha \beta \gamma \delta} 
			(\partial_{\alpha} \Lnenth)
			\GradEnt_{\beta} u_{\gamma} \vort_{\delta}
			-
			\frac{1}{\Enth} 
			\upepsilon^{\alpha \beta \gamma \delta} \GradEnt_{\beta} u_{\gamma} (\partial_{\alpha} \vort_{\delta})
			-
			\frac{1}{\Enth} 
			\upepsilon^{\alpha \beta \gamma \delta} 
			\GradEnt_{\beta} (\partial_{\alpha} u_{\gamma}) \vort_{\delta}
			\notag \\
		& \ \
			+
			\TempoverEnth_{;\Lnenth} \GradEnt_{\kappa} \GradEnt^{\kappa} (u^{\lambda} \partial_{\lambda} \Lnenth)
			+
			2 \TempoverEnth \GradEnt^{\kappa} (u^{\lambda} \partial_{\lambda} \GradEnt_{\kappa})
			+
			\TempoverEnth \GradEnt_{\kappa} \GradEnt^{\kappa} (\partial_{\lambda} u^{\lambda})
			\notag
			\\
			& \ \ 
			- 
			2 \speed^{-3} \speed_{;\Ent} \GradEnt_{\kappa} \GradEnt^{\kappa} (u^{\lambda} \partial_{\lambda} \Lnenth).
			\notag
	\end{align}
	We now multiply both sides of \eqref{E:SECONDSTEPDIVERGENCEENTROPYGRADIENTTRANSPORT}
	by $1/n$, commute the factor of $1/n$ under the operator $u^{\kappa} \partial_{\kappa}$
	on LHS~\eqref{E:SECONDSTEPDIVERGENCEENTROPYGRADIENTTRANSPORT},
	use equation
	\eqref{E:PARTIACLECURRENTCONSERVATION}
	(which in particular implies that 
	$u^{\kappa} \partial_{\kappa}(1/n) = (1/n) \partial_{\kappa} u^{\kappa}$),
	and use equation \eqref{E:ENTHALPYEVOLUTION}
	to replace the two factors of
	$u^{\lambda} \partial_{\lambda} \Lnenth$
	on the last and next-to-last lines of RHS~\eqref{E:SECONDSTEPDIVERGENCEENTROPYGRADIENTTRANSPORT}
	with
	$
	- \speed^2 \partial_{\lambda} u^{\lambda}
	$,
	thereby obtaining the following equation:
	\begin{align}	\label{E:THIRDSTEPDIVERGENCEENTROPYGRADIENTTRANSPORT}
		&
		u^{\kappa} \partial_{\kappa} 
		\left\lbrace
			\frac{1}{n} (\partial_{\lambda} \GradEnt^{\lambda})
			+
			\frac{1}{n} (\GradEnt^{\lambda} \partial_{\lambda} \Lnenth)
			-
			\frac{1}{n} \speed^{-2} (\GradEnt^{\lambda} \partial_{\lambda} \Lnenth)
		\right\rbrace
			\\
		& = 
			\frac{2}{n}
			\left\lbrace
				(\partial_{\kappa} \GradEnt^{\kappa}) (\partial_{\lambda} u^{\lambda})
				-
				(\partial_{\lambda} \GradEnt^{\kappa}) (\partial_{\kappa} u^{\lambda})
			\right\rbrace
				\notag \\
		& \ \
			+
					\frac{1}{n}
					\speed^{-2} 
					\left\lbrace
						(\GradEnt^{\kappa} \partial_{\kappa} u^{\lambda}) (\partial_{\lambda} \Lnenth)
						-
						(\partial_{\lambda} u^{\lambda}) (\GradEnt^{\kappa} \partial_{\kappa} \Lnenth)
					\right\rbrace
					\notag \\
		& \ \
				+
				\frac{1}{n}
				\speed^{-2}
				\left\lbrace
					(u^{\kappa} \partial_{\kappa} \Lnenth) (\partial_{\lambda} \GradEnt^{\lambda})
					-
					(\partial_{\lambda} \Lnenth) (u^{\kappa} \partial_{\kappa}  \GradEnt^{\lambda}) 
				\right\rbrace
					\notag \\
		& \ \
			+
			\frac{1}{n \Enth}
			\upepsilon^{\alpha \beta \gamma \delta} 
			(\partial_{\alpha} \Lnenth)
			\GradEnt_{\beta} u_{\gamma} \vort_{\delta}
			-
			\frac{1}{n \Enth} 
			\upepsilon^{\alpha \beta \gamma \delta} \GradEnt_{\beta} u_{\gamma} (\partial_{\alpha} \vort_{\delta})
			\notag
			\\
			& \ \
			-
			\frac{1}{n \Enth} 
			\upepsilon^{\alpha \beta \gamma \delta} 
			\GradEnt_{\beta} (\partial_{\alpha} u_{\gamma}) \vort_{\delta}
			\notag \\
		& \ \
			+
			\frac{2 \TempoverEnth}{n} \GradEnt^{\kappa} (u^{\lambda} \partial_{\lambda} \GradEnt_{\kappa})
			+ 
			\frac{2 \speed^{-1} \speed_{;\Ent}}{n}\GradEnt_{\kappa} \GradEnt^{\kappa} (\partial_{\lambda} u^{\lambda})
			-
			\frac{\speed^2 \TempoverEnth_{;\Lnenth} }{n}
			\GradEnt_{\kappa} \GradEnt^{\kappa} (\partial_{\lambda} u^{\lambda})
			\notag
			\\
			& \ \ 
			+
			\frac{\TempoverEnth}{n} \GradEnt_{\kappa} \GradEnt^{\kappa} (\partial_{\lambda} u^{\lambda}).
			\notag
	\end{align}
	Next, we use definitions 
	\eqref{E:UORTHGONALVORTICITYOFONEFORM}  
	and
	\eqref{E:MODIFIEDVORTICITYOFVORTICITY}
	and the identity \eqref{E:VELOCITYANDENTGRADIENTAREMINKOWSKIPERP}
	to obtain the following identity
	for the second product on the fourth line of
	RHS~\eqref{E:THIRDSTEPDIVERGENCEENTROPYGRADIENTTRANSPORT}:
	\begin{align} \label{E:MODIFIEDVORTID}
		-
		\frac{1}{n \Enth} 
		\upepsilon^{\alpha \beta \gamma \delta} \GradEnt_{\beta} u_{\gamma} (\partial_{\alpha} \vort_{\delta})
		&  =
			-
			\frac{1}{n \Enth} 
			\upepsilon^{\alpha \beta \gamma \delta} \GradEnt_{\alpha} u_{\beta} (\partial_{\gamma} \vort_{\delta})
				\\
		& = \frac{\mathcal{C}^{\kappa} \GradEnt_{\kappa}}{n \Enth} 
			-
			\frac{1}{n \Enth}
			\speed^{-2}
			\upepsilon^{\alpha \beta \gamma \delta} 
			\GradEnt_{\alpha}
			u_{\beta}
			(\partial_{\gamma} \Lnenth) \vort_{\delta}
			\notag	\\
		& \ \
			+
			\frac{(\Temp_{;\Lnenth} - \Temp) }{n \Enth}
			\GradEnt_{\kappa} \GradEnt^{\kappa} 
			(\partial_{\lambda} u^{\lambda})
			+
			\frac{(\Temp - \Temp_{;\Lnenth})}{n \Enth}
			\GradEnt^{\kappa} (\GradEnt^{\lambda} \partial_{\lambda} u_{\kappa}).
			\notag
	\end{align}
	Using \eqref{E:MODIFIEDVORTID}
	to substitute for the second product on the fourth line of
	\linebreak RHS~\eqref{E:THIRDSTEPDIVERGENCEENTROPYGRADIENTTRANSPORT},
	using \eqref{E:ENTROPYGRADIENTEVOLUTION} to express the 
	first product on the next-to-last line of 
	RHS~\eqref{E:THIRDSTEPDIVERGENCEENTROPYGRADIENTTRANSPORT} as
	$
	\frac{2 \TempoverEnth}{n} \GradEnt^{\kappa} (u^{\lambda} \partial_{\lambda} \GradEnt_{\kappa})
	=
	-
	\frac{2 \TempoverEnth }{n} \GradEnt^{\kappa} (\GradEnt^{\lambda} \partial_{\lambda} u_{\kappa})
	$,
	and noting that the terms in parentheses on LHS~\eqref{E:THIRDSTEPDIVERGENCEENTROPYGRADIENTTRANSPORT}
	are equal to $\mathcal{D}$
	(see \eqref{E:MODIFIEDDIVERGENCEOFENTROPY}),
  we arrive at the desired evolution equation \eqref{E:MODIFIEDDIVERGENCEENTROPYGRADIENTTRANSPORT}.
	\end{proof}

\section{Transport equation for the vorticity}
\label{S:TRANSPORTFORVORTICITY}
In this section, with the help of the preliminary identities of Lemma~\ref{L:IDENTITIES},
we derive equation \eqref{E:MAINTHMVORTICITYTRANSPORT}. We also derive some preliminary 
identities that, in the next section, we will use when deriving equation \eqref{E:MAINTHMTRANSPORTFORMODIFIEDVORTICITYOFVORTICITY}.
We collect all of these results in the following proposition.

	\begin{proposition}[Transport equation for the vorticity]
	\label{P:VORTICITYTRANSPORT}
	Assume that $(\Lnenth,\Ent,u^{\alpha})$ is a $C^3$ solution 
	to \eqref{E:ENTHALPYEVOLUTION}-\eqref{E:ENTROPYEVOLUTION} + \eqref{E:UISUNITLENGTH}.
	Then the rectangular components $\vort^{\alpha}$ of the vorticity vectorfield defined 
	in \eqref{E:VORTICITYDEF} verify the following transport equations:
	\begin{align}	\label{E:VORTICITYTRANSPORT}
		u^{\kappa} \partial_{\kappa} \vort^{\alpha} 
		& = 
			\vort^{\kappa} \partial_{\kappa} u^{\alpha}
			-
			(\partial_{\kappa} u^{\kappa}) \vort^{\alpha} 
			-	
			(\vort^{\kappa} \partial_{\kappa} \Lnenth) u^{\alpha}
			\\
			& \ \ 
			+
			(\Temp - \Temp_{;\Lnenth})
			\upepsilon^{\alpha \beta \gamma \delta} u_{\beta}
			(\partial_{\gamma} \Lnenth) \GradEnt_{\delta}
			+
			\TempoverEnth \vort^{\kappa} \GradEnt_{\kappa} u^{\alpha}.
			\notag
	\end{align}
	
	Moreover, the following identity holds:
		\begin{align}	\label{E:LIEDERIVATIVEVORTICITYONEFORM}
		(\mathcal{L}_u \vort_{\flat})_{\alpha}
		& = 
			\vort^{\kappa} \partial_{\kappa} u_{\alpha}
			+
			\vort^{\kappa} (\partial_{\alpha} u_{\kappa})
			-
			(\partial_{\kappa} u^{\kappa}) \vort_{\alpha} 
			+
			(u^{\kappa} \partial_{\kappa} u_{\lambda}) u_{\alpha} \vort^{\lambda}
			\\
			& \ \
			+
			(\Temp - \Temp_{;\Lnenth})
			\upepsilon_{\alpha}^{\ \beta \gamma \delta} 
			u_{\beta}
			(\partial_{\gamma} \Lnenth) \GradEnt_{\delta}.
			\notag
	\end{align}
	
	In addition, the following identity holds:
	\begin{align} \label{E:EXTERIORDERIVATIVELIEDERIVATIVEVORTICITYONEFORM}
	& (d \mathcal{L}_u \vort_{\flat})_{\alpha \beta}
		\\
	& =
			(\partial_{\alpha} \vort^{\kappa}) (\partial_{\kappa} u_{\beta})
			-
			(\partial_{\beta} \vort^{\kappa}) (\partial_{\kappa} u_{\alpha})
		\notag		 \\
	& \ \
			+
			\vort^{\kappa} \partial_{\kappa} \partial_{\alpha} u_{\beta}
			-
			\vort^{\kappa} \partial_{\kappa} \partial_{\beta} u_{\alpha}
			\notag
				\\
		& \ \
			+
			(\partial_{\alpha} \vort^{\kappa}) (\partial_{\beta} u_{\kappa})
			-
			(\partial_{\beta} \vort^{\kappa}) (\partial_{\alpha} u_{\kappa})
				\notag \\
		& \ \
			- 
			(\partial_{\alpha} \partial_{\kappa} u^{\kappa}) \vort_{\beta}
			+
			(\partial_{\beta} \partial_{\kappa} u^{\kappa}) \vort_{\alpha}
				\notag \\
		& \ \
			- 
			(\partial_{\kappa} u^{\kappa}) (\partial_{\alpha} \vort_{\beta})
			+
			(\partial_{\kappa} u^{\kappa}) (\partial_{\beta} \vort_{\alpha})
				\notag \\
		& \ \
			+ 
			(\partial_{\alpha} u_{\beta}) \vort^{\lambda}
			(u^{\kappa} \partial_{\kappa} u_{\lambda}) 
			-
			(\partial_{\beta} u_{\alpha}) \vort^{\lambda}
			(u^{\kappa} \partial_{\kappa} u_{\lambda}) 
			\notag \\
		& \ \
			+ 
			u_{\beta} (\partial_{\alpha} \vort^{\lambda})
			(u^{\kappa} \partial_{\kappa} u_{\lambda})
			-
			u_{\alpha} (\partial_{\beta} \vort^{\lambda})
			(u^{\kappa} \partial_{\kappa} u_{\lambda})
			\notag \\
		& \ \
			+ 
			u_{\beta} \vort^{\lambda}
			(\partial_{\alpha} u^{\kappa}) (\partial_{\kappa} u_{\lambda})
			-
			u_{\alpha} \vort^{\lambda}
			(\partial_{\beta} u^{\kappa}) (\partial_{\kappa} u_{\lambda})
				\notag \\
		&  \ \
			+ 
			u_{\beta} \vort^{\lambda}
			(u^{\kappa} \partial_{\kappa} \partial_{\alpha} u_{\lambda}) 
			-
			u_{\alpha} \vort^{\lambda}
			(u^{\kappa} \partial_{\kappa} \partial_{\beta} u_{\lambda})
				\notag
					\\
		& \ \
			+
			(\Temp_{\Lnenth} - \Temp_{;\Lnenth;\Lnenth})
			\upepsilon_{\beta \kappa}^{\ \ \gamma \delta} 
			u^{\kappa} 
			(\partial_{\alpha} \Lnenth)
			(\partial_{\gamma} \Lnenth) \GradEnt_{\delta}
			\notag
			\\
			& \ \
			+
			(\Temp_{;\Lnenth;\Lnenth} - \Temp_{\Lnenth})
			\upepsilon_{\alpha \kappa}^{\ \ \gamma \delta} 
			u^{\kappa} 
			(\partial_{\beta} \Lnenth)
			(\partial_{\gamma} \Lnenth) \GradEnt_{\delta}
			\notag
			\\
		& \ \
			+
			(\Temp_{;\Ent} - \Temp_{;\Lnenth;\Ent})
			\upepsilon_{\beta \kappa}^{\ \ \gamma \delta} 
			u^{\kappa} 
			\GradEnt_{\alpha}
			(\partial_{\gamma} \Lnenth) \GradEnt_{\delta}
			+
			(\Temp_{;\Lnenth;\Ent} - \Temp_{;\Ent})
			\upepsilon_{\alpha \kappa}^{\ \ \gamma \delta} 
			u^{\kappa} 
			\GradEnt_{\beta}
			(\partial_{\gamma} \Lnenth) \GradEnt_{\delta}
			\notag
				\\
		& \ \
			+
			(\Temp - \Temp_{;\Lnenth}) 
			\upepsilon_{\beta \kappa}^{\ \ \gamma \delta} 
			(\partial_{\alpha} u^{\kappa})
			(\partial_{\gamma} \Lnenth) \GradEnt_{\delta}
			+
			(\Temp_{;\Lnenth} - \Temp) 
			\upepsilon_{\alpha \kappa}^{\ \ \gamma \delta} 
			(\partial_{\beta} u^{\kappa})
			(\partial_{\gamma} \Lnenth) \GradEnt_{\delta}
			\notag 
			\\
		& \ \
			+
			(\Temp - \Temp_{;\Lnenth}) 
			\upepsilon_{\beta \kappa}^{\ \ \gamma \delta} 
			 u^{\kappa}
			(\partial_{\alpha} \partial_{\gamma} \Lnenth) \GradEnt_{\delta}
			+
			(\Temp_{;\Lnenth} - \Temp) 
			\upepsilon_{\alpha \kappa}^{\ \ \gamma \delta} 
			u^{\kappa}
			(\partial_{\beta} \partial_{\gamma} \Lnenth) \GradEnt_{\delta}
			\notag 
				\\
		& \ \
			+
			(\Temp - \Temp_{;\Lnenth}) 
			\upepsilon_{\beta \kappa}^{\ \ \gamma \delta} 
			 u^{\kappa}
			 (\partial_{\gamma} \Lnenth) (\partial_{\alpha} \GradEnt_{\delta})
			+
			(\Temp_{;\Lnenth} - \Temp) 
			\upepsilon_{\alpha \kappa}^{\ \ \gamma \delta} 
			u^{\kappa}
			(\partial_{\gamma} \Lnenth) (\partial_{\beta} \GradEnt_{\delta}).
			\notag 
	\end{align}
	
	Finally, the rectangular components $\uperpvort^{\alpha}(\vort)$ of the vorticity of the vorticity,
	which is defined 
	by \eqref{E:UORTHGONALVORTICITYOFONEFORM} and \eqref{E:VORTICITYDEF},
	verify the following transport equations:
	\begin{align} \label{E:VORTICITYOFVORTICITYFIRSTEVOLUTION}
		&u^{\kappa} \partial_{\kappa} \uperpvort^{\alpha}(\vort)
		 = 
			\uperpvort^{\kappa}(\vort) \partial_{\kappa} u^{\alpha}
			- 
			(\partial_{\kappa} u^{\kappa}) \uperpvort^{\alpha}(\vort)
			\\
			& \ \ 
			+
				u^{\alpha} 
				(u^{\kappa} \partial_{\kappa} u_{\beta}) 
				\uperpvort^{\beta}(\vort)
				\notag
					\\
	& \ \
			+ 
			\upepsilon^{\alpha \beta \gamma \delta} 
			u_{\beta}
			(\partial_{\gamma} \partial_{\kappa} u^{\kappa}) \vort_{\delta}
			-
			\upepsilon^{\alpha \beta \gamma \delta} 
			u_{\beta}
			(\vort^{\kappa} \partial_{\kappa} \partial_{\gamma} u_{\delta})
				\notag \\
	& \ \
			+ \upepsilon^{\alpha \beta \gamma \delta} (u^{\kappa} \partial_{\kappa} u_{\beta}) 
				(u^{\lambda} \partial_{\lambda} \vort_{\delta}) u_{\gamma}
			+ \upepsilon^{\alpha \beta \gamma \delta} (u^{\kappa} \partial_{\kappa} u_{\beta}) 
				\vort^{\lambda} (\partial_{\delta} u_{\lambda}) u_{\gamma}
			\notag		\\
	& \ \
			-
			\upepsilon^{\alpha \beta \gamma \delta} 
			u_{\beta}
			(\partial_{\gamma} \vort^{\kappa}) (\partial_{\delta} u_{\kappa})
		-
		\upepsilon^{\alpha \beta \gamma \delta} 
			u_{\beta}
			(\partial_{\gamma} \vort^{\kappa}) (\partial_{\kappa} u_{\delta})
				\notag \\
	& \ \
			+ 
			\upepsilon^{\alpha \beta \gamma \delta} 
			u_{\beta}
			(\partial_{\kappa} u^{\kappa}) (\partial_{\gamma} \vort_{\delta})
				\notag \\
		& \ \
			- 
			\upepsilon^{\alpha \beta \gamma \delta} 
			u_{\beta}
			(\partial_{\gamma} u_{\delta}) \vort^{\lambda}
			(u^{\kappa} \partial_{\kappa} u_{\lambda})
			\notag
				\\
		& \ \
			+
			(\Temp_{\Lnenth} - \Temp_{;\Lnenth;\Lnenth})
			\GradEnt^{\alpha}  
			(\upeta^{-1})^{\kappa \lambda} (\partial_{\kappa} \Lnenth)
			(\partial_{\lambda} \Lnenth)
			\notag
			\\
			& \ \ 
			+
			(\Temp_{\Lnenth} - \Temp_{;\Lnenth;\Lnenth})
			 \GradEnt^{\alpha}
			(u^{\kappa} \partial_{\kappa} \Lnenth)
			(u^{\lambda} \partial_{\lambda} \Lnenth)
			\notag
			\\
	& \ \
			+
			(\Temp_{;\Lnenth;\Lnenth} - \Temp_{\Lnenth})
			u^{\alpha} 
			(\GradEnt^{\kappa} \partial_{\kappa} \Lnenth)
			(u^{\lambda} \partial_{\lambda} \Lnenth) 
			\notag
			\\
			& \ \
			+
			(\Temp_{;\Lnenth;\Lnenth} - \Temp_{\Lnenth})
			((\upeta^{-1})^{\alpha \kappa}\partial_{\kappa} \Lnenth) 
			(\GradEnt^{\lambda} \partial_{\lambda} \Lnenth)
			\notag
			\\
	& \ \
			+
			(\Temp_{;\Ent} - \Temp_{;\Lnenth;\Ent})
			\GradEnt^{\alpha} 
			(\GradEnt^{\kappa} \partial_{\kappa} \Lnenth) 
			+
			(\Temp_{;\Lnenth;\Ent} - \Temp_{;\Ent})
			u^{\alpha} 
			\GradEnt_{\kappa} \GradEnt^{\kappa}
			(u^{\lambda} \partial_{\lambda} \Lnenth) 
			\notag
			\\
			& \ \
			+
			(\Temp_{;\Lnenth;\Ent} - \Temp_{;\Ent})
			\GradEnt_{\kappa} \GradEnt^{\kappa}
			((\upeta^{-1})^{\alpha \lambda} \partial_{\lambda} \Lnenth) 
				\notag
				\\
		& \ \
			+
			(\Temp - \Temp_{;\Lnenth}) 
			\GradEnt^{\alpha}
			(\partial_{\kappa} u^{\kappa})
			(u^{\lambda} \partial_{\lambda} \Lnenth) 
			+
			(\Temp_{;\Lnenth} - \Temp) 
			(\GradEnt^{\kappa} \partial_{\kappa} u^{\alpha})
			(u^{\lambda} \partial_{\lambda} \Lnenth) 
				\notag 
			\\
		& \ \
			+
			(\Temp - \Temp_{;\Lnenth})  
			\GradEnt^{\alpha} 
			((\upeta^{-1})^{\kappa \lambda} \partial_{\kappa} \partial_{\lambda} \Lnenth) 
			+
			(\Temp - \Temp_{;\Lnenth}) 
			\GradEnt^{\alpha}  
			(u^{\kappa} u^{\lambda} \partial_{\kappa} \partial_{\lambda} \Lnenth) 
				\notag 
				\\
		& \ \
			+
			(\Temp_{;\Lnenth} - \Temp) 
			 u^{\alpha}
			(\GradEnt^{\kappa} u^{\lambda} \partial_{\kappa} \partial_{\lambda} \Lnenth) 
			+
			(\Temp_{;\Lnenth} - \Temp) 
			 (\upeta^{-1})^{\alpha \lambda}
			(\GradEnt^{\kappa} \partial_{\kappa} \partial_{\lambda} \Lnenth) 
				\notag 
				\\
		& \ \
			+
			(\Temp - \Temp_{;\Lnenth}) 
			(\upeta^{-1})^{\kappa \lambda}
			(\partial_{\kappa} \Lnenth) (\partial_{\lambda}\GradEnt^{\alpha})
			+
			(\Temp - \Temp_{;\Lnenth})  
			(u^{\kappa} \partial_{\kappa} \Lnenth) (u^{\lambda} \partial_{\lambda} \GradEnt^{\alpha})
			\notag 
			\\
		& \ \
			+
			(\Temp_{;\Lnenth} - \Temp) 
			u^{\alpha} 
			(u^{\kappa} \partial_{\kappa} \Lnenth) (\partial_{\lambda} \GradEnt^{\lambda})
			+
			(\Temp_{;\Lnenth} - \Temp) 
			((\upeta^{-1})^{\alpha \kappa} \partial_{\kappa} \Lnenth) (\partial_{\lambda} \GradEnt^{\lambda})
			\notag 
			\\
		& \ \
			+
		(\Temp_{;\Lnenth} - \Temp) 
		 ((\upeta^{-1})^{\alpha \kappa} \partial_{\kappa} \Lnenth) u_{\beta} (u^{\lambda} \partial_{\lambda} \GradEnt^{\beta})
			\notag
			\\
			& \ \
			+
			(\Temp - \Temp_{;\Lnenth}) 
			u^{\alpha}
			(\upeta^{-1})^{\kappa \lambda} \partial_{\kappa} \Lnenth) u_{\beta} (\partial_{\lambda} \GradEnt^{\beta}).
			\notag 
\end{align}

\end{proposition}

\begin{remark}
	Note that RHS~\eqref{E:VORTICITYOFVORTICITYFIRSTEVOLUTION} features some terms that explicitly depend on two derivatives of $u$,
	falsely suggesting that there is derivative loss, that is, 
	that $\uperpvort^{\alpha}(\vort)$ cannot be more regular than $\partial^2 u$.
	For this reason, equation \eqref{E:VORTICITYOFVORTICITYFIRSTEVOLUTION}
	is not suitable for obtaining top-order energy estimates for
	$\uperpvort^{\alpha}(\vort)$.
	To overcome this difficulty, we will
	derive a $\mbox{\upshape transport}$-$\mbox{\upshape div}$-$\mbox{\upshape curl}$ system 
	for $\vort$ that does not lose derivatives; 
	see Prop.\,\ref{P:EQUATIONSFORMODIFIEDVORTICITYVORTICITY}.
	\end{remark}

\begin{proof}[Proof of Prop.\,\ref{P:VORTICITYTRANSPORT}]
	We first prove \eqref{E:VORTICITYTRANSPORT}.
	From definition \eqref{E:VORTICITYDEF} and the Lie differentiation formula \eqref{E:LIEDERIVATIVE},
	we deduce that
	\begin{align} \label{E:FIRSTSTEPVORTICITYTRANSPORT}
		u^{\kappa} \partial_{\kappa} \vort^{\alpha}
		-
		\vort^{\kappa} \partial_{\kappa} u^{\alpha}
		& = \mathcal{L}_u \vort^{\alpha}
		=  - \frac{1}{2} 
			\mathcal{L}_u 
			\left\lbrace
				\upepsilon^{\alpha \beta \gamma \delta} u_{\beta} (d (\Enth u_{\flat}))_{\gamma \delta}
			\right\rbrace.
	\end{align}
	Using \eqref{E:FIRSTSTEPVORTICITYTRANSPORT}, 
	the Leibniz rule for Lie derivatives,
	definition \eqref{E:VORTICITYDEF},
	\eqref{E:LIEUVOLUMEFORMINDICESUP},
	the first identity in \eqref{E:LIEUVELOCITYONEFORM},
	\eqref{E:LIEUOFEXTERIORDERIVATIVEOFENTHALPHYTIMESVELOCITYONEFORMEVOLUTION},
	and \eqref{E:CONTRACTIONVOLUMEFORMANDEXTERIORDERIVATIVEOFENTHALPHYTIMESVELOCITY},
	we compute that
	\begin{align}	\label{E:SECONDSTEPVORTICITYTRANSPORT}
		u^{\kappa} \partial_{\kappa} \vort^{\alpha} 
		& = 
			\vort^{\kappa} \partial_{\kappa} u^{\alpha}
			-
			(\partial_{\kappa} u^{\kappa}) \vort^{\alpha}
			+
			(u^{\kappa} \partial_{\kappa} u_{\lambda}) u^{\alpha} \vort^{\lambda}
			-
			(u^{\kappa} \partial_{\kappa} u_{\lambda}) u^{\lambda} \vort^{\alpha}
			\\
			& \ \ 
			-
			\Temp 
			\upepsilon^{\alpha \beta \gamma \delta}
			(u^{\kappa} \partial_{\kappa} u_{\beta})
			\GradEnt_{\gamma} u_{\delta}
			- 
			\Temp_{;\Lnenth}
			\upepsilon^{\alpha \beta \gamma \delta} u_{\beta}
			(\partial_{\gamma} \Lnenth) \GradEnt_{\delta}.
			\notag
	\end{align}
	Using \eqref{E:DERIVATIVEOFVELOCITYCONTRACTEDWITHVELOCITYISZERO}, we see that
	the fourth product on RHS~\eqref{E:SECONDSTEPVORTICITYTRANSPORT} vanishes.
	Next, we use 
	\eqref{E:VELOCITYEVOLUTION}
	and \eqref{E:VORTISORTHGONALTOU}
	to obtain the following identity for the third product on RHS~\eqref{E:SECONDSTEPVORTICITYTRANSPORT}:
	$
	(u^{\kappa} \partial_{\kappa} u_{\lambda}) u^{\alpha} \vort^{\lambda} 
	= - (\vort^{\kappa} \partial_{\kappa} \Lnenth) u^{\alpha}
		+
		\TempoverEnth \vort^{\kappa} \GradEnt_{\kappa} u^{\alpha}
	$.
	Next, we use
	\eqref{E:VELOCITYEVOLUTION}
	to obtain the following identity for the fifth product on RHS~\eqref{E:SECONDSTEPVORTICITYTRANSPORT}:
	$
	-
			\Temp 
			\upepsilon^{\alpha \beta \gamma \delta}
			(u^{\kappa} \partial_{\kappa} u_{\beta})
			\GradEnt_{\gamma} u_{\delta}
	= 	\Temp 
			\upepsilon^{\alpha \beta \gamma \delta}
			(\partial_{\beta} \Lnenth)
			\GradEnt_{\gamma} u_{\delta}
			=
			\Temp
			\upepsilon^{\alpha \beta \gamma \delta} u_{\beta}
			(\partial_{\gamma} \Lnenth) \GradEnt_{\delta}
	$.
	Substituting these two identities for the third and fifth products on RHS~\eqref{E:SECONDSTEPVORTICITYTRANSPORT},
	we arrive at the desired identity \eqref{E:VORTICITYTRANSPORT}.
	
	Equation \eqref{E:LIEDERIVATIVEVORTICITYONEFORM} follows from 
	the Lie derivative identity
\begin{align}
(\mathcal{L}_u \vort_{\flat})_{\alpha} = u^{\kappa} \partial_{\kappa} \vort_{\alpha} + \vort^{\kappa} \partial_{\alpha} u_{\kappa}
\notag
\end{align}	
	(see \eqref{E:LIEDERIVATIVE}),
	from using \eqref{E:DERIVATIVEOFVELOCITYCONTRACTEDWITHVELOCITYISZERO} to observe the vanishing of 
	the fourth product on RHS~\eqref{E:SECONDSTEPVORTICITYTRANSPORT},
	and from using the identity for the fifth product on RHS~\eqref{E:SECONDSTEPVORTICITYTRANSPORT}
	proved in the previous paragraph.
	
\eqref{E:EXTERIORDERIVATIVELIEDERIVATIVEVORTICITYONEFORM} then follows from taking the exterior derivative
of equation \eqref{E:LIEDERIVATIVEVORTICITYONEFORM} and carrying out straightforward computations.
		
To derive \eqref{E:VORTICITYOFVORTICITYFIRSTEVOLUTION}, we first
use definition \eqref{E:UORTHGONALVORTICITYOFONEFORM} to 
deduce
\begin{align} \label{E:LIEUVORTICITYOFVORTICITYDEFINITION}
		\mathcal{L}_u \uperpvort^{\alpha}(\vort)
		& = 
			- \frac{1}{2} \mathcal{L}_u (\upepsilon^{\alpha \beta \gamma \delta} u_{\beta} (d \vort_{\flat})_{\gamma \delta}).
\end{align}
Next, we use \eqref{E:LIEUVORTICITYOFVORTICITYDEFINITION},
the Leibniz rule for Lie derivatives,
\eqref{E:LIEUVOLUMEFORMINDICESUP},
the first equality in \eqref{E:LIEUVELOCITYONEFORM},
and the standard commutation identity
$\mathcal{L}_u d \vort_{\flat} = d \mathcal{L}_u \vort_{\flat}$
to deduce
\begin{align} \label{E:FIRSTCOMPUTATIONLIEUVORTICITYOFVORTICITY}
		\mathcal{L}_u \uperpvort^{\alpha}(\vort)
		& = 
			- (\partial_{\kappa} u^{\kappa}) \uperpvort^{\alpha}(\vort)
			- \frac{1}{2} \upepsilon^{\alpha \beta \gamma \delta} (u^{\kappa} \partial_{\kappa} u_{\beta}) (d \vort_{\flat})_{\gamma \delta}
			\\
			& \ \
			- \frac{1}{2} \upepsilon^{\alpha \beta \gamma \delta} u_{\beta} (d \mathcal{L}_u \vort_{\flat})_{\gamma \delta}.
			\notag
\end{align}
Next, using \eqref{E:CONTRACTIONOFEXTERIORDERIVAITVEOFVORTICITYOFVORTICITYAGAINSTVOLUMEFORM},
we express the second product
on RHS~\eqref{E:FIRSTCOMPUTATIONLIEUVORTICITYOFVORTICITY} as follows:
\begin{align} \label{E:CONTRACTIONVOLUMEFORMUDERIVATIVEOFUEXTERIORDERIVATIVEOFVORTICITYOFVORTICITY}
	- \frac{1}{2} \upepsilon^{\alpha \beta \gamma \delta} (u^{\kappa} \partial_{\kappa} u_{\beta}) (d \vort_{\flat})_{\gamma \delta}
	& = 
	- \uperpvort^{\alpha}(\vort) (u^{\kappa} \partial_{\kappa} u_{\beta}) u^{\beta}
	\\
	& \ \
	+
	u^{\alpha} (u^{\kappa} \partial_{\kappa} u_{\beta}) \uperpvort^{\beta}(\vort) 
	\notag
	\\
	& \ \
	-
	\upepsilon^{\alpha \beta \gamma \delta} (u^{\kappa} \partial_{\kappa} u_{\beta})
	(u^{\lambda} \partial_{\lambda} \vort_{\gamma}) u_{\delta}
	\notag
	\\
	& \ \
	+
	\upepsilon^{\alpha \beta \gamma \delta} (u^{\kappa} \partial_{\kappa} u_{\beta})
	(u^{\lambda} \partial_{\gamma} \vort_{\lambda}) u_{\delta}.
	\notag
\end{align}
Next, using \eqref{E:DERIVATIVEOFVELOCITYCONTRACTEDWITHVELOCITYISZERO},
we observe that the first product on 
RHS~\eqref{E:CONTRACTIONVOLUMEFORMUDERIVATIVEOFUEXTERIORDERIVATIVEOFVORTICITYOFVORTICITY} vanishes.
From this fact, 
the Lie derivative identity
$\mathcal{L}_u \uperpvort^{\alpha}(\vort) 
= 
u^{\kappa} \partial_{\kappa} \uperpvort^{\alpha}(\vort) 
- 
\uperpvort^{\kappa}(\vort) \partial_{\kappa} u^{\alpha}$
(see \eqref{E:LIEDERIVATIVE}),
\eqref{E:FIRSTCOMPUTATIONLIEUVORTICITYOFVORTICITY},
and 
\eqref{E:CONTRACTIONVOLUMEFORMUDERIVATIVEOFUEXTERIORDERIVATIVEOFVORTICITYOFVORTICITY},
we deduce
\begin{align} \label{E:SECONDCOMPUTATIONLIEUVORTICITYOFVORTICITY}
		u^{\kappa} \partial_{\kappa} \uperpvort^{\alpha}(\vort) 
		& = 
			\uperpvort^{\kappa}(\vort) \partial_{\kappa} u^{\alpha}
			- 
			(\partial_{\kappa} u^{\kappa}) \uperpvort^{\alpha}(\vort)
			\\
			& \ \
			+
			u^{\alpha} (u^{\kappa} \partial_{\kappa} u_{\beta}) \uperpvort^{\beta}(\vort) 
			\notag
			\\
	& \ \
	+
	\upepsilon^{\alpha \beta \gamma \delta} 
	(u^{\kappa} \partial_{\kappa} u_{\beta})
	(u^{\lambda} \partial_{\gamma} \vort_{\lambda}) u_{\delta}
	\notag
	\\
	& \ \
	-
	\upepsilon^{\alpha \beta \gamma \delta} (u^{\kappa} \partial_{\kappa} u_{\beta})
	(u^{\lambda} \partial_{\lambda} \vort_{\gamma}) u_{\delta}
		\notag \\ 
	& \ \
		- \frac{1}{2} \upepsilon^{\alpha \beta \gamma \delta} u_{\beta} (d \mathcal{L}_u \vort_{\flat})_{\gamma \delta}.
		\notag
\end{align}

Next, we use \eqref{E:TRANSFERDERIVATIVESFROMVORTICITYTOVELOCITY}
and the antisymmetry of $\upepsilon^{\cdots}$
to express the product on the third line of RHS~\eqref{E:SECONDCOMPUTATIONLIEUVORTICITYOFVORTICITY}
as
\begin{align}
	\upepsilon^{\alpha \beta \gamma \delta} 
	(u^{\kappa} \partial_{\kappa} u_{\beta})
	(u^{\lambda} \partial_{\gamma} \vort_{\lambda}) u_{\delta}
	=
	\upepsilon^{\alpha \beta \gamma \delta} (u^{\kappa} \partial_{\kappa} u_{\beta}) 
	\vort^{\lambda} (\partial_{\delta} u_{\lambda}) u_{\gamma},
	\notag
\end{align}
use the antisymmetry of $\upepsilon^{\cdots}$
to express the product on the fourth line of RHS~\eqref{E:SECONDCOMPUTATIONLIEUVORTICITYOFVORTICITY}
\begin{align}
-
	\upepsilon^{\alpha \beta \gamma \delta} (u^{\kappa} \partial_{\kappa} u_{\beta})
	(u^{\lambda} \partial_{\lambda} \vort_{\gamma}) u_{\delta}
=
 \upepsilon^{\alpha \beta \gamma \delta} (u^{\kappa} \partial_{\kappa} u_{\beta}) 
				(u^{\lambda} \partial_{\lambda} \vort_{\delta}) u_{\gamma},
				\notag
\end{align}
use \eqref{E:EXTERIORDERIVATIVELIEDERIVATIVEVORTICITYONEFORM} to substitute for the factor 
$(d \mathcal{L}_u \vort_{\flat})_{\gamma \delta}$
in the last product on RHS~
\linebreak \eqref{E:SECONDCOMPUTATIONLIEUVORTICITYOFVORTICITY},
and carry out straightforward computations,
thereby deducing that
\begin{align} \label{E:ALMOSTDONEVORTICITYOFVORTICITYFIRSTEVOLUTION}
		u^{\kappa} \partial_{\kappa} \uperpvort^{\alpha}(\vort)
		& = 
			\uperpvort^{\kappa}(\vort) \partial_{\kappa} u^{\alpha}
			- 
			(\partial_{\kappa} u^{\kappa}) \uperpvort^{\alpha}(\vort)
			\\
			& \ \
			+
				u^{\alpha} 
				(u^{\kappa} \partial_{\kappa} u_{\beta}) 
				\uperpvort^{\beta}(\vort)
				\notag
					\\
	& \ \
			+ 
			\upepsilon^{\alpha \beta \gamma \delta} 
			u_{\beta}
			(\partial_{\gamma} \partial_{\kappa} u^{\kappa}) \vort_{\delta}
			-
			\upepsilon^{\alpha \beta \gamma \delta} 
			u_{\beta}
			(\vort^{\kappa} \partial_{\kappa} \partial_{\gamma} u_{\delta})
				\notag \\
	& \ \
			+ \upepsilon^{\alpha \beta \gamma \delta} (u^{\kappa} \partial_{\kappa} u_{\beta}) 
				(u^{\lambda} \partial_{\lambda} \vort_{\delta}) u_{\gamma}
			\notag
			\\
			& \ \
			+ 
			\upepsilon^{\alpha \beta \gamma \delta} (u^{\kappa} \partial_{\kappa} u_{\beta}) 
			\vort^{\lambda} (\partial_{\delta} u_{\lambda}) u_{\gamma}
			\notag		\\
	& \ \
			-
			\upepsilon^{\alpha \beta \gamma \delta} 
			u_{\beta}
			(\partial_{\gamma} \vort^{\kappa}) (\partial_{\delta} u_{\kappa})
		-
		\upepsilon^{\alpha \beta \gamma \delta} 
			u_{\beta}
			(\partial_{\gamma} \vort^{\kappa}) (\partial_{\kappa} u_{\delta})
				\notag \\
	& \ \
			+ 
			\upepsilon^{\alpha \beta \gamma \delta} 
			u_{\beta}
			(\partial_{\kappa} u^{\kappa}) (\partial_{\gamma} \vort_{\delta})
				\notag \\
		& \ \
			- 
			\upepsilon^{\alpha \beta \gamma \delta} 
			u_{\beta}
			(\partial_{\gamma} u_{\delta}) \vort^{\lambda}
			(u^{\kappa} \partial_{\kappa} u_{\lambda})
			\notag
				\\
		& \ \
			+
			(\Temp_{;\Lnenth;\Lnenth} - \Temp_{\Lnenth})
			\upepsilon^{\alpha \beta \gamma \delta}
			\upepsilon_{\delta \kappa}^{\ \ \mu \nu} 
			u_{\beta}
			u^{\kappa} 
			(\partial_{\gamma} \Lnenth)
			(\partial_{\mu} \Lnenth) \GradEnt_{\nu}
			\notag
			\\
		& \ \
			+
			(\Temp_{;\Lnenth;\Ent} - \Temp_{;\Ent})
			\upepsilon^{\alpha \beta \gamma \delta}
			\upepsilon_{\delta \kappa}^{\ \ \mu \nu}
			u_{\beta}
			u^{\kappa} 
			\GradEnt_{\gamma}
			(\partial_{\mu} \Lnenth) \GradEnt_{\nu}
				\notag
				\\
		& \ \
			+
			(\Temp_{;\Lnenth} - \Temp) 
			\upepsilon^{\alpha \beta \gamma \delta}
			\upepsilon_{\delta \kappa}^{\ \ \mu \nu}
			u_{\beta}
			(\partial_{\gamma} u^{\kappa})
			(\partial_{\mu} \Lnenth) \GradEnt_{\nu}
				\notag 
			\\
		& \ \
			+
			(\Temp_{;\Lnenth} - \Temp) 
			\upepsilon^{\alpha \beta \gamma \delta}
			\upepsilon_{\delta \kappa}^{\ \ \mu \nu}
			u_{\beta}
			u^{\kappa}
			(\partial_{\gamma} \partial_{\mu} \Lnenth) \GradEnt_{\nu}
				\notag 
				\\
		& \ \
			+
			(\Temp_{;\Lnenth} - \Temp) 
			\upepsilon^{\alpha \beta \gamma \delta}
			\upepsilon_{\delta \kappa}^{\ \ \mu \nu}
			u_{\beta} u^{\kappa}
			(\partial_{\mu} \Lnenth) (\partial_{\gamma} \GradEnt_{\nu}).
			\notag 
\end{align}
Finally, we
use the identity
\begin{align}
\upepsilon^{\alpha \beta \gamma \delta}
\upepsilon_{\delta \kappa}^{\ \ \mu \nu}
& =
(\upeta^{-1})^{\nu \alpha} \updelta_{\kappa}^{\beta} (\upeta^{-1})^{\mu \gamma}
-
(\upeta^{-1})^{\nu \alpha} \updelta_{\kappa}^{\gamma} (\upeta^{-1})^{\mu \beta}
\notag
\\
& \ \
+
(\upeta^{-1})^{\nu \gamma} \updelta_{\kappa}^{\alpha} (\upeta^{-1})^{\mu \beta}
-
(\upeta^{-1})^{\nu \gamma} \updelta_{\kappa}^{\beta} (\upeta^{-1})^{\mu \alpha}
\notag
\\
& \ \
+
(\upeta^{-1})^{\nu \beta} \updelta_{\kappa}^{\gamma} (\upeta^{-1})^{\mu \alpha}
-
(\upeta^{-1})^{\nu \beta} \updelta_{\kappa}^{\alpha} (\upeta^{-1})^{\mu \gamma}
\notag
\end{align}
to substitute for the five products
$
\upepsilon^{\alpha \beta \gamma \delta}
\upepsilon_{\delta \kappa}^{\ \ \mu \nu}
$
on RHS~\eqref{E:ALMOSTDONEVORTICITYOFVORTICITYFIRSTEVOLUTION}.
Also using 
\eqref{E:UISUNITLENGTH},
\eqref{E:VELOCITYANDENTGRADIENTAREMINKOWSKIPERP},
and \eqref{E:DERIVATIVEOFVELOCITYCONTRACTEDWITHVELOCITYISZERO},
we arrive at the desired identity \eqref{E:VORTICITYOFVORTICITYFIRSTEVOLUTION}.

\end{proof}

\section{The $\mbox{\upshape transport}$-$\mbox{\upshape div}$-$\mbox{\upshape curl}$ system for the vorticity}
\label{S:TRANSPORTDIVCURLFORVORTICITY}
Our main goal in this section is to derive equations \eqref{E:MAINTHMDIVOFVORTICITY} and 
\eqref{E:MAINTHMTRANSPORTFORMODIFIEDVORTICITYOFVORTICITY}.
We accomplish this in Prop.\,\ref{P:EQUATIONSFORMODIFIEDVORTICITYVORTICITY}. 
Before proving the proposition, we will first establish some preliminary identities.

\subsection{Preliminary identities}
In the next lemma, we derive a large collection of identities that we will use in deriving
the transport equation verified by the vectorfield
$\mathcal{C}^{\alpha}$ defined in \eqref{E:MODIFIEDVORTICITYOFVORTICITY}.

\begin{lemma}[Identification of the null structure of some terms tied to 
	the $\mbox{\upshape transport}$-$\mbox{\upshape div}$-$\mbox{\upshape curl}$ system for the vorticity] 
	\label{L:IDOFNULLSTRUCTUREFORDIVCURLTRANSPORTSYSTEM}
	Assume that $(\Lnenth,\Ent,u^{\alpha})$ is a $C^2$ solution 
	to \eqref{E:ENTHALPYEVOLUTION}-\eqref{E:ENTROPYEVOLUTION} + \eqref{E:UISUNITLENGTH}.
	Then the following identities hold
	for some of the terms on the third through seventh lines of RHS~\eqref{E:VORTICITYOFVORTICITYFIRSTEVOLUTION}:
	\begin{subequations}
\begin{align} \label{E:TERM1ID}
	\upepsilon^{\alpha \beta \gamma \delta} 
	u_{\beta}
	(\partial_{\gamma} \partial_{\kappa} u^{\kappa}) \vort_{\delta}
	& = 		
	-
	u^{\kappa} \partial_{\kappa} 
	\left\lbrace
		\speed^{-2}
		\upepsilon^{\alpha \beta \gamma \delta} 
		u_{\beta}
		(\partial_{\gamma} \Lnenth) \vort_{\delta}
	\right\rbrace
		\\
		& \ \
		-
		2
		(\partial_{\kappa} u^{\kappa})
		\speed^{-2}
		\upepsilon^{\alpha \beta \gamma \delta} 
		u_{\beta}
		(\partial_{\gamma} \Lnenth) \vort_{\delta}
		\notag
		\\
& \ \
	+
		\speed^{-2}
		\upepsilon^{\alpha \beta \gamma \delta} 
		(u^{\kappa} \partial_{\kappa} u_{\beta})
		(\partial_{\gamma} \Lnenth) \vort_{\delta}
	\notag
		\\
& \ \
		+
		\speed^{-2}
		\upepsilon^{\alpha \beta \gamma \delta} 
		u_{\beta}
		(\partial_{\gamma} \Lnenth) 
		\vort^{\kappa} (\partial_{\delta} u_{\kappa})
			\notag \\
	& \ \
	+
	\speed^{-2}
	(\Temp - \Temp_{;\Lnenth})  
	(\GradEnt^{\kappa} \partial_{\kappa} \Lnenth)
	((\upeta^{-1})^{\alpha \lambda} \partial_{\lambda} \Lnenth) 
		\notag \\
& \ \
	+
	\speed^{-2}
	(\Temp - \Temp_{;\Lnenth}) 
	u^{\alpha}
	(\GradEnt^{\kappa} \partial_{\kappa} \Lnenth)
	(u^{\lambda} \partial_{\lambda} \Lnenth)
	\notag \\
& \ \
+ \speed^{-2}
	(\Temp_{;\Lnenth} - \Temp) 
	\GradEnt^{\alpha} 
	(u^{\kappa} \partial_{\kappa} \Lnenth)
	(u^{\lambda} \partial_{\lambda} \Lnenth)
	\notag \\
& \ \
	+	 
	\speed^{-2}
	(\Temp_{;\Lnenth} - \Temp) 
	\GradEnt^{\alpha}
	(\upeta^{-1})^{\kappa \lambda}
	(\partial_{\kappa} \Lnenth)
	(\partial_{\lambda} \Lnenth)
		\notag
		\\
	& \ \
	+
	\speed^{-2}
	\upepsilon^{\alpha \beta \gamma \delta} 
	u_{\beta}
	\left\lbrace
		(\partial_{\kappa} u^{\kappa}) (\partial_{\gamma} \Lnenth) 
		- 
		(\partial_{\gamma} u^{\kappa}) (\partial_{\kappa} \Lnenth) 
	\right\rbrace
	\vort_{\delta}	
		\notag
			\\
	& \ \
		+
	2 \speed^{-3} \speed_{;\Ent} 
	(u^{\kappa} \partial_{\kappa} \Lnenth)
	\upepsilon^{\alpha \beta \gamma \delta} 
	u_{\beta}
	\GradEnt_{\gamma}
	\vort_{\delta},
	\notag
\end{align}

\begin{align} \label{E:TERM2ID}
	-
	\upepsilon^{\alpha \beta \gamma \delta} 
	u_{\beta}
	(\vort^{\kappa} \partial_{\kappa} \partial_{\gamma} u_{\delta})
	& = 
		\frac{1}{\Enth} (\vort^{\kappa} \partial_{\kappa} \vort^{\alpha})
		-
		\frac{1}{\Enth} \vort^{\alpha} (\vort^{\kappa} \partial_{\kappa} \Lnenth)
			\\
	&  \ \
		-
		\frac{1}{\Enth}
		u^{\alpha} \vort^{\lambda} 
		(\vort^{\kappa} \partial_{\kappa} u_{\lambda})
		+
		\upepsilon^{\alpha \beta \gamma \delta} 
		u_{\beta}
		(\partial_{\gamma} \Lnenth)
		\vort^{\kappa} (\partial_{\delta} u_{\kappa})
				\notag
				\\
	& \ \
			-
			\TempoverEnth
			\upepsilon^{\alpha \beta \gamma \delta}
			u_{\beta}
			\GradEnt_{\gamma} 
			\vort^{\kappa} (\partial_{\delta} u_{\kappa}),
			\notag
\end{align}

\begin{align} \label{E:TERM3ID}
	\upepsilon^{\alpha \beta \gamma \delta} (u^{\kappa} \partial_{\kappa} u_{\beta}) 
	(u^{\lambda} \partial_{\lambda} \vort_{\delta}) u_{\gamma}
	& =
	-
	\upepsilon^{\alpha \beta \gamma \delta} 
	(\partial_{\beta} \Lnenth) 
	u_{\gamma} \vort^{\kappa} (\partial_{\delta} u_{\kappa}) 
	\\
	& \ \
	+
	(\partial_{\kappa} u^{\kappa})
	\upepsilon^{\alpha \beta \gamma \delta} 
	(\partial_{\beta} \Lnenth) 
	u_{\gamma}
	\vort_{\delta} 
	\notag
		\\
	& \ \
	+
	(\Temp - \Temp_{;\Lnenth})
	((\upeta^{-1})^{\alpha \kappa} \partial_{\kappa} \Lnenth) 
	(\GradEnt^{\lambda} \partial_{\lambda} \Lnenth) 
	\notag
	\\
	& \ \
	+
  (\Temp - \Temp_{;\Lnenth})
	u^{\alpha}
	(u^{\kappa} \partial_{\kappa} \Lnenth)
	(\GradEnt^{\lambda} \partial_{\lambda} \Lnenth) 
	\notag \\
& \ \
	+	
	(\Temp_{;\Lnenth} - \Temp)
	\GradEnt^{\alpha} 
	(u^{\kappa} \partial_{\kappa} \Lnenth)
	(u^{\lambda} \partial_{\lambda} \Lnenth)
	\notag
	\\
	& \ \
	+
	(\Temp_{;\Lnenth} - \Temp)
	\GradEnt^{\alpha} 
	(\upeta^{-1})^{\kappa \lambda}
	(\partial_{\kappa} \Lnenth) 
	(\partial_{\lambda} \Lnenth)
	\notag
		\\
& \ \
	+
	\TempoverEnth
	\upepsilon^{\alpha \beta \gamma \delta} 
	\GradEnt_{\beta}
	u_{\gamma}
	\vort^{\kappa} (\partial_{\delta} u_{\kappa})
		\notag
			\\
	& \ \
	-
	\TempoverEnth
	(\partial_{\kappa} u^{\kappa})
	\upepsilon^{\alpha \beta \gamma \delta} 
	\GradEnt_{\beta}
	u_{\gamma}
	\vort_{\delta} 
		\notag \\
& \ \
	+
	\TempoverEnth
	(\Temp_{;\Lnenth} - \Temp) 
	((\upeta^{-1})^{\kappa \alpha} \partial_{\kappa} \Lnenth) \GradEnt^{\lambda} \GradEnt_{\lambda}
	\notag
	\\
	& \ \
	+
	\TempoverEnth
	(\Temp_{;\Lnenth} - \Temp) 
	u^{\alpha}
 (u^{\kappa} \partial_{\kappa} \Lnenth) \GradEnt^{\lambda} \GradEnt_{\lambda} 
	\notag \\
& \ \
	+
	\TempoverEnth
	(\Temp - \Temp_{;\Lnenth}) 
	\GradEnt^{\alpha}
	(\GradEnt^{\kappa} \partial_{\kappa} \Lnenth),
	\notag
\end{align}

\begin{align} \label{E:TERM4ID}
	\upepsilon^{\alpha \beta \gamma \delta} (u^{\kappa} \partial_{\kappa} u_{\beta}) 
	u_{\gamma}
	\vort^{\lambda} (\partial_{\delta} u_{\lambda}) 
	& = 
	-
	\upepsilon^{\alpha \beta \gamma \delta} (\partial_{\beta} \Lnenth) 
	u_{\gamma}
	\vort^{\lambda} (\partial_{\delta} u_{\lambda})
	\\
	& \ \
	+
	\TempoverEnth
	\upepsilon^{\alpha \beta \gamma \delta} \GradEnt_{\beta} 
	u_{\gamma}
	\vort^{\lambda} (\partial_{\delta} u_{\lambda}),
	\notag
\end{align}

\begin{align} \label{E:TERM5ID}
	- \upepsilon^{\alpha \beta \gamma \delta} 
	u_{\beta}
	(\partial_{\gamma} \vort^{\kappa}) (\partial_{\kappa} u_{\delta})
	& = 
		-
		\upepsilon^{\alpha \beta \gamma \delta} 
		u_{\beta}
		(\partial_{\gamma} \vort^{\kappa}) (\partial_{\delta} u_{\kappa})
		\\
		& \ \
		-
		\upepsilon^{\alpha \beta \gamma \delta} 
		u_{\beta}
		(\partial_{\gamma} \Lnenth)
		\vort^{\kappa}
		(\partial_{\delta} u_{\kappa})
		\notag
			\\
& \ \
			-
			\frac{1}{\Enth} 
			(\vort^{\kappa} \partial_{\kappa} \vort^{\alpha}) 
			+
			\frac{1}{\Enth} 
			\vort^{\alpha}
			(\partial_{\kappa} \vort^{\kappa})
			\notag \\
& \ \
		-
		\frac{1}{\Enth} 
		\vort^{\alpha}
		\vort^{\lambda}
		(u^{\kappa} \partial_{\kappa} u_{\lambda}) 
		+
		\frac{1}{\Enth} 
		u^{\alpha} 
		\vort^{\lambda}
		(\vort^{\kappa} \partial_{\kappa} u_{\lambda})
		\notag		\\
	& \ \
		-
		\TempoverEnth
		\upepsilon^{\alpha \beta \gamma \delta} 
		u_{\beta}
		(\partial_{\gamma} u^{\kappa})
		\vort_{\kappa}
		\GradEnt_{\delta},
		\notag
\end{align}

\begin{align} \label{E:TERM6ID}
	\upepsilon^{\alpha \beta \gamma \delta} 
	u_{\beta}
	(\partial_{\kappa} u^{\kappa}) (\partial_{\gamma} \vort_{\delta})
	& = 
		-
		(\partial_{\kappa} u^{\kappa})
		\uperpvort^{\alpha}(\vort),
\end{align}

\begin{align} \label{E:TERM7ID}
	-
	\upepsilon^{\alpha \beta \gamma \delta} 
	u_{\beta}
	(\partial_{\gamma} u_{\delta}) \vort^{\lambda}
	(u^{\kappa} \partial_{\kappa} u_{\lambda}) 
	& =
	\frac{1}{\Enth}
	\vort^{\alpha}
	\vort^{\lambda}
	(u^{\kappa} \partial_{\kappa} u_{\lambda}).
\end{align}
\end{subequations}

Moreover, we have
\begin{subequations}
\begin{align} \label{E:TOPORDERENTHTERMKEYID}
			&
			(\Temp - \Temp_{;\Lnenth}) 
			\GradEnt^{\alpha} 
			((\upeta^{-1})^{\kappa \lambda} \partial_{\kappa} \partial_{\lambda} \Lnenth) 
			+
			(\Temp - \Temp_{;\Lnenth}) 
			\GradEnt^{\alpha}  
			(u^{\kappa} u^{\lambda} \partial_{\kappa} \partial_{\lambda} \Lnenth) 
				\\
		& = 
		 u^{\kappa} \partial_{\kappa}
		\left\lbrace
			(\Temp_{;\Lnenth} - \Temp) 
			\GradEnt^{\alpha} 
			(\partial_{\lambda} u^{\lambda})
		\right\rbrace
			\notag \\
	& \ \
		+
		(\Temp_{\Lnenth} - \Temp_{;\Lnenth;\Lnenth})
		\GradEnt^{\alpha}
		(u^{\kappa} \partial_{\kappa} \Lnenth)
		(\partial_{\lambda} u^{\lambda})
		+
		(\Temp - \Temp_{;\Lnenth}) 
		(u^{\kappa} \partial_{\kappa} \GradEnt^{\alpha})
		(\partial_{\lambda} u^{\lambda})
		\notag \\
	& \ \
		+
		(\Temp_{;\Lnenth} - \Temp) 
		\GradEnt^{\alpha} 
		(\partial_{\kappa} u^{\lambda}) 
		(\partial_{\lambda} u^{\kappa})
		+
		(\Temp_{;\Lnenth} - \Temp) 
		\GradEnt^{\alpha} 
		(u^{\kappa} \partial_{\kappa} u^{\lambda}) 
		(\partial_{\lambda} \Lnenth)
		\notag
		\\
		& \ \
		+
		(\Temp_{;\Lnenth} - \Temp) 
		\GradEnt^{\alpha} 
		(\partial_{\kappa} u^{\kappa}) (u^{\lambda} \partial_{\lambda} \Lnenth)
		\notag
			\\
	& \ \
		+
		(\Temp - \Temp_{;\Lnenth})  
		\TempoverEnth 
		\GradEnt^{\alpha} 
		(\partial_{\kappa} \GradEnt^{\kappa})
		+
		(\Temp - \Temp_{;\Lnenth}) 
		\TempoverEnth_{;\Lnenth} 
		\GradEnt^{\alpha}
		(\GradEnt^{\kappa} \partial_{\kappa} \Lnenth)
		+
		(\Temp - \Temp_{;\Lnenth}) 
		\TempoverEnth_{;\Ent} 
		\GradEnt^{\alpha}	
		\GradEnt_{\kappa} \GradEnt^{\kappa},
		\notag
\end{align}

\begin{align} \label{E:TWODERIVATIVEENTHID3}
			(\Temp_{;\Lnenth} - \Temp) 
			 u^{\alpha}
			(\GradEnt^{\kappa} u^{\lambda} \partial_{\kappa} \partial_{\lambda} \Lnenth) 
			& = 
				u^{\kappa} \partial_{\kappa}
				\left\lbrace
					(\Temp_{;\Lnenth} - \Temp) u^{\alpha} (\GradEnt^{\lambda} \partial_{\lambda} \Lnenth)
				\right\rbrace
					\\
		& \ \
			+
			(\Temp_{\Lnenth} - \Temp_{;\Lnenth;\Lnenth})
			u^{\alpha}
			(u^{\kappa} \partial_{\kappa} \Lnenth)
			(\GradEnt^{\lambda} \partial_{\lambda} \Lnenth)
			\notag
			\\
			& \ \
			+
			(\Temp - \Temp_{;\Lnenth}) (u^{\kappa} \partial_{\kappa} u^{\alpha}) 
			(\GradEnt^{\lambda} \partial_{\lambda} \Lnenth)
				\notag \\
		& \ \
			+
			(\Temp - \Temp_{;\Lnenth})
			u^{\alpha} (u^{\kappa} \partial_{\kappa} \GradEnt^{\lambda}) (\partial_{\lambda} \Lnenth),
			\notag
\end{align}

\begin{align} \label{E:TWODERIVATIVEENTHID4}
			(\Temp_{;\Lnenth} - \Temp) 
			(\upeta^{-1})^{\alpha \lambda}
			(\GradEnt^{\kappa} \partial_{\kappa} \partial_{\lambda} \Lnenth) 
			& =
			u^{\kappa} \partial_{\kappa}
			\left\lbrace
				(\Temp - \Temp_{;\Lnenth})  (\upeta^{-1})^{\alpha \lambda} \GradEnt^{\beta} (\partial_{\lambda} u_{\beta})
			\right\rbrace
			\\
	& \ \
		+
		(\Temp_{;\Lnenth;\Lnenth} - \Temp_{\Lnenth}) 
		(u^{\kappa} \partial_{\kappa} \Lnenth)
		(\upeta^{-1})^{\alpha \lambda} \GradEnt^{\beta} (\partial_{\lambda} u_{\beta})
		\notag
		\\
		& \ \
		+
		(\Temp_{;\Lnenth} - \Temp)
		(u^{\kappa} \partial_{\kappa} \GradEnt^{\beta})
		(\upeta^{-1})^{\alpha \lambda} (\partial_{\lambda} u_{\beta})
		\notag \\
	& \ \
		+
		(\Temp - \Temp_{;\Lnenth}) 
		\GradEnt^{\beta} ((\upeta^{-1})^{\alpha \lambda} \partial_{\lambda} u_{\beta}) (u^{\kappa} \partial_{\kappa} \Lnenth)
		\notag
		\\
		& \ \
		+
		(\Temp - \Temp_{;\Lnenth})  \GradEnt^{\beta} ((\upeta^{-1})^{\alpha \lambda} \partial_{\lambda} u^{\kappa}) 
		(\partial_{\kappa} u_{\beta})
		\notag \\
	& \ \
	+ 
	(\Temp - \Temp_{;\Lnenth}) \TempoverEnth ((\upeta^{-1})^{\alpha \lambda} \partial_{\lambda} \GradEnt^{\beta}) \GradEnt_{\beta}
	\notag
	\\
	& \ \
	+
	(\Temp_{;\Lnenth} - \Temp) \TempoverEnth_{;\Lnenth} ((\upeta^{-1})^{\alpha \lambda} \partial_{\lambda} \Lnenth) 
	\GradEnt_{\kappa} \GradEnt^{\kappa}
		\notag \\
&  \ \
	+
	(\Temp_{;\Lnenth} - \Temp) \TempoverEnth_{;\Ent} \GradEnt^{\alpha} \GradEnt_{\kappa} \GradEnt^{\kappa}
	\notag
	\\
	& \ \
	+
	2 (\Temp_{;\Lnenth} - \Temp) \TempoverEnth ((\upeta^{-1})^{\alpha \lambda} \partial_{\lambda} \GradEnt^{\kappa}) \GradEnt_{\kappa}.
	\notag
\end{align}

\end{subequations}

\noindent \underline{\textbf{Identities that reveal null form structure and cancellations}}.
\, The following identities hold:\footnote{Our labeling of the terms $\mathscr{Q}_2$, $\mathscr{Q}_3$, etc.\ 
	is tied to the order in which terms appear in our proof of \eqref{E:TRANSPORTFORMODIFIEDVORTICITYOFVORTICITY}.}
\begin{subequations}
\begin{align} \label{E:Q2ID}
		\mathscr{Q}_2
		&
		:=
		(\Temp_{;\Lnenth} - \Temp) 
		((\upeta^{-1})^{\alpha \kappa} \partial_{\kappa} \Lnenth) \partial_{\lambda} \GradEnt^{\lambda}
		+
		(\Temp - \Temp_{;\Lnenth})  
		(\upeta^{-1})^{\kappa \lambda}
		(\partial_{\kappa} \Lnenth) 
		(\partial_{\lambda} \GradEnt^{\alpha})
			\\
		& = 
			(\Temp_{;\Lnenth} - \Temp) 
			(\upeta^{-1})^{\alpha \kappa}
			\left\lbrace
				(\partial_{\kappa} \Lnenth) (\partial_{\lambda} \GradEnt^{\lambda})
				-
				(\partial_{\lambda} \Lnenth) (\partial_{\kappa} \GradEnt^{\lambda})
			\right\rbrace,
			\notag
\end{align}

\begin{align} \label{E:Q4ID}
	\mathscr{Q}_4
	&
	:=
			(\Temp_{\Lnenth} - \Temp_{;\Lnenth;\Lnenth})
			\GradEnt^{\alpha}  
			(\upeta^{-1})^{\kappa \lambda} (\partial_{\kappa} \Lnenth)
			(\partial_{\lambda} \Lnenth)
			\\
			& \ \
			+
			(\Temp_{\Lnenth} - \Temp_{;\Lnenth;\Lnenth})
			 \GradEnt^{\alpha}
			(u^{\kappa} \partial_{\kappa} \Lnenth)
			(u^{\lambda} \partial_{\lambda} \Lnenth)
			+
			(\Temp_{\Lnenth} - \Temp_{;\Lnenth;\Lnenth})
			\GradEnt^{\alpha}
			(u^{\kappa} \partial_{\kappa} \Lnenth)
			(\partial_{\lambda} u^{\lambda})
			\notag
				\\
	& = 	\speed^{-2}
				(\Temp_{\Lnenth} - \Temp_{;\Lnenth;\Lnenth})
				\GradEnt^{\alpha}
				(g^{-1})^{\kappa \lambda} (\partial_{\kappa} \Lnenth) (\partial_{\lambda} \Lnenth),
	\notag
\end{align}

\begin{align} \label{E:Q5ID}
	\mathscr{Q}_5
	&
	:=
	(\Temp_{;\Lnenth;\Lnenth} - \Temp_{\Lnenth})
	((\upeta^{-1})^{\alpha \kappa}\partial_{\kappa} \Lnenth) 
	(\GradEnt^{\lambda} \partial_{\lambda} \Lnenth)
	\\
	& \ \
	+
	(\Temp_{;\Lnenth;\Lnenth} - \Temp_{\Lnenth}) 
	(u^{\kappa} \partial_{\kappa} \Lnenth)
	(\upeta^{-1})^{\alpha \lambda} \GradEnt^{\beta} (\partial_{\lambda} u_{\beta})
	\notag
		\\
	& =
	(\Temp_{;\Lnenth;\Lnenth} - \Temp_{\Lnenth}) 
	\GradEnt^{\beta}
	u^{\kappa} 
	(\upeta^{-1})^{\alpha \lambda}
	\left\lbrace
		(\partial_{\kappa} \Lnenth)
		(\partial_{\lambda} u_{\beta})
		-
		(\partial_{\lambda} \Lnenth)
		(\partial_{\kappa} u_{\beta})
	\right\rbrace
		\notag \\
	& \ \
	+
	(\Temp_{;\Lnenth;\Lnenth} - \Temp_{\Lnenth})
	\TempoverEnth
	((\upeta^{-1})^{\alpha \kappa}\partial_{\kappa} \Lnenth) 
	\GradEnt^{\lambda} \GradEnt_{\lambda},
	\notag
\end{align}

\begin{align} \label{E:Q6ID}
	\mathscr{Q}_6
	&
	:=
	\speed^{-2}
	(\Temp_{;\Lnenth} - \Temp) 
	\GradEnt^{\alpha} 
	(u^{\kappa} \partial_{\kappa} \Lnenth)
	(u^{\lambda} \partial_{\lambda} \Lnenth)
	\\
	& \ \
	+	 
	\speed^{-2}
	(\Temp_{;\Lnenth} - \Temp) 
	\GradEnt^{\alpha}
	(\upeta^{-1})^{\kappa \lambda}
	(\partial_{\kappa} \Lnenth)
	(\partial_{\lambda} \Lnenth)
	+
	(\Temp - \Temp_{;\Lnenth})  
	\GradEnt^{\alpha} 
	(\partial_{\kappa} u^{\kappa})
	(\partial_{\lambda} u^{\lambda})
	\notag
		\\
& = 
	\speed^{-4}
	(\Temp_{;\Lnenth} - \Temp) 
	\GradEnt^{\alpha} 
	(g^{-1})^{\kappa \lambda} (\partial_{\kappa} \Lnenth) (\partial_{\lambda} \Lnenth),
	\notag
\end{align}

\begin{align} \label{E:Q7ID}
	\mathscr{Q}_7
	&
	:=
	(\Temp_{;\Lnenth} - \Temp)
	\GradEnt^{\alpha} 
	(u^{\kappa} \partial_{\kappa} \Lnenth)
	(u^{\lambda} \partial_{\lambda} \Lnenth)
	+
	(\Temp_{;\Lnenth} - \Temp)
	\GradEnt^{\alpha} 
	(\upeta^{-1})^{\kappa \lambda}
	(\partial_{\kappa} \Lnenth) 
	(\partial_{\lambda} \Lnenth)
	\\
	& \ \
	+
	(\Temp_{;\Lnenth} - \Temp) 
	\GradEnt^{\alpha} 
	(u^{\kappa} \partial_{\kappa} u^{\lambda}) (\partial_{\lambda} \Lnenth)
	\notag
		\\
	& =
	(\Temp_{;\Lnenth} - \Temp) 
	\TempoverEnth 
	\GradEnt^{\alpha} 
	(\GradEnt^{\kappa} \partial_{\kappa} \Lnenth),
	\notag
\end{align}

\begin{align} \label{E:Q9ID}
	\mathscr{Q}_9
	&
	:=
	(\Temp_{;\Lnenth} - \Temp) 
			(\GradEnt^{\kappa} \partial_{\kappa} u^{\alpha})
			(\partial_{\lambda} u^{\lambda})
			+
			(\Temp_{;\Lnenth} - \Temp) 
			(u^{\kappa} \partial_{\kappa} u^{\alpha})
			(\GradEnt^{\lambda} \partial_{\lambda} \Lnenth)
			\\
			& \ \
			+
			(\Temp - \Temp_{;\Lnenth}) 
			((\upeta^{-1})^{\kappa \lambda} \partial_{\kappa} u^{\alpha})
			\GradEnt^{\beta} 
			(\partial_{\lambda} u_{\beta})
			\notag
				\\
			& =
			(\Temp_{;\Lnenth} - \Temp)
			\GradEnt^{\kappa}
			\left\lbrace
				(\partial_{\kappa} u^{\alpha})
				(\partial_{\lambda} u^{\lambda})
				-	 
				 (\partial_{\lambda} u^{\alpha})
				 (\partial_{\kappa} u^{\lambda})
			\right\rbrace
			\notag \\
	& \ \
	+
	\frac{1}{\Enth} 
	(\Temp - \Temp_{;\Lnenth}) 
	\upepsilon^{\kappa \beta \gamma \delta} 
	(\partial_{\kappa} u^{\alpha})
	\GradEnt_{\beta} u_{\gamma} \vort_{\delta}
	+ 
	\TempoverEnth
	(\Temp_{;\Lnenth} - \Temp) 
	(u^{\kappa} \partial_{\kappa} u^{\alpha})
	\GradEnt^{\lambda} \GradEnt_{\lambda},
	\notag
\end{align}

\begin{align} \label{E:Q11ID}
	\mathscr{Q}_{11}
	&
	:=
	(\Temp - \Temp_{;\Lnenth})  
	(u^{\kappa} \partial_{\kappa} \GradEnt^{\alpha})
	(\partial_{\lambda} u^{\lambda})
	+
	(\Temp - \Temp_{;\Lnenth}) 
	\GradEnt^{\beta} ((\upeta^{-1})^{\alpha \lambda} \partial_{\lambda} u^{\kappa}) (\partial_{\kappa} u_{\beta})
		\\
	& = 
	(\Temp_{;\Lnenth} - \Temp) 
	\GradEnt^{\beta}
	(\upeta^{-1})^{\alpha \kappa} 
	\left\lbrace
		(\partial_{\kappa} u_{\beta})
		(\partial_{\lambda} u^{\lambda})
		- 
		(\partial_{\lambda} u_{\beta}) 
		(\partial_{\kappa} u^{\lambda})
	\right\rbrace,
	\notag
\end{align}

\begin{align} \label{E:Q12ID}
		\mathscr{Q}_{12}
		& :=
		2
		(\Temp_{;\Lnenth} - \Temp) 
		(\partial_{\kappa} u^{\kappa})
		(\upeta^{-1})^{\alpha \lambda} 
		\GradEnt^{\beta} 
		(\partial_{\lambda} u_{\beta})
				\\
	& \ \
		+
		\speed^{-2}
		(\Temp - \Temp_{;\Lnenth}) 
		(\GradEnt^{\kappa} \partial_{\kappa} \Lnenth)
		((\upeta^{-1})^{\alpha \lambda} \partial_{\lambda} \Lnenth) 
		\notag
		\\
		& \ \
		+
		(\Temp_{;\Lnenth} - \Temp)
		(u^{\kappa} \partial_{\kappa} \GradEnt^{\beta})
		((\upeta^{-1})^{\alpha \lambda} \partial_{\lambda} u_{\beta})
		\notag
			\\
	& =
		(\Temp_{;\Lnenth} - \Temp)
		(\upeta^{-1})^{\alpha \kappa} 
		\GradEnt^{\beta}	
			\left\lbrace
				(\partial_{\kappa} u_{\beta})
				(\partial_{\lambda} u^{\lambda})
				-
				(\partial_{\lambda} u_{\beta})
				(\partial_{\kappa} u^{\lambda})
			\right\rbrace
		\notag \\
	& \ \
		+ 
		\speed^{-2} 
		(\Temp_{;\Lnenth} - \Temp)
		(\upeta^{-1})^{\alpha \kappa} 
		\GradEnt^{\beta} u^{\lambda}
		\left\lbrace
			(\partial_{\lambda} u_{\beta}) (\partial_{\kappa} \Lnenth)
			-
			(\partial_{\kappa} u_{\beta}) (\partial_{\lambda} \Lnenth)
		\right\rbrace
		\notag \\
	& \ \
		+
		\TempoverEnth 
		\speed^{-2} 
		(\Temp - \Temp_{;\Lnenth}) \GradEnt_{\kappa} \GradEnt^{\kappa} ((\upeta^{-1})^{\alpha \lambda} \partial_{\lambda} \Lnenth),
	  \notag 
\end{align}

\begin{align} \label{E:Q13ID}
	\mathscr{Q}_{13}
	&
	:=
	(\Temp - \Temp_{;\Lnenth}) 
	u^{\alpha} 
	(\partial_{\kappa} u^{\kappa})
	(\GradEnt^{\lambda} \partial_{\lambda} \Lnenth)
	+
	(\Temp - \Temp_{;\Lnenth}) 
	u^{\alpha} 
	(u^{\kappa} \partial_{\kappa} u^{\lambda}) 
	\GradEnt^{\beta} 
	(\partial_{\lambda} u_{\beta})
		\\
	& =
	(\Temp_{;\Lnenth} - \Temp) 
	u^{\alpha} 
	\GradEnt^{\beta}
	u^{\lambda} 
	\left\lbrace
		(\partial_{\kappa} u^{\kappa})
		(\partial_{\lambda} u_{\beta}) 
		-
	  (\partial_{\lambda} u^{\kappa}) 
		(\partial_{\kappa} u_{\beta})
	\right\rbrace
		\notag \\
& \ \
	+
	\TempoverEnth 
	(\Temp - \Temp_{;\Lnenth}) 
	u^{\alpha} 
	\GradEnt_{\kappa} \GradEnt^{\kappa}
	(\partial_{\lambda} u^{\lambda}),
	\notag
\end{align}

\begin{align} \label{E:Q14ID}
		\mathscr{Q}_{14}
		&	
		:=
			(\Temp - \Temp_{;\Lnenth}) 
			u^{\alpha} 
			(\partial_{\kappa} u^{\kappa})
			(\GradEnt^{\lambda} \partial_{\lambda} \Lnenth)
			+
			(\Temp - \Temp_{;\Lnenth}) 
			u^{\alpha} 
		  (u^{\kappa} \partial_{\kappa} \Lnenth)
			(\GradEnt^{\lambda} \partial_{\lambda} \Lnenth)
			\\
			& \ \
			+
			(\Temp - \Temp_{;\Lnenth}) 
			u^{\alpha}
			(\upeta^{-1})^{\kappa \lambda}
			(\partial_{\kappa} \Lnenth) u_{\beta} (\partial_{\lambda} \GradEnt^{\beta})
			\notag
				\\
		& = 
				n (\Temp - \Temp_{;\Lnenth}) u^{\alpha} (u^{\kappa} \partial_{\kappa} \Lnenth) \mathcal{D}
				\notag \\
		& \ \ 
		+
		(\Temp_{;\Lnenth} - \Temp) 
		u^{\alpha}
		u^{\kappa} 
		\left\lbrace
			(\partial_{\kappa} \Lnenth) (\partial_{\lambda} \GradEnt^{\lambda})
			-
			(\partial_{\lambda} \Lnenth) (\partial_{\kappa} \GradEnt^{\lambda})
		\right\rbrace,
		\notag
\end{align}

\begin{align} \label{E:Q15ID}
		\mathscr{Q}_{15}
		&
		:=
		(\partial_{\kappa} u^{\kappa})
		\upepsilon^{\alpha \beta \gamma \delta} 
		(\partial_{\beta} \Lnenth) 
		u_{\gamma}
		\vort_{\delta}  
		+
		\speed^{-2}
		\upepsilon^{\alpha \beta \gamma \delta} 
		(u^{\kappa} \partial_{\kappa} u_{\beta})
		(\partial_{\gamma} \Lnenth) \vort_{\delta}
			\\
		& = 
		\speed^{-2}
		\TempoverEnth
		\upepsilon^{\alpha \beta \gamma \delta} 
		\GradEnt_{\beta}
		(\partial_{\gamma} \Lnenth) \vort_{\delta},
		\notag
\end{align}

\begin{align} \label{E:Q16ID}
		  \mathscr{Q}_{16}
			& :=
			-
			\speed^{-2}
			u^{\alpha}
			\upepsilon^{\sigma \beta \gamma \delta} 
			(u^{\kappa} \partial_{\kappa} u_{\sigma})
			u_{\beta}
			(\partial_{\gamma} \Lnenth) \vort_{\delta}
				\\
			& =
			-
			\speed^{-2}
			\TempoverEnth
			u^{\alpha}
			\upepsilon^{\kappa \beta \gamma \delta} 
			\GradEnt_{\kappa}
			u_{\beta}
			(\partial_{\gamma} \Lnenth) \vort_{\delta},
			\notag
\end{align}

\begin{align} \label{E:Q18ID}
	\mathscr{Q}_{18}
	&
	:=
	(\Temp - \Temp_{;\Lnenth})
	((\upeta^{-1})^{\alpha \kappa} \partial_{\kappa} \Lnenth) 
	(\GradEnt^{\lambda} \partial_{\lambda} \Lnenth) 
	\\
	& \ \
	+
	(\Temp_{;\Lnenth} - \Temp) 
	((\upeta^{-1})^{\alpha \kappa} \partial_{\kappa} \Lnenth) u_{\beta} (u^{\lambda} \partial_{\lambda} \GradEnt^{\beta})
	\notag
		\\
	& =
	\TempoverEnth 
	(\Temp - \Temp_{;\Lnenth})
	\GradEnt_{\kappa} \GradEnt^{\kappa}
	((\upeta^{-1})^{\alpha \lambda} \partial_{\lambda} \Lnenth),
	\notag
\end{align}

\begin{align} \label{E:Q19ID}
			\mathscr{Q}_{19}
			& 
			:=
			(\Temp_{;\Lnenth} - \Temp) 
			u^{\alpha} 
			(u^{\kappa} \partial_{\kappa} u_{\sigma})
			\GradEnt^{\sigma} 
			(\partial_{\lambda} u^{\lambda})
			\\
			& \ \
			+
			\speed^{-2}
			(\Temp - \Temp_{;\Lnenth}) 
			u^{\alpha}
			(\GradEnt^{\kappa} \partial_{\kappa} \Lnenth)
			(u^{\lambda} \partial_{\lambda} \Lnenth)
			\notag
				\\
			& =
			\TempoverEnth 
			(\Temp_{;\Lnenth} - \Temp) 
			u^{\alpha} 
			\GradEnt_{\kappa} \GradEnt^{\kappa}
			(\partial_{\lambda} u^{\lambda}),
			\notag
\end{align}

\begin{align} \label{E:Q20ID}
		\mathscr{Q}_{20}
		& :=
		(\Temp - \Temp_{;\Lnenth}) \GradEnt^{\beta} ((\upeta^{-1})^{\alpha \lambda} \partial_{\lambda} u_{\beta}) (u^{\kappa} \partial_{\kappa} \Lnenth)
		\\
		& \ \
		+
		(\Temp - \Temp_{;\Lnenth}) 
		(u^{\kappa} \partial_{\kappa} \Lnenth) (u^{\lambda} \partial_{\lambda} \GradEnt^{\alpha})
		\notag
		\\
		& 	= 0,
		\notag
\end{align}

\begin{align}  \label{E:Q21ID}
	\mathscr{Q}_{21}
	&:=
	(\Temp_{;\Lnenth} - \Temp) 
		u^{\alpha} 
		(u^{\kappa} \partial_{\kappa} u_{\beta})
		u^{\beta} (\GradEnt^{\lambda} \partial_{\lambda} \Lnenth)
		\\
		&	= 0.
		\notag
\end{align}
\end{subequations}

\end{lemma}

\begin{proof}
We split the proof into many pieces.

\noindent $\bullet$ \textbf{Proof of \eqref{E:TERM1ID}}:	
We first use equation \eqref{E:ENTHALPYEVOLUTION} to deduce
\begin{align} \label{E:TERM1IDFIRSTSTEP}
	\upepsilon^{\alpha \beta \gamma \delta}
	u_{\beta}
	(\partial_{\gamma} \partial_{\kappa} u^{\kappa}) \vort_{\delta}
	& = 		
	-
	\upepsilon^{\alpha \beta \gamma \delta} 
	u_{\beta}
	\left\lbrace
		\partial_{\gamma} (\speed^{-2} u^{\kappa} \partial_{\kappa} \Lnenth)
	\right\rbrace
	\vort_{\delta}
		\\
& =
	-
	\upepsilon^{\alpha \beta \gamma \delta} 
	u_{\beta}
	(\speed^{-2} u^{\kappa} \partial_{\kappa} \partial_{\gamma} \Lnenth) \vort_{\delta}
	\notag	\\
& \ \
	+
	2 \speed^{-3} \speed_{;\Lnenth} 
	(u^{\kappa} \partial_{\kappa} \Lnenth)
	\upepsilon^{\alpha \beta \gamma \delta} 
	u_{\beta}
	(\partial_{\gamma} \Lnenth)
	\vort_{\delta}
	\notag
	\\
	& \ \
	+
	2 \speed^{-3} \speed_{;\Ent} 
	(u^{\kappa} \partial_{\kappa} \Lnenth)
	\upepsilon^{\alpha \beta \gamma \delta} 
	u_{\beta}
	\GradEnt_{\gamma}
	\vort_{\delta}
		\notag \\
& \ \
	-
	\speed^{-2}
	\upepsilon^{\alpha \beta \gamma \delta} 
	u_{\beta}
	(\partial_{\gamma} u^{\kappa}) 
	(\partial_{\kappa} \Lnenth) 
	\vort_{\delta}.
		\notag
\end{align}
Next, we rewrite the first term on RHS~\eqref{E:TERM1IDFIRSTSTEP}
as a perfect $u^{\kappa} \partial_{\kappa}$ derivative
plus error terms, thereby obtaining, with the help of 
\eqref{E:ENTROPYEVOLUTION}, 
the following identity:
\begin{align} \label{E:TERM1IDSECONDSTEP}
\upepsilon^{\alpha \beta \gamma \delta}
	u_{\beta}
	(\partial_{\gamma} \partial_{\kappa} u^{\kappa}) \vort_{\delta}
	& = 	 
		-
		u^{\kappa} \partial_{\kappa}
		\left\lbrace
		\upepsilon^{\alpha \beta \gamma \delta} 
		\speed^{-2} 
		u_{\beta}
		(\partial_{\gamma} \Lnenth) \vort_{\delta}
	\right\rbrace
	\\
	& \ \
	+
	\speed^{-2} 
	\upepsilon^{\alpha \beta \gamma \delta} 
	(u^{\kappa} \partial_{\kappa} u_{\beta})
	(\partial_{\gamma} \Lnenth) 
	\vort_{\delta}
	\notag
		\\
& \ \
	+
	\speed^{-2}
	\upepsilon^{\alpha \beta \gamma \delta} 
	u_{\beta}
	(\partial_{\gamma} \Lnenth)
	(u^{\kappa} \partial_{\kappa} \vort_{\delta})
	\notag
	\\
	& \ \
	-
	\speed^{-2}
	\upepsilon^{\alpha \beta \gamma \delta} 
	u_{\beta}
	(\partial_{\gamma} u^{\kappa}) 
	(\partial_{\kappa} \Lnenth) 
	\vort_{\delta}
		\notag \\
	& \ \
	+
	2 \speed^{-3} \speed_{;\Ent} 
	(u^{\kappa} \partial_{\kappa} \Lnenth)
	\upepsilon^{\alpha \beta \gamma \delta} 
	u_{\beta}
	\GradEnt_{\gamma}
	\vort_{\delta}.
	\notag
\end{align}
Using equation \eqref{E:VORTICITYTRANSPORT}
to substitute for the factor $u^{\kappa} \partial_{\kappa} \vort_{\delta}$
in the third product on RHS~\eqref{E:TERM1IDSECONDSTEP},
we deduce
\begin{align} \label{E:TERM1IDSTHIRDSTEP}
	\upepsilon^{\alpha \beta \gamma \delta}
	u_{\beta}
	(\partial_{\gamma} \partial_{\kappa} u^{\kappa}) \vort_{\delta}
	& = 
		-
		u^{\kappa} \partial_{\kappa}
		\left\lbrace
			\speed^{-2}
			\upepsilon^{\alpha \beta \gamma \delta} 
			u_{\beta}
			(\partial_{\gamma} \Lnenth) 
			\vort_{\delta}
	\right\rbrace
	\\
	& \ \
	+
	\speed^{-2} 
	\upepsilon^{\alpha \beta \gamma \delta} 
	(u^{\kappa} \partial_{\kappa} u_{\beta})
	(\partial_{\gamma} \Lnenth) 
	\vort_{\delta}	
	\notag
		\\
& \ \
	+
	\speed^{-2}
	\upepsilon^{\alpha \beta \gamma \delta} 
	u_{\beta}
	(\partial_{\gamma} \Lnenth)
	(\vort^{\kappa} \partial_{\kappa} u_{\delta})
	\notag
	\\
	& \ \
	-
	\speed^{-2}
	(\partial_{\kappa} u^{\kappa}) 
	\upepsilon^{\alpha \beta \gamma \delta} 
	u_{\beta}
	(\partial_{\gamma} \Lnenth)
	\vort_{\delta} 
	\notag \\
	&  
	\ \
	+
	\speed^{-2}
	(\Temp - \Temp_{;\Lnenth}) 
	\upepsilon^{\alpha \beta \gamma \delta} 
	\upepsilon_{\delta \nu}^{\ \ \kappa \lambda}
	u_{\beta} u^{\nu}
	(\partial_{\gamma} \Lnenth)
	(\partial_{\kappa} \Lnenth) \GradEnt_{\lambda}
	\notag
	\\
	& \ \
	-
	\speed^{-2}
	\upepsilon^{\alpha \beta \gamma \delta} 
	u_{\beta}
	(\partial_{\gamma} u^{\kappa}) 
	(\partial_{\kappa} \Lnenth) 
	\vort_{\delta}
		\notag
		\\
& \ \
	+
	2 \speed^{-3} \speed_{;\Ent} 
	(u^{\kappa} \partial_{\kappa} \Lnenth)
	\upepsilon^{\alpha \beta \gamma \delta} 
	u_{\beta}
	\GradEnt_{\gamma}
	\vort_{\delta}.
	\notag
\end{align}
Next, 
using the identity
\eqref{E:VORTICITYCONTRACTEDPARITALALPHAUBETAMINUSPARTIALBETAUALPHA},
we express the third product on RHS~\eqref{E:TERM1IDSTHIRDSTEP} as follows:
\begin{align} \label{E:TERM1IDFOURTHSTEP}
\speed^{-2}
\upepsilon^{\alpha \beta \gamma \delta} 
u_{\beta}
(\partial_{\gamma} \Lnenth)
(\vort^{\kappa} \partial_{\kappa} u_{\delta})
& =
	\speed^{-2}
	\upepsilon^{\alpha \beta \gamma \delta} 
	u_{\beta}
	(\partial_{\gamma} \Lnenth)
	\vort^{\kappa} (\partial_{\delta} u_{\kappa}).
\end{align}	
Next, using the identity
\begin{align}
	- \upepsilon^{\alpha \beta \gamma \delta} 
		\upepsilon_{\delta \nu}^{\ \ \kappa \lambda}
	& =
	(\upeta^{-1})^{\lambda \beta} \updelta_{\nu}^{\alpha} (\upeta^{-1})^{\kappa \gamma}
	-
	(\upeta^{-1})^{\lambda \beta} \updelta_{\nu}^{\gamma} (\upeta^{-1})^{\kappa \alpha}
	\notag
		\\
& \ \
	+
	(\upeta^{-1})^{\lambda \gamma} \updelta_{\nu}^{\beta} (\upeta^{-1})^{\kappa \alpha}
	-
	(\upeta^{-1})^{\lambda \gamma} \updelta_{\nu}^{\alpha} (\upeta^{-1})^{\kappa \beta}
	\notag
	\\
	& \ \
	+
	(\upeta^{-1})^{\lambda \alpha} \updelta_{\nu}^{\gamma} (\upeta^{-1})^{\kappa \beta}
	-
	(\upeta^{-1})^{\lambda \alpha} \updelta_{\nu}^{\beta} (\upeta^{-1})^{\kappa \gamma}
	\notag
\end{align}
and equations 
\eqref{E:UISUNITLENGTH}, 
\eqref{E:VELOCITYANDENTGRADIENTAREMINKOWSKIPERP},
and
\eqref{E:VORTISORTHGONALTOU},
we express the third-from-last product on RHS~\eqref{E:TERM1IDSTHIRDSTEP} as follows:
\begin{align} \label{E:TERM1IDFIFTHSTEP}
	&
	\speed^{-2}
	(\Temp - \Temp_{;\Lnenth}) 
	\upepsilon^{\alpha \beta \gamma \delta} 
	\upepsilon_{\delta \nu}^{\ \ \kappa \lambda}
	u_{\beta} u^{\nu}
	(\partial_{\gamma} \Lnenth)
	(\partial_{\kappa} \Lnenth) \GradEnt_{\lambda}
	\\
	 & = 
	\speed^{-2}
	(\Temp - \Temp_{;\Lnenth}) 
	(\GradEnt^{\kappa} \partial_{\kappa} \Lnenth)
	((\upeta^{-1})^{\alpha \lambda} \partial_{\lambda} \Lnenth) 
	\notag
		\\
 & \ \
	+
	\speed^{-2}
	(\Temp - \Temp_{;\Lnenth}) 
	u^{\alpha}
	(\GradEnt^{\kappa} \partial_{\kappa} \Lnenth)
	(u^{\lambda} \partial_{\lambda} \Lnenth)
	\notag \\
 & \ \
		+ \speed^{-2}
	(\Temp_{;\Lnenth} - \Temp) 
	\GradEnt^{\alpha} 
	(u^{\kappa} \partial_{\kappa} \Lnenth)
	(u^{\lambda} \partial_{\lambda} \Lnenth)
	\notag \\
 \ \
& \ \
	+	 
	\speed^{-2}
	(\Temp_{;\Lnenth} - \Temp) 
	\GradEnt^{\alpha}
	(\upeta^{-1})^{\kappa \lambda}
	(\partial_{\kappa} \Lnenth)
	(\partial_{\lambda} \Lnenth).
	\notag
\end{align}

Using \eqref{E:TERM1IDFOURTHSTEP} and \eqref{E:TERM1IDFIFTHSTEP}
to substitute for the relevant products on \linebreak RHS~\eqref{E:TERM1IDSTHIRDSTEP},
adding and subtracting
$
\speed^{-2}
\upepsilon^{\alpha \beta \gamma \delta} 
u_{\beta}
\vort_{\delta}
(\partial_{\kappa} u^{\kappa}) (\partial_{\gamma} \Lnenth) 
$,
and reorganizing the terms,
we arrive at the desired identity \eqref{E:TERM1ID}.

\medskip

\noindent $\bullet$ \textbf{Proof of \eqref{E:TERM2ID}}:
We first use \eqref{E:GRADIENTOFUCONTRACTEDAGAINSTVOLUMEFORM}
to deduce
\begin{align} \label{E:TERM2IDFIRSTSTEP}
	&
	\upepsilon^{\alpha \beta \gamma \delta} 
	u_{\beta}
	(\vort^{\kappa} \partial_{\kappa} \partial_{\gamma} u_{\delta})
		\\
	& = 
	u_{\beta}
	\vort^{\kappa} \partial_{\kappa} 
	\left\lbrace
		\frac{1}{\Enth}
			\vort^{\alpha} u^{\beta} 
			- 
			\frac{1}{\Enth}
			u^{\alpha} \vort^{\beta}
			-
			\upepsilon^{\alpha \beta \gamma \delta} (\partial_{\gamma} \Lnenth) u_{\delta}
			+
			\TempoverEnth
			\upepsilon^{\alpha \beta \gamma \delta}
			\GradEnt_{\gamma} u_{\delta}
	\right\rbrace.
	\notag
\end{align}
The desired identity \eqref{E:TERM2ID}
now follows from 
\eqref{E:LOGARITHMICENTHALPY},
\eqref{E:UISUNITLENGTH},
\eqref{E:VORTISORTHGONALTOU},
\eqref{E:DERIVATIVEOFVELOCITYCONTRACTEDWITHVELOCITYISZERO},
\eqref{E:TRANSFERDERIVATIVESFROMVORTICITYTOVELOCITY},
\eqref{E:VORTICITYCONTRACTEDPARITALALPHAUBETAMINUSPARTIALBETAUALPHA},
\eqref{E:TERM2IDFIRSTSTEP},
and straightforward calculations.

\medskip

\noindent $\bullet$ \textbf{Proof of \eqref{E:TERM3ID}}:
We first use equation 
\eqref{E:VELOCITYEVOLUTION} to 
substitute for the factor
$
u^{\kappa} \partial_{\kappa} u_{\beta}
$
on LHS~\eqref{E:TERM3ID}, thereby obtaining the identity
\begin{align} \label{E:TERM3IDFIRSTSTEP}
	\upepsilon^{\alpha \beta \gamma \delta} (u^{\kappa} \partial_{\kappa} u_{\beta}) 
	(u^{\lambda} \partial_{\lambda} \vort_{\delta}) u_{\gamma}
	& =
	-
	\upepsilon^{\alpha \beta \gamma \delta} 
	(\partial_{\beta} \Lnenth) 
	(u^{\lambda} \partial_{\lambda} \vort_{\delta}) u_{\gamma}
	\\
	& \ \
	+
	\TempoverEnth
	\upepsilon^{\alpha \beta \gamma \delta} 
	\GradEnt_{\beta}
	(u^{\lambda} \partial_{\lambda} \vort_{\delta}) u_{\gamma}.
	\notag
\end{align}
We then use equation \eqref{E:VORTICITYTRANSPORT}
to substitute for the two factors of
$u^{\lambda} \partial_{\lambda} \vort_{\delta}$ 
on RHS~\eqref{E:TERM3IDFIRSTSTEP},
which yields the identity
\begin{align} \label{E:TERM3IDSECONDSTEP}
	\upepsilon^{\alpha \beta \gamma \delta} (u^{\kappa} \partial_{\kappa} u_{\beta}) 
	(u^{\lambda} \partial_{\lambda} \vort_{\delta}) u_{\gamma}
	& =
	-
	\upepsilon^{\alpha \beta \gamma \delta} 
	(\partial_{\beta} \Lnenth) 
	u_{\gamma}
	(\vort^{\kappa} \partial_{\kappa} u_{\delta}) 
	\\
	& \ \ 
	+
	(\partial_{\kappa} u^{\kappa})
	\upepsilon^{\alpha \beta \gamma \delta} 
	(\partial_{\beta} \Lnenth) 
	u_{\gamma} \vort_{\delta}
	\notag
		\\
	& \ \
	+
	(\Temp_{;\Lnenth} - \Temp) 
	\upepsilon^{\alpha \beta \gamma \delta} 
	\upepsilon_{\delta \nu}^{\ \ \kappa \lambda} 
	(\partial_{\beta} \Lnenth) 
	u^{\nu}
	(\partial_{\kappa} \Lnenth) \GradEnt_{\lambda}
	u_{\gamma}
	\notag
	\\
	& \ \
	+
	\TempoverEnth
	\upepsilon^{\alpha \beta \gamma \delta} 
	\GradEnt_{\beta}
	u_{\gamma}
	(\vort^{\kappa} \partial_{\kappa} u_{\delta}) 
		\notag
			\\
	& \ \
	-
	\TempoverEnth
	(\partial_{\kappa} u^{\kappa})
	\upepsilon^{\alpha \beta \gamma \delta} 
	\GradEnt_{\beta}
	u_{\gamma} \vort_{\delta} 
	\notag
	\\
	&  \ \
	+
	\TempoverEnth
	(\Temp - \Temp_{;\Lnenth}) 
	\upepsilon^{\alpha \beta \gamma \delta} 
	\upepsilon_{\delta \nu}^{\ \ \kappa \lambda} u^{\nu}
	(\partial_{\kappa} \Lnenth) \GradEnt_{\lambda}
	\GradEnt_{\beta} u_{\gamma}.
	\notag
\end{align}
Next, we use the identity
\eqref{E:VORTICITYCONTRACTEDPARITALALPHAUBETAMINUSPARTIALBETAUALPHA}
to express the first and fourth products on RHS~\eqref{E:TERM3IDSECONDSTEP} as follows:
\begin{align} \label{E:IDFORFIRSTPRODUCTONRHSTERM3IDSECONDSTEP}
	-
	\upepsilon^{\alpha \beta \gamma \delta} 
	(\partial_{\beta} \Lnenth) 
	(\vort^{\lambda} \partial_{\lambda} u_{\delta}) u_{\gamma}
& =
-
	\upepsilon^{\alpha \beta \gamma \delta} 
	(\partial_{\beta} \Lnenth) 
	\vort^{\lambda} (\partial_{\delta} u_{\lambda}) u_{\gamma},
		\\
	\TempoverEnth
	\upepsilon^{\alpha \beta \gamma \delta} 
	\GradEnt_{\beta}
	u_{\gamma}
	(\vort^{\kappa} \partial_{\kappa} u_{\delta}) 
	& = 
	\TempoverEnth
	\upepsilon^{\alpha \beta \gamma \delta} 
	\GradEnt_{\beta}
	u_{\gamma} \vort^{\kappa} (\partial_{\delta} u_{\kappa}) .
	\label{E:IDFORFOURTHPRODUCTONRHSTERM3IDSECONDSTEP}
\end{align}
We then use the identity
\begin{align} 
	\upepsilon^{\alpha \beta \gamma \delta} 
	\upepsilon_{\delta \nu}^{\ \ \kappa \lambda}
	& =
		(\upeta^{-1})^{\lambda \beta} \updelta_{\nu}^{\gamma} (\upeta^{-1})^{\kappa \alpha}
		-
		(\upeta^{-1})^{\lambda \beta} \updelta_{\nu}^{\alpha} (\upeta^{-1})^{\kappa \gamma}
		\notag
			\\
& \ \
	+
	(\upeta^{-1})^{\lambda \gamma} \updelta_{\nu}^{\alpha} (\upeta^{-1})^{\kappa \beta}
	-
	(\upeta^{-1})^{\lambda \gamma} \updelta_{\nu}^{\beta} (\upeta^{-1})^{\kappa \alpha}
	\notag
	\\
	& \ \
	+
	(\upeta^{-1})^{\lambda \alpha} \updelta_{\nu}^{\beta} (\upeta^{-1})^{\kappa \gamma}
	-
	(\upeta^{-1})^{\lambda \alpha} \updelta_{\nu}^{\gamma} (\upeta^{-1})^{\kappa \beta}
	\notag
\end{align}
and equations 
\eqref{E:UISUNITLENGTH}
and
\eqref{E:VELOCITYANDENTGRADIENTAREMINKOWSKIPERP}
to express the third product on RHS~\eqref{E:TERM3IDSECONDSTEP} as follows:
\begin{align} \label{E:IDFORTHIRDPRODUCTONRHSTERM3IDSECONDSTEP}
	&
	(\Temp_{;\Lnenth} - \Temp) 
	\upepsilon^{\alpha \beta \gamma \delta} 
	\upepsilon_{\delta \nu}^{\ \ \kappa \lambda} 
	(\partial_{\beta} \Lnenth) 
	u^{\nu}
	(\partial_{\kappa} \Lnenth) \GradEnt_{\lambda}
	u_{\gamma}
		\\
	& =  
	(\Temp - \Temp_{;\Lnenth})
	((\upeta^{-1})^{\alpha \kappa} \partial_{\kappa} \Lnenth) 
	(\GradEnt^{\lambda} \partial_{\lambda} \Lnenth) 
	\notag
		\\
& \ \
	+
  (\Temp - \Temp_{;\Lnenth})
	u^{\alpha}
	(u^{\kappa} \partial_{\kappa} \Lnenth)
	(\GradEnt^{\lambda} \partial_{\lambda} \Lnenth) 
	\notag \\
& \ \
	+	
	(\Temp_{;\Lnenth} - \Temp)
	\GradEnt^{\alpha} 
	(u^{\kappa} \partial_{\kappa} \Lnenth)
	(u^{\lambda} \partial_{\lambda} \Lnenth)
		\notag \\
&  \ \
	+
	(\Temp_{;\Lnenth} - \Temp)
	\GradEnt^{\alpha} 
	(\upeta^{-1})^{\kappa \lambda}
	(\partial_{\kappa} \Lnenth) 
	(\partial_{\lambda} \Lnenth).
	\notag
\end{align}
Similarly, we express the last product on RHS~\eqref{E:TERM3IDSECONDSTEP} as follows:
\begin{align} \label{E:IDFORLASTPRODUCTONRHSTERM3IDSECONDSTEP}
	\TempoverEnth
	(\Temp - \Temp_{;\Lnenth}) 
	\upepsilon^{\alpha \beta \gamma \delta} 
	\upepsilon_{\delta \nu}^{\ \ \kappa \lambda} u^{\nu}
	(\partial_{\kappa} \Lnenth) \GradEnt_{\lambda}
	\GradEnt_{\beta} u_{\gamma}
	& =  
	\TempoverEnth
	(\Temp_{;\Lnenth} - \Temp) 
	((\upeta^{-1})^{\kappa \alpha} \partial_{\kappa} \Lnenth) \GradEnt^{\lambda} \GradEnt_{\lambda}
	\\
	& \ \
	+
	\TempoverEnth
	(\Temp_{;\Lnenth} - \Temp) 
	u^{\alpha}
	(u^{\kappa} \partial_{\kappa} \Lnenth) \GradEnt^{\lambda} \GradEnt_{\lambda} 
	\notag
		\\
& \ \
	+
\TempoverEnth
(\Temp - \Temp_{;\Lnenth}) 
\GradEnt^{\alpha}
(\GradEnt^{\kappa} \partial_{\kappa} \Lnenth).
	\notag
\end{align}

Using \eqref{E:IDFORFIRSTPRODUCTONRHSTERM3IDSECONDSTEP}-\eqref{E:IDFORFOURTHPRODUCTONRHSTERM3IDSECONDSTEP}
and
\eqref{E:IDFORTHIRDPRODUCTONRHSTERM3IDSECONDSTEP}-\eqref{E:IDFORLASTPRODUCTONRHSTERM3IDSECONDSTEP}
to substitute for
the relevant products on RHS~\eqref{E:TERM3IDSECONDSTEP},
we arrive at the desired identity \eqref{E:TERM3ID}.

\medskip

\noindent $\bullet$ \textbf{Proof of \eqref{E:TERM4ID}}:
\eqref{E:TERM4ID} follows easily from using equation \eqref{E:VELOCITYEVOLUTION} 
to substitute for the factor
$u^{\kappa} \partial_{\kappa} u_{\beta}$ 
on the LHS.

\medskip

\noindent $\bullet$ \textbf{Proof of \eqref{E:TERM5ID}}:
We first use the identity \eqref{E:PARITALALPHAUBETAMINUSPARTIALBETAUALPHA}
to deduce
\begin{align} \label{E:TERM5IDFIRSTSTEP}
	\upepsilon^{\alpha \beta \gamma \delta} 
	u_{\beta}
	(\partial_{\gamma} \vort^{\kappa}) (\partial_{\kappa} u_{\delta})
	& = 
		\upepsilon^{\alpha \beta \gamma \delta} 
		u_{\beta}
		(\partial_{\gamma} \vort^{\kappa}) (\partial_{\delta} u_{\kappa})
		\\
		& \ \
		+
		\frac{1}{\Enth}
		\upepsilon^{\alpha \beta \gamma \delta} 
		\upepsilon_{\kappa \delta \theta \lambda} u^{\theta} \vort^{\lambda}
		u_{\beta}
		(\partial_{\gamma} \vort^{\kappa}) 
		\notag
			\\
	& \ \
		+
		\upepsilon^{\alpha \beta \gamma \delta} 
		u_{\beta}
		(\partial_{\gamma} \vort^{\kappa})  u_{\kappa} (\partial_{\delta} \Lnenth)
		-
		\TempoverEnth
		\upepsilon^{\alpha \beta \gamma \delta} 
		u_{\beta}
		(\partial_{\gamma} \vort^{\kappa})
		u_{\kappa}
		\GradEnt_{\delta}.
			\notag
\end{align}
Next, we note the identity
\begin{align}
\upepsilon^{\alpha \beta \gamma \delta} 
\upepsilon_{\kappa \delta \theta \lambda}
	=
	-
	\upepsilon^{\alpha \beta \gamma \delta} 
	\upepsilon_{\kappa \lambda \theta \delta} 
	& =
	\updelta_{\kappa}^{\alpha} \updelta_{\lambda}^{\beta} \updelta_{\theta}^{\gamma}
	-
	\updelta_{\kappa}^{\beta} \updelta_{\lambda}^{\alpha} \updelta_{\theta}^{\gamma}
	\notag
		\\
& \ \
	+
	\updelta_{\kappa}^{\beta} \updelta_{\lambda}^{\gamma} \updelta_{\theta}^{\alpha}
	-
	\updelta_{\kappa}^{\alpha} \updelta_{\lambda}^{\gamma} \updelta_{\theta}^{\beta}
	+
	\updelta_{\kappa}^{\gamma} \updelta_{\lambda}^{\alpha} \updelta_{\theta}^{\beta}
	-
	\updelta_{\kappa}^{\gamma} \updelta_{\lambda}^{\beta} \updelta_{\theta}^{\alpha},
\notag
\end{align}
which, in view of \eqref{E:UISUNITLENGTH},
\eqref{E:VORTISORTHGONALTOU},
and \eqref{E:TRANSFERDERIVATIVESFROMVORTICITYTOVELOCITY},
allows us to express the second product on RHS~\eqref{E:TERM5IDFIRSTSTEP} as follows:
\begin{align} \label{E:TERM5IDTRICKYTERMID}
		\frac{1}{\Enth}
		\upepsilon^{\alpha \beta \gamma \delta} 
		\upepsilon_{\kappa \delta \theta \lambda} u^{\theta} \vort^{\lambda}
		u_{\beta}
		(\partial_{\gamma} \vort^{\kappa}) 
	& = 
			-
			\frac{1}{\Enth} 
			\vort^{\alpha}
			u_{\lambda}
			(u^{\kappa} \partial_{\kappa} \vort^{\lambda}) 
			\\
			& \ \
			+
			\frac{1}{\Enth} 
			u^{\alpha} 
			u_{\lambda}
			(\vort^{\kappa} \partial_{\kappa} \vort^{\lambda})
			\notag
				\\
	& \ \
			+
			\frac{1}{\Enth} 
			(\vort^{\kappa} \partial_{\kappa} \vort^{\alpha}) 
			-
			\frac{1}{\Enth} 
			\vort^{\alpha}
			(\partial_{\kappa} \vort^{\kappa})
			\notag
				\\
	& = 
			\frac{1}{\Enth} 
			\vort^{\alpha}
			\vort^{\lambda}
			(u^{\kappa} \partial_{\kappa} u_{\lambda}) 
			-
			\frac{1}{\Enth} 
			u^{\alpha} 
			\vort^{\lambda}
			(\vort^{\kappa} \partial_{\kappa} u_{\lambda})
				\notag \\
	& \ \
			+
			\frac{1}{\Enth} 
			(\vort^{\kappa} \partial_{\kappa} \vort^{\alpha}) 
			-
			\frac{1}{\Enth} 
			\vort^{\alpha}
			(\partial_{\kappa} \vort^{\kappa}).
			\notag
\end{align}
Using \eqref{E:TERM5IDTRICKYTERMID}
to substitute for the second product on RHS~\eqref{E:TERM5IDFIRSTSTEP},
and using \eqref{E:TRANSFERDERIVATIVESFROMVORTICITYTOVELOCITY}
to express the third product on RHS~\eqref{E:TERM5IDFIRSTSTEP} as
$
\upepsilon^{\alpha \beta \gamma \delta} 
u_{\beta}
(\partial_{\gamma} \vort^{\kappa}) 
u_{\kappa}
(\partial_{\delta} \Lnenth) 
=
-
\upepsilon^{\alpha \beta \gamma \delta} 
u_{\beta}
\vort^{\kappa}
(\partial_{\gamma} u_{\kappa}) (\partial_{\delta} \Lnenth)
=
\upepsilon^{\alpha \beta \gamma \delta} 
u_{\beta}
(\partial_{\gamma} \Lnenth)
\vort^{\kappa}
(\partial_{\delta} u_{\kappa}) 
$
and the last product on RHS~\eqref{E:TERM5IDFIRSTSTEP}
as
$
-
		\TempoverEnth
		\upepsilon^{\alpha \beta \gamma \delta} 
		u_{\beta}
		(\partial_{\gamma} \vort^{\kappa})
		 u_{\kappa}
		\GradEnt_{\delta}
=	\TempoverEnth
		\upepsilon^{\alpha \beta \gamma \delta} 
		u_{\beta}
		(\partial_{\gamma} u^{\kappa})
		 \vort_{\kappa}
		\GradEnt_{\delta}		
$,
we arrive at the desired identity
\eqref{E:TERM5ID}.

\medskip

\noindent $\bullet$ \textbf{Proof of \eqref{E:TERM6ID}}:
\eqref{E:TERM6ID} follows from definition \eqref{E:UORTHGONALVORTICITYOFONEFORM}
with $V_{\delta} := \vort_{\delta}$.

\medskip

\noindent $\bullet$ \textbf{Proof of \eqref{E:TERM7ID}}:
\eqref{E:TERM7ID} is a straightforward consequence of
\eqref{E:GRADIENTOFUCONTRACTEDAGAINSTVOLUMEFORM},
\eqref{E:UISUNITLENGTH}, and \eqref{E:VORTISORTHGONALTOU}.

\medskip

\noindent $\bullet$ \textbf{Proof of \eqref{E:TOPORDERENTHTERMKEYID}}:
We first use \eqref{E:FLATWAVEOPERATOROFLOGENTHID} to express LHS~\eqref{E:TOPORDERENTHTERMKEYID} as follows:
\begin{align} \label{E:FIRSTSTEPTOPORDERENTHTERMKEYID}
		&
		(\Temp - \Temp_{;\Lnenth}) 
		\GradEnt^{\alpha}
		((\upeta^{-1})^{\kappa \lambda} \partial_{\kappa} \partial_{\lambda} \Lnenth)
		+
		(\Temp - \Temp_{;\Lnenth}) 
		\GradEnt^{\alpha}
		(u^{\kappa} u^{\lambda} \partial_{\kappa} \partial_{\lambda} \Lnenth)
			\\
		& = 
		(\Temp_{;\Lnenth} - \Temp) 
		\GradEnt^{\alpha} 
		(u^{\lambda} \partial_{\lambda} \partial_{\kappa} u^{\kappa})
		+
		(\Temp_{;\Lnenth} - \Temp) 
		\GradEnt^{\alpha}
		(\partial_{\kappa} u^{\lambda}) (\partial_{\lambda} u^{\kappa}) 
			\notag \\
	& \ \
		+
		(\Temp_{;\Lnenth} - \Temp) 
		\GradEnt^{\alpha}
		(u^{\kappa} \partial_{\kappa} u^{\lambda}) (\partial_{\lambda} \Lnenth)
		+
		(\Temp_{;\Lnenth} - \Temp) 
		\GradEnt^{\alpha}
		(\partial_{\kappa} u^{\kappa}) (u^{\lambda} \partial_{\lambda} \Lnenth)
		\notag
			\\
	& \ \
		+
		(\Temp - \Temp_{;\Lnenth}) 
		\TempoverEnth
		\GradEnt^{\alpha}
		(\partial_{\kappa} \GradEnt^{\kappa})
		+
		(\Temp - \Temp_{;\Lnenth}) 
		\TempoverEnth_{;\Lnenth} 
		\GradEnt^{\alpha}
		(\GradEnt^{\kappa} \partial_{\kappa} \Lnenth)
		+
		(\Temp - \Temp_{;\Lnenth}) 
		\TempoverEnth_{;\Ent} 
		\GradEnt^{\alpha}
		\GradEnt_{\kappa} \GradEnt^{\kappa}.
		\notag
	\end{align}
	Next, with the help of equation \eqref{E:ENTROPYEVOLUTION}, we rewrite the first
	product on \linebreak RHS~\eqref{E:FIRSTSTEPTOPORDERENTHTERMKEYID} as follows:
	\begin{align} \label{E:SECONDSTEPTOPORDERENTHTERMKEYID}
		(\Temp_{;\Lnenth} - \Temp) 
		\GradEnt^{\alpha} 
		(u^{\lambda} \partial_{\lambda} \partial_{\kappa} u^{\kappa})
		&
		= 
		u^{\kappa} \partial_{\kappa} 
		\left\lbrace
			(\Temp_{;\Lnenth} - \Temp) 
			\GradEnt^{\alpha} 
			(\partial_{\lambda} u^{\lambda})
		\right\rbrace
			\\
		& \ \
			+ 
			(\Temp_{\Lnenth} - \Temp_{;\Lnenth;\Lnenth})
				\GradEnt^{\alpha}  
				(u^{\kappa} \partial_{\kappa} \Lnenth)
				(\partial_{\lambda} u^{\lambda})
				\notag
				\\
				& \ \
				+
				(\Temp - \Temp_{;\Lnenth}) 
				(u^{\kappa} \partial_{\kappa} \GradEnt^{\alpha})
				(\partial_{\lambda} u^{\lambda}).
				\notag
	\end{align}
	Using \eqref{E:SECONDSTEPTOPORDERENTHTERMKEYID} to substitute for the first
	product on RHS~\eqref{E:FIRSTSTEPTOPORDERENTHTERMKEYID}, we arrive at the
	desired identity \eqref{E:TOPORDERENTHTERMKEYID}.

\medskip

\noindent $\bullet$ \textbf{Proof of \eqref{E:TWODERIVATIVEENTHID3}}:
\eqref{E:TWODERIVATIVEENTHID3} is a straightforward consequence of equation \eqref{E:ENTROPYEVOLUTION}.

\medskip

\noindent $\bullet$ \textbf{Proof of \eqref{E:TWODERIVATIVEENTHID4}}:
We first differentiate equation \eqref{E:IDVELOCITYDERIVATIVEOFVELOCITYCONTRACTEDWITHENTROPYGRADIENT}
with $(\upeta^{-1})^{\alpha \lambda} \partial_{\lambda}$ 
and then multiply the resulting identity by $(\Temp_{;\Lnenth} - \Temp)$
to obtain
\begin{align} \label{E:FIRSTSTEPTWODERIVATIVEENTHID4}
	(\Temp_{;\Lnenth} - \Temp)
	(\upeta^{-1})^{\alpha \lambda}
	(\GradEnt^{\kappa} \partial_{\kappa} \partial_{\lambda} \Lnenth)
	& = 
		(\Temp - \Temp_{;\Lnenth})
		(\upeta^{-1})^{\alpha \lambda}
		\GradEnt^{\beta} 
		(u^{\kappa} \partial_{\kappa} \partial_{\lambda} u_{\beta})
		\\
		& \ \
		+
		(\Temp - \Temp_{;\Lnenth})
		((\upeta^{-1})^{\alpha \lambda} \partial_{\lambda} \GradEnt^{\kappa}) (\partial_{\kappa} \Lnenth)
		\notag
			\\
	& \ \
		+
		(\Temp - \Temp_{;\Lnenth})
		((\upeta^{-1})^{\alpha \lambda} \partial_{\lambda} \GradEnt^{\beta}) (u^{\kappa} \partial_{\kappa} u_{\beta})
		\notag
		\\
		& \ \
		+
		(\Temp - \Temp_{;\Lnenth})
		\GradEnt^{\beta} ((\upeta^{-1})^{\alpha \lambda} \partial_{\lambda} u^{\kappa}) (\partial_{\kappa} u_{\beta})
		\notag \\
& \ \
	+
	(\Temp_{;\Lnenth} - \Temp)
	\TempoverEnth_{;\Lnenth} 
	((\upeta^{-1})^{\alpha \lambda} \partial_{\lambda} \Lnenth) \GradEnt_{\kappa} \GradEnt^{\kappa}
	\notag
	\\
	& \ \
	+
	(\Temp_{;\Lnenth} - \Temp)
	\TempoverEnth_{;\Ent}
	\GradEnt^{\alpha} \GradEnt_{\kappa} \GradEnt^{\kappa}
		\notag \\
& \ \
	+
	2 
	(\Temp_{;\Lnenth} - \Temp)
	\TempoverEnth 
	((\upeta^{-1})^{\alpha \lambda} \partial_{\lambda} \GradEnt^{\kappa}) \GradEnt_{\kappa}.
	\notag
\end{align}
Next, with the help of equation \eqref{E:ENTROPYEVOLUTION}, we rewrite the
first product on \linebreak RHS~\eqref{E:FIRSTSTEPTWODERIVATIVEENTHID4} as follows:
\begin{align} \label{E:SECONDSTEPTWODERIVATIVEENTHID4}
		&
		(\Temp - \Temp_{;\Lnenth})
		(\upeta^{-1})^{\alpha \lambda}
		\GradEnt^{\beta} 
		(u^{\kappa} \partial_{\kappa} \partial_{\lambda} u_{\beta})
			\\
		& 
		=
		 u^{\kappa} \partial_{\kappa} 
		\left\lbrace
			(\Temp - \Temp_{;\Lnenth})
			(\upeta^{-1})^{\alpha \lambda}
			\GradEnt^{\beta} (\partial_{\lambda} u_{\beta})
		\right\rbrace
			\notag \\
		& \ \
			+ 
			(\Temp_{;\Lnenth;\Lnenth} - \Temp_{\Lnenth}) 
				(u^{\kappa} \partial_{\kappa} \Lnenth)
				(\upeta^{-1})^{\alpha \lambda}
				\GradEnt^{\beta} (\partial_{\lambda} u_{\beta})
					\notag \\
		& \ \
				+ 
				(\Temp_{;\Lnenth} - \Temp) 
				(u^{\kappa} \partial_{\kappa} \GradEnt^{\beta})
				(\upeta^{-1})^{\alpha \lambda}
				(\partial_{\lambda} u_{\beta}).
				\notag
\end{align}
Next, we use equations
\eqref{E:VELOCITYEVOLUTION} 
and \eqref{E:TRANSFERDERIVATIVESFROMENTROPYGRADIENTTOVORTICITY}
to express the sum of the second and third products on RHS~\eqref{E:FIRSTSTEPTWODERIVATIVEENTHID4}
as follows:
\begin{align}  \label{E:THIRDSTEPTWODERIVATIVEENTHID4}
	&
	(\Temp - \Temp_{;\Lnenth})
	((\upeta^{-1})^{\alpha \lambda} \partial_{\lambda} \GradEnt^{\kappa}) (\partial_{\kappa} \Lnenth)
	+
	(\Temp - \Temp_{;\Lnenth})
	((\upeta^{-1})^{\alpha \lambda} \partial_{\lambda} \GradEnt^{\beta}) (u^{\kappa} \partial_{\kappa} u_{\beta})	
		\\
	& =
		(\Temp_{;\Lnenth} - \Temp)
		((\upeta^{-1})^{\alpha \lambda} \partial_{\lambda} \GradEnt^{\beta}) u_{\beta} (u^{\kappa} \partial_{\kappa} \Lnenth)
		+
		(\Temp - \Temp_{;\Lnenth})
		\TempoverEnth
		((\upeta^{-1})^{\alpha \lambda} \partial_{\lambda} \GradEnt^{\beta}) \GradEnt_{\beta}
		\notag
		\\
	& =
		(\Temp - \Temp_{;\Lnenth})
		\GradEnt^{\beta}
		((\upeta^{-1})^{\alpha \lambda} \partial_{\lambda} u_{\beta}) (u^{\kappa} \partial_{\kappa} \Lnenth)
		+
		(\Temp - \Temp_{;\Lnenth})
		\TempoverEnth
		((\upeta^{-1})^{\alpha \lambda} \partial_{\lambda} \GradEnt^{\beta}) \GradEnt_{\beta}.
		\notag
\end{align}
Using \eqref{E:SECONDSTEPTWODERIVATIVEENTHID4}
to substitute for the first product on RHS~\eqref{E:FIRSTSTEPTWODERIVATIVEENTHID4},
and using \eqref{E:THIRDSTEPTWODERIVATIVEENTHID4}
to substitute for the second and third products on RHS~\eqref{E:FIRSTSTEPTWODERIVATIVEENTHID4},
we arrive at the desired identity \eqref{E:TWODERIVATIVEENTHID4}.

\medskip

\noindent $\bullet$ \textbf{Proof of \eqref{E:Q2ID}}:
We simply use \eqref{E:ENTSYMMETRYOFMIXEDPARTIALS} to express 
the second product on LHS~\eqref{E:Q2ID}
as follows: 
\begin{align}
(\Temp - \Temp_{;\Lnenth}) 
(\upeta^{-1})^{\kappa \lambda}
(\partial_{\kappa} \Lnenth) (\partial_{\lambda}\GradEnt^{\alpha})
= 
(\Temp - \Temp_{;\Lnenth})
(\upeta^{-1})^{\alpha \kappa}
(\partial_{\lambda} \Lnenth)
(\partial_{\kappa} \GradEnt^{\lambda}).
\notag
\end{align}

\medskip

\noindent $\bullet$ \textbf{Proof of \eqref{E:Q4ID}}:
We use equation \eqref{E:ENTHALPYEVOLUTION}
to substitute for the last factor
$\partial_{\lambda} u^{\lambda}$ on LHS~\eqref{E:Q4ID}
and then appeal to equation \eqref{E:INVERSEACOUSTICALMETRIC}.

\medskip

\noindent $\bullet$ \textbf{Proof of \eqref{E:Q5ID}}:
We first use 
\eqref{E:IDVELOCITYDERIVATIVEOFVELOCITYCONTRACTEDWITHENTROPYGRADIENT}
to express the first product on LHS~\eqref{E:Q5ID}
as follows:
\begin{align} \label{E:Q5IDFIRSTSTEP}
	&
	(\Temp_{;\Lnenth;\Lnenth} - \Temp_{\Lnenth})
	((\upeta^{-1})^{\alpha \kappa}\partial_{\kappa} \Lnenth) 
	(\GradEnt^{\lambda} \partial_{\lambda} \Lnenth)
	\\
	&	= 
	(\Temp_{\Lnenth} - \Temp_{;\Lnenth;\Lnenth})
	((\upeta^{-1})^{\alpha \kappa}\partial_{\kappa} \Lnenth) 
	(u^{\lambda} \partial_{\lambda} u_{\beta}) \GradEnt^{\beta}
	\notag
		\\
& \ \
	+
	(\Temp_{;\Lnenth;\Lnenth} - \Temp_{\Lnenth})
	\TempoverEnth
	((\upeta^{-1})^{\alpha \kappa}\partial_{\kappa} \Lnenth) 
	\GradEnt^{\lambda} \GradEnt_{\lambda}.
	\notag
\end{align}
Using \eqref{E:Q5IDFIRSTSTEP}
to substitute for the first product on LHS~\eqref{E:Q5ID},
we arrive at the desired identity.

\medskip

\noindent $\bullet$ \textbf{Proof of \eqref{E:Q6ID}}:
To prove \eqref{E:Q6ID},
we first use equation \eqref{E:ENTHALPYEVOLUTION}
to express the last product on
LHS~\eqref{E:Q6ID} as follows:
\begin{align} \label{E:Q6IDFIRSTSTEP}
	(\Temp - \Temp_{;\Lnenth}) 
	\GradEnt^{\alpha} 
	(\partial_{\kappa} u^{\kappa})
	(\partial_{\lambda} u^{\lambda})
& =
	\speed^{-4}
	(\Temp - \Temp_{;\Lnenth}) 
	\GradEnt^{\alpha} 
	(u^{\kappa} \partial_{\kappa} \Lnenth)
	(u^{\lambda} \partial_{\lambda} \Lnenth).
\end{align}
Using \eqref{E:Q6IDFIRSTSTEP} to substitute for 
the last product on LHS~\eqref{E:Q6ID}
and appealing to equation \eqref{E:INVERSEACOUSTICALMETRIC},
we arrive at the desired identity.

\medskip

\noindent $\bullet$ \textbf{Proof of \eqref{E:Q7ID}}:
We first use \eqref{E:VELOCITYEVOLUTION} to substitute for the 
factor $u^{\kappa} \partial_{\kappa} u^{\lambda}$ in the last product on LHS~\eqref{E:Q7ID},
thereby obtaining the following identity:
\begin{align} \label{E:FIRSTSTEPQ7ID}
(\Temp_{;\Lnenth} - \Temp) 
\GradEnt^{\alpha} 
(u^{\kappa} \partial_{\kappa} u^{\lambda}) (\partial_{\lambda} \Lnenth)
& =
	(\Temp - \Temp_{;\Lnenth}) 
	\GradEnt^{\alpha} 
	(\upeta^{-1})^{\kappa \lambda}
	(\partial_{\kappa} \Lnenth)
	(\partial_{\lambda} \Lnenth)
	\\
	& \ \
	+
	(\Temp - \Temp_{;\Lnenth}) 
	\GradEnt^{\alpha}
	(u^{\kappa} \partial_{\kappa} \Lnenth)
	(u^{\lambda} \partial_{\lambda} \Lnenth)
	\notag
		\\
& \ \
	+
	(\Temp_{;\Lnenth} - \Temp) 
	\TempoverEnth 
	\GradEnt^{\alpha} 
	(\GradEnt^{\kappa} \partial_{\kappa} \Lnenth).
	\notag
\end{align}
Using \eqref{E:FIRSTSTEPQ7ID} to substitute for the last product on 
LHS~\eqref{E:Q7ID}, we arrive at the desired identity.

\medskip

\noindent $\bullet$ \textbf{Proof of \eqref{E:Q9ID}}:
We first use 
\eqref{E:ENTSYMMETRYOFMIXEDPARTIALS},
\eqref{E:TRANSFERDERIVATIVESFROMENTROPYGRADIENTTOVORTICITY},
and the first equality in \eqref{E:ANOTHERENTROPYGRADIENTEVOLUTION}
to express the last product on LHS~\eqref{E:Q9ID} as follows:
\begin{align} \label{E:FIRSTSTEPQ9ID}
	(\Temp - \Temp_{;\Lnenth}) 
	((\upeta^{-1})^{\kappa \lambda} \partial_{\kappa} u^{\alpha})
	\GradEnt^{\beta} 
	(\partial_{\lambda} u_{\beta})
	& 
	=
	(\Temp_{;\Lnenth} - \Temp) 
	(\partial_{\kappa} u^{\alpha})
	(u^{\beta} \partial_{\beta} \GradEnt^{\kappa})
		\\
	& = 
	(\Temp - \Temp_{;\Lnenth}) 
	(\partial_{\kappa} u^{\alpha})
	(\GradEnt^{\lambda} \partial_{\lambda} u^{\kappa})
	\notag
	\\
	& \ \
	+
	\frac{1}{\Enth} 
	(\Temp - \Temp_{;\Lnenth}) 
	\upepsilon^{\kappa \beta \gamma \delta} 
	(\partial_{\kappa} u^{\alpha})
	\GradEnt_{\beta} u_{\gamma} \vort_{\delta}
		\notag \\
 & \ \
	+
	(\Temp - \Temp_{;\Lnenth}) 
	(u^{\kappa} \partial_{\kappa} u^{\alpha})
	(\GradEnt^{\lambda} \partial_{\lambda} \Lnenth)
	\notag
	\\
	& \ \
	+ 
	\TempoverEnth
	(\Temp_{;\Lnenth} - \Temp) 
	(u^{\kappa} \partial_{\kappa} u^{\alpha})
	\GradEnt^{\lambda} \GradEnt_{\lambda}.
	\notag
\end{align}
Using \eqref{E:FIRSTSTEPQ9ID} to substitute for the last product on LHS~\eqref{E:Q9ID},
we arrive at the desired identity.

\medskip

\noindent $\bullet$ \textbf{Proof of \eqref{E:Q11ID}}:
We use
\eqref{E:ENTSYMMETRYOFMIXEDPARTIALS}
and
\eqref{E:TRANSFERDERIVATIVESFROMENTROPYGRADIENTTOVORTICITY}
to express the first product on LHS~\eqref{E:Q11ID} as follows:
\begin{align} \label{E:FIRSTSTEPQ11ID}
	(\Temp - \Temp_{;\Lnenth}) 
	(u^{\kappa} \partial_{\kappa} \GradEnt^{\alpha})
	(\partial_{\lambda} u^{\lambda})
	& =
	(\Temp - \Temp_{;\Lnenth}) 
	(u^{\kappa} (\upeta^{-1})^{\alpha \beta} \partial_{\beta} \GradEnt_{\kappa})
	(\partial_{\lambda} u^{\lambda})
		\\
	& =
	(\Temp_{;\Lnenth} - \Temp) 
	\GradEnt^{\beta}
	((\upeta^{-1})^{\alpha \kappa} \partial_{\kappa} u_{\beta})
	(\partial_{\lambda} u^{\lambda}).
	\notag
\end{align}
Using \eqref{E:FIRSTSTEPQ11ID} to substitute for 
the first product on LHS~\eqref{E:Q11ID},
we conclude the desired identity.

\medskip

\noindent $\bullet$ \textbf{Proof of \eqref{E:Q12ID}}:
To prove \eqref{E:Q12ID},
we first note the following identity,
which we derive below:
\begin{align} \label{E:FIRSTSTEPQ12ID}
		&
		\speed^{-2}
		(\Temp - \Temp_{;\Lnenth}) 
		(\GradEnt^{\kappa} \partial_{\kappa} \Lnenth)
		((\upeta^{-1})^{\alpha \lambda} \partial_{\lambda} \Lnenth)
		\\
		& \ \
		+
		(\Temp_{;\Lnenth} - \Temp)
		(u^{\kappa} \partial_{\kappa} \GradEnt^{\beta})
		((\upeta^{-1})^{\alpha \lambda} \partial_{\lambda} u_{\beta})
		\notag
			\\
	& = 
		(\Temp - \Temp_{;\Lnenth})
		\GradEnt^{\kappa}
		(\partial_{\beta} u_{\kappa})
		((\upeta^{-1})^{\alpha \lambda} \partial_{\lambda} u^{\beta})
			\notag \\
	& \ \
		+ 
		\speed^{-2} (\Temp_{;\Lnenth} - \Temp)  
		(u^{\lambda} \partial_{\lambda} u_{\beta}) \GradEnt^{\beta} ((\upeta^{-1})^{\alpha \kappa} \partial_{\kappa}  \Lnenth)
		\notag
		\\
		& \ \
		+ 
		\speed^{-2} (\Temp - \Temp_{;\Lnenth}) 
		((\upeta^{-1})^{\alpha \kappa} \partial_{\kappa} u_{\beta}) \GradEnt^{\beta} (u^{\lambda} \partial_{\lambda} \Lnenth)
		\notag \\
	&  \ \
		+
		(\Temp - \Temp_{;\Lnenth}) 
		((\upeta^{-1})^{\alpha \kappa} \partial_{\kappa} u_{\beta}) \GradEnt^{\beta} (\partial_{\lambda} u^{\lambda})
		+
		\TempoverEnth \speed^{-2} (\Temp - \Temp_{;\Lnenth}) \GradEnt^{\beta} \GradEnt_{\beta} 
		((\upeta^{-1})^{\alpha \kappa} \partial_{\kappa} \Lnenth).
		\notag
\end{align}
Using \eqref{E:FIRSTSTEPQ12ID}
to substitute for the sum of the second and third products on LHS~\eqref{E:Q12ID},
we conclude the desired identity \eqref{E:Q12ID}.

It remains for us to prove \eqref{E:FIRSTSTEPQ12ID}.
To proceed, we first use 
\eqref{E:ENTSYMMETRYOFMIXEDPARTIALS}
and
\eqref{E:TRANSFERDERIVATIVESFROMENTROPYGRADIENTTOVORTICITY}
to express the second product on LHS~\eqref{E:FIRSTSTEPQ12ID} as follows:
\begin{align} \label{E:SECONDSTEPQ12ID}
&
(\Temp_{;\Lnenth} - \Temp) (u^{\kappa} \partial_{\kappa} \GradEnt^{\beta}) 
((\upeta^{-1})^{\alpha \lambda} \partial_{\lambda} u_{\beta})
	\\
	& 
	= 
	(\Temp_{;\Lnenth} - \Temp) (u^{\kappa} \partial_{\beta} \GradEnt_{\kappa}) 
	((\upeta^{-1})^{\alpha \lambda} \partial_{\lambda} u^{\beta})
	\notag
		\\
& = 
	(\Temp - \Temp_{;\Lnenth}) (\GradEnt^{\beta} \partial_{\lambda} u_{\beta}) 
	((\upeta^{-1})^{\alpha \kappa} \partial_{\kappa} u^{\lambda})
		\notag \\
	& 
	= (\Temp - \Temp_{;\Lnenth}) ((\upeta^{-1})^{\alpha \kappa} \partial_{\kappa} u_{\beta}) 
	\GradEnt^{\beta} (\partial_{\lambda} u^{\lambda})
	 \notag \\
	&  \ \
		+ 
		(\Temp_{;\Lnenth} - \Temp) ((\upeta^{-1})^{\alpha \kappa} \partial_{\kappa} u_{\beta}) \GradEnt^{\beta} 
		(\partial_{\lambda} u^{\lambda})
			\notag \\
	& \ \
		+
		(\Temp - \Temp_{;\Lnenth}) (\GradEnt^{\beta} \partial_{\lambda} u_{\beta}) 
		((\upeta^{-1})^{\alpha \kappa} \partial_{\kappa} u^{\lambda}),
		\notag
\end{align}
where to obtain the last equality, we have added and subtracted
\[
(\Temp_{;\Lnenth} - \Temp) ((\upeta^{-1})^{\alpha \kappa} \partial_{\kappa} u_{\beta}) 
\GradEnt^{\beta} (\partial_{\lambda} u^{\lambda}).
\]
Next, we use equation \eqref{E:ENTHALPYEVOLUTION}
to substitute for the factor $\partial_{\lambda} u^{\lambda}$ in the first product on RHS~\eqref{E:SECONDSTEPQ12ID},
which allows us to express the product as follows:
\begin{align} \label{E:THIRDSTEPQ12ID}
	&
	(\Temp - \Temp_{;\Lnenth}) ((\upeta^{-1})^{\alpha \kappa} \partial_{\kappa} u_{\beta}) \GradEnt^{\beta} 
	(\partial_{\lambda} u^{\lambda}) 
	\\
	& 
	= 
	\speed^{-2} (\Temp_{;\Lnenth} - \Temp) ((\upeta^{-1})^{\alpha \kappa} \partial_{\kappa} u_{\beta}) \GradEnt^{\beta} 
	(u^{\lambda} \partial_{\lambda} \Lnenth)
	\notag
		\\
	& 
	= 
	\speed^{-2} (\Temp_{;\Lnenth} - \Temp)
		(u^{\lambda} \partial_{\lambda} u_{\beta}) \GradEnt^{\beta} ((\upeta^{-1})^{\alpha \kappa} \partial_{\kappa} \Lnenth)  
	 \notag	\\
	& \ \
		+ 
		\speed^{-2} (\Temp - \Temp_{;\Lnenth})
		(u^{\lambda} \partial_{\lambda} u_{\beta}) \GradEnt^{\beta} ((\upeta^{-1})^{\alpha \kappa} \partial_{\kappa} \Lnenth) 
			\notag \\
	& \ \
		+ 
		\speed^{-2} 
		(\Temp_{;\Lnenth} - \Temp) 
		((\upeta^{-1})^{\alpha \kappa} \partial_{\kappa} u_{\beta}) \GradEnt^{\beta} (u^{\lambda} \partial_{\lambda} \Lnenth),
		\notag 
\end{align}
where to obtain the last equality, we have added and subtracted
\[
\speed^{-2} (\Temp - \Temp_{;\Lnenth}) 
(u^{\lambda} \partial_{\lambda} u_{\beta}) \GradEnt^{\beta} ((\upeta^{-1})^{\alpha \kappa} \partial_{\kappa} \Lnenth).
\]
Next, we use equation \eqref{E:IDVELOCITYDERIVATIVEOFVELOCITYCONTRACTEDWITHENTROPYGRADIENT}
to express the first product on RHS~\eqref{E:THIRDSTEPQ12ID} as follows:
\begin{align}	 \label{E:SIXTHSTEPQ12ID}
	&
	\speed^{-2} (\Temp_{;\Lnenth} - \Temp)
	(u^{\lambda} \partial_{\lambda} u_{\beta}) \GradEnt^{\beta} ((\upeta^{-1})^{\alpha \kappa} \partial_{\kappa} \Lnenth)
	\\
	& = 
	\speed^{-2} (\Temp - \Temp_{;\Lnenth})
	(\GradEnt^{\kappa} \partial_{\kappa} \Lnenth) ((\upeta^{-1})^{\alpha \lambda} \partial_{\lambda} \Lnenth)
	\notag
		\\
& \ \
	+
	\TempoverEnth
	\speed^{-2} (\Temp_{;\Lnenth} - \Temp)
	\GradEnt_{\kappa} \GradEnt^{\kappa} ((\upeta^{-1})^{\alpha \lambda} \partial_{\lambda} \Lnenth).
	\notag
\end{align}
Combining
\eqref{E:SECONDSTEPQ12ID}-\eqref{E:SIXTHSTEPQ12ID}, we find that
\begin{align} \label{E:SEVENTHSTEPQ12ID}
		&
		\speed^{-2}
		(\Temp - \Temp_{;\Lnenth}) 
		(\GradEnt^{\kappa} \partial_{\kappa} \Lnenth)
		((\upeta^{-1})^{\alpha \lambda} \partial_{\lambda} \Lnenth)
			\\
	& = 
		(\Temp_{;\Lnenth} - \Temp)
		(u^{\kappa} \partial_{\kappa} \GradEnt^{\beta})
		((\upeta^{-1})^{\alpha \lambda} \partial_{\lambda} u_{\beta})
			\notag 
				\\
	& \ \
		+
		(\Temp - \Temp_{;\Lnenth}) ((\upeta^{-1})^{\alpha \kappa} \partial_{\kappa} u_{\beta}) \GradEnt^{\beta} 
		(\partial_{\lambda} u^{\lambda}) 
		+
		(\Temp_{;\Lnenth} - \Temp) (\GradEnt^{\beta} \partial_{\lambda} u_{\beta}) ((\upeta^{-1})^{\alpha \kappa} \partial_{\kappa} u^{\lambda})
		\notag \\
	& \ \
		+
		\speed^{-2} (\Temp_{;\Lnenth} - \Temp)
		(u^{\lambda} \partial_{\lambda} u_{\beta}) \GradEnt^{\beta} ((\upeta^{-1})^{\alpha \kappa} \partial_{\kappa} \Lnenth) 
		\notag
		\\
		& \ \
		+ 
		\speed^{-2} (\Temp - \Temp_{;\Lnenth})
		((\upeta^{-1})^{\alpha \kappa} \partial_{\kappa} u_{\beta}) \GradEnt^{\beta} (u^{\lambda} \partial_{\lambda} \Lnenth)
		\notag \\
& \ \
	+
	\TempoverEnth
	\speed^{-2} (\Temp - \Temp_{;\Lnenth})
	\GradEnt_{\kappa} \GradEnt^{\kappa} ((\upeta^{-1})^{\alpha \lambda} \partial_{\lambda} \Lnenth).
	\notag
\end{align}
Using \eqref{E:SEVENTHSTEPQ12ID}
to substitute for the first product on LHS~\eqref{E:FIRSTSTEPQ12ID},
we deduce
\begin{align} \label{E:EIGHTSTEPQ12ID}
		&
		\speed^{-2}
		(\Temp - \Temp_{;\Lnenth}) 
		(\GradEnt^{\kappa} \partial_{\kappa} \Lnenth)
		((\upeta^{-1})^{\alpha \lambda} \partial_{\lambda} \Lnenth)
		\\
		& \ \
		+
		(\Temp_{;\Lnenth} - \Temp)
		(u^{\kappa} \partial_{\kappa} \GradEnt^{\beta})
		((\upeta^{-1})^{\alpha \lambda} \partial_{\lambda} u_{\beta})
		\notag
			\\
	& = 
		2 (\Temp_{;\Lnenth} - \Temp)
		(u^{\kappa} \partial_{\kappa} \GradEnt^{\beta})
		((\upeta^{-1})^{\alpha \lambda} \partial_{\lambda} u_{\beta})
			\notag 
				\\
	& \ \
		+
		(\Temp - \Temp_{;\Lnenth}) ((\upeta^{-1})^{\alpha \kappa} \partial_{\kappa} u_{\beta}) \GradEnt^{\beta} (\partial_{\lambda} u^{\lambda}) 
		+
		(\Temp_{;\Lnenth} - \Temp) (\GradEnt^{\beta} \partial_{\lambda} u_{\beta}) (\upeta^{-1})^{\alpha \kappa} (\partial_{\kappa} u^{\lambda})
		\notag \\
	& \ \
		+ 
		\speed^{-2} ( \Temp_{;\Lnenth} - \Temp)
		(u^{\lambda} \partial_{\lambda} u_{\beta}) \GradEnt^{\beta} ((\upeta^{-1})^{\alpha \kappa} \partial_{\kappa} \Lnenth) 
		\notag
		\\
		& \ \
		+ 
		\speed^{-2} (\Temp - \Temp_{;\Lnenth})
		((\upeta^{-1})^{\alpha \kappa} \partial_{\kappa} u_{\beta}) \GradEnt^{\beta} (u^{\lambda} \partial_{\lambda} \Lnenth)
		\notag \\
& \ \
	+
	\TempoverEnth
	\speed^{-2} (\Temp - \Temp_{;\Lnenth})
	\GradEnt_{\kappa} \GradEnt^{\kappa} ((\upeta^{-1})^{\alpha \lambda} \partial_{\lambda} \Lnenth).
	\notag
\end{align}
Next, we use 
\eqref{E:ENTSYMMETRYOFMIXEDPARTIALS}
and
\eqref{E:TRANSFERDERIVATIVESFROMENTROPYGRADIENTTOVORTICITY}
to express the first product on RHS~\eqref{E:EIGHTSTEPQ12ID} as
\begin{align} \label{E:NINTHSTEPQ12ID}
		&
		2 (\Temp_{;\Lnenth} - \Temp)
		(u^{\kappa} \partial_{\kappa} \GradEnt^{\beta})
		((\upeta^{-1})^{\alpha \lambda} \partial_{\lambda} u_{\beta})
		\\
		& =
	2 (\Temp_{;\Lnenth} - \Temp)
		(u^{\kappa} \partial_{\beta} \GradEnt_{\kappa})
		((\upeta^{-1})^{\alpha \lambda} \partial_{\lambda} u^{\beta})
		\notag
			\\
& =
	2 (\Temp - \Temp_{;\Lnenth})
		(\GradEnt^{\kappa} \partial_{\beta} u_{\kappa})
		((\upeta^{-1})^{\alpha \lambda} \partial_{\lambda} u^{\beta}).
		\notag
\end{align}
Using \eqref{E:NINTHSTEPQ12ID} to substitute for the first product on RHS~\eqref{E:EIGHTSTEPQ12ID},
we arrive at the desired identity \eqref{E:FIRSTSTEPQ12ID}. This completes the proof of \eqref{E:Q12ID}.

\medskip

\noindent $\bullet$ \textbf{Proof of \eqref{E:Q13ID}}:
We use the identity \eqref{E:IDVELOCITYDERIVATIVEOFVELOCITYCONTRACTEDWITHENTROPYGRADIENT}
to substitute for the factor $\GradEnt^{\lambda} \partial_{\lambda} \Lnenth$ on LHS~\eqref{E:Q13ID},
thus obtaining
\begin{align} \label{E:FIRSTSTEPQ13ID}
	(\Temp - \Temp_{;\Lnenth}) 
	u^{\alpha} 
	(\partial_{\kappa} u^{\kappa})
	(\GradEnt^{\lambda} \partial_{\lambda} \Lnenth)
	& =
	(\Temp_{;\Lnenth} - \Temp) 
	u^{\alpha} 
	(\partial_{\kappa} u^{\kappa})
	\GradEnt^{\beta}
	(u^{\lambda} \partial_{\lambda} u_{\beta}) 
	\\
	& \ \
	+
	(\Temp - \Temp_{;\Lnenth}) 
	\TempoverEnth 
	u^{\alpha} 
	(\partial_{\kappa} u^{\kappa})
	\GradEnt^{\lambda} \GradEnt_{\lambda}.
	\notag
\end{align}
Using \eqref{E:FIRSTSTEPQ13ID} to substitute for
the first product on LHS~\eqref{E:Q13ID}, we arrive at the desired identity.

\medskip

\noindent $\bullet$ \textbf{Proof of \eqref{E:Q14ID}}:
We first use equation \eqref{E:ENTSYMMETRYOFMIXEDPARTIALS} to
express the last product on LHS~\eqref{E:Q14ID} as follows:
\begin{align} \label{E:FIRSTQ14ID}
		(\Temp - \Temp_{;\Lnenth}) &
		u^{\alpha}
		(\upeta^{-1})^{\kappa \lambda}
		(\partial_{\kappa} \Lnenth) u_{\beta} (\partial_{\lambda} \GradEnt^{\beta})
			\\
		& =
		(\Temp - \Temp_{;\Lnenth}) 
		u^{\alpha}
		(\partial_{\kappa} \Lnenth) (u^{\beta} \partial_{\beta} \GradEnt^{\kappa})
			\notag \\
		& =
		(\Temp - \Temp_{;\Lnenth}) 
		u^{\alpha}
		(u^{\kappa} \partial_{\kappa} \Lnenth) (\partial_{\lambda} \GradEnt^{\lambda})
			\notag \\
		& \ \ 
		+
		(\Temp_{;\Lnenth} - \Temp) 
		u^{\alpha}
		u^{\lambda} 
		\left\lbrace
			(\partial_{\lambda} \Lnenth) (\partial_{\kappa} \GradEnt^{\kappa})
			-
			(\partial_{\kappa} \Lnenth) (\partial_{\lambda} \GradEnt^{\kappa})
		\right\rbrace,
		\notag
\end{align}
where to obtain the second equality in \eqref{E:FIRSTQ14ID},
we added and subtracted
$
(\Temp_{;\Lnenth} - \Temp) 
u^{\alpha}
(u^{\kappa} \partial_{\kappa} \Lnenth) (\partial_{\lambda} \GradEnt^{\lambda})
$.
Next, we solve for $\partial_{\lambda} \GradEnt^{\lambda}$ in terms
of the remaining terms in definition \eqref{E:MODIFIEDDIVERGENCEOFENTROPY}
and then use the resulting identity to algebraically substitute for the factor
$\partial_{\lambda} \GradEnt^{\lambda}$ in the first product on
RHS~\eqref{E:FIRSTQ14ID}, 
which yields the identity
\begin{align} \label{E:SECONDQ14ID}
(\Temp - \Temp_{;\Lnenth}) 
u^{\alpha}
(u^{\kappa} \partial_{\kappa} \Lnenth) (\partial_{\lambda} \GradEnt^{\lambda})
& =
	n
	(\Temp - \Temp_{;\Lnenth}) 
	u^{\alpha}
	(u^{\kappa} \partial_{\kappa} \Lnenth) \mathcal{D}
	\\
	& \ \
	+
	(\Temp_{;\Lnenth} - \Temp) 
	u^{\alpha}
	(u^{\kappa} \partial_{\kappa} \Lnenth) (\GradEnt^{\lambda} \partial_{\lambda} \Lnenth)
	\notag
		\\
&  \ \
	+
	\speed^{-2}
	(\Temp - \Temp_{;\Lnenth}) 
	u^{\alpha}
	(u^{\kappa} \partial_{\kappa} \Lnenth) (\GradEnt^{\lambda} \partial_{\lambda} \Lnenth).
	\notag
\end{align}
Next, we use equation \eqref{E:ENTHALPYEVOLUTION}
to substitute for the factor $u^{\kappa} \partial_{\kappa} \Lnenth$
in the last product on RHS~\eqref{E:SECONDQ14ID}, which yields the identity
\begin{align} \label{E:THIRDQ14ID}
(\Temp - \Temp_{;\Lnenth}) 
u^{\alpha}
(u^{\kappa} \partial_{\kappa} \Lnenth) (\partial_{\lambda} \GradEnt^{\lambda})
& =
	n
	(\Temp - \Temp_{;\Lnenth}) 
	u^{\alpha}
	(u^{\kappa} \partial_{\kappa} \Lnenth) \mathcal{D}
	\\
	& \ \
	+
	(\Temp_{;\Lnenth} - \Temp) 
	u^{\alpha}
	(u^{\kappa} \partial_{\kappa} \Lnenth) (\GradEnt^{\lambda} \partial_{\lambda} \Lnenth)
	\notag
		\\
& \ \
	+
	(\Temp_{;\Lnenth} - \Temp) 
	u^{\alpha}
	(\partial_{\kappa} u^{\kappa}) (\GradEnt^{\lambda} \partial_{\lambda} \Lnenth).
	\notag
\end{align}
Substituting RHS~\eqref{E:THIRDQ14ID} for the first product
on RHS~\eqref{E:FIRSTQ14ID} and then using
the resulting identity to substitute for the
last product on LHS~\eqref{E:Q14ID},
we arrive at the desired identity \eqref{E:Q14ID}.

\medskip

\noindent $\bullet$ \textbf{Proof of \eqref{E:Q15ID}}:
We first use equation \eqref{E:VELOCITYEVOLUTION} 
to substitute for the factor of
$u^{\kappa} \partial_{\kappa} u_{\beta}$
in the second product on LHS~\eqref{E:Q15ID},
which yields the identity
\begin{align} \label{E:FIRSTSTEPQ15ID}
		&
		(\partial_{\kappa} u^{\kappa})
		\upepsilon^{\alpha \beta \gamma \delta} 
		(\partial_{\beta} \Lnenth) 
		u_{\gamma}
		\vort_{\delta} 
		+
		\speed^{-2}
		\upepsilon^{\alpha \beta \gamma \delta} 
		(u^{\kappa} \partial_{\kappa} u_{\beta})
		(\partial_{\gamma} \Lnenth) \vort_{\delta}
			\\
		& = 
		(\partial_{\kappa} u^{\kappa})
		\upepsilon^{\alpha \beta \gamma \delta} 
		(\partial_{\beta} \Lnenth) 
		u_{\gamma}
		\vort_{\delta} 
		-
		\speed^{-2}
		(u^{\kappa} \partial_{\kappa} \Lnenth)
		\upepsilon^{\alpha \beta \gamma \delta} 
		u_{\beta}
		(\partial_{\gamma} \Lnenth) 
		\vort_{\delta}
		\notag
			\\
	& \ \
		+
		\speed^{-2}
		\TempoverEnth
		\upepsilon^{\alpha \beta \gamma \delta} 
		\GradEnt_{\beta}
		(\partial_{\gamma} \Lnenth) \vort_{\delta}.
		\notag
\end{align}
Using equation \eqref{E:ENTHALPYEVOLUTION} to substitute for
the factor $\partial_{\kappa} u^{\kappa}$ in the first product on RHS~\eqref{E:FIRSTSTEPQ15ID}
and taking into account the antisymmetry of $\epsilon$,
we see that the first and second products on
RHS~\eqref{E:FIRSTSTEPQ15ID} cancel, 
which yields the desired identity \eqref{E:Q15ID}.

\medskip

\noindent $\bullet$ \textbf{Proof of \eqref{E:Q16ID}}:
We simply use equation \eqref{E:VELOCITYEVOLUTION} 
to substitute for the factor
$u^{\kappa} \partial_{\kappa} u_{\sigma}$
on LHS~\eqref{E:Q16ID}.

\medskip

\noindent $\bullet$ \textbf{Proof of \eqref{E:Q18ID}}:
We simply multiply equation \eqref{E:IDVELOCITYDERIVATIVEOFENTROPYGRADIENTCONTRACTEDWITHVELOCITY}
by 
$
(\Temp_{;\Lnenth} - \Temp)
	(\upeta^{-1})^{\alpha \kappa} \partial_{\kappa} \Lnenth 
$.

\medskip

\noindent $\bullet$ \textbf{Proof of \eqref{E:Q19ID}}:
We use equation \eqref{E:IDVELOCITYDERIVATIVEOFVELOCITYCONTRACTEDWITHENTROPYGRADIENT}
to substitute for the factor \linebreak
$
(u^{\kappa} \partial_{\kappa} u_{\sigma}) \GradEnt^{\sigma}
$
in the first product on LHS~\eqref{E:Q19ID}
and equation \eqref{E:ENTHALPYEVOLUTION}
to substitute for the factor
$
u^{\lambda} \partial_{\lambda} \Lnenth
$
in the second product on LHS~\eqref{E:Q19ID}.

\medskip

\noindent $\bullet$ \textbf{Proof of \eqref{E:Q20ID}}:
We simply use equation \eqref{E:ENTROPYGRADIENTEVOLUTION}
to substitute for the factor
$u^{\lambda} \partial_{\lambda} \GradEnt^{\alpha}$
in the second product on LHS~\eqref{E:Q20ID}.

\medskip

\noindent $\bullet$ \textbf{Proof of \eqref{E:Q21ID}}:
\eqref{E:Q21ID} follows from \eqref{E:DERIVATIVEOFVELOCITYCONTRACTEDWITHVELOCITYISZERO}.

\end{proof}

\subsection{The $\mbox{\upshape transport}$-$\mbox{\upshape div}$-$\mbox{\upshape curl}$ system}
Armed with Lemma~\ref{L:IDOFNULLSTRUCTUREFORDIVCURLTRANSPORTSYSTEM},
we now derive the main result of this section.

\begin{proposition}[The $\mbox{\upshape transport}$-$\mbox{\upshape div}$-$\mbox{\upshape curl}$ system for the vorticity]
	\label{P:EQUATIONSFORMODIFIEDVORTICITYVORTICITY}
Assume that $(\Lnenth,\Ent,u^{\alpha})$ is a $C^3$ solution 
to \eqref{E:ENTHALPYEVOLUTION}-\eqref{E:ENTROPYEVOLUTION} + \eqref{E:UISUNITLENGTH}.
Then the divergence of the vorticity vectorfield $\vort^{\alpha}$ defined in \eqref{E:VORTICITYDEF} verifies the following identity:
\begin{align} \label{E:DIVOFVORTICITY}
	\partial_{\alpha} \vort^{\alpha}
	& = - \vort^{\kappa} \partial_{\kappa} \Lnenth
		+
		2 \TempoverEnth \vort^{\kappa} \GradEnt_{\kappa}.
\end{align}

Moreover, the rectangular components $\mathcal{C}^{\alpha}$ of the modified vorticity of the vorticity,
which is defined in \eqref{E:MODIFIEDVORTICITYOFVORTICITY},
verify the following transport equations:
\begin{align}  \label{E:TRANSPORTFORMODIFIEDVORTICITYOFVORTICITY}
		u^{\kappa} \partial_{\kappa} \mathcal{C}^{\alpha}
		& = 
		\mathcal{C}^{\kappa} \partial_{\kappa} u^{\alpha}
		- 
		2 (\partial_{\kappa} u^{\kappa}) \mathcal{C}^{\alpha}
		+
		u^{\alpha} 
		(u^{\kappa} \partial_{\kappa} u_{\lambda}) 
		\mathcal{C}^{\lambda}
					\\
		& \ \
			-
			2 \upepsilon^{\alpha \beta \gamma \delta} 
			u_{\beta}
			(\partial_{\gamma} \vort^{\kappa}) (\partial_{\delta} u_{\kappa})
		\notag \\
		& \ \
			+
		(\Temp_{;\Lnenth} - \Temp) 
		\left\lbrace
			(\upeta^{-1})^{\alpha \kappa}
			+
			2 u^{\alpha} u^{\kappa}
		\right\rbrace
			\left\lbrace
				(\partial_{\kappa} \Lnenth) (\partial_{\lambda} \GradEnt^{\lambda})
				-
				(\partial_{\lambda} \Lnenth) (\partial_{\kappa} \GradEnt^{\lambda})
			\right\rbrace
			\notag \\
		& \ \
		+
		n (\Temp - \Temp_{;\Lnenth}) u^{\alpha} (u^{\kappa} \partial_{\kappa} \Lnenth) \mathcal{D} 
			\notag
			\\
		& \ \
		+
		(\Temp - \Temp_{;\Lnenth}) 
		\TempoverEnth 
		\GradEnt^{\alpha} 
		\partial_{\kappa} \GradEnt^{\kappa}
		+
		(\Temp_{;\Lnenth} - \Temp) 
		\TempoverEnth 
		((\upeta^{-1})^{\alpha \lambda} \partial_{\lambda} \GradEnt^{\kappa}) 
		\GradEnt_{\kappa}
			\notag \\
	& \ \
		+
		\mathfrak{Q}_{(\mathcal{C}^{\alpha})}
		+
		\mathfrak{L}_{(\mathcal{C}^{\alpha})},
		\notag
	\end{align}
where $\mathfrak{Q}_{(\mathcal{C}^{\alpha})}$ is the linear combination of 
\textbf{null forms} defined by
\begin{align} \label{E:VORTICITYOFVORTICITYEASYNULLFORMS}
	\mathfrak{Q}_{(\mathcal{C}^{\alpha})}
	& := 
			-
			\speed^{-2}
			\upepsilon^{\kappa \beta \gamma \delta} 
			(\partial_{\kappa} u^{\alpha})
			u_{\beta}
			(\partial_{\gamma} \Lnenth) \vort_{\delta}
			\\
			& \ \
		+
		(\speed^{-2} + 2)
		\upepsilon^{\alpha \beta \gamma \delta} 
		u_{\beta}
		(\partial_{\gamma} \Lnenth) \vort^{\kappa} (\partial_{\delta} u_{\kappa})
		\notag
			\\
	& \ \
	+
	\speed^{-2}
	\upepsilon^{\alpha \beta \gamma \delta} 
	u_{\beta}
	\vort_{\delta}
	\left\lbrace
		(\partial_{\kappa} u^{\kappa}) (\partial_{\gamma} \Lnenth)
		- 
		(\partial_{\gamma} u^{\kappa}) (\partial_{\kappa} \Lnenth)
	\right\rbrace
	\notag
			\\
	& \ \
		+
	\left\lbrace	
		(\Temp_{;\Lnenth;\Lnenth} - \Temp_{\Lnenth}) 
		+
		\speed^{-2} 
		(\Temp - \Temp_{;\Lnenth})
	\right\rbrace
	(\upeta^{-1})^{\alpha \lambda}
	\GradEnt^{\beta}
	u^{\kappa} 
		\notag \\
	& 
	\ \ \ \ \ 
	\times
	\left\lbrace
		(\partial_{\kappa} \Lnenth)
		(\partial_{\lambda} u_{\beta})
		-
		(\partial_{\lambda} \Lnenth)
		(\partial_{\kappa} u_{\beta})
	\right\rbrace
	\notag
		\\
	& \ \
		+
		(\Temp_{;\Lnenth} - \Temp) 
		\GradEnt^{\kappa}
		u^{\lambda}
		\left\lbrace
			(\partial_{\kappa} u^{\alpha})
			(\partial_{\lambda} \Lnenth) 
			-
			(\partial_{\lambda} u^{\alpha})
			(\partial_{\kappa} \Lnenth)
		\right\rbrace
		\notag
		\\
	& \ \
		+
	(\Temp_{;\Lnenth} - \Temp) 
	\left\lbrace
		(\upeta^{-1})^{\alpha \kappa} 
		+
		u^{\alpha} u^{\kappa}
	\right\rbrace
	\GradEnt^{\beta}
		\notag \\
	& 
	\ \ \ \ \ 
	\times
	\left\lbrace
		(\partial_{\kappa} u_{\beta})
		(\partial_{\lambda} u^{\lambda})
		- 
		(\partial_{\lambda} u_{\beta}) 
		(\partial_{\kappa} u^{\lambda})
	\right\rbrace
	\notag
		\\
	& \ \
		+
		(\Temp_{;\Lnenth} - \Temp) 
		\GradEnt^{\alpha} 
		\left\lbrace
			(\partial_{\kappa} u^{\lambda}) (\partial_{\lambda} u^{\kappa})
				-
			(\partial_{\kappa} u^{\kappa})
			(\partial_{\lambda} u^{\lambda})
		\right\rbrace
				\notag \\
	& \ \
		+
		(\Temp_{;\Lnenth} - \Temp)
			\GradEnt^{\kappa}
			\left\lbrace
				(\partial_{\kappa} u^{\alpha})
				(\partial_{\lambda} u^{\lambda})
				-	 
				 (\partial_{\lambda} u^{\alpha})
				 (\partial_{\kappa} u^{\lambda})
			\right\rbrace
			\notag \\
& \ \
	+
	\GradEnt^{\alpha}
		\left\lbrace
			\speed^{-2}
			(\Temp_{\Lnenth} - \Temp_{;\Lnenth;\Lnenth})
			+
			\speed^{-4} 
			(\Temp_{;\Lnenth} - \Temp)
		\right\rbrace
		(g^{-1})^{\kappa \lambda} (\partial_{\kappa} \Lnenth) (\partial_{\lambda} \Lnenth),
		\notag 
\end{align}
and $\mathfrak{L}_{(\mathcal{C}^{\alpha})}$, which is at most linear in the derivatives of the solution
variables, is defined by
\begin{align} \label{E:VORTICITYOFVORTICITYLINEARTERMS}
	\mathfrak{L}_{(\mathcal{C}^{\alpha})}
	& := 
		\frac{2 \TempoverEnth}{\Enth}
				\vort^{\kappa} \GradEnt_{\kappa}
				\vort^{\alpha}
		-
		\frac{2}{\Enth} \vort^{\alpha} (\vort^{\kappa} \partial_{\kappa} \Lnenth) 
			\\
	& \ \
		+
		2 \speed^{-3} \speed_{;\Ent} 
	(u^{\kappa} \partial_{\kappa} \Lnenth)
	\upepsilon^{\alpha \beta \gamma \delta} 
	u_{\beta}
	\GradEnt_{\gamma}
	\vort_{\delta}
	\notag
		\\
	& \ \
			-
			2 \TempoverEnth
			\upepsilon^{\alpha \beta \gamma \delta}
			u_{\beta}
			\GradEnt_{\gamma} 
			\vort^{\kappa} (\partial_{\delta} u_{\kappa})
			-
		\TempoverEnth
		(\partial_{\kappa} u^{\kappa})
		\upepsilon^{\alpha \beta \gamma \delta} 
		\GradEnt_{\beta}
		u_{\gamma}
		\vort_{\delta} 
			\notag
			\\
	& \ \
	+
	\frac{1}{\Enth} 
	(\Temp - \Temp_{;\Lnenth}) 
	\upepsilon^{\kappa \beta \gamma \delta} 
	(\partial_{\kappa} u^{\alpha})
	\GradEnt_{\beta} u_{\gamma} \vort_{\delta}
	+
		\speed^{-2}
		\TempoverEnth
		\upepsilon^{\alpha \beta \gamma \delta} 
		\GradEnt_{\beta}
		(\partial_{\gamma} \Lnenth) \vort_{\delta}
		\notag
		\\
		& \ \
	-
			\speed^{-2}
			\TempoverEnth
			u^{\alpha}
			\upepsilon^{\kappa \beta \gamma \delta} 
			\GradEnt_{\kappa}
			u_{\beta}
			(\partial_{\gamma} \Lnenth) \vort_{\delta}
		\notag
		\\
& \ \
	+ 
	\TempoverEnth
	(\Temp_{;\Lnenth} - \Temp) 
	\GradEnt_{\kappa} \GradEnt^{\kappa}
	(u^{\lambda} \partial_{\lambda} u^{\alpha})
	\notag
		\\
& \ \
	+
	\TempoverEnth
	(\Temp_{;\Lnenth} - \Temp) 
	u^{\alpha}
	\GradEnt_{\kappa} \GradEnt^{\kappa}
	(u^{\lambda} \partial_{\lambda} \Lnenth) 
	+
	(\Temp_{;\Lnenth;\Ent} - \Temp_{;\Ent})
	u^{\alpha} 
	\GradEnt_{\kappa} \GradEnt^{\kappa}
	(u^{\lambda} \partial_{\lambda} \Lnenth) 
		\notag
		\\
	& \ \
			+
			(\Temp_{;\Ent} - \Temp_{;\Lnenth;\Ent})
			\GradEnt^{\alpha} 
			(\GradEnt^{\kappa} \partial_{\kappa} \Lnenth) 
			+
		(\Temp - \Temp_{;\Lnenth}) 
		\TempoverEnth_{;\Lnenth} 
		\GradEnt^{\alpha}
		(\GradEnt^{\kappa} \partial_{\kappa} \Lnenth)
							\notag
				\\
	& \ \
		+
	\TempoverEnth
	(\Temp_{;\Lnenth;\Lnenth} - \Temp_{\Lnenth})
	\GradEnt_{\kappa} \GradEnt^{\kappa}
	((\upeta^{-1})^{\alpha \lambda}\partial_{\lambda} \Lnenth) 
		+
		(\Temp_{;\Lnenth;\Ent} - \Temp_{;\Ent})
		\GradEnt_{\kappa} \GradEnt^{\kappa}
		((\upeta^{-1})^{\alpha \lambda} \partial_{\lambda} \Lnenth) 
		\notag \\
& \ \
	+
	\TempoverEnth 
	\speed^{-2} 
	(\Temp - \Temp_{;\Lnenth}) \GradEnt_{\kappa} \GradEnt^{\kappa} ((\upeta^{-1})^{\alpha \lambda} \partial_{\lambda} \Lnenth)
	+
	(\Temp_{;\Lnenth} - \Temp) \TempoverEnth_{;\Lnenth} 
	\GradEnt_{\kappa} \GradEnt^{\kappa} ((\upeta^{-1})^{\alpha \lambda} \partial_{\lambda} \Lnenth).
		\notag
\end{align}
\end{proposition}

\begin{proof}
We split the proof into several pieces.

\noindent $\bullet$ \textbf{Proof of \eqref{E:DIVOFVORTICITY}}:
First, from definition \eqref{E:VORTICITYDEF}
and the antisymmetry of $\upepsilon^{\kappa \lambda \gamma \delta}$,
we deduce
\begin{align} \label{E:FIRSTIDDIVOFVORTICITY}
	\partial_{\kappa} \vort^{\kappa}
	& = 
	- 
	\upepsilon^{\kappa \lambda \gamma \delta} 
	(\partial_{\kappa} u_{\lambda}) 
	\partial_{\gamma} (\Enth u_{\delta}).
\end{align}
Next, using \eqref{E:CONTRACTIONVOLUMEFORMANDEXTERIORDERIVATIVEOFENTHALPHYTIMESVELOCITY},
we deduce that
\begin{align} \label{E:SECONDIDDIVOFVORTICITY}
	\mbox{RHS~\eqref{E:FIRSTIDDIVOFVORTICITY}} 
	& = 
		\vort^{\lambda} (u^{\kappa} \partial_{\kappa} u_{\lambda})
		-
		\vort^{\kappa} u^{\lambda} (\partial_{\kappa} u_{\lambda})
		-
		\Temp
		\upepsilon^{\kappa \lambda \gamma \delta}
		(\partial_{\kappa} u_{\lambda}) 
		\GradEnt_{\gamma} u_{\delta}.
\end{align}
Using \eqref{E:DERIVATIVEOFVELOCITYCONTRACTEDWITHVELOCITYISZERO},
we see that the second product on RHS~\eqref{E:SECONDIDDIVOFVORTICITY} vanishes.
Moreover, using equation \eqref{E:VELOCITYEVOLUTION} and the identity \eqref{E:VORTISORTHGONALTOU},
we can express the first product on RHS~\eqref{E:SECONDIDDIVOFVORTICITY}
as follows:
\begin{align} \label{E:THIRDIDDIVOFVORTICITY}
	\vort^{\lambda} (u^{\kappa} \partial_{\kappa} u_{\lambda})
	& = - \vort^{\kappa} \partial_{\kappa} \Lnenth
		+
		\TempoverEnth \vort^{\kappa} \GradEnt_{\kappa}.
\end{align}
In addition, using definition \eqref{E:TEMPOVERENTH} and the identity \eqref{E:GRADIENTOFUCONTRACTEDAGAINSTVOLUMEFORMANDVELOCITY},
we can express the last product on RHS~\eqref{E:SECONDIDDIVOFVORTICITY}
as follows:
\begin{align} \label{E:FOURTHIDDIVOFVORTICITY}
	-
	\Temp
	\upepsilon^{\kappa \lambda \gamma \delta}
	(\partial_{\kappa} u_{\lambda}) 
	\GradEnt_{\gamma} u_{\delta}
	& = 
	\TempoverEnth \vort^{\kappa} \GradEnt_{\kappa}.
\end{align}
Combining these calculations, we arrive at the desired identity \eqref{E:DIVOFVORTICITY}.

\medskip

\noindent $\bullet$ \textbf{Proof of \eqref{E:TRANSPORTFORMODIFIEDVORTICITYOFVORTICITY}}:
The proof is a series of lengthy calculations in which we observe many cancellations.
We start by using \eqref{E:TERM1ID}-\eqref{E:TERM7ID}
to substitute for all of the terms on the third through seventh
lines of RHS~\eqref{E:VORTICITYOFVORTICITYFIRSTEVOLUTION}
except for the term
$-
			\upepsilon^{\alpha \beta \gamma \delta} 
			u_{\beta}
			(\partial_{\gamma} \vort^{\kappa}) 
			(\partial_{\delta} u_{\kappa})$
from the fifth line,
which we leave as is.
We also use \eqref{E:DIVOFVORTICITY}
to express the fourth product 
on RHS~\eqref{E:TERM5ID} as
$
\frac{1}{\Enth} 
\vort^{\alpha}
(\partial_{\kappa} \vort^{\kappa})
=
-
\frac{1}{\Enth} 
\vort^{\alpha}
(\vort^{\kappa} \partial_{\kappa} \Lnenth)
+
\frac{2 \TempoverEnth}{\Enth} 
\vort^{\alpha}
\vort^{\kappa} \GradEnt_{\kappa}
$,
and we use
\eqref{E:TOPORDERENTHTERMKEYID}-\eqref{E:TWODERIVATIVEENTHID4}
to substitute for the four products 
(which depend on the second derivatives of $\Lnenth$)
on the sixth-to-last and fifth-to-last
lines of RHS~\eqref{E:VORTICITYOFVORTICITYFIRSTEVOLUTION},
thereby obtaining the following equation
(where at this stage in the argument, we have simply performed a term-by-term substitution
and have not yet organized the terms):
\begingroup
\allowdisplaybreaks
\begin{align} \label{E:RETRYVORTICITYOFVORTICITYFIRSTEVOLUTION}
		u^{\kappa} \partial_{\kappa} \uperpvort^{\alpha}(\vort)
		& = 
			\uperpvort^{\kappa}(\vort) \partial_{\kappa} u^{\alpha}
			- 
			(\partial_{\kappa} u^{\kappa}) \uperpvort^{\alpha}(\vort)
			\\
			& \ \
			+
				u^{\alpha} 
				(u^{\kappa} \partial_{\kappa} u_{\beta}) 
				\uperpvort^{\beta}(\vort)
				\notag
				\\
		& \ \
			-
			u^{\kappa} \partial_{\kappa} 
			\left\lbrace
				\speed^{-2}
				\upepsilon^{\alpha \beta \gamma \delta} 
				u_{\beta}
				(\partial_{\gamma} \Lnenth) 
				\vort_{\delta}
			\right\rbrace
			\notag
			\\
			& \ \
			-
		2
		(\partial_{\kappa} u^{\kappa})
		\speed^{-2}
		\upepsilon^{\alpha \beta \gamma \delta} 
		u_{\beta}
		(\partial_{\gamma} \Lnenth) \vort_{\delta}
		\notag \\
& \ \
	+
		\speed^{-2}
		\upepsilon^{\alpha \beta \gamma \delta} 
		(u^{\kappa} \partial_{\kappa} u_{\beta})
		(\partial_{\gamma} \Lnenth) \vort_{\delta}
	\notag
		\\
& \ \
		+
		\speed^{-2}
		\upepsilon^{\alpha \beta \gamma \delta} 
		u_{\beta}
		(\partial_{\gamma} \Lnenth) 
		\vort^{\kappa} 
		(\partial_{\delta} u_{\kappa})
			\notag \\
	& \ \
	+
	\speed^{-2}
	(\Temp - \Temp_{;\Lnenth}) 
	(\GradEnt^{\kappa} \partial_{\kappa} \Lnenth)
	((\upeta^{-1})^{\alpha \lambda} \partial_{\lambda} \Lnenth) 
		\notag \\
& \ \
	+
	\speed^{-2}
	(\Temp - \Temp_{;\Lnenth}) 
	u^{\alpha}
	(\GradEnt^{\kappa} \partial_{\kappa} \Lnenth)
	(u^{\lambda} \partial_{\lambda} \Lnenth)
	\notag \\
& \ \
+ \speed^{-2}
	(\Temp_{;\Lnenth} - \Temp) 
	\GradEnt^{\alpha} 
	(u^{\kappa} \partial_{\kappa} \Lnenth)
	(u^{\lambda} \partial_{\lambda} \Lnenth)
	\notag \\
& \ \
	+	 
	\speed^{-2}
	(\Temp_{;\Lnenth} - \Temp) 
	\GradEnt^{\alpha}
	(\upeta^{-1})^{\kappa \lambda}
	(\partial_{\kappa} \Lnenth)
	(\partial_{\lambda} \Lnenth)
		\notag
		\\
	& \ \
	+
	\speed^{-2}
	\upepsilon^{\alpha \beta \gamma \delta} 
	u_{\beta}
	\left\lbrace
		(\partial_{\kappa} u^{\kappa}) (\partial_{\gamma} \Lnenth) 
		- 
		(\partial_{\gamma} u^{\kappa}) (\partial_{\kappa} \Lnenth) 
	\right\rbrace
	\vort_{\delta}	
		\notag
			\\
	& \ \
		+
	2 \speed^{-3} \speed_{;\Ent} 
	(u^{\kappa} \partial_{\kappa} \Lnenth)
	\upepsilon^{\alpha \beta \gamma \delta} 
	u_{\beta}
	\GradEnt_{\gamma}
	\vort_{\delta}
	\notag
		\notag \\
& \ \
		+ 
		\frac{1}{\Enth} (\vort^{\kappa} \partial_{\kappa} \vort^{\alpha})
		-
		\frac{1}{\Enth} (\vort^{\kappa} \partial_{\kappa} \Lnenth)  \vort^{\alpha}
			\notag \\
	&  \ \
		-
		\frac{1}{\Enth}
		u^{\alpha} \vort^{\lambda}
		(\vort^{\kappa} \partial_{\kappa} u_{\lambda})
		+
		\upepsilon^{\alpha \beta \gamma \delta} 
		u_{\beta}
		(\partial_{\gamma} \Lnenth)
		\vort^{\kappa} (\partial_{\delta} u_{\kappa})
				\notag
				\\
	& \ \
			-
			\TempoverEnth
			\upepsilon^{\alpha \beta \gamma \delta}
			u_{\beta}
			\GradEnt_{\gamma} 
			\vort^{\kappa} 
			(\partial_{\delta} u_{\kappa})
			\notag
			\\
	& \ \
	-
	\upepsilon^{\alpha \beta \gamma \delta} 
	(\partial_{\beta} \Lnenth) 
	u_{\gamma} \vort^{\kappa} (\partial_{\delta} u_{\kappa}) 
	+
	(\partial_{\kappa} u^{\kappa})
	\upepsilon^{\alpha \beta \gamma \delta} 
	(\partial_{\beta} \Lnenth) 
	u_{\gamma}
	\vort_{\delta} 
		\notag \\
	& \ \
	+
	(\Temp - \Temp_{;\Lnenth})
	((\upeta^{-1})^{\alpha \kappa} \partial_{\kappa} \Lnenth) 
	(\GradEnt^{\lambda} \partial_{\lambda} \Lnenth) 
	\notag
	\\
	& \ \
	+
  (\Temp - \Temp_{;\Lnenth})
	u^{\alpha}
	(u^{\kappa} \partial_{\kappa} \Lnenth)
	(\GradEnt^{\lambda} \partial_{\lambda} \Lnenth) 
	\notag \\
& \ \
	+	
	(\Temp_{;\Lnenth} - \Temp)
	\GradEnt^{\alpha} 
	(u^{\kappa} \partial_{\kappa} \Lnenth)
	(u^{\lambda} \partial_{\lambda} \Lnenth)
	\notag
	\\
	& \ \
	+
	(\Temp_{;\Lnenth} - \Temp)
	\GradEnt^{\alpha} 
	(\upeta^{-1})^{\kappa \lambda}
	(\partial_{\kappa} \Lnenth) 
	(\partial_{\lambda} \Lnenth)
	\notag
		\\
& \ \
	+
	\TempoverEnth
	\upepsilon^{\alpha \beta \gamma \delta} 
	\GradEnt_{\beta}
	u_{\gamma}
	\vort^{\kappa} (\partial_{\delta} u_{\kappa})
		\notag
			\\
	& \ \
	-
	\TempoverEnth
	(\partial_{\kappa} u^{\kappa})
	\upepsilon^{\alpha \beta \gamma \delta} 
	\GradEnt_{\beta}
	u_{\gamma}
	\vort_{\delta} 
		\notag \\
& \ \
	+
	\TempoverEnth
	(\Temp_{;\Lnenth} - \Temp) 
	((\upeta^{-1})^{\kappa \alpha} \partial_{\kappa} \Lnenth) \GradEnt^{\lambda} \GradEnt_{\lambda}
    \notag
    \\
    & \ \	
	+
	\TempoverEnth
	(\Temp_{;\Lnenth} - \Temp) 
	u^{\alpha}
 (u^{\kappa} \partial_{\kappa} \Lnenth) \GradEnt^{\lambda} \GradEnt_{\lambda} 
	+
	\TempoverEnth
	(\Temp - \Temp_{;\Lnenth}) 
	\GradEnt^{\alpha}
	(\GradEnt^{\kappa} \partial_{\kappa} \Lnenth)
	\notag
		\\
& \ \
	-
	\upepsilon^{\alpha \beta \gamma \delta} (\partial_{\beta} \Lnenth) 
	u_{\gamma}
	\vort^{\lambda} (\partial_{\delta} u_{\lambda})
	+
	\TempoverEnth
	\upepsilon^{\alpha \beta \gamma \delta} \GradEnt_{\beta} 
	u_{\gamma}
	\vort^{\lambda} (\partial_{\delta} u_{\lambda})
		\notag \\
& \ \
			-
			\upepsilon^{\alpha \beta \gamma \delta} 
			u_{\beta}
			(\partial_{\gamma} \vort^{\kappa}) (\partial_{\delta} u_{\kappa})
			\notag	\\
& \ \
		-
		\upepsilon^{\alpha \beta \gamma \delta} 
		u_{\beta}
		(\partial_{\gamma} \vort^{\kappa}) (\partial_{\delta} u_{\kappa})
		-
		\upepsilon^{\alpha \beta \gamma \delta} 
		u_{\beta}
		(\partial_{\gamma} \Lnenth)
		\vort^{\kappa}
		(\partial_{\delta} u_{\kappa})
		\notag	\\
& \ \
			-
			\frac{1}{\Enth} 
			(\vort^{\kappa} \partial_{\kappa} \vort^{\alpha}) 
			-
			\frac{1}{\Enth} 
			\vort^{\alpha}
			(\vort^{\kappa} \partial_{\kappa} \Lnenth)
			+
			\frac{2 \TempoverEnth}{\Enth} 
			\vort^{\alpha}
			\vort^{\kappa} \GradEnt_{\kappa}
			\notag \\
& \ \
		-
		\frac{1}{\Enth} 
		\vort^{\alpha}
		\vort^{\lambda}
		(u^{\kappa} \partial_{\kappa} u_{\lambda}) 
		+
		\frac{1}{\Enth} 
		u^{\alpha} 
		\vort^{\lambda}
		(\vort^{\kappa} \partial_{\kappa} u_{\lambda})
		\notag		\\
	& \ \
		-
		\TempoverEnth
		\upepsilon^{\alpha \beta \gamma \delta} 
		u_{\beta}
		(\partial_{\gamma} u^{\kappa})
		\vort_{\kappa}
		\GradEnt_{\delta}
		\notag
		\\
	& \ \
		-
		(\partial_{\kappa} u^{\kappa})
		\uperpvort^{\alpha}(\vort)
			\notag \\
	& \ \
		+
		\frac{1}{\Enth}
		\vort^{\alpha}
		\vort^{\lambda}
		(u^{\kappa} \partial_{\kappa} u_{\lambda})
			\notag \\
	& \ \
			+
			(\Temp_{\Lnenth} - \Temp_{;\Lnenth;\Lnenth})
			\GradEnt^{\alpha}  
			(\upeta^{-1})^{\kappa \lambda} (\partial_{\kappa} \Lnenth) 
			(\partial_{\lambda} \Lnenth)
			\notag
			\\
			& \ \
			+
			(\Temp_{\Lnenth} - \Temp_{;\Lnenth;\Lnenth})
			 \GradEnt^{\alpha}
			(u^{\kappa} \partial_{\kappa} \Lnenth)
			(u^{\lambda} \partial_{\lambda} \Lnenth)
			\notag
			\\
	& \ \
			+
			(\Temp_{;\Lnenth;\Lnenth} - \Temp_{\Lnenth})
			u^{\alpha} 
			(\GradEnt^{\kappa} \partial_{\kappa} \Lnenth)
			(u^{\lambda} \partial_{\lambda} \Lnenth) 
			\notag
			\\
			& \ \
			+
			(\Temp_{;\Lnenth;\Lnenth} - \Temp_{\Lnenth})
			((\upeta^{-1})^{\alpha \kappa}\partial_{\kappa} \Lnenth) 
			(\GradEnt^{\lambda} \partial_{\lambda} \Lnenth)
			\notag
			\\
	& \ \
			+
			(\Temp_{;\Ent} - \Temp_{;\Lnenth;\Ent})
			\GradEnt^{\alpha} 
			(\GradEnt^{\kappa} \partial_{\kappa} \Lnenth) 
			+
			(\Temp_{;\Lnenth;\Ent} - \Temp_{;\Ent})
			u^{\alpha} 
			\GradEnt_{\kappa} \GradEnt^{\kappa}
			(u^{\lambda} \partial_{\lambda} \Lnenth) 
			\notag
			\\
			& \ \
			+
			(\Temp_{;\Lnenth;\Ent} - \Temp_{;\Ent})
			\GradEnt_{\kappa} \GradEnt^{\kappa}
			((\upeta^{-1})^{\alpha \lambda} \partial_{\lambda} \Lnenth) 
				\notag
				\\
		& \ \
			+
			(\Temp - \Temp_{;\Lnenth}) 
			\GradEnt^{\alpha}
			(\partial_{\kappa} u^{\kappa})
			(u^{\lambda} \partial_{\lambda} \Lnenth) 
			\notag
			\\
			& \ \
			+
			(\Temp_{;\Lnenth} - \Temp) 
			(\GradEnt^{\kappa} \partial_{\kappa} u^{\alpha})
			(u^{\lambda} \partial_{\lambda} \Lnenth) 
				\notag 
			\\
		& \ \
			+
			u^{\kappa} \partial_{\kappa}
		\left\lbrace
			(\Temp_{;\Lnenth} - \Temp) 
			\GradEnt^{\alpha} 
			(\partial_{\lambda} u^{\lambda})
		\right\rbrace
			\notag \\
	& \ \
		+ 
		(\Temp_{\Lnenth} - \Temp_{;\Lnenth;\Lnenth})
		\GradEnt^{\alpha}
		(u^{\kappa} \partial_{\kappa} \Lnenth)
		(\partial_{\lambda} u^{\lambda})
		\notag
		\\
		& \ \
		+
		(\Temp - \Temp_{;\Lnenth}) 
		(u^{\kappa} \partial_{\kappa} \GradEnt^{\alpha})
		(\partial_{\lambda} u^{\lambda})
		\notag \\
	& \ \
		+
		(\Temp_{;\Lnenth} - \Temp) 
		\GradEnt^{\alpha} 
		(\partial_{\kappa} u^{\lambda}) (\partial_{\lambda} u^{\kappa})
		\notag
		\\
		& \ \
		+
		(\Temp_{;\Lnenth} - \Temp) 
		\GradEnt^{\alpha} 
		(u^{\kappa} \partial_{\kappa} u^{\lambda}) (\partial_{\lambda} \Lnenth)
		+
		(\Temp_{;\Lnenth} - \Temp) 
		\GradEnt^{\alpha} 
		(\partial_{\kappa} u^{\kappa}) (u^{\lambda} \partial_{\lambda} \Lnenth)
		\notag
			\\
	& \ \
		+
		(\Temp - \Temp_{;\Lnenth}) 
		\TempoverEnth 
		\GradEnt^{\alpha} 
		(\partial_{\kappa} \GradEnt^{\kappa})
		+
		(\Temp - \Temp_{;\Lnenth}) 
		\TempoverEnth_{;\Lnenth} 
		\GradEnt^{\alpha}
		(\GradEnt^{\kappa} \partial_{\kappa} \Lnenth)
		\notag
		\\
		& \ \
		+
		(\Temp - \Temp_{;\Lnenth}) 
		\TempoverEnth_{;\Ent} 
		\GradEnt^{\alpha}	
		\GradEnt_{\kappa} \GradEnt^{\kappa}
			\notag \\
		& \ \
				+
				u^{\kappa} \partial_{\kappa}
				\left\lbrace
					(\Temp_{;\Lnenth} - \Temp) u^{\alpha} (\GradEnt^{\lambda} \partial_{\lambda} \Lnenth)
				\right\rbrace
				\notag	\\
		& \ \
			+
			(\Temp_{\Lnenth} - \Temp_{;\Lnenth;\Lnenth})
			u^{\alpha}
			(u^{\kappa} \partial_{\kappa} \Lnenth)
			(\GradEnt^{\lambda} \partial_{\lambda} \Lnenth)
			\notag
			\\
			& \ \
			+
			(\Temp - \Temp_{;\Lnenth}) (u^{\kappa} \partial_{\kappa} u^{\alpha}) 
			(\GradEnt^{\lambda} \partial_{\lambda} \Lnenth)
			\notag
			\\
			& \ \
			+
			(\Temp - \Temp_{;\Lnenth})
			u^{\alpha} (u^{\kappa} \partial_{\kappa} \GradEnt^{\lambda}) (\partial_{\lambda} \Lnenth)
			\notag
				\\
			& \ \
			+ 
			u^{\kappa} \partial_{\kappa}
			\left\lbrace
				(\Temp - \Temp_{;\Lnenth}) (\upeta^{-1})^{\alpha \lambda} \GradEnt^{\beta} (\partial_{\lambda} u_{\beta})
			\right\rbrace
			\notag \\
	& \ \
		+
		(\Temp_{;\Lnenth;\Lnenth} - \Temp_{\Lnenth}) 
		(u^{\kappa} \partial_{\kappa} \Lnenth)
		\GradEnt^{\beta} ((\upeta^{-1})^{\alpha \lambda} \partial_{\lambda} u_{\beta})
		\notag
		\\
		& \ \
		+
		(\Temp_{;\Lnenth} - \Temp)
		(u^{\kappa} \partial_{\kappa} \GradEnt^{\beta})
		((\upeta^{-1})^{\alpha \lambda} \partial_{\lambda} u_{\beta})
		\notag \\
	& \ \
		+
		(\Temp - \Temp_{;\Lnenth}) \GradEnt^{\beta} ((\upeta^{-1})^{\alpha \lambda} \partial_{\lambda} u_{\beta}) (u^{\kappa} \partial_{\kappa} \Lnenth)
		\notag
		\\
		& \ \
		+
		(\Temp - \Temp_{;\Lnenth}) \GradEnt^{\beta} ((\upeta^{-1})^{\alpha \lambda} \partial_{\lambda} u^{\kappa}) (\partial_{\kappa} u_{\beta})
		\notag \\
	& \ \
	+
	(\Temp - \Temp_{;\Lnenth}) \TempoverEnth ((\upeta^{-1})^{\alpha \lambda} \partial_{\lambda} \GradEnt^{\beta}) \GradEnt_{\beta}
	\notag
	\\
	& \ \ 
	+
	(\Temp_{;\Lnenth} - \Temp) \TempoverEnth_{;\Lnenth} ((\upeta^{-1})^{\alpha \lambda} \partial_{\lambda} \Lnenth) \GradEnt_{\kappa} \GradEnt^{\kappa}
		\notag \\
&  \ \
	+
	(\Temp_{;\Lnenth} - \Temp) \TempoverEnth_{;\Ent} \GradEnt^{\alpha} \GradEnt_{\kappa} \GradEnt^{\kappa}
	+
	2 (\Temp_{;\Lnenth} - \Temp) \TempoverEnth ((\upeta^{-1})^{\alpha \lambda} \partial_{\lambda} \GradEnt^{\kappa}) \GradEnt_{\kappa}
	\notag	
		\\
		& \ \
			+
			(\Temp - \Temp_{;\Lnenth}) 
			(\upeta^{-1})^{\kappa \lambda}
			(\partial_{\kappa} \Lnenth) (\partial_{\lambda}\GradEnt^{\alpha})
			\notag
			\\
			& \ \
			+
			(\Temp - \Temp_{;\Lnenth}) 
			(u^{\kappa} \partial_{\kappa} \Lnenth) (u^{\lambda} \partial_{\lambda} \GradEnt^{\alpha})
			\notag 
			\\
		& \ \
			+
			(\Temp_{;\Lnenth} - \Temp) 
			u^{\alpha} 
			(u^{\kappa} \partial_{\kappa} \Lnenth) (\partial_{\lambda} \GradEnt^{\lambda})
			\notag
			\\
			& \ \
			+
			(\Temp_{;\Lnenth} - \Temp) 
			((\upeta^{-1})^{\alpha \kappa} \partial_{\kappa} \Lnenth) (\partial_{\lambda} \GradEnt^{\lambda})
			\notag 
			\\
		& \ \
			+
		(\Temp_{;\Lnenth} - \Temp) 
		 ((\upeta^{-1})^{\alpha \kappa} \partial_{\kappa} \Lnenth) u_{\beta} (u^{\lambda} \partial_{\lambda} \GradEnt^{\beta})
			\notag
			\\
			& \ \
			+
			(\Temp - \Temp_{;\Lnenth}) 
			u^{\alpha}
			(\upeta^{-1})^{\kappa \lambda}
			 (\partial_{\kappa} \Lnenth) u_{\beta} (\partial_{\lambda} \GradEnt^{\beta}).
			\notag 
\end{align}
\endgroup

Next, we bring the four perfect-derivative terms 
$u^{\kappa} \partial_{\kappa} \lbrace \cdots \rbrace$
on  RHS \linebreak \eqref{E:RETRYVORTICITYOFVORTICITYFIRSTEVOLUTION} over to the left-hand side, 
which yields the equation

\begin{align} \label{E:MAINTERMSWRITTENVORTICITYOFVORTICITYEVOLUTION}
		u^{\kappa} \partial_{\kappa} 
		& \Big\lbrace
			\uperpvort^{\alpha}(\vort)
			+
			\speed^{-2}
			\upepsilon^{\alpha \beta \gamma \delta} 
			u_{\beta}
			(\partial_{\gamma} \Lnenth) \vort_{\delta}
			+
			(\Temp - \Temp_{;\Lnenth}) 
			\GradEnt^{\alpha} 
			(\partial_{\lambda} u^{\lambda})
				\\
		& \ \ \ \ \ \ 
			+
			(\Temp - \Temp_{;\Lnenth}) 
			u^{\alpha} 
			(\GradEnt^{\lambda} \partial_{\lambda} \Lnenth)
			+
			(\Temp_{;\Lnenth} - \Temp) 
			(\upeta^{-1})^{\alpha \lambda} 
			\GradEnt^{\beta} 
			(\partial_{\lambda} u_{\beta})
		\Big\rbrace
			\notag \\
		& = 
			\uperpvort^{\kappa}(\vort) \partial_{\kappa} u^{\alpha}
			- 
			2 (\partial_{\kappa} u^{\kappa}) \uperpvort^{\alpha}(\vort)
			+
			u^{\alpha} 
			(u^{\kappa} \partial_{\kappa} u_{\beta}) 
			\uperpvort^{\beta}(\vort)
			+
			\cdots,
				\notag 
\end{align}
where the terms $\cdots$ do not involve $\uperpvort(\vort)$. 
Next, we solve for $\uperpvort(\vort)$ in terms
of the remaining terms in definition \eqref{E:MODIFIEDVORTICITYOFVORTICITY}
and then use the resulting identity to algebraically substitute for the four instances of $\uperpvort(\vort)$
in equation \eqref{E:MAINTERMSWRITTENVORTICITYOFVORTICITYEVOLUTION}
(note in particular 
that the terms in braces on LHS~\eqref{E:MAINTERMSWRITTENVORTICITYOFVORTICITYEVOLUTION} are equal to $\mathcal{C}^{\alpha}$).
In total, this yields the following equation, where we have placed the terms
generated by the algebraic substitution 
on the first through tenth lines of RHS~\eqref{E:SECONDSTEPVORTICITYOFVORTICITYFIRSTEVOLUTION}:
\begingroup
\allowdisplaybreaks
\begin{align} \label{E:SECONDSTEPVORTICITYOFVORTICITYFIRSTEVOLUTION}
		u^{\kappa} \partial_{\kappa} \mathcal{C}^{\alpha}
		& = 
			\mathcal{C}^{\kappa} \partial_{\kappa} u^{\alpha}
			- 
			2 (\partial_{\kappa} u^{\kappa}) \mathcal{C}^{\alpha}
			+
			u^{\alpha} 
			(u^{\kappa} \partial_{\kappa} u_{\beta}) 
			\mathcal{C}^{\beta}
				\\
		& \ \
			- 
			\speed^{-2}
			\upepsilon^{\kappa \beta \gamma \delta} 
			(\partial_{\kappa} u^{\alpha})
			u_{\beta}
			(\partial_{\gamma} \Lnenth) \vort_{\delta}
			+
			(\Temp_{;\Lnenth} - \Temp) 
			(\GradEnt^{\kappa} \partial_{\kappa} u^{\alpha}) 
			(\partial_{\lambda} u^{\lambda})
				\notag \\
		& \ \
			+
			(\Temp_{;\Lnenth} - \Temp) (u^{\kappa} \partial_{\kappa} u^{\alpha}) (\GradEnt^{\lambda} \partial_{\lambda} \Lnenth)
			+
			(\Temp - \Temp_{;\Lnenth}) (\upeta^{-1})^{\kappa \lambda} (\partial_{\kappa} u^{\alpha}) \GradEnt^{\beta} (\partial_{\lambda} u_{\beta})
			\notag \\
			& \ \
			+ 2 (\partial_{\kappa} u^{\kappa}) 
				\speed^{-2}
				\upepsilon^{\alpha \beta \gamma \delta} 
				u_{\beta}
				(\partial_{\gamma} \Lnenth) 
				\vort_{\delta}
		+
		2 
		(\Temp - \Temp_{;\Lnenth}) 
		\GradEnt^{\alpha}
		(\partial_{\kappa} u^{\kappa})
		(\partial_{\lambda} u^{\lambda})
			\notag \\
	& \ \
		+
		2 
		(\Temp - \Temp_{;\Lnenth}) u^{\alpha} 
		(\partial_{\kappa} u^{\kappa}) 
		(\GradEnt^{\lambda} \partial_{\lambda} \Lnenth)
		\notag
		\\
		& \ \
		+
		2 
		(\Temp_{;\Lnenth} - \Temp) 
		(\partial_{\kappa} u^{\kappa}) 
		\GradEnt^{\beta} ((\upeta^{-1})^{\alpha \lambda} \partial_{\lambda} u_{\beta})
		\notag \\
		& \ \
		-	
		u^{\alpha} 
		(u^{\kappa} \partial_{\kappa} u_{\sigma}) 
		\speed^{-2}
		\upepsilon^{\sigma \beta \gamma \delta} 
		u_{\beta}
		(\partial_{\gamma} \Lnenth) \vort_{\delta}
		\notag
		\\
		& \ \
		+
		(\Temp_{;\Lnenth} - \Temp) 
		u^{\alpha} 
		(u^{\kappa} \partial_{\kappa} u_{\sigma})
		 \GradEnt^{\sigma}
		(\partial_{\lambda} u^{\lambda})
			\notag \\
	& \ \
		+
		(\Temp_{;\Lnenth} - \Temp) 
		u^{\alpha} 
		(u^{\kappa} \partial_{\kappa} u_{\beta})
		u^{\beta} (\GradEnt^{\lambda} \partial_{\lambda} \Lnenth)
		\notag
		\\
		& \ \
		+
		(\Temp - \Temp_{;\Lnenth}) 
		u^{\alpha} 
		(u^{\kappa} \partial_{\kappa} u^{\lambda})
		\GradEnt^{\beta} (\partial_{\lambda} u_{\beta})
				\notag \\
		& \ \
		-
		2
		(\partial_{\kappa} u^{\kappa})
		\speed^{-2}
		\upepsilon^{\alpha \beta \gamma \delta} 
		u_{\beta}
		(\partial_{\gamma} \Lnenth) \vort_{\delta}
		\notag \\
& \ \
	+
		\speed^{-2}
		\upepsilon^{\alpha \beta \gamma \delta} 
		(u^{\kappa} \partial_{\kappa} u_{\beta})
		(\partial_{\gamma} \Lnenth) \vort_{\delta}
	\notag
		\\
& \ \
		+
		\speed^{-2}
		\upepsilon^{\alpha \beta \gamma \delta} 
		u_{\beta}
		(\partial_{\gamma} \Lnenth) 
		\vort^{\kappa} 
		(\partial_{\delta} u_{\kappa})
			\notag \\
	& \ \
	+
	\speed^{-2}
	(\Temp - \Temp_{;\Lnenth}) 
	(\GradEnt^{\kappa} \partial_{\kappa} \Lnenth)
	((\upeta^{-1})^{\alpha \lambda} \partial_{\lambda} \Lnenth) 
		\notag \\
& \ \
	+
	\speed^{-2}
	(\Temp - \Temp_{;\Lnenth}) 
	u^{\alpha}
	(\GradEnt^{\kappa} \partial_{\kappa} \Lnenth)
	(u^{\lambda} \partial_{\lambda} \Lnenth)
	\notag \\
& \ \
+ \speed^{-2}
	(\Temp_{;\Lnenth} - \Temp) 
	\GradEnt^{\alpha} 
	(u^{\kappa} \partial_{\kappa} \Lnenth)
	(u^{\lambda} \partial_{\lambda} \Lnenth)
	\notag \\
& \ \
	+	 
	\speed^{-2}
	(\Temp_{;\Lnenth} - \Temp) 
	\GradEnt^{\alpha}
	(\upeta^{-1})^{\kappa \lambda}
	(\partial_{\kappa} \Lnenth)
	(\partial_{\lambda} \Lnenth)
		\notag
		\\
	& \ \
	+
	\speed^{-2}
	\upepsilon^{\alpha \beta \gamma \delta} 
	u_{\beta}
	\left\lbrace
		(\partial_{\kappa} u^{\kappa}) (\partial_{\gamma} \Lnenth) 
		- 
		(\partial_{\gamma} u^{\kappa}) (\partial_{\kappa} \Lnenth) 
	\right\rbrace
	\vort_{\delta}	
		\notag
			\\
	& \ \
		+
	2 \speed^{-3} \speed_{;\Ent} 
	(u^{\kappa} \partial_{\kappa} \Lnenth)
	\upepsilon^{\alpha \beta \gamma \delta} 
	u_{\beta}
	\GradEnt_{\gamma}
	\vort_{\delta}
	\notag
		\notag \\
& \ \
		+ 
		\frac{1}{\Enth} (\vort^{\kappa} \partial_{\kappa} \vort^{\alpha})
		-
		\frac{1}{\Enth} \vort^{\alpha} (\vort^{\kappa} \partial_{\kappa} \Lnenth) 
			\notag \\
	&  \ \
		-
		\frac{1}{\Enth}
		u^{\alpha} \vort^{\lambda}
		(\vort^{\kappa} \partial_{\kappa} u_{\lambda})
		+
		\upepsilon^{\alpha \beta \gamma \delta} 
		u_{\beta}
		(\partial_{\gamma} \Lnenth)
		\vort^{\kappa} (\partial_{\delta} u_{\kappa})
				\notag
				\\
	& \ \
			-
			\TempoverEnth
			\upepsilon^{\alpha \beta \gamma \delta}
			u_{\beta}
			\GradEnt_{\gamma} 
			\vort^{\kappa} (\partial_{\delta} u_{\kappa})
			\notag
			\\
	& \ \
	-
	\upepsilon^{\alpha \beta \gamma \delta} 
	(\partial_{\beta} \Lnenth) 
	u_{\gamma} \vort^{\kappa} (\partial_{\delta} u_{\kappa}) 
	+
	(\partial_{\kappa} u^{\kappa})
	\upepsilon^{\alpha \beta \gamma \delta} 
	(\partial_{\beta} \Lnenth) 
	u_{\gamma}
	\vort_{\delta} 
		\notag \\
	& \ \
	+
	(\Temp - \Temp_{;\Lnenth})
	((\upeta^{-1})^{\alpha \kappa} \partial_{\kappa} \Lnenth) 
	(\GradEnt^{\lambda} \partial_{\lambda} \Lnenth) 
	+
  (\Temp - \Temp_{;\Lnenth})
	u^{\alpha}
	(u^{\kappa} \partial_{\kappa} \Lnenth)
	(\GradEnt^{\lambda} \partial_{\lambda} \Lnenth) 
	\notag \\
& \ \
	+	
	(\Temp_{;\Lnenth} - \Temp)
	\GradEnt^{\alpha} 
	(u^{\kappa} \partial_{\kappa} \Lnenth)
	(u^{\lambda} \partial_{\lambda} \Lnenth)
	+
	(\Temp_{;\Lnenth} - \Temp)
	\GradEnt^{\alpha} 
	(\upeta^{-1})^{\kappa \lambda}
	(\partial_{\kappa} \Lnenth) 
	(\partial_{\lambda} \Lnenth)
	\notag
		\\
& \ \
	+
	\TempoverEnth
	\upepsilon^{\alpha \beta \gamma \delta} 
	\GradEnt_{\beta}
	u_{\gamma}
	\vort^{\kappa} (\partial_{\delta} u_{\kappa})
		\notag
			\\
	& \ \
	-
	\TempoverEnth
	(\partial_{\kappa} u^{\kappa})
	\upepsilon^{\alpha \beta \gamma \delta} 
	\GradEnt_{\beta}
	u_{\gamma}
	\vort_{\delta} 
		\notag \\
& \ \
	+
	\TempoverEnth
	(\Temp_{;\Lnenth} - \Temp) 
	((\upeta^{-1})^{\kappa \alpha} \partial_{\kappa} \Lnenth) \GradEnt^{\lambda} \GradEnt_{\lambda}
	+
	\TempoverEnth
	(\Temp_{;\Lnenth} - \Temp) 
	u^{\alpha}
 (u^{\kappa} \partial_{\kappa} \Lnenth) \GradEnt^{\lambda} \GradEnt_{\lambda} 
		\notag
		\\
		& \ \
	+
	\TempoverEnth
	(\Temp - \Temp_{;\Lnenth}) 
	\GradEnt^{\alpha}
	(\GradEnt^{\kappa} \partial_{\kappa} \Lnenth)
	\notag
		\\
& \ \
	-
	\upepsilon^{\alpha \beta \gamma \delta} (\partial_{\beta} \Lnenth) 
	u_{\gamma}
	\vort^{\lambda} (\partial_{\delta} u_{\lambda})
	+
	\TempoverEnth
	\upepsilon^{\alpha \beta \gamma \delta} \GradEnt_{\beta} 
	u_{\gamma}
	\vort^{\lambda} (\partial_{\delta} u_{\lambda})
		\notag \\
& \ \
			-
			\upepsilon^{\alpha \beta \gamma \delta} 
			u_{\beta}
			(\partial_{\gamma} \vort^{\kappa}) (\partial_{\delta} u_{\kappa})
			\notag	\\
& \ \
		-
		\upepsilon^{\alpha \beta \gamma \delta} 
		u_{\beta}
		(\partial_{\gamma} \vort^{\kappa}) (\partial_{\delta} u_{\kappa})
		-
		\upepsilon^{\alpha \beta \gamma \delta} 
		u_{\beta}
		(\partial_{\gamma} \Lnenth)
		\vort^{\kappa}
		(\partial_{\delta} u_{\kappa})
		\notag	\\
& \ \
			-
			\frac{1}{\Enth} 
			(\vort^{\kappa} \partial_{\kappa} \vort^{\alpha}) 
			-
			\frac{1}{\Enth} 
			\vort^{\alpha}
			(\vort^{\kappa} \partial_{\kappa} \Lnenth)
			+
			\frac{2 \TempoverEnth}{\Enth} 
			\vort^{\alpha}
			\vort^{\kappa} \GradEnt_{\kappa}
			\notag \\
& \ \
		-
		\frac{1}{\Enth} 
		\vort^{\alpha}
		\vort^{\lambda}
		(u^{\kappa} \partial_{\kappa} u_{\lambda}) 
		+
		\frac{1}{\Enth} 
		u^{\alpha} 
		\vort^{\lambda}
		(\vort^{\kappa} \partial_{\kappa} u_{\lambda})
		\notag		\\
	& \ \
		-
		\TempoverEnth
		\upepsilon^{\alpha \beta \gamma \delta} 
		u_{\beta}
		(\partial_{\gamma} u^{\kappa})
		\vort_{\kappa}
		\GradEnt_{\delta}
		\notag
		\\
	& \ \
		+
		\frac{1}{\Enth}
		\vort^{\alpha}
		\vort^{\lambda}
		(u^{\kappa} \partial_{\kappa} u_{\lambda})
			\notag \\
	& \ \
			+
			(\Temp_{\Lnenth} - \Temp_{;\Lnenth;\Lnenth})
			\GradEnt^{\alpha}  
			(\upeta^{-1})^{\kappa \lambda} 
			(\partial_{\kappa} \Lnenth) 
			(\partial_{\lambda} \Lnenth)
		\notag
		\\
		& \ \
			+
			(\Temp_{\Lnenth} - \Temp_{;\Lnenth;\Lnenth})
			 \GradEnt^{\alpha}
			(u^{\kappa} \partial_{\kappa} \Lnenth)
			(u^{\lambda} \partial_{\lambda} \Lnenth)
			\notag
			\\
	& \ \
			+
			(\Temp_{;\Lnenth;\Lnenth} - \Temp_{\Lnenth})
			u^{\alpha} 
			(\GradEnt^{\kappa} \partial_{\kappa} \Lnenth)
			(u^{\lambda} \partial_{\lambda} \Lnenth) 
		\notag
		\\
		& \ \
			+
			(\Temp_{;\Lnenth;\Lnenth} - \Temp_{\Lnenth})
			((\upeta^{-1})^{\alpha \kappa}\partial_{\kappa} \Lnenth) 
			(\GradEnt^{\lambda} \partial_{\lambda} \Lnenth)
			\notag
			\\
	& \ \
			+
			(\Temp_{;\Ent} - \Temp_{;\Lnenth;\Ent})
			\GradEnt^{\alpha} 
			(\GradEnt^{\kappa} \partial_{\kappa} \Lnenth) 
			+
			(\Temp_{;\Lnenth;\Ent} - \Temp_{;\Ent})
			u^{\alpha} 
			\GradEnt_{\kappa} \GradEnt^{\kappa}
			(u^{\lambda} \partial_{\lambda} \Lnenth) 
		\notag
		\\
		& \ \
			+
			(\Temp_{;\Lnenth;\Ent} - \Temp_{;\Ent})
			\GradEnt_{\kappa} \GradEnt^{\kappa}
			((\upeta^{-1})^{\alpha \lambda} \partial_{\lambda} \Lnenth) 
				\notag
				\\
		& \ \
			+
			(\Temp - \Temp_{;\Lnenth}) 
			\GradEnt^{\alpha}
			(\partial_{\kappa} u^{\kappa})
			(u^{\lambda} \partial_{\lambda} \Lnenth) 
			+
			(\Temp_{;\Lnenth} - \Temp) 
			(\GradEnt^{\kappa} \partial_{\kappa} u^{\alpha})
			(u^{\lambda} \partial_{\lambda} \Lnenth) 
				\notag 
			\\
	& \ \
		+ 
		(\Temp_{\Lnenth} - \Temp_{;\Lnenth;\Lnenth})
		\GradEnt^{\alpha}
		(u^{\kappa} \partial_{\kappa} \Lnenth)
		(\partial_{\lambda} u^{\lambda})
		+
		(\Temp - \Temp_{;\Lnenth}) 
		(u^{\kappa} \partial_{\kappa} \GradEnt^{\alpha})
		(\partial_{\lambda} u^{\lambda})
		\notag \\
	& \ \
		+
		(\Temp_{;\Lnenth} - \Temp) 
		\GradEnt^{\alpha} 
		(\partial_{\kappa} u^{\lambda}) 
		(\partial_{\lambda} u^{\kappa})
		+
		(\Temp_{;\Lnenth} - \Temp) 
		\GradEnt^{\alpha} 
		(u^{\kappa} \partial_{\kappa} u^{\lambda}) 
		(\partial_{\lambda} \Lnenth)
		\notag
		\\
		& \ \
		+
		(\Temp_{;\Lnenth} - \Temp) 
		\GradEnt^{\alpha} 
		(\partial_{\kappa} u^{\kappa}) 
		(u^{\lambda} \partial_{\lambda} \Lnenth)
		\notag
			\\
	& \ \
		+
		(\Temp - \Temp_{;\Lnenth}) 
		\TempoverEnth 
		\GradEnt^{\alpha} 
		(\partial_{\kappa} \GradEnt^{\kappa})
		+
		(\Temp - \Temp_{;\Lnenth}) 
		\TempoverEnth_{;\Lnenth} 
		\GradEnt^{\alpha}
		(\GradEnt^{\kappa} \partial_{\kappa} \Lnenth)
		\notag
		\\
		& \ \
		+
		(\Temp - \Temp_{;\Lnenth}) 
		\TempoverEnth_{;\Ent} 
		\GradEnt^{\alpha}	
		\GradEnt_{\kappa} \GradEnt^{\kappa}
			\notag \\
		& \ \
			+
			(\Temp_{\Lnenth} - \Temp_{;\Lnenth;\Lnenth})
			u^{\alpha}
			(u^{\kappa} \partial_{\kappa} \Lnenth)
			(\GradEnt^{\lambda} \partial_{\lambda} \Lnenth)
			+
			(\Temp - \Temp_{;\Lnenth}) 
			(u^{\kappa} \partial_{\kappa} u^{\alpha}) 
			(\GradEnt^{\lambda} \partial_{\lambda} \Lnenth)
		\notag
		\\
		& \ \
			+
			(\Temp - \Temp_{;\Lnenth})
			u^{\alpha} 
			(u^{\kappa} \partial_{\kappa} \GradEnt^{\lambda}) 
			(\partial_{\lambda} \Lnenth)
			\notag
				\\
		& \ \
		+
		(\Temp_{;\Lnenth;\Lnenth} - \Temp_{\Lnenth}) 
		(u^{\kappa} \partial_{\kappa} \Lnenth)
		\GradEnt^{\beta} ((\upeta^{-1})^{\alpha \lambda} \partial_{\lambda} u_{\beta})
		\notag
		\\
		& \ \
		+
		(\Temp_{;\Lnenth} - \Temp)
		(u^{\kappa} \partial_{\kappa} \GradEnt^{\beta})
		((\upeta^{-1})^{\alpha \lambda} \partial_{\lambda} u_{\beta})
		\notag \\
	& \ \
		+
		(\Temp - \Temp_{;\Lnenth}) \GradEnt^{\beta} ((\upeta^{-1})^{\alpha \lambda} \partial_{\lambda} u_{\beta}) (u^{\kappa} \partial_{\kappa} \Lnenth)
		\notag
		\\
		& \ \
		+ 
		(\Temp - \Temp_{;\Lnenth}) \GradEnt^{\beta} ((\upeta^{-1})^{\alpha \lambda} \partial_{\lambda} u^{\kappa}) (\partial_{\kappa} u_{\beta})
		\notag \\
	& \ \
	+
	(\Temp - \Temp_{;\Lnenth}) \TempoverEnth ((\upeta^{-1})^{\alpha \lambda} \partial_{\lambda} \GradEnt^{\beta}) \GradEnt_{\beta}
	+
	(\Temp_{;\Lnenth} - \Temp) \TempoverEnth_{;\Lnenth} ((\upeta^{-1})^{\alpha \lambda} \partial_{\lambda} \Lnenth) \GradEnt_{\kappa} \GradEnt^{\kappa}
		\notag \\
&  \ \
	+
	(\Temp_{;\Lnenth} - \Temp) \TempoverEnth_{;\Ent} \GradEnt^{\alpha} \GradEnt_{\kappa} \GradEnt^{\kappa}
	+
	2 (\Temp_{;\Lnenth} - \Temp) \TempoverEnth ((\upeta^{-1})^{\alpha \lambda} \partial_{\lambda} \GradEnt^{\kappa}) \GradEnt_{\kappa}
	\notag	
		\\
		& \ \
			+
			(\Temp - \Temp_{;\Lnenth}) 
			(\upeta^{-1})^{\kappa \lambda}
			(\partial_{\kappa} \Lnenth) 
			(\partial_{\lambda}\GradEnt^{\alpha})
			+
			(\Temp - \Temp_{;\Lnenth}) 
			(u^{\kappa} \partial_{\kappa} \Lnenth) 
			(u^{\lambda} \partial_{\lambda} \GradEnt^{\alpha})
			\notag 
			\\
		& \ \
			+
			(\Temp_{;\Lnenth} - \Temp) 
			u^{\alpha} 
			(u^{\kappa} \partial_{\kappa} \Lnenth) 
			(\partial_{\lambda} \GradEnt^{\lambda})
			+
			(\Temp_{;\Lnenth} - \Temp) 
			((\upeta^{-1})^{\alpha \kappa} \partial_{\kappa} \Lnenth) 
			(\partial_{\lambda} \GradEnt^{\lambda})
			\notag 
			\\
		& \ \
			+
		(\Temp_{;\Lnenth} - \Temp) 
		 ((\upeta^{-1})^{\alpha \kappa} \partial_{\kappa} \Lnenth) u_{\beta} (u^{\lambda} \partial_{\lambda} \GradEnt^{\beta})
		\notag
		\\
		& \ \
			+
			(\Temp - \Temp_{;\Lnenth}) 
			u^{\alpha}
			(\upeta^{-1})^{\kappa \lambda}
			 (\partial_{\kappa} \Lnenth) u_{\beta} (\partial_{\lambda} \GradEnt^{\beta}).
			\notag 
\end{align}
\endgroup

Next, we reorganize the terms on RHS~\eqref{E:SECONDSTEPVORTICITYOFVORTICITYFIRSTEVOLUTION}
to obtain the equation
\begin{align} \label{E:REORGANIZEDVORTICITYTRANSPORT}
	u^{\kappa} \partial_{\kappa} \mathcal{C}^{\alpha}
		& = 
			\mathcal{C}^{\kappa} \partial_{\kappa} u^{\alpha}
			- 
			2 (\partial_{\kappa} u^{\kappa}) \mathcal{C}^{\alpha}
			+
			u^{\alpha} 
			(u^{\kappa} \partial_{\kappa} u_{\beta}) 
			\mathcal{C}^{\beta}
			+
			\sum_{i=1}^{21} \mathscr{Q}_i
			+
			\mathscr{L},
\end{align}
where
\begingroup
\allowdisplaybreaks
\begin{align} \label{E:Q1DEF}
	\mathscr{Q}_1
	& :=
			-
			2 \upepsilon^{\alpha \beta \gamma \delta} 
			u_{\beta}
			(\partial_{\gamma} \vort^{\kappa}) (\partial_{\delta} u_{\kappa}),
				\\
	\mathscr{Q}_2
	& := 
		(\Temp_{;\Lnenth} - \Temp) 
		((\upeta^{-1})^{\alpha \kappa} \partial_{\kappa} \Lnenth) \partial_{\lambda} \GradEnt^{\lambda}
	+
	(\Temp - \Temp_{;\Lnenth}) 
	(\upeta^{-1})^{\kappa \lambda}
	(\partial_{\kappa} \Lnenth) (\partial_{\lambda}\GradEnt^{\alpha}),
		\label{E:Q2DEF} \\
\mathscr{Q}_3
& := 
	(\Temp_{;\Lnenth} - \Temp) 
	u^{\alpha} 
	(u^{\kappa} \partial_{\kappa} \Lnenth) (\partial_{\lambda} \GradEnt^{\lambda})
	+
	(\Temp - \Temp_{;\Lnenth})
	u^{\alpha} (u^{\kappa} \partial_{\kappa} \GradEnt^{\lambda}) (\partial_{\lambda} \Lnenth),
	 \label{E:Q3DEF} \\
\mathscr{Q}_4
	& : = 
			(\Temp_{\Lnenth} - \Temp_{;\Lnenth;\Lnenth})
			\GradEnt^{\alpha}  
			(\upeta^{-1})^{\kappa \lambda} (\partial_{\kappa} \Lnenth)
			(\partial_{\lambda} \Lnenth) 
				\label{E:Q4DEF} \\
				& \ \
			+
			(\Temp_{\Lnenth} - \Temp_{;\Lnenth;\Lnenth})
			 \GradEnt^{\alpha}
			(u^{\kappa} \partial_{\kappa} \Lnenth)
			(u^{\lambda} \partial_{\lambda} \Lnenth)
			+ 
			(\Temp_{\Lnenth} - \Temp_{;\Lnenth;\Lnenth})
			\GradEnt^{\alpha}
			(u^{\kappa} \partial_{\kappa} \Lnenth)
			(\partial_{\lambda} u^{\lambda}),
			\notag
			\\
	\mathscr{Q}_5
	& : =
			(\Temp_{;\Lnenth;\Lnenth} - \Temp_{\Lnenth})
			(\upeta^{-1})^{\alpha \kappa}
			(\partial_{\kappa} \Lnenth) 
			(\GradEnt^{\lambda} \partial_{\lambda} \Lnenth)
				\label{E:Q5DEF} \\
			& \ \
				+
			(\Temp_{;\Lnenth;\Lnenth} - \Temp_{\Lnenth}) 
			(u^{\kappa} \partial_{\kappa} \Lnenth)
			\GradEnt^{\beta} ((\upeta^{-1})^{\alpha \lambda} \partial_{\lambda} u_{\beta}),
			\notag 
			\\
	\mathscr{Q}_6
	& :=
	\speed^{-2}
	(\Temp_{;\Lnenth} - \Temp) 
	\GradEnt^{\alpha} 
	(u^{\kappa} \partial_{\kappa} \Lnenth)
	(u^{\lambda} \partial_{\lambda} \Lnenth)
		\label{E:Q6DEF} \\
	& \ \
	+	 
	\speed^{-2}
	(\Temp_{;\Lnenth} - \Temp) 
	\GradEnt^{\alpha}
	(\upeta^{-1})^{\kappa \lambda}
	(\partial_{\kappa} \Lnenth)
	(\partial_{\lambda} \Lnenth)
	+
	(\Temp - \Temp_{;\Lnenth}) 
	\GradEnt^{\alpha} 
	(\partial_{\kappa} u^{\kappa})
	(\partial_{\lambda} u^{\lambda}),
	\notag
	\\
\mathscr{Q}_7
&:=
	(\Temp_{;\Lnenth} - \Temp)
	\GradEnt^{\alpha} 
	(u^{\kappa} \partial_{\kappa} \Lnenth)
	(u^{\lambda} \partial_{\lambda} \Lnenth)
		\label{E:Q7DEF} \\
		& \ \
	+
	(\Temp_{;\Lnenth} - \Temp)
	\GradEnt^{\alpha} 
	(\upeta^{-1})^{\kappa \lambda}
	(\partial_{\kappa} \Lnenth) 
	(\partial_{\lambda} \Lnenth)
	+
	(\Temp_{;\Lnenth} - \Temp) 
	\GradEnt^{\alpha} 
	(u^{\kappa} \partial_{\kappa} u^{\lambda}) (\partial_{\lambda} \Lnenth),
	\notag
	\\
\mathscr{Q}_8
	& :=
		(\Temp_{;\Lnenth} - \Temp) 
		\GradEnt^{\alpha} 
		(\partial_{\kappa} u^{\lambda}) (\partial_{\lambda} u^{\kappa}) 
			+
			(\Temp - \Temp_{;\Lnenth}) 
			\GradEnt^{\alpha} 
			(\partial_{\kappa} u^{\kappa})
			(\partial_{\lambda} u^{\lambda}),
				\label{E:Q8DEF} \\
\mathscr{Q}_9
& := 
	(\Temp_{;\Lnenth} - \Temp) 
			(\GradEnt^{\kappa} \partial_{\kappa} u^{\alpha})
			(\partial_{\lambda} u^{\lambda})
			+
			(\Temp_{;\Lnenth} - \Temp) 
			(u^{\kappa} \partial_{\kappa} u^{\alpha})
			(\GradEnt^{\lambda} \partial_{\lambda} \Lnenth)
				\label{E:Q9DEF} \\
				& \ \
			+
			(\Temp - \Temp_{;\Lnenth}) 
			((\upeta^{-1})^{\kappa \lambda} \partial_{\kappa} u^{\alpha})
			\GradEnt^{\beta} 
			(\partial_{\lambda} u_{\beta}),
	\notag
	\\
\mathscr{Q}_{10}
& :=
	(\Temp_{;\Lnenth} - \Temp) 
	(\GradEnt^{\kappa} \partial_{\kappa} u^{\alpha})
	(u^{\lambda} \partial_{\lambda} \Lnenth) 
	+
	(\Temp - \Temp_{;\Lnenth})
	(u^{\kappa} \partial_{\kappa} u^{\alpha}) (\GradEnt^{\lambda} \partial_{\lambda} \Lnenth),
	  \label{E:Q10DEF} \\
\mathscr{Q}_{11}
& := 
		(\Temp - \Temp_{;\Lnenth}) 
		(u^{\kappa} \partial_{\kappa} \GradEnt^{\alpha})
		(\partial_{\lambda} u^{\lambda})
			 \label{E:Q11DEF} \\
	& \ \
		+
		(\Temp - \Temp_{;\Lnenth}) \GradEnt^{\beta} ((\upeta^{-1})^{\alpha \lambda} \partial_{\lambda} u^{\kappa}) (\partial_{\kappa} u_{\beta}),
	\notag
	\\
\mathscr{Q}_{12}
& := 
			2
			(\Temp_{;\Lnenth} - \Temp) 
			(\partial_{\kappa} u^{\kappa})
			(\upeta^{-1})^{\alpha \lambda} 
			\GradEnt^{\beta} 
			(\partial_{\lambda} u_{\beta})
				\label{E:Q12DEF} \\
	& \ \
		+
		\speed^{-2}
		(\Temp - \Temp_{;\Lnenth}) 
		(\GradEnt^{\kappa} \partial_{\kappa} \Lnenth)
		((\upeta^{-1})^{\alpha \lambda} \partial_{\lambda} \Lnenth) 
		\notag
		\\
		& \ \
		+
		(\Temp_{;\Lnenth} - \Temp)
		(u^{\kappa} \partial_{\kappa} \GradEnt^{\beta})
		((\upeta^{-1})^{\alpha \lambda} \partial_{\lambda} u_{\beta}),
		\notag
			\\
\mathscr{Q}_{13}
& :=
		  (\Temp - \Temp_{;\Lnenth}) 
			u^{\alpha} 
			(\partial_{\kappa} u^{\kappa})
			(\GradEnt^{\lambda} \partial_{\lambda} \Lnenth)
			+
			(\Temp - \Temp_{;\Lnenth}) 
			u^{\alpha} 
			(u^{\kappa} \partial_{\kappa} u^{\lambda}) 
			\GradEnt^{\beta} 
			(\partial_{\lambda} u_{\beta}),
				\label{E:Q13DEF} \\
\mathscr{Q}_{14}
& := 	(\Temp - \Temp_{;\Lnenth}) 
			u^{\alpha} 
			(\partial_{\kappa} u^{\kappa})
			(\GradEnt^{\lambda} \partial_{\lambda} \Lnenth)
			+
			(\Temp - \Temp_{;\Lnenth}) 
			u^{\alpha} 
		  (u^{\kappa} \partial_{\kappa} \Lnenth)
			(\GradEnt^{\lambda} \partial_{\lambda} \Lnenth)
				\label{E:Q14DEF} \\
		& \ \
			+
			(\Temp - \Temp_{;\Lnenth}) 
			u^{\alpha}
			(\upeta^{-1})^{\kappa \lambda}
			(\partial_{\kappa} \Lnenth) u_{\beta} (\partial_{\lambda} \GradEnt^{\beta}),
	\notag
	\\
	\mathscr{Q}_{15}	
	& :=
		(\partial_{\kappa} u^{\kappa})
		\upepsilon^{\alpha \beta \gamma \delta} 
		(\partial_{\beta} \Lnenth) 
		u_{\gamma}
		\vort_{\delta} 
		+
	  \speed^{-2}
		\upepsilon^{\alpha \beta \gamma \delta} 
		(u^{\kappa} \partial_{\kappa} u_{\beta})
		(\partial_{\gamma} \Lnenth) \vort_{\delta},
			\label{E:Q15DEF} \\
	\mathscr{Q}_{16}
	& :=
			-
			\speed^{-2}
			u^{\alpha}
			\upepsilon^{\sigma \beta \gamma \delta} 
			(u^{\kappa} \partial_{\kappa} u_{\sigma})
			u_{\beta}
			(\partial_{\gamma} \Lnenth) \vort_{\delta},
			\label{E:Q16DEF} \\
	\mathscr{Q}_{17}
	& := 
		-
			\speed^{-2}
			\upepsilon^{\kappa \beta \gamma \delta} 
			(\partial_{\kappa} u^{\alpha})
			u_{\beta}
			(\partial_{\gamma} \Lnenth) \vort_{\delta}
			\label{E:Q17DEF} \\
	& \ \
		+
		(\speed^{-2} + 2)
		\upepsilon^{\alpha \beta \gamma \delta} 
		u_{\beta}
		(\partial_{\gamma} \Lnenth) \vort^{\kappa} (\partial_{\delta} u_{\kappa})
			\notag \\
	& \ \
	+
	\speed^{-2}
	\upepsilon^{\alpha \beta \gamma \delta} 
	u_{\beta}
	\vort_{\delta}
	\left\lbrace
		(\partial_{\kappa} u^{\kappa}) (\partial_{\gamma} \Lnenth)
		- 
		(\partial_{\gamma} u^{\kappa}) (\partial_{\kappa} \Lnenth)
	\right\rbrace,
	\notag
		\\
	\mathscr{Q}_{18}
	& := 
	(\Temp - \Temp_{;\Lnenth})
	((\upeta^{-1})^{\alpha \kappa} \partial_{\kappa} \Lnenth) 
	(\GradEnt^{\lambda} \partial_{\lambda} \Lnenth) 
		\label{E:Q18DEF} \\
	& \ \
	+
	(\Temp_{;\Lnenth} - \Temp) 
	((\upeta^{-1})^{\alpha \kappa} \partial_{\kappa} \Lnenth) u_{\beta} (u^{\lambda} \partial_{\lambda} \GradEnt^{\beta}),
	\notag
	\\
	\mathscr{Q}_{19}
	& := 
		(\Temp_{;\Lnenth} - \Temp) 
			u^{\alpha} 
			(u^{\kappa} \partial_{\kappa} u_{\sigma})
			\GradEnt^{\sigma} 
			(\partial_{\lambda} u^{\lambda})
				\label{E:Q19DEF} \\
	& \ \
			+
			\speed^{-2}
			(\Temp - \Temp_{;\Lnenth}) 
			u^{\alpha}
			(\GradEnt^{\kappa} \partial_{\kappa} \Lnenth)
			(u^{\lambda} \partial_{\lambda} \Lnenth),
	\notag
	\\
	\mathscr{Q}_{20}
	& := 
		(\Temp - \Temp_{;\Lnenth}) \GradEnt^{\beta} ((\upeta^{-1})^{\alpha \lambda} \partial_{\lambda} u_{\beta}) (u^{\kappa} \partial_{\kappa} \Lnenth)
		\label{E:Q20DEF}\\
		& \ \
		+
		(\Temp - \Temp_{;\Lnenth}) 
		(u^{\kappa} \partial_{\kappa} \Lnenth) (u^{\lambda} \partial_{\lambda} \GradEnt^{\alpha}),
	\notag
	\\
	\mathscr{Q}_{21}
	& := (\Temp_{;\Lnenth} - \Temp) 
		u^{\alpha} 
		(u^{\kappa} \partial_{\kappa} u_{\beta})
		(u^{\beta} \GradEnt^{\lambda} \partial_{\lambda} \Lnenth),
		\label{E:Q21DEF}
\end{align}
and
\begin{align} \label{E:VORTOFVORTFIRSTLINEARTERMDEF}
	\mathscr{L} 
	& := 
		-
		\frac{2}{\Enth} 
		\vort^{\alpha}
		(\vort^{\kappa} \partial_{\kappa} \Lnenth)
		+
		\frac{2 \TempoverEnth}{\Enth} 
		\vort^{\alpha}
		\vort^{\kappa} \GradEnt_{\kappa}
			\\
	& \ \
		+
		2 \speed^{-3} \speed_{;\Ent} 
		(u^{\kappa} \partial_{\kappa} \Lnenth)
		\upepsilon^{\alpha \beta \gamma \delta} 
		u_{\beta}
		\GradEnt_{\gamma}
		\vort_{\delta}
			-
			\TempoverEnth
			\upepsilon^{\alpha \beta \gamma \delta}
			u_{\beta}
			\GradEnt_{\gamma} 
			\vort^{\kappa} (\partial_{\delta} u_{\kappa})
			\notag
			\\
	& \ \
	+
	2 \TempoverEnth
	\upepsilon^{\alpha \beta \gamma \delta} 
	\GradEnt_{\beta}
	u_{\gamma}
	\vort^{\kappa} (\partial_{\delta} u_{\kappa})
		\notag
			\\
	& \ \
	-
	\TempoverEnth
	(\partial_{\kappa} u^{\kappa})
	\upepsilon^{\alpha \beta \gamma \delta} 
	\GradEnt_{\beta}
	u_{\gamma}
	\vort_{\delta} 
		\notag \\
& \ \
	+
	\TempoverEnth
	(\Temp_{;\Lnenth} - \Temp) 
	((\upeta^{-1})^{\kappa \alpha} \partial_{\kappa} \Lnenth) \GradEnt^{\lambda} \GradEnt_{\lambda}
	+
	\TempoverEnth
	(\Temp_{;\Lnenth} - \Temp) 
	u^{\alpha}
 (u^{\kappa} \partial_{\kappa} \Lnenth) \GradEnt^{\lambda} \GradEnt_{\lambda} 
	\notag
	\\
	& \ \
	+
	\TempoverEnth
	(\Temp - \Temp_{;\Lnenth}) 
	\GradEnt^{\alpha}
	(\GradEnt^{\kappa} \partial_{\kappa} \Lnenth)
	\notag
		\\
	& \ \
		 -
		\TempoverEnth
		\upepsilon^{\alpha \beta \gamma \delta} 
		u_{\beta}
		(\partial_{\gamma} u^{\kappa})
		\vort_{\kappa}
		\GradEnt_{\delta}
		\notag
			\\
& \ \
			+
			(\Temp_{;\Ent} - \Temp_{;\Lnenth;\Ent})
			\GradEnt^{\alpha} 
			(\GradEnt^{\kappa} \partial_{\kappa} \Lnenth) 
			+
			(\Temp_{;\Lnenth;\Ent} - \Temp_{;\Ent})
			u^{\alpha} 
			\GradEnt_{\kappa} \GradEnt^{\kappa}
			(u^{\lambda} \partial_{\lambda} \Lnenth) 
			\notag
			\\
			& \ \
			+
			(\Temp_{;\Lnenth;\Ent} - \Temp_{;\Ent})
			\GradEnt_{\kappa} \GradEnt^{\kappa}
			((\upeta^{-1})^{\alpha \lambda} \partial_{\lambda} \Lnenth) 
				\notag \\
	& \ \
		+
		(\Temp - \Temp_{;\Lnenth}) 
		\TempoverEnth 
		\GradEnt^{\alpha} 
		(\partial_{\kappa} \GradEnt^{\kappa})
		+
		(\Temp - \Temp_{;\Lnenth}) 
		\TempoverEnth_{;\Lnenth} 
		\GradEnt^{\alpha}
		(\GradEnt^{\kappa} \partial_{\kappa} \Lnenth)
		+
		(\Temp - \Temp_{;\Lnenth}) 
		\TempoverEnth_{;\Ent} 
		\GradEnt^{\alpha}	
		\GradEnt_{\kappa} \GradEnt^{\kappa}
		\notag				
			\\
	& \ \
	+
	(\Temp - \Temp_{;\Lnenth}) \TempoverEnth ((\upeta^{-1})^{\alpha \lambda} \partial_{\lambda} \GradEnt^{\beta}) \GradEnt_{\beta}
	+
	(\Temp_{;\Lnenth} - \Temp) \TempoverEnth_{;\Lnenth} ((\upeta^{-1})^{\alpha \lambda} \partial_{\lambda} \Lnenth) \GradEnt_{\kappa} \GradEnt^{\kappa}
		\notag \\
&  \ \
	+
	(\Temp_{;\Lnenth} - \Temp) \TempoverEnth_{;\Ent} \GradEnt^{\alpha} \GradEnt_{\kappa} \GradEnt^{\kappa}
	+
	2 (\Temp_{;\Lnenth} - \Temp) \TempoverEnth ((\upeta^{-1})^{\alpha \lambda} \partial_{\lambda} \GradEnt^{\kappa}) \GradEnt_{\kappa}.
	\notag
\end{align}
\endgroup

Note that the terms on RHSs~\eqref{E:Q1DEF}-\eqref{E:Q21DEF} are precisely quadratic in the
first-order derivatives of the solution variables
$(\Lnenth,u^{\alpha},\vort^{\alpha},\GradEnt^{\alpha})_{\alpha=0,1,2,3}$ 
while the terms on RHS~\eqref{E:VORTOFVORTFIRSTLINEARTERMDEF} are at most linear
in the derivatives of the solution variables. 
We will now show that $\mathscr{Q}_1$, $\mathscr{Q}_2$, $\cdots$, $\mathscr{Q}_{21}$
can be expressed as null forms 
or terms that are at most linear
in the derivatives of the solution variables. 
To this end, we simply use
\eqref{E:Q2ID}-\eqref{E:Q21ID}
to algebraically substitute for
$\mathscr{Q}_2$,
$\mathscr{Q}_4$,
$\mathscr{Q}_5$,
$\mathscr{Q}_6$,
$\mathscr{Q}_7$,
$\mathscr{Q}_9$,
$\mathscr{Q}_{11}$,
$\mathscr{Q}_{12}$,
$\mathscr{Q}_{13}$,
$\mathscr{Q}_{14}$,
$\mathscr{Q}_{15}$,
$\mathscr{Q}_{16}$,
$\mathscr{Q}_{18}$,
$\mathscr{Q}_{19}$,
$\mathscr{Q}_{20}$,
and
$\mathscr{Q}_{21}$ 
(we do not substitute for
$\mathscr{Q}_1$,
$\mathscr{Q}_3$,
$\mathscr{Q}_8$,
$\mathscr{Q}_{10}$,
and
$\mathscr{Q}_{17}$
since these terms are already manifestly linear combinations of null forms).
Following this substitution, there are only two kinds of
terms on RHS~\eqref{E:REORGANIZEDVORTICITYTRANSPORT}:
null forms and terms that 
are at most linear
in the derivatives of the solution variables.
We now place all null forms on RHS~\eqref{E:VORTICITYOFVORTICITYEASYNULLFORMS}
except for null forms that involve the derivatives of
$\vort$
or
$\GradEnt$; these null forms we place directly on RHS~\eqref{E:TRANSPORTFORMODIFIEDVORTICITYOFVORTICITY}.
We then place all terms that are linear in
$
\mathcal{C}
$,
linear in
$
\mathcal{D}
$,
linear in the first-order derivatives of $\vort$,
or
linear in the first-order derivatives of $\GradEnt$
directly on RHS~\eqref{E:TRANSPORTFORMODIFIEDVORTICITYOFVORTICITY}.
Finally, we place all remaining terms,
which are at most linear
in the derivatives of the solution variables
and do not depend on the derivatives of
$\vort$ or $\GradEnt$, on RHS~\eqref{E:VORTICITYOFVORTICITYLINEARTERMS}.
This completes the proof of the proposition.

\end{proof}

\section{Local well-posedness with additional regularity for the vorticity and entropy}
\label{S:WELLPOSEDNESS}
Our main goal in this section is to prove Theorem~\ref{T:UPGRADEDLOCALWELLPOSEDNESS}, which
is a local well-posedness result for the relativistic Euler equations
based on our new formulation of the equations, that is, based on the equations of Theorem~\ref{T:NEWFORMULATIONRELEULER}. 
The main new feature of
Theorem~\ref{T:UPGRADEDLOCALWELLPOSEDNESS} compared to standard local well-posedness results for the
relativistic Euler equations
(see Theorem~\ref{T:STANDLOCALWELLPOSEDNESS} for a statement of standard local well-posedness)
is that it yields an extra degree of differentiability for the vorticity
and the entropy, assuming that the initial vorticity and entropy enjoy the same extra differentiability.
We stress that this gain in regularity holds even though the logarithmic enthalpy and four-velocity
do not generally enjoy the same gain. As we described in Subsect.\,\ref{SS:SHOCKWAVES}, 
this extra regularity for the vorticity and the entropy is essential for the study of shocks 
in more than one spatial dimension.

For convenience, instead of proving local well-posedness for the relativistic Euler equations on 
the standard Minkowski spacetime background, 
we instead consider the spacetime background $(\mathbb{R} \times \mathbb{T}^3,\upeta)$,
where the ``spatial manifold'' $\mathbb{T}^3$ is the standard three-dimensional torus and, 
relative to standard coordinates on $\mathbb{R} \times \mathbb{T}^3$, 
$\upeta_{\alpha \beta} := \mbox{\upshape diag}(-1,1,1,1)$ is the standard Minkowski metric.
Thus, strictly speaking, in this section, $\upeta$ denotes a tensor on a different manifold
compared to the rest of the paper, but this minor change has no substantial bearing on the
discussion. In particular, the relativistic Euler equations on
$(\mathbb{R} \times \mathbb{T}^3,\upeta)$ take the same form that they take in
Theorem~\ref{T:NEWFORMULATIONRELEULER}. The advantage of the compact spatial topology is 
that it allows for a simplified approach to some technical aspects of the proof of local well-posedness.
However, the arguments that we give in this section feature all of the main ideas needed 
to prove local well-posedness on the standard Minkowski spacetime background
(in which the spacetime manifold is diffeomorphic to $\mathbb{R}^{1+3}$).

\subsection{Notation, norms, and basic tools from analysis}
\label{SS:NOTATIONANDNORM}

\subsubsection{Notation}
\label{SSS:NOTATION}
Throughout this section, 
$\lbrace x^{\alpha} \rbrace_{\alpha = 0,1,2,3}$
denote standard rectangular coordinates on $\mathbb{R} \times \mathbb{T}^3$, 
where $\lbrace x^a \rbrace_{a=1,2,3}$ are standard local coordinates
on $\mathbb{T}^3$, and we often use the alternate notation $t := x^0$.
Note that even though $\lbrace x^a \rbrace_{a=1,2,3}$ are only locally defined on $\mathbb{T}^3$, 
the coordinate partial derivative vectorfields $\lbrace \partial_a \rbrace_{a=1,2,3}$ 
can be extended to a smooth global
frame on $\mathbb{T}^3$; by a slight abuse of notation, we will denote
the globally defined ``spatial'' frame by $\lbrace \partial_a \rbrace_{a=1,2,3}$,
and the corresponding globally defined ``spacetime frame'' by
$\lbrace \partial_{\alpha} \rbrace_{\alpha=0,1,2,3}$.
Also, we often use the alternate partial derivative notation $\partial_t := \partial_0$.
\begin{align} \label{E:SIGMATDEF}
	\Sigma_t 
	& := \lbrace (t,x) \ | \ x \in \mathbb{T}^3 \rbrace
\end{align}
denotes the standard flat constant-time hypersurface. 

Throughout Sect.\,\ref{S:WELLPOSEDNESS}, we use the same conventions for lowering and raising indices
stated in Subsect.\,\ref{SS:NOTATIONANDCONVENTIONS}, 
i.e., we lower and raise indices with the Minkowski metric and its inverse.
Note that for Latin ``spatial'' indices, this is equivalent to lowering and raising 
via the Euclidean metric $\updelta_{ij} = \mbox{\upshape diag}(1,1,1)$
and its inverse $\updelta^{ij} = \mbox{\upshape diag}(1,1,1)$.
Finally, we note that we sometimes identify the Euclidean metric or its inverse
with the Kronecker delta.

To each ``spatial multi-index'' $\vec{I} = (\iota_1,\iota_2,\iota_3)$,
where the $\iota_a$ are non-negative integers,
we associate the spatial differential operator 
$\partial_{\vec{I}} := \partial_1^{\iota_1} \partial_2^{\iota_2} \partial_3^{\iota_3}$.
Note that $\partial_{\vec{I}}$ is an operator of order $|\vec{I}| := \iota_1 + \iota_2 + \iota_3$.

If $V$ is a spacetime vectorfield or a one-form, 
then $\underline{V}$ denotes the $\upeta$-orthogonal projection of $V$ onto $\Sigma_t$,
that is, the ``spatial part'' of $V$.
For example, $\underline{\vort}$ is the vectorfield on $\Sigma_t$ 
with rectangular components
$\underline{\vort}^i := \vort^i$ for $i=1,2,3$.
Moreover, we use the notation
\begin{align} \label{E:STANDARDCURL}
	\curlthreed^i(W) := \varepsilon^{ijk} \partial_j W_k
\end{align}
to denote the standard Euclidean curl operator acting on one-forms on $\Sigma_t$,
where $\upepsilon^{ijk}$ is the fully antisymmetric symbol normalized by
$\upepsilon^{123} = 1$.

\subsubsection{Norms}
\begin{definition}[Lebesgue and Sobolev norms]
	\label{D:LEBESGUEANDSOBOLEVNORMS}
We define the following \linebreak Lebesgue norms for scalar functions $f$:
\begin{align} 
		\| f \|_{L^{\infty}(\mathbb{T}^3)}
	& := 
		\mbox{ess sup}_{x \in \mathbb{T}^3}
			|f(x)|,
		\label{E:LINFTYNORM} 
			\\
	\| f \|_{L^2(\mathbb{T}^3)}
	& := 
		\left\lbrace
			\int_{\mathbb{T}^3}
				f^2(x)
			\, dx
		\right\rbrace^{1/2},
			\label{E:L2NORM} 
\end{align}
where in the rest of Sect.\,\ref{S:WELLPOSEDNESS}, 
$dx := dx^1 dx^2 dx^3$ denotes the standard volume form on $\mathbb{T}^3$
induced by the Euclidean metric $\mbox{\upshape diag}(1,1,1)$. 

\begin{remark}[Extending the definitions of the norms from $\mathbb{T}^3$ to $\Sigma_t$]
	\label{R:NORMSALONGSIGMAT}
	In our proof of local well-posedness, 
	we will use norms in which 
	the manifold $\mathbb{T}^3$ from 
	Def.\,\ref{D:LEBESGUEANDSOBOLEVNORMS} is replaced
	with the constant time slice $\Sigma_t = \lbrace t \rbrace \times \mathbb{T}^3$,
	which is diffeomorphic to $\mathbb{T}^3$. We will not explicitly define these norms
	along $\Sigma_t$ since their definitions are obvious analogs of the ones appearing in
	Def.\,\ref{D:LEBESGUEANDSOBOLEVNORMS}. For example,
	$
	\| f \|_{L^2(\Sigma_t)}
	:= 
		\left\lbrace
			\int_{\Sigma_t}
				f^2(t,x)
			\, dx
		\right\rbrace^{1/2}
		$,
		which is also equal to
		$\left\lbrace
			\int_{\mathbb{T}^3}
				f^2(t,x)
			\, dx
		\right\rbrace^{1/2}
	$. Here, we are using that the volume form induced by the Minkowski metric on $\Sigma_t$
	equals $dx$.
	Similar remarks apply to other norms on 
	$\mathbb{T}^3$ introduced later in this subsubsection.
\end{remark}

We define the following Sobolev norms for integers $r \geq 0$:
\begin{subequations}
\begin{align}
	\| f \|_{H^r(\mathbb{T}^3)}
	& := 
		\left\lbrace
			\sum_{|\vec{I}| \leq r}
			\| \partial_{\vec{I}} f \|_{L^2(\mathbb{R}^3)}^2
		\right\rbrace^{1/2},
		\label{E:HMNORM}
		\\
\| f \|_{\dot{H}^r(\mathbb{T}^3)}
	& := 
		\left\lbrace
			\sum_{|\vec{I}| = r}
			\| \partial_{\vec{I}} f \|_{L^2(\mathbb{R}^3)}^2
		\right\rbrace^{1/2}.
		\label{E:HOMOGENEOUSHMNORM}
\end{align}
\end{subequations}

If $r \in \mathbb{R}$ is not an integer, then we define\footnote{As is well known, 
when $r$ is an integer, RHS~\eqref{E:NOTINTEGERHMNORM} defines a norm
that is equivalent to the norm defined in \eqref{E:HMNORM}.}
\begin{align} \label{E:NOTINTEGERHMNORM}
\| f \|_{H^r(\mathbb{T}^3)}
	& := 
		\left\lbrace
			\sum_{(k_1,k_2,k_3) \in \mathbb{Z}^3}
				(1 + |k|^2)^r \left|\hat{f}(k_1,k_2,k_3) \right|^2
		\right\rbrace^{1/2},
\end{align}
where 
$\hat{f}(k_1,k_2,k_3) 
:=
\int_{\mathbb{T}^3} 
	f(x) e^{- 2 \pi i \sum_{a=1}^3 x^a k_a} 
\, dx
$
is the spatial Fourier transform of $f$
and
$|k|^2 := \sum_{a=1}^3 k_a^2$.

\end{definition}

If $U = (U_1,\cdots,U_m)$ is an array of scalar-valued functions
and $\| \cdot \|$ denotes any of the norms introduced in this subsubsection,
then we define
\begin{align} \label{E:NORMOFARRAY}
	\| U \|
	:= \sum_{a=1}^m \| U_a \|.
\end{align}

\begin{definition}[Some additional function spaces]
	\label{D:CONTINOUOUSMAPSINTOBANACHSPACES}
	If $\mathfrak{B}$ is a Banach space with norm $\| \cdot \|_{\mathfrak{B}}$
	and $r \geq 0$ is an integer,
	then $C^r([0,T],\mathfrak{B})$ denotes the space of 
	$r$-times continuously differentiable functions from $[0,T]$ to $\mathfrak{B}$.
	We omit the superscript when $r=0$.
	We denote the corresponding norm of an element $f$ of this space by
	$\| f \|_{C^r([0,T],\mathfrak{B})}
	: =
	\max_{t \in [0,T]}
	\sum_{k=0}^r
	\|  f^{(k)}(t) \|_{\mathfrak{B}}
	$,
	where $f^{(k)}$ denotes the $k^{th}$ derivative of $f$ with respect to $t$.

	$L^{\infty}([0,T],\mathfrak{B})$
	denotes the space of 
	functions from $[0,T]$ to $\mathfrak{B}$
	that are essentially bounded over the interval $[0,T]$.
	We denote the corresponding norm of an element $f$ of this space by
	$\| f \|_{L^{\infty}([0,T],\mathfrak{B})}
	: =
	\mbox{ess sup}_{t \in [0,T]}
	\|  f(t) \|_{\mathfrak{B}}$.
	
	$C^r(\mathbb{T}^3)$ denotes the space of functions on $\mathbb{T}^3$
	that are $r$-times continuously differentiable.	
	We omit the superscript when $r=0$.
	We denote the corresponding norm of an element $f$ of this space by
	\[
	\| f \|_{C^r(\mathbb{T}^3)} := 
	\sum_{|\vec{I}| \leq r} \max_{x \in \mathbb{T}^3} |\partial_{\vec{I}}f(x)|.
	\]
\end{definition}

We now fix, for the rest of Sect.\,\ref{S:WELLPOSEDNESS}, an integer $N$ subject to
\begin{align} \label{E:NISATLEASTTHREE}
	N \geq 3.
\end{align}

\subsubsection{Basic analytical tools}
In our analysis, we will rely on the following standard results;
see, e.g., \cites{AdamsSobolev,PalaisSeminar,TaylorPDE3} for proofs.

\begin{lemma} [Sobolev embedding, product, difference, and interpolation estimates]
\label{L:STANDARDSOBOLEVRESULTS}
If $r > 3/2$, then $H^r(\mathbb{T}^{3})$ continuously embeds into
$C(\mathbb{T}^{3})$,
and there exists a constant $C_r > 0$ such that the following estimate
holds for $v \in H^r(\mathbb{T}^{3})$:
\begin{align} \label{E:STANDARDSOBOLEVEMBEDDING}
	\| v \|_{C(\mathbb{T}^{3})}
	& 
	\leq 
	C_r \| v \|_{H^r(\mathbb{T}^{3})}.
\end{align}

Let $r \geq 0$ be an integer and let $v:= (v_1,\cdots,v_A)$
and $w:= (w_1,\cdots,w_B)$
be finite-dimensional arrays of real-valued
functions on $\mathbb{T}^{3}$ such that $v_a \in \dot{H}^r(\mathbb{T}^{3}) \cap C(\mathbb{T}^{3})$
for $1 \leq a \leq A$ and 
$w_b \in C(\mathbb{T}^{3})$
$1 \leq b \leq B$.
Let
\begin{align} \label{E:ORDERROPERATORS}
	\mathcal{I}_r 
	& := \left\lbrace (\vec{I}_1,\cdots,\vec{I}_A) \ | \ \sum_{a=1}^A |\vec{I}_a| = r \right\rbrace.
\end{align}
Assume that $w(\mathbb{T}^3) \subset \mbox{\upshape int} \mathscr{K}$,
where $\mathscr{K}$ is a compact subset of $\mathbb{R}^B$,
and let $\mathrm{f}$ be a smooth real-valued function 
on an open subset of $\mathbb{R}^B$ containing $\mathscr{K}$.
Then the following estimate holds:
	\begin{align} \label{E:WITHSMOOTHFUNCTIONFRAMENORML2PRODUCTBOUNDINERMSOFLINFINITYANDHSDOT}
		&
		\max_{(\vec{I}_1,\cdots,\vec{I}_A) \in \mathcal{I}_r}
		\left\| 
			\mathrm{f}(w) \prod_{a=1}^A \partial_{\vec{I}_a} v_a 
		\right\|_{L^2(\mathbb{T}^3)}
			\\
		& 
		\leq
		C_{\mathrm{f},\mathscr{K},r}
		\sum_{a=1}^A \| v_a \|_{\dot{H}^r(\mathbb{T}^{3})} 
		\prod_{b \neq a} \| v_b \|_{C(\mathbb{T}^{3})}.
		\notag
	\end{align}
	
	Moreover, under the same assumptions stated in the previous paragraph, 
	if $(\vec{I}_1,\cdots,\vec{I}_A) \in \mathcal{I}_r$, 
	then the map $(v,w) \rightarrow \mathrm{f}(w) \prod_{a=1}^A \partial_{\vec{I}_a} v_a$
	is continuous from $\left(\dot{H}^r(\mathbb{T}^{3}) \cap C(\mathbb{T}^{3})\right)^A 
	\times \left(C(\mathbb{T}^{3})\right)^B$ 
	to $L^2(\mathbb{T}^{3})$.
	In particular,
	let $\updelta = \updelta_w > 0$ be such that the following holds:\footnote{Such a $\updelta > 0$ exists due to the compactness
of $w(\mathbb{T}^3)$ and $\mathscr{K}$, where the compactness of
$w(\mathbb{T}^3)$ follows from the assumption that the $v_a$ are continuous.}
if
$d(p,w(\mathbb{T}^3)) < \updelta$,
$d(q,w(\mathbb{T}^3)) < \updelta$,
and $d(p,q) < \updelta$,
where $d$ is the standard Euclidean distance function on $\mathbb{R}^B$,
then the straight line segment 
joining $p$ to $q$ is contained in $\mbox{\upshape int} \mathscr{K}$.
Then if $(v,w)$ and $(\widetilde{v},\widetilde{w})$ are two array pairs
of the type described in the previous paragraph such that
$
		\| w - \widetilde{w} \|_{C(\mathbb{T}^3)}
		\leq \updelta
	$,
	and if $r > 3/2$,
	then the following estimate holds
	(where the function $\mathrm{f}$ is assumed to be
	the same in both appearances on LHS~\eqref{E:CONTINUITYESTIMATEWITHSMOOTHFUNCTIONFRAMENORML2PRODUCTBOUNDINERMSOFLINFINITYANDHMDOT}
	and $\mathcal{I}_r$ is defined by \eqref{E:ORDERROPERATORS}):
	\begin{align} \label{E:CONTINUITYESTIMATEWITHSMOOTHFUNCTIONFRAMENORML2PRODUCTBOUNDINERMSOFLINFINITYANDHMDOT}
		\max_{(\vec{I}_1,\cdots,\vec{I}_A) \in \mathcal{I}_r}
		&
		\left\| 
			\mathrm{f}(w) \prod_{a=1}^A \partial_{\vec{I}_a} v_a 
			-
			\mathrm{f}(\widetilde{w}) \prod_{a=1}^A \partial_{\vec{I}_a} \widetilde{v}_a 
		\right\|_{L^2(\mathbb{T}^3)}
			\\
		& \leq
		C_{\mathrm{f},\mathscr{K},\| v \|_{H^r(\mathbb{T}^{3})},\| \widetilde{v} \|_{H^r(\mathbb{T}^{3})},A,r}
		\left\lbrace
			\| v - \widetilde{v} \|_{H^r(\mathbb{T}^{3})}
			+
			\| w - \widetilde{w} \|_{C(\mathbb{T}^{3})}
		\right\rbrace.
		\notag
	\end{align}
	
Furthermore, if $r > 3/2$ and $v_a \in H^r(\mathbb{T}^{3})$ for $a=1,2$, then
$v_1 v_2 \in H^r(\mathbb{T}^{3})$,
and there exists a constant $C_r > 0$ such that
\begin{align} \label{E:SOBOLEVALGEBRAPROP}
	\| v_1 v_2 \|_{H^r(\mathbb{T}^{3})}
	& \leq C_r \| v_1 \|_{H^r(\mathbb{T}^{3})} \| v_2 \|_{H^r(\mathbb{T}^{3})},
\end{align} 
and function multiplication $(v_1,v_2) \rightarrow v_1 v_2$
is a continuous map from \linebreak $H^r(\mathbb{T}^{3}) \times H^r(\mathbb{T}^{3})$ to $H^r(\mathbb{T}^{3})$.

Finally, if $0 \leq s \leq r$ and $v \in H^r(\mathbb{T}^{3})$, then 
there exists a constant $C_{r,s} > 0$ such that
\begin{align} \label{E:STANDARDSOBOLEVINTERPOLATION}
	\| v \|_{H^s(\mathbb{T}^3)}
	& \leq C_{r,s} 
		\| v \|_{L^2(\mathbb{T}^3)}^{1-\frac{s}{r}}
		\| v \|_{H^r(\mathbb{T}^3)}^{\frac{s}{r}}.
\end{align}
\end{lemma}

\begin{remark}[The same estimates hold along $\Sigma_t$]
	All of the results of Lemma \linebreak \ref{L:STANDARDSOBOLEVRESULTS}
	hold verbatim if we replace $\mathbb{T}^3$ by $\Sigma_t$
	throughout.
\end{remark}

\subsubsection{An $L^2$-in-time continuity result for transport equations}
We will use the following simple technical result in our proof of local well-posedness.

\begin{lemma}[An $L^2$-in-time continuity result for transport equations]
	\label{L:CONTINUITYRESULTFORTRANSPORTEQUATIONS}
	Let $T > 0$. Assume that $\mathscr{F} \in L^{\infty}\big([0,T],L^2(\mathbb{T}^3) \big)$,
	and let $f$ be the solution to the inhomogeneous transport equation initial value problem
	\begin{align}	
		u^{\alpha} \partial_{\alpha} f
		& = \mathscr{F},
			\\
		f|_{\Sigma_0}
		& := \mathring{f} \in L^2(\Sigma_0).
	\end{align}
	Assume further that $u^{\alpha} \in L^{\infty}\big([0,T],C^1(\mathbb{T}^3) \big)$ for $\alpha = 0,1,2,3$.
	Then
	\begin{align} \label{E:L2CONTINUITYFORTRANSPROTEQUATION}
		f \in C\big([0,T],L^2(\mathbb{T}^3)\big).
	\end{align}
\end{lemma}	

\begin{proof}
	We will prove right continuity at $t=0$; continuity at any other time
	$t \in (0,T]$ could be proved using similar arguments.
	More precisely, we will show that
	\begin{align} \label{E:ENERGYESTIMATETRANSPORTORIGINALDATASUBTRACTEDOFFLIMCONTINUITYESTIMATE}
		\lim_{t \downarrow 0}
		\left\| 
			f(t,\cdot) 
			- 
			\mathring{f}
		\right\|_{L^2(\mathbb{T}^3)}
		& = 0.
	\end{align}
	To proceed, we let $\lbrace \mathring{f}_k \rbrace_{k = 1}^{\infty} \subset C^{\infty}(\mathbb{T}^3)$
	be a sequence of smooth functions such that
	\begin{align} \label{E:SMOOTHFUNCTIONSNEARDATA}
		\|  \mathring{f} - \mathring{f}_k \|_{L^2(\Sigma_0)}
		& \leq \frac{1}{k}.
	\end{align}
	Note that
	\begin{align} \label{E:TRANSPORTEQUATIONWITH}	
		u^{\alpha} \partial_{\alpha} (f - \mathring{f}_k)
		& = 
				-
				u^a \partial_a \mathring{f}_k
				+
				\mathscr{F}.
	\end{align}
	Hence, a standard integration by parts argument 
	based on the divergence identity
	\begin{align} \label{E:DIVIDL2CONTINUITYFORTRANSPORT}
	\partial_t 
	\left\lbrace
		(f - \mathring{f}_k)
	\right\rbrace^2
	& =
	\left\lbrace
					\partial_a
					\left(
						\frac{u^a}{u^0}
					\right)
				\right\rbrace
				(f - \mathring{f}_k)^2
					\\
	& \ \
	+
	2 
	\frac{(f - \mathring{f}_k)}{u^0}
	\left\lbrace
		-
		u^a \partial_a \mathring{f}_k
		+
		\mathscr{F}
	\right\rbrace
		\notag \\
& \ \
	-
	\partial_a
	\left\lbrace
			\left(
						\frac{u^a}{u^0}
					\right)
	(f - \mathring{f}_k)^2
	\right\rbrace
	\notag
	\end{align}
	yields that
	for $0 \leq t \leq T$, we have
	\begin{align} \label{E:BASICENERGYIDTRANSPORT}
		\| f - \mathring{f}_k \|_{L^2(\Sigma_t)}^2
		& =
		\| \mathring{f} - \mathring{f}_k \|_{L^2(\Sigma_0)}^2
			\\
	& \ \
		+
		\int_{\tau = 0}^t
			\int_{\Sigma_{\tau}}
				\left\lbrace
					\partial_a
					\left(
						\frac{u^a}{u^0}
					\right)
				\right\rbrace
				(f - \mathring{f}_k)^2
			\, dx
		\, d \tau
		\notag	\\
	& \	\
		+
		2
		\int_{\tau = 0}^t
			\int_{\Sigma_{\tau}}
				\frac{(f - \mathring{f}_k)}{u^0}
				\left\lbrace
					-
					u^a \partial_a \mathring{f}_k
					+
					\mathscr{F}
				\right\rbrace
			\, dx
		\, d \tau.
		\notag
	\end{align}
	In particular, from 
	\eqref{E:SMOOTHFUNCTIONSNEARDATA},
	\eqref{E:BASICENERGYIDTRANSPORT},
	our assumptions on $\mathscr{F}$ and $u^{\alpha}$,
	and Young's inequality, 
	we find that if $0 \leq t \leq T$,
	then there is a constant $C_T$ 
	(independent of $k$)
	such that
	\begin{align} \label{E:FIRSTINEQUALITYDERIVEDFROMBASICENERGYIDTRANSPORT}
		\| f - \mathring{f}_k \|_{L^2(\Sigma_t)}^2
		& 
		\leq
		\frac{1}{k^2}
		+
		C_T
		\int_{\tau = 0}^t
			\left\lbrace
				1
				+
				\| \mathring{f}_k \|_{H^1(\Sigma_0)}^2
			\right\rbrace
		\, d \tau
			\\
	& \ \
		+
		C_T
		\int_{\tau = 0}^t
			\| f - \mathring{f}_k \|_{L^2(\Sigma_{\tau})}^2
		\, d \tau.
		\notag
	\end{align}
	From \eqref{E:FIRSTINEQUALITYDERIVEDFROMBASICENERGYIDTRANSPORT}
	and Gronwall's inequality,
	we deduce (allowing $C_T$ to vary from line to line in the rest of the proof) that
	if $0 \leq t \leq T$,
	then the following inequality holds:
	\begin{align} \label{E:BASICENERGYTRANSPORTGRONWALLED}
		\| f - \mathring{f}_k \|_{L^2(\Sigma_t)}^2
		& \leq 
			\left\lbrace
				\frac{1}{k^2}
				+
				C_T t
				\left(
				1
				+
				\| \mathring{f}_k \|_{H^1(\Sigma_0)}^2
			\right)
			\right\rbrace
			\exp(C_T t).
	\end{align}
	From \eqref{E:BASICENERGYTRANSPORTGRONWALLED},
	\eqref{E:SMOOTHFUNCTIONSNEARDATA},
	and the triangle inequality,
	it follows that
	\begin{align} \label{E:TRANSPORTEQUATIONBOUNDEDL2NROM}
		\lim_{t \to 0^+} \sup_{0 \leq \tau \leq t}  \| f - \mathring{f} \|_{L^2(\Sigma_{\tau})}
		& \leq \frac{2}{k}.
	\end{align}
	Finally, allowing $k \to \infty$ in \eqref{E:TRANSPORTEQUATIONBOUNDEDL2NROM},
	we conclude \eqref{E:ENERGYESTIMATETRANSPORTORIGINALDATASUBTRACTEDOFFLIMCONTINUITYESTIMATE}.
	We have therefore proved the lemma.
	
\end{proof}

\subsection{The regime of hyperbolicity}
\label{SS:REGIMEOFHYPERBOLICITY}
Our proof of well-posedness relies on a standard assumption, namely that the
solution lies in the interior of the region of state space where the equations
are hyperbolic without degeneracy. This notion is precisely captured by the next
definition.

\begin{definition}[Regime of hyperbolicity]
	\label{D:REGIMEOFHYPERBOLICITY}
	We define the regime of hyperbolicity $\mathcal{H}$ to be the following subset of state-space:
	\begin{align} \label{E:REGIMEOFHYPERBOLICITY}
		\mathcal{H}
		:= \left\lbrace
					(\Lnenth,\Ent,u^1,u^2,u^3) \in \mathbb{R}^5
					\ | \ 0 < \speed(\Lnenth,\Ent) \leq 1
			 \right\rbrace.
	\end{align}
\end{definition}

\subsection{Standard local well-posedness}
\label{SS:STANDARDLOCALWELLPOSEDNESS}
Our principal goal in this subsection is to state Theorem~\ref{T:UPGRADEDLOCALWELLPOSEDNESS},
which is our main local well-posedness result exhibiting the gain
in regularity for the vorticity and entropy. Most aspects of the theorem are standard.
We summarize these standard aspects in Theorem~\ref{T:STANDLOCALWELLPOSEDNESS},
which will serve as a precursor to our proof of Theorem~\ref{T:UPGRADEDLOCALWELLPOSEDNESS}.

\begin{remark}[Some non-standard aspects of Theorem~\ref{T:UPGRADEDLOCALWELLPOSEDNESS}]
One of the non-standard aspects of Theorem~\ref{T:UPGRADEDLOCALWELLPOSEDNESS}
is that it shows the continuous time-dependence of the top-order derivatives of
$\vort$ and $\Ent$ in the norm $\| \cdot \|_{L^2(\Sigma_t)}$.
The proof relies on some results that are not easy to locate in the literature,
tied in part to the fact that the required estimates are of elliptic-hyperbolic type.
In our proof of Theorem~\ref{T:UPGRADEDLOCALWELLPOSEDNESS}, we will show how
to obtain these top-order time-continuity results.
A second non-standard aspect of Theorem~\ref{T:UPGRADEDLOCALWELLPOSEDNESS}
is that the $\mbox{\upshape transport}$-$\mbox{\upshape div}$-$\mbox{\upshape curl}$ systems 
(specifically \eqref{E:MAINTHMSMODIFIEDDIVERGENCEENTROPYGRADIENTTRANSPORT}-\eqref{E:MAINTHMUPERPCURLOFENTROPYGRADIENTISZERO} 
and \eqref{E:MAINTHMDIVOFVORTICITY}-\eqref{E:MAINTHMTRANSPORTFORMODIFIEDVORTICITYOFVORTICITY})
leading to the gain in
regularity for $\vort$ and $\Ent$
involve \emph{spacetime} divergence and curl operators.
Hence, additional arguments are needed to
obtain the needed \emph{spatial} elliptic estimates along 
$\Sigma_t$; 
the key ingredients in this vein 
are provided by the identity \eqref{E:ALGEBRAICEXPRESSIONFORICOMMUTEDGEOMETRICDIVERGENCEOFVORTICITY} 
and Lemma~\ref{L:COMPARISONBASEDONELLIPTIC}.
\end{remark}

\begin{remark}[The ``fundamental'' initial data]
	In the rest of Sect.\,\ref{S:WELLPOSEDNESS}, we view
	$\mathring{\Lnenth} := \Lnenth|_{\Sigma_0}$,
	$\mathring{\Ent} := \Ent|_{\Sigma_0}$,
	and
	$\mathring{u}^i := u^i|_{\Sigma_0}$
	to be the ``fundamental'' initial data in the following sense:
	with the help of the relativistic Euler equations 
	\eqref{E:ENTHALPYEVOLUTION}-\eqref{E:ENTROPYEVOLUTION} + \eqref{E:UISUNITLENGTH},
	along $\Sigma_0$, all of the other quantities that are relevant for our analysis
	can be expressed in term of the fundamental initial data;
	see Lemma~\ref{L:EXPRESSIONSFORALGEBRAICQUANTITIESINTERMSOFHYPERBOLICANDELLIPTIC}.
\end{remark}

\begin{theorem}[Standard local well-posedness]
	\label{T:STANDLOCALWELLPOSEDNESS}
	Let
	$\mathring{\Lnenth} := \Lnenth|_{\Sigma_0}$,
	$\mathring{\Ent} := \Ent|_{\Sigma_0}$,
	and
	$\mathring{u}^i := u^i|_{\Sigma_0}$
	be initial data\footnote{The datum $u^0|_{\Sigma_0}$ is determined from the other data by virtue of the constraint \eqref{E:UISUNITLENGTH}.} 
	for the relativistic Euler equations
	\eqref{E:ENTHALPYEVOLUTION}-\eqref{E:ENTROPYEVOLUTION} + \eqref{E:UISUNITLENGTH}.
	Assume that for some integer $N \geq 3$, we have
	\begin{align} \label{E:FIRSTTHMSTANDARDDATAASSUMPTIONS} 
		\mathring{\Lnenth},
			\,
		\mathring{\Ent},
			\,
		\mathring{u}^i
		& \in H^N(\Sigma_0).
	\end{align}
	Assume moreover that there is a compact subset 
	$\mathfrak{K} \subset \mbox{\upshape int} \mathcal{H}$ 
	(where $\mbox{int} \mathcal{H}$ is the interior of $\mathcal{H}$)
	such that for all $p \in \Sigma_0$, we have
	\[
	(\mathring{\Lnenth}(p),\mathring{\Ent}(p),\mathring{u}^1(p),\mathring{u}^2(p),\mathring{u}^3(p)) 
	\in \mbox{\upshape int} \mathfrak{K}.
	\]
	Then there exists a time $T > 0$ depending only on\footnote{In fact, using
	additional arguments not presented here, one can show that
	for any fixed real number $r > 5/2$,
	the time of existence can be controlled by a function of
	$\mathfrak{K}$, 
	$\| \mathring{\Lnenth} \|_{H^r(\Sigma_0)}$,
	$\| \mathring{\Ent} \|_{H^r(\Sigma_0)}$,
	and $\| \mathring{u}^i \|_{H^r(\Sigma_0)}$. Of course, if the initial data enjoy additional Sobolev regularity,
	then the additional regularity persists in the solution during its classical lifespan.
	\label{FN:SHARPERTIMEOFEXISTENCENORMS}}
	$\mathfrak{K}$, 
	$\| \mathring{\Lnenth} \|_{H^3(\Sigma_0)}$,
	$\| \mathring{\Ent} \|_{H^3(\Sigma_0)}$,
	and $\| \mathring{u}^i \|_{H^3(\Sigma_0)}$,
	such that a unique classical solution 
	$(\Lnenth,\Ent,u^{\alpha},\vort^{\alpha})$
	exists on the slab
	$[0,T] \times \mathbb{T}^3$ 
	and satisfies
	$(\Lnenth(p),\Ent(p),u^1(p),u^2(p),u^3(p)) \in \mbox{int} \mathfrak{K}$
	for $p \in [0,T] \times \mathbb{T}^3$.
	Moreover, the solution depends continuously on the initial data,\footnote{In particular,
	there is a $\left(H^3(\Sigma_0) \right)^5$-neighborhood of 
	$(\mathring{\Lnenth},\mathring{\Ent},\mathring{u}^i)$
	such that all data in the neighborhood launch solutions that 
	exist on the same slab $[0,T] \times \mathbb{T}^3$
	and, assuming also that the data belong to $\left(H^N(\Sigma_0) \right)^5$, 
	enjoy the regularity properties stated in the theorem. \label{FN:CONTINDEPONDATA}} 
	and its components relative to the standard coordinates enjoy the following regularity properties:
	\begin{subequations}
	\begin{align}
		\Lnenth,
			\, 
		\Ent,
			\,
		u^{\alpha} & \in C\big([0,T],H^N(\mathbb{T}^3)\big),
		\label{E:FIRSTTHMSTANDARDCONTINUITYINHN} \\
		\GradEnt^{\alpha},
			\,
		\vort^{\alpha} 
		& \in C\big([0,T],H^{N-1}(\mathbb{T}^3)\big).
		\label{E:NOTYETGAINOFONEDERIVATIVE}
	\end{align}
	\end{subequations}
\end{theorem}

\begin{proof}[Discussion of the proof]
Theorem~\ref{T:STANDLOCALWELLPOSEDNESS} is standard.
Readers can consult, for example, \cite{jS2008a}
for detailed proofs in the case of the relativistic Euler equations
on a family of conformally flat\footnote{More precisely, in \cite{jS2008a},
the spacetime metrics are scalar function multiples of the Minkowski metric
on $\mathbb{R}^{1+3}$.} 
spacetimes.
The main step in the proof is deriving a priori energy estimates for
linearized versions of a first-order formulation of the equations,
such as \eqref{E:ENTHALPYEVOLUTION}-\eqref{E:ENTROPYEVOLUTION} + \eqref{E:UISUNITLENGTH}.
For a first-order formulation that is equivalent (for $C^1$ solutions)
to \eqref{E:ENTHALPYEVOLUTION}-\eqref{E:ENTROPYEVOLUTION} + \eqref{E:UISUNITLENGTH},
this step was carried out in detail in \cite{jS2008a}
using the method of energy currents, 
a technique that originated in the context of
the relativistic Euler equations in Christodoulou's
foundational work \cite{dC2007} on shock formation.

\end{proof}

\begin{remark}[$C^{\infty}$ data give rise to $C^{\infty}$ solutions]
	\label{R:CINFINITYDATAGIVECINFINITYSOLUTIONS}
	In view of the Sobolev embedding result \eqref{E:STANDARDSOBOLEVEMBEDDING},
	we see that Theorem~\ref{T:STANDLOCALWELLPOSEDNESS}
	implies that $C^{\infty}$ initial data give rise
	to (local-in-time) $C^{\infty}$ solutions.
\end{remark}

We now state our main local well-posedness theorem. Its proof is located
in Subsect.\,\ref{SS:PROOFOFTHEOREMUPGRADEDLOCALWELLPOSEDNESS}.

\begin{theorem}[Local well-posedness with improved regularity for the entropy and vorticity]
	\label{T:UPGRADEDLOCALWELLPOSEDNESS}
	Assume the hypotheses of Theorem~\ref{T:STANDLOCALWELLPOSEDNESS},
	but in addition to \eqref{E:FIRSTTHMSTANDARDDATAASSUMPTIONS},
	assume also that the initial vorticity and entropy
	enjoy one extra degree of Sobolev regularity. That is, 
	assume that for some integer $N \geq 3$ and $i=1,2,3$, 
	we have
	\begin{subequations}
	\begin{align}
		\mathring{\Lnenth},
			\,
		\mathring{u}^i 
		& \in H^N(\Sigma_0),
			\label{E:STANDARDDATAASSUMPTIONS} \\
		\mathring{\Ent} \in H^{N+1}(\Sigma_0),
			\,
			\mathring{\vort}^i
		& \in H^N(\Sigma_0),
		\label{E:GAINOFONEDERIVATIVEDATAASSUMPTIONS}
	\end{align}
	\end{subequations}
	where $\vort$ is defined in \eqref{E:VORTICITYDEF} and $\mathring{\vort}^i := \vort|_{\Sigma_0}^i$.
	
	Then the conclusions of Theorem~\ref{T:STANDLOCALWELLPOSEDNESS} hold,
	and the solution's components relative to standard coordinates enjoy the following regularity properties for $\alpha=0,1,2,3,$
	where $T > 0$ is the same time from Theorem~\ref{T:STANDLOCALWELLPOSEDNESS}:
	\begin{subequations}
	\begin{align}
		\Lnenth, 
		\,
		u^{\alpha} 
		& \in C\big([0,T],H^N(\mathbb{T}^3)\big),
		&&
			\label{E:STANDARDCONTINUITYINHN} \\
		\Ent 
		& \in C\big([0,T],H^{N+1}(\mathbb{T}^3)\big),
		&
		\GradEnt^{\alpha},
		\,
		\vort^{\alpha} 
		& \in C\big([0,T],H^N(\mathbb{T}^3)\big).
		\label{E:GAINOFONEDERIVATIVE}
	\end{align}
	\end{subequations}
	In particular, according to \eqref{E:GAINOFONEDERIVATIVE},
	the additional regularity of the entropy and vorticity featured in the initial data
	is propagated by the flow of the equations.
	Moreover, the solution depends continuously on the initial data
	relative the norms corresponding to \eqref{E:STANDARDCONTINUITYINHN}-\eqref{E:GAINOFONEDERIVATIVE}.
\end{theorem}

\subsection{A new inverse Riemannian metric and the classification of various combinations of solution variables}
\label{SS:NEWRIEMANNMETRICANDCLASSIFICATIONOFSOLUTIONVARIABLES}
In our proof of Theorem~\ref{T:UPGRADEDLOCALWELLPOSEDNESS},
when controlling the top-order derivatives of the
vorticity and entropy, we will rely on ``geometrically sharp'' elliptic estimates
in which the precise details of the principal coefficients of the elliptic operators are important
for our arguments. Due to the quasilinear nature of the relativistic Euler equations,
these precise elliptic estimates involve the inverse Riemannian metric $G^{-1}$ from the next definition.
In particular, we will need to use $G^{-1}$-based norms
when proving that the top-order derivatives
of $\vort$ and $\GradEnt$ are continuous in time with values in $L^2(\mathbb{T}^3)$
(these facts are contained within the statement \eqref{E:GAINOFONEDERIVATIVE});
the role of $G^{-1}$ in our analysis
will become clear in Subsect.\,\ref{SS:PROOFOFTHEOREMUPGRADEDLOCALWELLPOSEDNESS}.

\begin{definition}[An inverse Riemannian metric on $\Sigma_t$]
	\label{D:INVERSERIEMANNIANONSIGMAT}
	On each $\Sigma_t$, we define the inverse Riemannian metric $G^{-1}$ 
	as follows:
	\begin{align} \label{E:INVERSERIEMANNIANONSIGMAT}
		(G^{-1})^{ij}
		& := \updelta^{ij} - \frac{u^i u^j}{(u^0)^2},
	\end{align}
	where $\updelta^{ij} := \mbox{\upshape diag}(1,1,1)$
	is the standard Kronecker delta.
\end{definition}

\begin{remark}
	From the relation $\upeta_{\alpha \beta} u^{\alpha} u^{\beta} = - 1$,
	one can easily show that 
	$G^{-1}$ is Riemannian, that is, of signature $(+,+,+)$.
\end{remark}

In proving that the solution depends continuously on the initial data,
we will use a modified version of Kato's framework 
\cites{tK1970,tK1973,tK1975}. His framework
was designed to handle hyperbolic systems, while our formulation of
the relativistic Euler equations is elliptic-hyperbolic. For this reason,
we find it convenient to divide the solution variables into various
classes, which we provide in the next definition. Roughly, we will handle
the ``hyperbolic quantities'' using Kato's framework, and to handle the
remaining quantities, we will use elliptic estimates and algebraic relationships
to control them in terms of the hyperbolic quantities.

\begin{definition}[Classification of various combinations of solution variables]
	\label{D:CLASSIFICATIONOFSOLUTIONVARIABLES}
We define the \emph{hyperbolic quantities} $\mathbf{H}$,
the \emph{elliptic quantities} $\mathbf{E}$,
and the \emph{algebraic quantities}
$\mathbf{A}_{\mathbf{H}}$,
$\mathbf{A}_{\mathbf{H},\mathbf{E}}$,
and $\mathbf{A}$
as follows,
where the Euclidean curl operator $\curlthreed$ is defined in \eqref{E:STANDARDCURL}:
\begin{subequations}
\begin{align}	
	\mathbf{H}
	& := (\Lnenth,\Ent,u^a,\partial_a \Lnenth,\partial_a u^b,\vort^a,\GradEnt^a,\mathcal{C}^a,\mathcal{D})_{a,b=1,2,3},
	\label{E:HYPERBOLICVARS}
		\\
	\mathbf{E}
	& := (\partial_a \vort_b,\partial_a \GradEnt_b)_{a,b=1,2,3},
	\label{E:ELLIPTICVARS}
		\\
	\mathbf{A}_{\mathbf{H}}
	& := \left(u^0 - 1, \vort^0, \GradEnt^0, \mathcal{C}^0,
				\partial_t \Lnenth, \partial_t u^{\alpha}, \partial_a u^0, \partial_t \Ent \right)_{\alpha=0,1,2,3;a=1,2,3}
	\label{E:ALGEBRAICVARSHYPERBOLICONLY}
		\\
	& \ \ \ \
		\cup \left((G^{-1})^{cd} \partial_c \vort_d,
				(G^{-1})^{cd} \partial_c \GradEnt_d,
				\curlthreed^a(\underline{\vort}),
				\curlthreed^a(\underline{\GradEnt})\right)_{a=1,2,3},
					\notag \\
	\mathbf{A}_{\mathbf{H},\mathbf{E}}
	& := (\partial_t \vort_{\alpha},\partial_a \vort_0,\partial_t \GradEnt_{\alpha},\partial_a \GradEnt_0,
				\partial_b \vort^b, \partial_b \GradEnt^b)_{\alpha=0,1,2,3;a=1,2,3},
				\label{E:ALGEBRAICVARSHYPERBOLICELLIPTIC}
					\\
	\mathbf{A}
	& := \mathbf{A}_{\mathbf{H}} \cup \mathbf{A}_{\mathbf{H},\mathbf{E}}.
	\label{E:ALLALGEBRAICVARIABLES}
\end{align}
\end{subequations}

\end{definition}

Some remarks are in order.
\begin{itemize}
	\item The point of introducing the algebraic quantities 
		$\mathbf{A}$ is that, by virtue of the relativistic Euler equations, 
		they can be algebraically expressed in terms of $\mathbf{H}$ and $\mathbf{E}$
		(and thus are redundant);
		see Lemma~\ref{L:EXPRESSIONSFORALGEBRAICQUANTITIESINTERMSOFHYPERBOLICANDELLIPTIC}.
		We stress that in \eqref{E:ALGEBRAICVARSHYPERBOLICONLY},
		it is crucial that the inverse metric $G^{-1}$
		is the one from Def.\,\ref{D:INVERSERIEMANNIANONSIGMAT};
		the proof of \eqref{E:ALGEBRAICHYPERBOLICONLYSMOOTHFUNCTIONEXPRESSION}
		will clarify that it is essential that the inverse metric is precisely $G^{-1}$.
	\item The elliptic quantities $\mathbf{E}$ can be controlled 
		(in appropriate Sobolev norms)
		in terms of $\mathbf{H}$ via elliptic estimates; see Lemma~\ref{L:COMPARISONBASEDONELLIPTIC} and its proof.
	\item The hyperbolic quantities $\mathbf{H}$ solve evolution equations 
		with source terms that depend on $\mathbf{H}$ and $\mathbf{E}$.
		In view of the previous point, we see that one can bound the source terms
		(in appropriate Sobolev norms) in terms of $\mathbf{H}$.
		This will allow us to derive a closed system of energy inequalities
		that can be used to estimate $\mathbf{H}$. In view of the previous two points,
		we see that the estimates for $\mathbf{H}$ imply corresponding estimates for
		$\mathbf{E}$ and $\mathbf{A}$.
\end{itemize}

\begin{remark}[The hyperbolic quantities verify first-order hyperbolic equations]
	\label{R:HYPERBOLICVARIABLESSOLVEFIRSTORDERTRANSPORT}
In our proof of local well-posedness, we will use the fact
that the hyperbolic quantities $\mathbf{H}$ solve first-order hyperbolic equations.
More precisely, the elements $\Lnenth$, $\Ent$, and $u^a$ of \eqref{E:HYPERBOLICVARS}
satisfy the first-order hyperbolic system
\eqref{E:ENTHALPYEVOLUTION}-\eqref{E:ENTROPYEVOLUTION} + \eqref{E:UISUNITLENGTH},
the elements
$
\partial_a \Lnenth
$
and
$\partial_a u^b$
satisfy hyperbolic equations obtained by taking one spatial derivative of the
equations \eqref{E:ENTHALPYEVOLUTION}-\eqref{E:ENTROPYEVOLUTION} + \eqref{E:UISUNITLENGTH},
and
$
\GradEnt^a
$,
$
\vort^a
$,
$
\mathcal{C}^a
$,
and
$
\mathcal{D}
$
respectively satisfy the (spatial components of the)
transport equations
\eqref{E:MAINTHMENTROPYGRADIENTEVOLUTION},
\eqref{E:MAINTHMVORTICITYTRANSPORT},
\eqref{E:MAINTHMTRANSPORTFORMODIFIEDVORTICITYOFVORTICITY},
and
\eqref{E:MAINTHMSMODIFIEDDIVERGENCEENTROPYGRADIENTTRANSPORT};
it is in this sense that we consider the variables 
$\mathbf{H}$ to be ``hyperbolic.''
\end{remark}

\begin{lemma}[Expressions for the algebraic quantities in terms of the hyperbolic and elliptic quantities]
	\label{L:EXPRESSIONSFORALGEBRAICQUANTITIESINTERMSOFHYPERBOLICANDELLIPTIC}
	Assume that $(\Lnenth,\Ent,u^{\alpha})$ is a smooth solution 
	to \eqref{E:ENTHALPYEVOLUTION}-\eqref{E:ENTROPYEVOLUTION} + \eqref{E:UISUNITLENGTH}.
	Then we can express
	\begin{subequations}
	\begin{align} 
		\mathbf{A}_{\mathbf{H}} 
		& = \mathrm{f}(\mathbf{H}),
			\label{E:ALGEBRAICHYPERBOLICONLYSMOOTHFUNCTIONEXPRESSION} \\
		\mathbf{A}_{\mathbf{H},\mathbf{E}} 
		& = \mathrm{f}(\mathbf{H},\mathbf{E}),
			\label{E:ALGEBRAICHYPERBOLICANDELLIPTICSMOOTHFUNCTIONEXPRESSION} \\
	\mathbf{A}
	& = \mathrm{f}(\mathbf{H},\mathbf{E}),
	\label{E:ALLALGEBRAICSMOOTHFUNCTIONEXPRESSION}
	\end{align}
	\end{subequations}
	where in \eqref{E:ALGEBRAICHYPERBOLICONLYSMOOTHFUNCTIONEXPRESSION}-\eqref{E:ALLALGEBRAICSMOOTHFUNCTIONEXPRESSION},
	$\mathrm{f}$ is a schematically denoted smooth function that 
	satisfies $\mathrm{f}(0) = 0$ and that is allowed to vary from line to line.
	
	Moreover, let $\vec{I}$ be a spatial multi-index with $|\vec{I}| \geq 1$. Then
	\begin{align} \label{E:ALGEBRAICEXPRESSIONFORICOMMUTEDGEOMETRICDIVERGENCEOFVORTICITY} 
		&
		(G^{-1})^{ab} \partial_a \partial_{\vec{I}} \vort_b,
			\,
		(G^{-1})^{ab} \partial_a \partial_{\vec{I}} \GradEnt_b,
			\,
		\curlthreed^i(\partial_{\vec{I}} \underline{\vort})
			\\
		& = \mathop{\sum_{M=1}^{|\vec{I}|}}_{|\vec{J}_1| + \cdots + |\vec{J}_M| = |\vec{I}|}
				\mathrm{f}_{\vec{J}_1,\cdots,\vec{J}_M}(\mathbf{H}) 
				\prod_{m=1}^M
				\partial_{\vec{J}_m} \mathbf{H}.
				\notag
	\end{align}
	where $\mathrm{f}_{\vec{J}_1,\cdots,\vec{J}_M}$ are schematically denoted smooth functions 
	(not necessarily vanishing at $0$)
	and $\prod_{m=1}^M \partial_{\vec{J}_m} \mathbf{H}$ schematically denotes an order $M$ monomial in the derivatives of the
	elements of $\mathbf{H}$.
\end{lemma}

\begin{proof}
	Throughout this proof, $\mathrm{f}$ is a smooth function that 
	can vary from line to line and satisfies $\mathrm{f}(0) = 0$
	(except that the functions $\mathrm{f}_{\vec{J}_1,\cdots,\vec{J}_M}$ on
	RHS~\eqref{E:ALGEBRAICEXPRESSIONFORICOMMUTEDGEOMETRICDIVERGENCEOFVORTICITY}
	do not necessarily satisfy $\mathrm{f}_{\vec{J}_1,\cdots,\vec{J}_M}(0)=0$).
	Moreover, $\mathbf{H}$ and $\mathbf{E}$
	are as defined in \eqref{E:HYPERBOLICVARS} and \eqref{E:ELLIPTICVARS}.
	
	We first prove \eqref{E:ALGEBRAICHYPERBOLICONLYSMOOTHFUNCTIONEXPRESSION}. 
	We must show that the elements of \eqref{E:ALGEBRAICVARSHYPERBOLICONLY}
	can be written as smooth functions of the elements of \eqref{E:HYPERBOLICVARS} that vanish at $0$.
	We first note that by the normalization condition $\upeta_{\kappa \lambda} u^{\kappa} u^{\lambda} = -1$,
	$u^0 - 1$ is a smooth function of the spatial components of $u$ that vanishes when $u^1=u^2=u^3 = 0$.
	From this fact and
	the identity $u^{\kappa} \GradEnt_{\kappa} = 0$ (see \eqref{E:VELOCITYANDENTGRADIENTAREMINKOWSKIPERP}),
	we deduce that $\GradEnt^0$ is a smooth function of the spatial components of $u$ and $\GradEnt$
	that vanishes at $0$.
	A similar result holds for $\vort^0$ by virtue of \eqref{E:VORTISORTHGONALTOU}.
	Next, we note that, in view of the above discussion and the
	discussion surrounding equation \eqref{E:A0DET},
	we can solve for the time derivatives of $\Lnenth$, $\Ent$, and $u^{\alpha}$ in terms of their spatial derivatives.
	Thus far, we have shown that
	$u^0 - 1, \vort^0, \GradEnt^0, \partial_t \Lnenth, \partial_t u^{\alpha}, \partial_a u^0, \partial_t \Ent$
	can be expressed as $\mathrm{f}(\mathbf{H})$. In the rest of the proof, we will use these facts
	without explicitly mentioning them every time. Next, we use
	definitions \eqref{E:UORTHGONALVORTICITYOFONEFORM} and \eqref{E:MODIFIEDVORTICITYOFVORTICITY} to deduce that 
	$u^{\kappa} \mathcal{C}_{\kappa} = \mathrm{f}(\mathbf{H})$.
	Using this equation to algebraically solve for $\mathcal{C}^0$, 
	we deduce that $\mathcal{C}^0 = \mathrm{f}(\mathbf{H})$, as desired.
	We will now show that
	$(G^{-1})^{cd} \partial_c \GradEnt_d = \mathrm{f}(\mathbf{H})$.
	To begin, we use definition \eqref{E:MODIFIEDDIVERGENCEOFENTROPY}
	to deduce that
	$
	\partial_i \GradEnt^i
	=
	\partial_\alpha \GradEnt^\alpha - \partial_t \GradEnt^0
	=
		n \mathcal{D} 
		- 
		\GradEnt^\kappa \partial_\kappa \Lnenth 
		+ 
		\speed^{-2} \GradEnt^\kappa \partial_\kappa \Lnenth 
		- 
		\partial_t \GradEnt^0
	= 
	\mathrm{f}(\mathbf{H})
	-
	\partial_t \GradEnt^0
	$.
	Next, using the identity $\partial_t = \frac{u^{\kappa} \partial_{\kappa}}{u^0} - \frac{u^i \partial_i}{u^0}$
	and the evolution equation \eqref{E:MAINTHMENTROPYGRADIENTEVOLUTION}
	with $\alpha = 0$,
	we find that
	$
	\partial_t \GradEnt^0
	=
	\mathrm{f}(\mathbf{H})
	- 
	\frac{u^i \partial_i \GradEnt^0}{u^0}
	$.
	Moreover, using \eqref{E:VELOCITYANDENTGRADIENTAREMINKOWSKIPERP},
	we find that $\GradEnt^0 = \frac{\GradEnt_j u^j}{u^0}$,
	from which we deduce that
	$
	\frac{u^i \partial_i \GradEnt^0}{u^0}
	= 
		\mathrm{f}(\mathbf{H})
		+
	 \frac{u^i u^j \partial_i \GradEnt_j }{(u^0)^2}
	$.
	Combining the above calculations, we find that
	$
	\partial_i \GradEnt^i
	-
	\frac{u^i u^j \partial_i \GradEnt_j }{(u^0)^2}
	=
	\mathrm{f}(\mathbf{H})
	$
	which, in view of definition \eqref{E:INVERSERIEMANNIANONSIGMAT},
	yields the desired relation $(G^{-1})^{cd} \partial_c \GradEnt_d = \mathrm{f}(\mathbf{H})$.
	The relation $(G^{-1})^{cd} \partial_c \vort_d  = \mathrm{f}(\mathbf{H})$
	can be proved using a similar argument based on equations \eqref{E:MAINTHMVORTICITYTRANSPORT}
	and \eqref{E:MAINTHMDIVOFVORTICITY}, and we omit the details.
	To show that
	$\curlthreed^a(\underline{\vort}) = \mathrm{f}(\mathbf{H})$,
	we first note that by definition
	\eqref{E:STANDARDCURL},
	it suffices to show that $\partial_i \vort_j - \partial_j \vort_i = \mathrm{f}(\mathbf{H})$
	for $i,j=1,2,3$.
	To proceed, we use \eqref{E:VANDUAREORTHOGONALANTISYMMETRICPARTOFGRADIENTVINTERMSOFVORTICITYANDUPARTOFV}
	with $V := \vort$ 
	(which is applicable in view of \eqref{E:VORTISORTHGONALTOU}),
	definition \eqref{E:MODIFIEDVORTICITYOFVORTICITY},
	and the transport equation \eqref{E:MAINTHMVORTICITYTRANSPORT}
	to deduce that
	$
	\partial_i \vort_j - \partial_j \vort_i 
	= 
	\upepsilon_{i j \gamma \delta} u^{\gamma} \uperpvort^{\delta}(\vort)
	+
	u_{j} u^{\kappa} \partial_{\kappa} \vort_{i}
	-
	u_{i} u^{\kappa} \partial_{\kappa} \vort_{j}
	+
	\mathrm{f}(\mathbf{H})
	=
	\mathrm{f}(\mathbf{H})$,
	which is the desired result.
	The fact that
	$
	\curlthreed^a(\underline{\GradEnt})
	=
	0
	=
	\mathrm{f}(\mathbf{H})
	$
	is a trivial consequence of the symmetry property \eqref{E:ENTSYMMETRYOFMIXEDPARTIALS}
	and definition \eqref{E:STANDARDCURL}.
	We have therefore proved \eqref{E:ALGEBRAICHYPERBOLICONLYSMOOTHFUNCTIONEXPRESSION}.

	We now prove \eqref{E:ALGEBRAICHYPERBOLICANDELLIPTICSMOOTHFUNCTIONEXPRESSION}.
	We must show that elements of \eqref{E:ALGEBRAICVARSHYPERBOLICELLIPTIC}
	can be written as smooth functions of the elements of \eqref{E:HYPERBOLICVARS}
	and the elements of \eqref{E:ELLIPTICVARS} that vanish at $0$.
	To handle
	$
	\partial_t \vort_i
	$,
	we use the identity 
	$\partial_t = \frac{u^{\kappa} \partial_{\kappa}}{u^0} - \frac{u^j \partial_j}{u^0}$ 
	and the transport equation \eqref{E:MAINTHMVORTICITYTRANSPORT}
	to deduce that
	$
	\partial_t \vort_i
	=
	\frac{u^{\kappa} \partial_{\kappa} \vort_i}{u^0}
	+
	\mathrm{f}(\mathbf{H},\mathbf{E})
	=
	\mathrm{f}(\mathbf{H},\mathbf{E})
	$
	as desired.
	To handle
	$\partial_t \vort_0$,
	we simply use \eqref{E:VORTISORTHGONALTOU}
	to obtain the identity $\vort^0 = \frac{\vort_j u^j}{u^0}$,
	differentiate this identity with respect to $\partial_t$,
	and then use the already proven facts that $\vort_j$, $u^{\alpha} - \updelta_0^{\alpha}$, 
	and their time derivatives are equal to $\mathrm{f}(\mathbf{H},\mathbf{E})$.
	Similarly, by differentiating the identity 
	$\vort^0 = \frac{\vort_j u^j}{u^0}$
	with $\partial_a$,
	we conclude that $\partial_a \vort_0 = \mathrm{f}(\mathbf{H},\mathbf{E})$.
	The relations
	$\partial_t \GradEnt_{\alpha} = \mathrm{f}(\mathbf{H},\mathbf{E})$
	and
	$\partial_a \GradEnt_0 = \mathrm{f}(\mathbf{H},\mathbf{E})$
	can be proved
	using a similar argument based on
	equations \eqref{E:VELOCITYANDENTGRADIENTAREMINKOWSKIPERP}
	and \eqref{E:MAINTHMENTROPYGRADIENTEVOLUTION},
	and we omit the details.
	The facts that
	$\partial_b \vort^b = \mathrm{f}(\mathbf{H},\mathbf{E})$
	and
	$\partial_b \GradEnt^b = \mathrm{f}(\mathbf{H},\mathbf{E})$
	follow trivially from the definitions.
	We have therefore proved \eqref{E:ALGEBRAICHYPERBOLICANDELLIPTICSMOOTHFUNCTIONEXPRESSION}.
	
	\eqref{E:ALLALGEBRAICSMOOTHFUNCTIONEXPRESSION} then follows from definition \eqref{E:ALLALGEBRAICVARIABLES}
	and \eqref{E:ALGEBRAICHYPERBOLICONLYSMOOTHFUNCTIONEXPRESSION}-\eqref{E:ALGEBRAICHYPERBOLICANDELLIPTICSMOOTHFUNCTIONEXPRESSION}.
	
	To prove \eqref{E:ALGEBRAICEXPRESSIONFORICOMMUTEDGEOMETRICDIVERGENCEOFVORTICITY},
	we first note that definition \eqref{E:ALGEBRAICVARSHYPERBOLICONLY}
	and \eqref{E:ALGEBRAICHYPERBOLICONLYSMOOTHFUNCTIONEXPRESSION} 
	imply that
	$(G^{-1})^{ab} \partial_a \vort_b$,
	$(G^{-1})^{ab} \partial_a \GradEnt_b$,
	and
	$\curlthreed^i(\underline{\vort})$
	are all of the form $\mathrm{f}(\mathbf{H})$.
	Hence, \eqref{E:ALGEBRAICEXPRESSIONFORICOMMUTEDGEOMETRICDIVERGENCEOFVORTICITY}
	follows from the Leibniz and chain rules
	and the definition \eqref{E:HYPERBOLICVARS} of $\mathbf{H}$.

\end{proof}

\subsection{Elliptic estimates and the corresponding energies}
\label{SS:ENERGYNORMTIEDTOELLITPICESTIMATES}
In this subsection, we construct the energies
that we will use to control the top-order derivatives of 
the vorticity and entropy; see Def.\,\ref{D:ELLIPTICNORM}.
The proof that the energies are coercive
relies on elliptic estimates; see the proof of Lemma~\ref{L:COMPARISONBASEDONELLIPTIC}.
We start by defining a bilinear form on the relevant Hilbert space of functions.
Lemma~\ref{L:COMPARISONBASEDONELLIPTIC} shows that the bilinear form
induces a norm on the Hilbert space.

\begin{definition}[A new Hilbert space inner product]
\label{D:INNERPRODUCT}
Let $(\underline{\vort},\underline{\GradEnt})$
denote the array of spatial components of the vorticity and entropy gradient 
(i.e., the $\upeta$-orthogonally projection of $(\vort,\GradEnt)$ onto $\Sigma_t$, as in Subsubsect.\,\ref{SSS:NOTATION}).
Let $\upalpha > 0$ be a parameter and let $M^{-1}(t,\cdot)$ be an inverse Riemannian metric on $\Sigma_t$.
We define the following bilinear form
on the corresponding Hilbert space $\left(H^N(\Sigma_t) \right)^3 \times \left(H^N(\Sigma_t) \right)^3$:
\begin{align} \label{E:INNERPRODUCT}
	&
	\left
	\langle 
		\left(\underline{\vort},\underline{\GradEnt} \right), 
		\left(\underline{\widetilde{\vort}},\underline{\widetilde{\GradEnt}} \right) 
	\right\rangle_{M^{-1};\upalpha}(t)
		\\
	& : = 
			\upalpha
			\sum_{|\vec{I}|=N-1}
			\int_{\Sigma_t}
				\left\lbrace
					(M^{-1})^{ab} \partial_a \partial_{\vec{I}} \vort_b
				\right\rbrace
				\left\lbrace
					(M^{-1})^{cd} \partial_c \partial_{\vec{I}} \widetilde{\vort}_d
				\right\rbrace
			\, dx
			\notag	\\
		& \ \
			+
			\upalpha
			\sum_{|\vec{I}|=N-1}
			\int_{\Sigma_t}
				\left\lbrace
					(M^{-1})^{ab} \partial_a \partial_{\vec{I}} \GradEnt_b
				\right\rbrace
				\left\lbrace
					(M^{-1})^{cd} \partial_c \partial_{\vec{I}} \widetilde{\GradEnt}_d
				\right\rbrace
			\, dx
				\notag \\
		& \ \
			+
			\upalpha
			\sum_{|\vec{I}|=N-1}
			\int_{\Sigma_t}
				(M^{-1})^{ab} 
				(M^{-1})^{cd} 
				\upepsilon_{aci} 
				\upepsilon_{bdj}
				\curlthreed^i(\partial_{\vec{I}} \underline{\vort}) 
				\curlthreed^j
				(\partial_{\vec{I}} \underline{\widetilde{\vort}})
			\, dx
			\notag \\
		& 
			\	\
			+
			\upalpha
			\sum_{|\vec{I}|=N-1}
			\int_{\Sigma_t}
			 (M^{-1})^{ab} 
				(M^{-1})^{cd} 
				\upepsilon_{aci}  
				\upepsilon_{bdj}
				\curlthreed^i(\partial_{\vec{I}} \underline{\GradEnt}) 
				\curlthreed^j(\partial_{\vec{I}} \underline{\widetilde{\GradEnt}})
			\, dx
				\notag \\
		& \ \
			+
			\sum_{|\vec{I}| \leq N-1}	
			\int_{\Sigma_t}
				\updelta^{ab}
				(\partial_{\vec{I}} \vort_a)
				(\partial_{\vec{I}} \widetilde{\vort}_b)
			\, dx
			+
			\sum_{|\vec{I}| \leq N-1}	
			\int_{\Sigma_t}
				\updelta^{ab}
				(\partial_{\vec{I}} \GradEnt_a)
				(\partial_{\vec{I}} \widetilde{\GradEnt}_b)
			\, dx,
			\notag
	\end{align}
	where $\updelta^{ab}$ is the standard Kronecker delta
	and $\upepsilon_{abc}$ is the fully antisymmetric symbol normalized by
	$\upepsilon_{123} = 1$.
\end{definition}

We now define the family of energies that we will use to
control the top-order derivatives of the vorticity and entropy.

\begin{definition}[``Elliptic'' energy]
	\label{D:ELLIPTICNORM}
	Let $N \geq 3$ be an integer, let $\upalpha > 0$ be a parameter (below we will choose it to be sufficiently small), 
	and let $M^{-1}(t,\cdot)$ be a $C^1$ inverse Riemannian metric on $\Sigma_t$.
	We define the square of the ``elliptic'' energy 
	$\mathbb{E}_{N;M^{-1};\upalpha}[(\underline{\vort},\underline{\GradEnt})] 
	=
	\mathbb{E}_{N;M^{-1};\upalpha}[(\underline{\vort},\underline{\GradEnt})](t)
	\geq 0$ as follows:
	\begin{align} \label{E:TOPORDERENERGY}
		\mathbb{E}_{N;M^{-1};\upalpha}^2[(\underline{\vort},\underline{\GradEnt})](t)
		& := 
		\left\langle 
			\left(\underline{\vort},\underline{\GradEnt}\right),\left(\underline{\vort},\underline{\GradEnt}\right) 
		\right\rangle_{M^{-1};\upalpha}(t).
	\end{align}
	\end{definition}

In the next lemma, with the help of elliptic estimates,
we exhibit the coercivity of 
$\mathbb{E}_{N;M^{-1};\upalpha}[(\underline{\vort},\underline{\GradEnt})](t)$.
The lemma shows in particular that if $\upalpha > 0$ is sufficiently small 
(depending on the inverse Riemannian metric $M^{-1}$), 
then the bilinear form from
Def.\,\ref{D:INNERPRODUCT} is a Hilbert space inner product.

\begin{lemma}[Energy-norm comparison estimate based on elliptic estimates]
	\label{L:COMPARISONBASEDONELLIPTIC}
	Let $T > 0$, and let $M^{-1} = M^{-1}(t,x)$ be an inverse Riemannian metric
	defined for $(t,x) \in [0,T] \times \mathbb{T}^3$.
	Let $\uplambda$ be the infimum of the eigenvalues of the $3 \times 3$ matrix 
	$(M^{-1})^{ij}(t,x)$ over $(t,x) \in [0,T] \times \mathbb{T}^3$, 
	and let $\Lambda$ be the supremum of the eigenvalues of the $3 \times 3$ matrix 
	$(M^{-1})^{ij}(t,x)$ over $(t,x) \in [0,T] \times \mathbb{T}^3$,
	and assume that $0 < \uplambda \leq \Lambda < \infty$.
	Let $\mathbb{E}_{N;M^{-1};\upalpha}[(\underline{\vort},\underline{\GradEnt})]$ 
	be as in Def.\,\ref{D:ELLIPTICNORM}.
	There exist a small constant $\upalpha_* > 0$ and a large constant $C > 0$
	such that $\upalpha_*^{-1}$ and $C$
	depend continuously in an increasing fashion on 
	\textbf{i)} $\max_{i,j=1,2,3} \| (M^{-1})^{ij} \|_{C\big([0,T],C^1(\mathbb{T}^3)\big)}$;
	\textbf{ii)} $\Lambda$;
	and \textbf{iii)} $\uplambda^{-1}$, 
	such that the following comparison estimates hold
	for $t \in [0,T]$:
	\begin{subequations}
	\begin{align} \label{E:ENERGYNORMCOMPARISON}
			&
			\mathbb{E}_{N;M^{-1};\upalpha_*}[(\underline{\vort},\underline{\GradEnt})](t)
			\leq
			C
			\sum_{a=1}^3
			\left\|
				\vort^a
			\right\|_{H^N(\Sigma_t)}
			+
			C
			\sum_{a=1}^3
			\left\|
				\GradEnt_a
			\right\|_{H^N(\Sigma_t)},
				\\
			&
			\sum_{a=1}^3
			\left\|
				\vort^a
			\right\|_{H^N(\Sigma_t)}
			+
			\sum_{a=1}^3
			\left\|
				\GradEnt_a
			\right\|_{H^N(\Sigma_t)}
			\leq 
			C \mathbb{E}_{N;M^{-1};\upalpha_*}[(\underline{\vort},\underline{\GradEnt})](t).
			\label{E:SECONDENERGYNORMCOMPARISON}
	\end{align}
	\end{subequations}
\end{lemma}

\begin{proof}
	We prove only \eqref{E:SECONDENERGYNORMCOMPARISON}
	since \eqref{E:ENERGYNORMCOMPARISON} can be proved using similar but simpler arguments.
	Throughout the proof, $C > 0$ denotes a constant with the dependence-properties stated in the lemma.
	To proceed, we note the following divergence identity
	for one-forms $V$ on $\Sigma_t$, 
	which can be directly verified:
	\begin{align} \label{E:NONFLATELLIPTICIDENTITY}
		&
		(M^{-1})^{ab}
		(M^{-1})^{cd}
		(\partial_a V_b)
		(\partial_c V_d)
			\\
		& 
		\ \
		+
		\overbrace{
		\frac{1}{2}
		(M^{-1})^{ab}
		(M^{-1})^{cd}
		\upepsilon_{aci}
		\upepsilon_{bdj}
		\curlthreed^i(V)
		\curlthreed^j(V)}^{\frac{1}{2}
		(M^{-1})^{ab}
		(M^{-1})^{cd}
		(\partial_a V_c - \partial_c V_a)
		(\partial_b V_d - \partial_d V_b)}
		\notag	\\
	& =
		(M^{-1})^{ab}
		(M^{-1})^{cd}
		(\partial_a V_c)
		(\partial_b V_d)
			\notag \\
		& 
		\	\
		+ 
		\frac{1}{2} 
		\left\lbrace
			\partial_a 
			\left[
			(M^{-1})^{ab}
			(M^{-1})^{cd}
			\right]
		\right\rbrace
		\left[
			V_c
			\partial_b V_d
			+
			V_c
			\partial_d V_b
		\right]
			\notag \\
		& \ \
		+ 
		\frac{1}{2} 
		\left\lbrace
			\partial_c 
			\left[
			(M^{-1})^{ab}
			(M^{-1})^{cd}
			\right]
		\right\rbrace
		\left[
			V_a
			\partial_b V_d
			+
			V_a
			\partial_d V_b
		\right]
		\notag
			\\
		& \ \
		-
		\frac{1}{2} 
		\left\lbrace
			\partial_b
			\left[
			(M^{-1})^{ab}
			(M^{-1})^{cd}
			\right]
		\right\rbrace
		\left[
		V_a
		\partial_c V_d
		+
		V_c
		\partial_a V_d
		\right]
		\notag \\
		& \ \
		-
		\frac{1}{2} 
		\left\lbrace
			\partial_d
			\left[
			(M^{-1})^{ab}
			(M^{-1})^{cd}
			\right]
		\right\rbrace
		\left[
			V_a
			\partial_c V_b
			+
			V_c
			\partial_a V_b
		\right]
		\notag
		\\
		& \ \
		+
		\frac{1}{2} 
		\partial_b 
		\left\lbrace
			(M^{-1})^{ab}
			(M^{-1})^{cd}
			\left[
			V_a
			\partial_c V_d
			+
			V_c
			\partial_a V_d
			\right]
		\right\rbrace
		\notag 
			\\
		& \ \
		+
		\frac{1}{2} 
		\partial_d 
		\left\lbrace
			(M^{-1})^{ab}
			(M^{-1})^{cd}
			\left[
			V_a
			\partial_c V_b
			+
			V_c
			\partial_a V_b
			\right]
		\right\rbrace
		\notag
		\\
		& 
		\ \
		-
		\frac{1}{2} 
		\partial_a 
		\left\lbrace
			(M^{-1})^{ab}
			(M^{-1})^{cd}
			\left[
			V_c
			\partial_b V_d
			+
			V_c
			\partial_d V_b
			\right]
		\right\rbrace
			\notag 
			\\
		& \ \
		-
		\frac{1}{2} 
		\partial_c 
		\left\lbrace
			(M^{-1})^{ab}
			(M^{-1})^{cd}
			\left[
			V_a
			\partial_b V_d
			+
			V_a
			\partial_d V_b
			\right]
		\right\rbrace.
			\notag 
\end{align}
	We now integrate \eqref{E:NONFLATELLIPTICIDENTITY} over $\Sigma_t$ with respect to $dx$
	and note that the integrals of the last four (perfect spatial derivative) terms on the right-hand side vanish.
	In view of our assumptions on the eigenvalues of $(M^{-1})^{ij}(t,\cdot)$, 
	we see that the integral of the first term
	$
	(M^{-1})^{ab}
		(M^{-1})^{cd}
		(\partial_a V_c)
		(\partial_b V_d)
	$
	on RHS~\eqref{E:NONFLATELLIPTICIDENTITY}
	is 
	$
	\geq 
	\uplambda^2
	\sum_{a,b=1}^3
		\left\|
			\partial_a V_b
		\right\|_{L^2(\Sigma_t)}^2
	$.
	Also using Young's inequality, we see that
	the integrals of the second through fifth terms
	on RHS~\eqref{E:NONFLATELLIPTICIDENTITY}
	(in which a derivative falls on $M^{-1}$)
	are collectively bounded from below by
	$\geq 
		-
		\frac{\uplambda^2}{2}
		\sum_{a,b=1}^3
		\left\|
			\partial_a V_b
		\right\|_{L^2(\Sigma_t)}^2
		-
		\frac{C}{\uplambda^2}
		\sum_{a=1}^3
		\left\|
			V_a
		\right\|_{L^2(\Sigma_t)}^2
		$.
	It follows that the integral of
	\eqref{E:NONFLATELLIPTICIDENTITY}
	is bounded from below by
	\[
	\geq 
		\frac{\uplambda^2}{2}
		\sum_{a,b=1}^3
		\left\|
			\partial_a V_b
		\right\|_{L^2(\Sigma_t)}^2
		-
		\frac{C}{\uplambda^2}
		\sum_{a=1}^3
		\left\|
			V_a
		\right\|_{L^2(\Sigma_t)}^2.
	\]
	The desired estimate \eqref{E:SECONDENERGYNORMCOMPARISON} now follows from	
	these considerations with $\underline{\vort}$ and $\underline{\GradEnt}$ in the role of $V$,
	and definitions \eqref{E:INNERPRODUCT} and \eqref{E:TOPORDERENERGY},
	where $\upalpha := \upalpha_* > 0$ is chosen so that
	$
	\upalpha_* \frac{C}{\uplambda^2} = \frac{1}{2}
	$,
	and $\frac{C}{\uplambda^2}$ is the (absolute value of the)
	coefficient from the previous inequality.
	We clarify that, by our conventions,
	factors of $\frac{1}{\uplambda^2}$
	can be absorbed into the constant $C$
	on RHS~\eqref{E:SECONDENERGYNORMCOMPARISON}.
	
\end{proof}

In the next lemma, we show that some Sobolev norms of the elliptic variables
$\mathbf{E}$ can be bounded by corresponding Sobolev norms of the hyperbolic
variables $\mathbf{H}$. We also derive related estimates for the difference of two solutions.
The main ingredients in the proofs are the elliptic estimates provided by Lemma~\ref{L:COMPARISONBASEDONELLIPTIC}.

\begin{lemma}[Controlling Sobolev norms of the elliptic variables in terms of Sobolev norms of the hyperbolic variables] 
		\label{L:SOBOLEVNORMSOFELLIPTICINTERMSOFHYPERBOLIC}\hfill

\medskip		

\noindent \textbf{(A)}. 
	Let
	$\mathring{\Lnenth} := \Lnenth|_{\Sigma_0}$,
	$\mathring{\Ent} := \Ent|_{\Sigma_0}$,
	and
	$\mathring{u}^i := u^i|_{\Sigma_0}$
	be initial data
	for the relativistic Euler equations
	\eqref{E:ENTHALPYEVOLUTION}-\eqref{E:ENTROPYEVOLUTION} + \eqref{E:UISUNITLENGTH},
	let $\mathring{\vort}^i := \vort^i|_{\Sigma_0}$,
	and let $(\Lnenth,\Ent,u^{\alpha})$ be the solution
	provided by Theorem~\ref{T:STANDLOCALWELLPOSEDNESS}.
	In particular, let $N \geq 3$ be an integer,
	let $[0,T] \times\mathbb{T}^3$
	be the slab of existence provided by the theorem, and let
	$\mathfrak{K}$ be the set featured in 
	Theorem ~\ref{T:STANDLOCALWELLPOSEDNESS}.
	Assume in addition that the rectangular components of 
	the initial data are elements of $C^{\infty}(\Sigma_0)$,
	and note that by Theorem~\ref{T:STANDLOCALWELLPOSEDNESS}
	and the Sobolev embedding result \eqref{E:STANDARDSOBOLEVEMBEDDING},
	the rectangular components of the solution belong to $C^{\infty}([0,T] \times \mathbb{T}^3)$.
	Let $\mathbf{E}$ and $\mathbf{H}$
	be the corresponding elliptic and hyperbolic variables as defined in 
	Def.\,\ref{D:CLASSIFICATIONOFSOLUTIONVARIABLES}.
	Then there exists a constant $C > 0$, depending only on: 
	\begin{enumerate}
	\item $N$ 
	\item $\mathfrak{K}$
	\item
	$\norm{\mathring{\Lnenth}}_{H^{N}(\Sigma_0)}
		+
		\norm{\mathring{\Ent}}_{H^{N+1}(\Sigma_0)}
		+
		\sum_{a=1}^3 \norm{\mathring{u}^a}_{H^{N}(\Sigma_0)}
		+
		\sum_{a=1}^3 \norm{\mathring{\vort}^a}_{H^{N}(\Sigma_0)}
		$
	\item  
	\begin{align*}
		&
		\| \Lnenth \|_{C\big([0,T],C^1(\mathbb{T}^3)\big)}
		+
		\| \Ent \|_{C\big([0,T],C^1(\mathbb{T}^3)\big)}
		+
		\sum_{a=1}^3 \| u^a \|_{C\big([0,T],C^1(\mathbb{T}^3)\big)}
			\\
	& \ \
		+
		\sum_{a=1}^3 \| \GradEnt^a \|_{C\big([0,T],C^1(\mathbb{T}^3)\big)}
		+
		\sum_{a=1}^3 \| \vort^a \|_{C\big([0,T],C^1(\mathbb{T}^3)\big)},
	\end{align*}
	\end{enumerate}
	such that the following estimate holds for $t \in [0,T]$:
	\begin{align} \label{E:KEYINHOMOGENEOUSTERMSINGLEELLIPTICVARIABLESESTIMATE}
		\norm{\mathbf{E}}_{H^{N-1}(\Sigma_t)} \leq C \norm{\mathbf{H}}_{H^{N-1}(\Sigma_t)}.
	\end{align}
	
	\medskip
	
	\textbf{(B)}. For $i=1,2$, let $(\Lnenth_{(i)},\Ent_{(i)},u_{(i)})$ be classical solutions
	to the relativistic Euler equations
	\eqref{E:ENTHALPYEVOLUTION}-\eqref{E:ENTROPYEVOLUTION} + \eqref{E:UISUNITLENGTH}
	that have the properties stated in part (A). Assume that the 
	slab of existence $[0,T] \times\mathbb{T}^3$ is the same for both solutions
	and that the set $\mathfrak{K}$ is the same for both solutions,
	that is, that there exists a compact set $\mathfrak{K} \subset \mbox{\upshape int} \mathcal{H}$ 
	such that for $i=1,2$, we have
	$(\Lnenth_{(i)},\Ent_{(i)},u_{(i)}^1,u_{(i)}^2,u_{(i)}^3)([0,T] \times \mathbb{T}^3) \subset \mbox{\upshape int} \mathfrak{K}$.
	Let $\mathbf{E}_{(i)}$ and $\mathbf{H}_{(i)}$ be the corresponding 
	elliptic and hyperbolic variables as defined in 
	Def.\,\ref{D:CLASSIFICATIONOFSOLUTIONVARIABLES}.
	Then there exist constants $\updelta > 0$ and $C > 0$, 
	depending only on:
	\begin{enumerate}
	\item $N$ 
	\item $\mathfrak{K}$
	\item
	\begin{align*}
	\sum_{i=1}^2
	\Big\lbrace
	\norm{\mathring{\Lnenth}_{(i)}}_{H^{N}(\Sigma_0)}
		+
		\norm{\mathring{\Ent}_{(i)}}_{H^{N+1}(\Sigma_0)}
		+
		\sum_{a=1}^3 \norm{\mathring{u}_{(i)}^a}_{H^{N}(\Sigma_0)}
		+
		\sum_{a=1}^3 \norm{\mathring{\vort}_{(i)}^a}_{H^{N}(\Sigma_0)}
	\Big\rbrace
	\end{align*}
	\item  
	\begin{align*}
		&
		\sum_{i=1}^2
		\Big\lbrace
		\| \Lnenth_{(i)} \|_{C\big([0,T],C^1(\mathbb{T}^3)\big)}
		+
		\| \Ent_{(i)} \|_{C \big([0,T],C^1(\mathbb{T}^3)\big)}
		+
		\sum_{a=1}^3 \| u_{(i)}^a \|_{C \big([0,T],C^1(\mathbb{T}^3)\big)}
			\\
	& \ \
		+
		\sum_{a=1}^3 \| \GradEnt_{(i)}^a \|_{C \big([0,T],C^1(\mathbb{T}^3)\big)}
		+
		\sum_{a=1}^3 \| \vort_{(i)}^a \|_{C \big([0,T],C^1(\mathbb{T}^3)\big)}
		\Big\rbrace
		\end{align*}
	\item 
	\begin{align*}
	&
	\sum_{i=1}^2
	\Big\lbrace
	\norm{\Lnenth_{(i)}}_{C\big([0,T],H^N(\mathbb{T}^3)\big)}
	+
	\sum_{a=1}^3
	\norm{u_{(i)}^a}_{C\big([0,T],H^N(\mathbb{T}^3)\big)}
	\\
	& 
	\ \
	+
	\norm{\Ent_{(i)}}_{C\big([0,T],H^{N+1}(\mathbb{T}^3)\big)}  
	+
	\sum_{a=1}^3
	\norm{\vort_{(i)}^a}_{C\big([0,T],H^N(\mathbb{T}^3)\big)}
	\Big\rbrace,
	\end{align*}
\end{enumerate}
such that 
if $\| \mathbf{H}_{(1)} - \mathbf{H}_{(2)} \|_{C(\Sigma_t)} \leq \updelta$,
then the following estimate holds for $t \in [0,T]$:
\begin{align} \label{E:KEYINHOMOGENEOUSTERMELLIPTICVARIABLESESTIMATE}	
	\|\mathbf{E}_{(1)} - \mathbf{E}_{(2)} \|_{H^{N-1}(\Sigma_{t})}
	& \leq
		C \|\mathbf{H}_{(1)} - \mathbf{H}_{(2)} \|_{H^{N-1}(\Sigma_{t})}.
\end{align}	
\end{lemma}

\begin{proof}
Throughout this proof, $C$ denotes a constant with the dependence-properties
stated in the lemma. We begin by establishing \eqref{E:KEYINHOMOGENEOUSTERMSINGLEELLIPTICVARIABLESESTIMATE}.
Invoking definitions 
\eqref{E:HYPERBOLICVARS},
\eqref{E:ELLIPTICVARS},
\eqref{E:INNERPRODUCT}, 
and \eqref{E:TOPORDERENERGY},
using the fact that $\curlthreed(\underline{\GradEnt}) = 0$ (see \eqref{E:ENTSYMMETRYOFMIXEDPARTIALS}),
and using the estimate \eqref{E:SECONDENERGYNORMCOMPARISON} with $M^{-1} := G^{-1}$
and with $\upalpha_* > 0$ as in the statement of Lemma~\ref{L:COMPARISONBASEDONELLIPTIC}
(where $G^{-1}$ is defined in Def.\,\ref{D:INVERSERIEMANNIANONSIGMAT}, and we stress that 
the proof of \eqref{E:SECONDENERGYNORMCOMPARISON} relied on elliptic estimates),
we find that
\begin{align} 
		\|\mathbf{E} \|_{H^{N-1}(\Sigma_t)}
		& \leq
		C \mathbb{E}_{N;G^{-1};\upalpha_*}[(\underline{\vort},\underline{\GradEnt})](t)
		\label{E:GDIVERGENCEVARIABLESINTERMSOFHYPERBOLICVARIABLES}
			\\
		& 
		\leq
			C
			\sum_{|\vec{I}|=N-1}
			\left\|
				(G^{-1})^{ab} \partial_a \partial_{\vec{I}} \vort_b 
			\right\|_{L^2(\Sigma_t)}
			\notag \\
		& \ \
			+
			C
			\sum_{|\vec{I}|=N-1}
			\left\|
				(G^{-1})^{ab} \partial_a \partial_{\vec{I}} \GradEnt_b 
			\right\|_{L^2(\Sigma_t)}	
				\notag \\
		& \ \
			+
			C
			\sum_{|\vec{I}|=N-1}
			\sum_{a=1}^3
			\left\|
				\curlthreed^a(\partial_{\vec{I}} \underline{\vort})
			\right\|_{L^2(\Sigma_t)}
				\notag \\
	& \ \
			+
			C \|\mathbf{H}  \|_{H^{N-1}(\Sigma_t)}.
			\notag
\end{align}
Next, using \eqref{E:ALGEBRAICEXPRESSIONFORICOMMUTEDGEOMETRICDIVERGENCEOFVORTICITY},
we see that the terms 
$(G^{-1})^{ab} \partial_a \partial_{\vec{I}} \vort_b$,
$(G^{-1})^{ab} \partial_a \partial_{\vec{I}} \GradEnt_b$,
and
$\curlthreed^a(\partial_{\vec{I}} \underline{\vort})$
on RHS~\eqref{E:GDIVERGENCEVARIABLESINTERMSOFHYPERBOLICVARIABLES}
are smooth functions of $\mathbf{H}$ and its spatial derivatives.
Thus, using inequality \eqref{E:WITHSMOOTHFUNCTIONFRAMENORML2PRODUCTBOUNDINERMSOFLINFINITYANDHSDOT} 
to bound RHS~\eqref{E:ALGEBRAICEXPRESSIONFORICOMMUTEDGEOMETRICDIVERGENCEOFVORTICITY}
in the norm 
$\| \cdot \|_{L^2(\Sigma_t)}$,
we arrive at the desired estimate \eqref{E:KEYINHOMOGENEOUSTERMSINGLEELLIPTICVARIABLESESTIMATE}.
We stress that RHS~\eqref{E:WITHSMOOTHFUNCTIONFRAMENORML2PRODUCTBOUNDINERMSOFLINFINITYANDHSDOT} 
is \emph{linear} in the order $r$ derivatives of the solution; this is the reason
that RHS~\eqref{E:KEYINHOMOGENEOUSTERMSINGLEELLIPTICVARIABLESESTIMATE}
is linear in $\norm{\mathbf{H}}_{H^{N-1}(\Sigma_t)}$.

We now prove \eqref{E:KEYINHOMOGENEOUSTERMELLIPTICVARIABLESESTIMATE}.
For $i=1,2$, we let $G_{(i)}^{-1}$ denote the inverse Riemannian metric corresponding to the $i^{th}$ solution,
that is, the inverse Riemannian metric 
whose rectangular components are formed by evaluating RHS~\eqref{E:INVERSERIEMANNIANONSIGMAT}
at the solution corresponding to the labeling index $i$.
To proceed, we use definitions 
\eqref{E:HYPERBOLICVARS},
\eqref{E:ELLIPTICVARS},
\eqref{E:INNERPRODUCT}, 
and \eqref{E:TOPORDERENERGY},
the fact that $\curlthreed(\underline{\GradEnt}_{(1)}) = \curlthreed(\underline{\GradEnt}_{(2)}) = 0$ (see \eqref{E:ENTSYMMETRYOFMIXEDPARTIALS}), 
and the comparison estimate \eqref{E:SECONDENERGYNORMCOMPARISON} with 
$M^{-1} := G_{(1)}^{-1}$
and with $\upalpha_* > 0$ as in the statement of Lemma~\ref{L:COMPARISONBASEDONELLIPTIC}
to deduce that
\begin{align} \label{E:FIRSTSTEPKEYINHOMOGENEOUSTERMELLIPTICVARIABLESDIFFERENCEESTIMATE}
		&
		\|\mathbf{E}_{(1)} - \mathbf{E}_{(2)} \|_{H^{N-1}(\Sigma_t)}
			\\
		& \leq
		C \mathbb{E}_{N;G_{(1)}^{-1};\upalpha_*}
		[(\underline{\vort}_{(1)} - \underline{\vort}_{(2)},\underline{\GradEnt}_{(1)} - \underline{\GradEnt}_{(2)})](t)
			\notag 
		\\
		& \leq C
			\sum_{|\vec{I}|=N-1}
			\left\|
				(G_{(1)}^{-1})^{ab} \partial_a \partial_{\vec{I}} 
				\left(
					\vort_{(1)b} - \vort_{(2)b} 
				\right)
			\right\|_{L^2(\Sigma_t)}
				\notag \\
		& \ \
			+
			C
			\sum_{|\vec{I}|=N-1}
			\left\|
				(G_{(1)}^{-1})^{ab} \partial_a \partial_{\vec{I}} 
				\left(
					\GradEnt_{(1)b} - \GradEnt_{(2)b}
				\right)
			\right\|_{L^2(\Sigma_t)}
			\notag \\
	& \ \
			+
			C
			\sum_{|\vec{I}|=N-1}
			\sum_{a=1}^3
			\left\|
				\curlthreed^a\left(
				\partial_{\vec{I}} 
					(\underline{\vort}_{(1)}
					-
					\underline{\vort}_{(2)})
				\right)
			\right\|_{L^2(\Sigma_t)}
			\notag \\
	& \ \
			+
			C \|\mathbf{H}_{(1)} - \mathbf{H}_{(2)} \|_{H^{N-1}(\Sigma_t)}.
			\notag
\end{align}
Next, using the triangle inequality, we find that
\begin{align} \label{E:SECONDSTEPKEYINHOMOGENEOUSTERMELLIPTICVARIABLESDIFFERENCEESTIMATE}
	&
	\mbox{RHS~\eqref{E:FIRSTSTEPKEYINHOMOGENEOUSTERMELLIPTICVARIABLESDIFFERENCEESTIMATE}}
		\\
	& \leq
	C\sum_{|\vec{I}|=N-1}
	\left\|
		(G_{(1)}^{-1})^{ab} \partial_a \partial_{\vec{I}} \vort_{(1)b} 
		- 
		(G_{(2)}^{-1})^{ab} \partial_a \partial_{\vec{I}} \vort_{(2)b}
	\right\|_{L^2(\Sigma_t)}
		\notag \\
& \ \
	+
	C\sum_{|\vec{I}|=N-1}
	\left\|
		(G_{(2)}^{-1})^{ab} - (G_{(1)}^{-1})^{ab}
	\right\|_{C(\Sigma_t)}
	\left\|
		\partial_a \partial_{\vec{I}} \vort_{(2)b}
	\right\|_{L^2(\Sigma_t)}
	\notag \\
	& 
	\ \
	+
	C\sum_{|\vec{I}|=N-1}
	\left\|
		(G_{(1)}^{-1})^{ab} \partial_a \partial_{\vec{I}} \GradEnt_{(1)b} 
		- 
		(G_{(2)}^{-1})^{ab} \partial_a \partial_{\vec{I}} \GradEnt_{(2)b}
	\right\|_{L^2(\Sigma_t)}
	\notag \\
	&
	\ \
	+ C
	\sum_{|\vec{I}|=N-1}
	\left\|
		(G_{(2)}^{-1})^{ab} - (G_{(1)}^{-1})^{ab}
	\right\|_{C(\Sigma_t)}
	\left\|
		 \partial_a \partial_{\vec{I}}  \GradEnt_{(2)b}
		\right\|_{L^2(\Sigma_t)}
	\notag \\
	& \ \
		+
			C
			\sum_{|\vec{I}|=N-1}
			\sum_{a=1}^3
			\left\|
				\curlthreed^a(\partial_{\vec{I}} \underline{\vort}_{(1)}) - \curlthreed^a(\partial_{\vec{I}} \underline{\vort}_{(2)})
			\right\|_{L^2(\Sigma_t)}
				\notag \\
	& \ \
		+
		C \|\mathbf{H}_{(1)} - \mathbf{H}_{(2)} \|_{H^{N-1}(\Sigma_t)}.
	\notag	
\end{align}
Using the assumed bounds 
$
\sum_{|\vec{I}|=N-1}
\sum_{a,b=1}^3
	\left\|
		\partial_a \partial_{\vec{I}} \vort_{(2)b}
	\right\|_{L^2(\Sigma_t)}
\leq C
$
and
$
\sum_{|\vec{I}|=N-1}	
\sum_{a,b=1}^3
\left\|
		 \partial_a \partial_{\vec{I}}  \GradEnt_{(2)b}
		\right\|_{L^2(\Sigma_t)}
\leq C
$,
\eqref{E:ALGEBRAICEXPRESSIONFORICOMMUTEDGEOMETRICDIVERGENCEOFVORTICITY},
\eqref{E:STANDARDSOBOLEVEMBEDDING},
and
\eqref{E:CONTINUITYESTIMATEWITHSMOOTHFUNCTIONFRAMENORML2PRODUCTBOUNDINERMSOFLINFINITYANDHMDOT}
(where the hypotheses needed to invoke \eqref{E:CONTINUITYESTIMATEWITHSMOOTHFUNCTIONFRAMENORML2PRODUCTBOUNDINERMSOFLINFINITYANDHMDOT} 
are satisfied if $\| \mathbf{H}_{(1)} - \mathbf{H}_{(2)} \|_{C(\Sigma_t)}$ is sufficiently small),
we see that the terms on the first, third, and fifth lines of RHS~\eqref{E:SECONDSTEPKEYINHOMOGENEOUSTERMELLIPTICVARIABLESDIFFERENCEESTIMATE}
are $\leq C \|\mathbf{H}_{(1)} - \mathbf{H}_{(2)} \|_{H^{N-1}(\Sigma_t)}$
as desired. To handle the terms on the second and fourth lines 
of RHS~\eqref{E:SECONDSTEPKEYINHOMOGENEOUSTERMELLIPTICVARIABLESDIFFERENCEESTIMATE},
we use the assumed bounds
\[
\sum_{|\vec{I}|=N-1}
\sum_{a,b=1}^3
	\left\|
		\partial_a \partial_{\vec{I}} \vort_{(2)b}
	\right\|_{L^2(\Sigma_t)}
\leq C,
	\,
\sum_{|\vec{I}|=N-1}	
\sum_{a,b=1}^3
\left\|
		 \partial_a \partial_{\vec{I}}  \GradEnt_{(2)b}
		\right\|_{L^2(\Sigma_t)}
\leq C,
\]
the mean value theorem estimate
$\left|(G_{(2)}^{-1})^{ab} - (G_{(1)}^{-1})^{ab} \right| \leq C \left|\mathbf{H}_{(1)} - \mathbf{H}_{(2)}\right|$ 
(where we are using that RHS~\eqref{E:INVERSERIEMANNIANONSIGMAT}
can be viewed as a smooth function of $(u^1,u^2,u^3)$),
and the Sobolev embedding result \eqref{E:STANDARDSOBOLEVEMBEDDING}
to deduce that the terms on the second and fourth lines of 
RHS~\eqref{E:SECONDSTEPKEYINHOMOGENEOUSTERMELLIPTICVARIABLESDIFFERENCEESTIMATE} 
are 
$
\leq 
C \|\mathbf{H}_{(1)} - \mathbf{H}_{(2)} \|_{C(\Sigma_t)}
\leq 
C \|\mathbf{H}_{(1)} - \mathbf{H}_{(2)} \|_{H^{N-1}(\Sigma_t)}$
as desired. We have therefore proved \eqref{E:KEYINHOMOGENEOUSTERMELLIPTICVARIABLESESTIMATE}. 
\end{proof}

\subsection{Energies for the wave equations via the vectorfield multiplier method}
\label{SS:ENERGIESFORWAVE}
In this subsection, we derive a priori estimates for our new formulation
of the relativistic Euler equations. The main result is provided by the next proposition.
The proposition shows in particular that the
vorticity and entropy are one degree more differentiable compared to the
standard estimates that follow from first-order formulations of the equations.
The main analytic tools in the proof of the proposition 
are the elliptic estimates from Subsect.\,\ref{SS:ENERGYNORMTIEDTOELLITPICESTIMATES}
and the vectorfield method for wave equations (see Subsubsect.\,\ref{SSS:ESTIMATESFORCOVARIANTWAVEEQUATIONS}).

\begin{proposition}[A priori estimates for solutions to the relativistic Euler equations]
	\label{P:APRIORIESTIMATES}
	Let 
	$\mathring{\Lnenth} := \Lnenth|_{\Sigma_0}$,
	$\mathring{\Ent} := \Ent|_{\Sigma_0}$,
	and
	$\mathring{u}^i := u^i|_{\Sigma_0}$
	be initial data for the relativistic Euler equations
	\eqref{E:ENTHALPYEVOLUTION}-\eqref{E:ENTROPYEVOLUTION} + \eqref{E:UISUNITLENGTH}
	obeying the assumptions of Theorem~\ref{T:STANDLOCALWELLPOSEDNESS},
	and let $(\Lnenth,\Ent,u^0,u^1,u^2,u^3)$ be the corresponding solution.
	In particular, let $N \geq 3$ be an integer,
	let $[0,T] \times\mathbb{T}^3$
	be the slab of existence provided by the theorem, and let
	$\mathfrak{K}$ be the set featured in theorem.
	Assume in addition that the components of the initial data relative to standard coordinates
	 belong to $C^{\infty}(\mathbb{T}^3)$
	and note that by Remark~\ref{R:CINFINITYDATAGIVECINFINITYSOLUTIONS},
	the solution components belong to $C^{\infty}([0,T] \times \mathbb{T}^3)$.
	Let $\vort$ be the vorticity (see definition~\ref{D:VORTICITYDEF}),
	and let $\mathring{\vort}^i := \vort^i|_{\Sigma_0}$ be its initial spatial components.
	
	Then there exists a constant $C > 0$, depending only on: 
	\begin{enumerate}
	\item $N$ 
	\item $\mathfrak{K}$
	\item
	$\norm{\mathring{\Lnenth}}_{H^{N}(\Sigma_0)}
		+
		\sum_{a=1}^3 \norm{\mathring{u}^a}_{H^{N}(\Sigma_0)}
		+
		\norm{\mathring{\Ent}}_{H^{N+1}(\Sigma_0)}
		+
		\sum_{a=1}^3 \norm{\mathring{\vort}^a}_{H^{N}(\Sigma_0)}$
	\item  
	\begin{align*}
		&
		\| \Lnenth \|_{C\big([0,T],C^1(\mathbb{T}^3)\big)}
		+
		\sum_{a=1}^3 \| u^a \|_{C\big([0,T],C^1(\mathbb{T}^3)\big)}
		+
		\| \Ent \|_{C\big([0,T],C^1(\mathbb{T}^3)\big)}
			\\
	& \ \
		+
		\sum_{a=1}^3 \| \GradEnt^a \|_{C\big([0,T],C^1(\mathbb{T}^3)\big)}
		+
		\sum_{a=1}^3 \| \vort^a \|_{C\big([0,T],C^1(\mathbb{T}^3)\big)}
	\end{align*}
	\end{enumerate}
	such that for $t \in [0,T]$, the components of the solution
	relative to the standard coordinates verify the following estimates:
	\begin{align}
	\label{E:APRIORIESTIMATE}
	&
	\norm{\Lnenth}_{H^{N}(\Sigma_t)} 
	+
	\sum_{a=0}^3
	\norm{u^{\alpha} - \updelta_0^{\alpha}}_{H^{N}(\Sigma_t)} 
	+
	\norm{\Ent}_{H^{N+1}(\Sigma_t)} 
		\\
	& \ \
	+
	\sum_{\alpha=0}^3
	\norm{\GradEnt^{\alpha}}_{H^N(\Sigma_t)} 
	+
	\sum_{\alpha=0}^3
	\norm{\vort^{\alpha}}_{H^{N}(\Sigma_t)}
		\notag \\
	& 
	\leq C \exp(Ct)
	\leq C \exp(CT)
	:= C_*,
	\notag
	\end{align}
	where $\updelta_0^{\alpha}$ is the Kronecker delta.
\end{proposition}

The proof of Prop.\,\ref{P:APRIORIESTIMATES} is located in Subsubsect.\,\ref{SSS:PROOFOFPROPAPRIORIESTIMATES}.
We will first derive some preliminary results.
We start by noting that we can rewrite
the spatial components of
\eqref{E:MAINTHMCOVARIANTWAVEENTHALPY},
\eqref{E:MAINTHMCOVARIANTWAVEVELOCITY},
\eqref{E:MAINTHMENTROPYGRADIENTEVOLUTION},
\eqref{E:MAINTHMVORTICITYTRANSPORT},
\eqref{E:MAINTHMSMODIFIEDDIVERGENCEENTROPYGRADIENTTRANSPORT},
and
\eqref{E:MAINTHMTRANSPORTFORMODIFIEDVORTICITYOFVORTICITY}	
in concise form as follows,
where $\ErrorTerm$ denotes a smooth function of its arguments that is free
to vary from line to line and that satisfies $\ErrorTerm(0) = 0$,
$\underline{V}$ denotes $\upeta$-orthogonal projection of $V$ onto constant-time hypersurfaces
(see Subsubsect.\,\ref{SSS:NOTATION}),
and the hyperbolic variables $\mathbf{H}$ and the elliptic variables $\mathbf{E}$
are as in Def.\,\ref{D:CLASSIFICATIONOFSOLUTIONVARIABLES}:
\begin{subequations}
\begin{align}
	\label{E:SYMBOLICMAINTHMCOVARIANTWAVEENTHALPY}
	\Square_g \Lnenth & = \ErrorTerm(\mathbf{H}),
\\
	\label{E:SYMBOLICMAINTHMCOVARIANTWAVEVELOCITY}
	\Square_g \underline{u} & = \ErrorTerm(\mathbf{H}),
\\
	\label{E:SYMBOLICMAINTHMENTROPYGRADIENTEVOLUTION}
	u^{\alpha} \partial_{\alpha} \underline{\GradEnt} & = \ErrorTerm(\mathbf{H}),
\\
	\label{E:SYMBOLICMAINTHMVORTICITYTRANSPORT}
	u^{\alpha} \partial_{\alpha} \underline{\vort} & = \ErrorTerm(\mathbf{H}),
\\
	\label{E:SYMBOLICMAINTHMSMODIFIEDDIVERGENCEENTROPYGRADIENTTRANSPORT}
	u^{\alpha} \partial_{\alpha} \mathcal{D} & = \ErrorTerm(\mathbf{H},\mathbf{E}),
\\
	\label{E:SYMBOLICMAINTHMTRANSPORTFORMODIFIEDVORTICITYOFVORTICITY}
	u^{\alpha} \partial_{\alpha} \underline{\mathcal{C}} & = \ErrorTerm(\mathbf{H},\mathbf{E}).
\end{align}
\end{subequations}

The crux of the proof of Prop.\,\ref{P:APRIORIESTIMATES}
is to derive energy estimates for the covariant wave equations
\eqref{E:SYMBOLICMAINTHMCOVARIANTWAVEENTHALPY} and \eqref{E:SYMBOLICMAINTHMCOVARIANTWAVEVELOCITY},
energy estimates for the transport equations
\eqref{E:SYMBOLICMAINTHMENTROPYGRADIENTEVOLUTION}, 
\eqref{E:SYMBOLICMAINTHMVORTICITYTRANSPORT},
\eqref{E:SYMBOLICMAINTHMSMODIFIEDDIVERGENCEENTROPYGRADIENTTRANSPORT},
and
\eqref{E:SYMBOLICMAINTHMTRANSPORTFORMODIFIEDVORTICITYOFVORTICITY},
and elliptic estimates
to handle the terms $\mathbf{E}$ on
RHSs~\eqref{E:SYMBOLICMAINTHMSMODIFIEDDIVERGENCEENTROPYGRADIENTTRANSPORT} and 
\eqref{E:SYMBOLICMAINTHMTRANSPORTFORMODIFIEDVORTICITYOFVORTICITY}.
We have already derived the necessary elliptic estimates in Subsect.\,\ref{SS:ENERGYNORMTIEDTOELLITPICESTIMATES}.
In the next three subsections, we will outline the energy estimates,
which are standard. 

\subsubsection{Energy estimates for covariant wave equations}
\label{SSS:ESTIMATESFORCOVARIANTWAVEEQUATIONS} 
The wave operator in \eqref{E:SYMBOLICMAINTHMCOVARIANTWAVEENTHALPY} and \eqref{E:SYMBOLICMAINTHMCOVARIANTWAVEVELOCITY} is with respect to the acoustical metric $g$ introduced in Definition \ref{D:ACOUSTICALMETRICANDINVERSE}. These are covariant
wave equations for the scalar quantities $\Lnenth$ and $u^\alpha$. Estimates for such
equations can be derived by using the well-known vectorfield multiplier method\footnote{In deriving a priori estimates, 
in addition to the multiplier method,
we will use only the simplest version of the vectorfield commutator method. 
Specifically, we will commute the equations only with the coordinate spatial derivative operators $\partial_{\vec{I}}$.} 
for wave equations, which we outline in this subsubsection. 

Let $\varphi$ be any element of $\lbrace \Lnenth,u^1,u^2,u^3 \rbrace$
(in practice, we will not need to derive separate energy estimates for $u^0$ since estimates
for $u^0$ can be obtained as a consequence of the estimates for the spatial components of $u$ and
the normalization condition $\upeta_{\kappa \lambda} u^{\kappa} u^{\lambda} = -1$).
We start by defining the energy-momentum tensor associated to a scalar function $\varphi$:
\begin{align}
	\label{E:ENERGYMOMENTUMTENSOR}
	T_{\alpha \beta}
	=
	T_{\alpha \beta}[\varphi]
	&: = (\partial_\alpha \varphi) (\partial_\beta \varphi)
	-\frac{1}{2} g_{\alpha\beta} (g^{-1})^{\mu \nu} (\partial_\mu \varphi) (\partial_\nu \varphi).
\end{align}
A crucial property of $T_{\alpha \beta}$ is that it satisfies 
the \emph{dominant energy condition}: 
$T_{\alpha \beta} X^\alpha Y^\beta \geq 0$ whenever the vectorfields $X$ and $Y$ 
are future-directed\footnote{By a ``future-directed'' vectorfield $X$, we mean that $X^0 > 0$.} 
and timelike\footnote{$X$ is defined to be timelike with respect to $g$ if
$g_{\alpha \beta} X^{\alpha} X^{\beta} < 0$.} with respect to $g$. In practice, the
dominant energy condition allows one to construct energies that are \emph{coercive}
along causal (with respect to $g$) hypersurfaces;\footnote{By a ``causal hypersurface,'' we mean a hypersurface whose
future-directed unit normal is either timelike with respect to $g$ or null with respect to $g$ at each point.}
see equation \eqref{E:ENERGYEXPLICITDENSITYCALCULATION} below for the energy 
that we use in deriving a priori estimates for $\Lnenth$ and $u$.

Next, for any vectorfield $X$ (soon to be employed in the role of a ``multiplier vectorfield''), 
we let $^{(X)} \mkern-1mu \pi$ be its deformation tensor relative to $g$,
which takes the following form relative to arbitrary coordinates:
\begin{align} \label{E:DEFTENSOR}
	{^{(X)} \mkern-1mu \pi_{\alpha\beta}} 
	& := g_{\beta\mu} \nabla_\alpha X^\mu + g_{\alpha\mu}\nabla_\beta X^\mu. 
\end{align}
In \eqref{E:DEFTENSOR} and in the rest of this subsubsection,
$\nabla$ is the covariant derivative induced by $g$. 
Next, we define the \emph{energy current} vectorfield corresponding to $X$ as follows:
\begin{align} \label{E:ENERGYCURRENT}
	{^{(X)} \mkern-4mu J^{\alpha}}
	=
	{^{(X)} \mkern-4mu J^{\alpha}}
	[\varphi]
	& := (g^{-1})^{\alpha \mu} T_{\mu \beta}[\varphi] X^\beta
		-
		X^{\alpha} \varphi^2.
\end{align}
From straightforward computations, we derive the following identity:
\begin{align} \label{E:DIVOFENERGYCURRENT}
	\nabla_\alpha {^{(X)} \mkern-4mu J^{\alpha}} 
	& =  
	(\square_g \varphi) \, X^\alpha \partial_\alpha \varphi
		+
		\frac{1}{2} (g^{-1})^{\alpha\gamma} (g^{-1})^{\beta\delta} T_{\alpha \beta} {^{(X)} \mkern-1mu \pi_{\gamma \delta}} 
			\\
	& \ \
		-
		(\nabla_{\alpha} X^{\alpha}) \varphi^2
		-
		2 \varphi (X^\alpha \partial_\alpha \varphi).
		\notag
\end{align}
Applying the divergence theorem on the spacetime slab $[0,T] \times \mathbb{T}^3$
and using \eqref{E:DIVOFENERGYCURRENT},
we deduce the following identity:
\begin{align} \label{E:DIVTHMGENERAL}
	& \int_{\Sigma_t} g_{\alpha\beta} {^{(X)} \mkern-4mu J^{\alpha}}[\varphi] \hat{N}^{\beta} \, d \mu_{\underline{g}}
		\\
	& =
	\int_{\Sigma_0} g_{\alpha\beta} {^{(X)} \mkern-4mu J^{\alpha}}[\varphi] \hat{N}^{\beta} \, d \mu_{\underline{g}}
		\notag \\
	& \ \
		-
		\int_{[0,t] \times \mathbb{T}^3}
			\left\lbrace
				(\square_g \varphi) \, X^\alpha \partial_\alpha \varphi
				+
				\frac{1}{2} (g^{-1})^{\alpha\gamma} (g^{-1})^{\beta\delta} T_{\alpha \beta} {^{(X)} \mkern-1mu \pi_{\gamma \delta}} 
			\right\rbrace
		\, d \mu_g
		\notag
			\\
	& \ \
		+
		\int_{[0,t] \times \mathbb{T}^3}
			\left\lbrace
				(\nabla_{\alpha} X^{\alpha}) \varphi^2
				+
				2 \varphi (X^\alpha \partial_\alpha \varphi)
			\right\rbrace
		\, d \mu_g.
		\notag
\end{align}
In \eqref{E:DIVTHMGENERAL},
$d \mu_g$ is the volume form that $g$ induces on [0,t] $\times \mathbb{T}^3$,
$\hat{N}$ is the future-directed unit normal to $\Sigma_t$ with respect to the metric $g$, 
and $d \mu_{\underline{g}}$ is the volume form that $\underline{g}$ induces on $\Sigma_t$,
where $\underline{g}$ is the first fundamental form of $\Sigma_t$, that is,
$\underline{g}_{ij} := g_{ij}$ for $1 \leq i,j \leq 3$.
We also note that relative to the standard coordinates,
$\hat{N}^{\alpha} = - \frac{(g^{-1})^{\alpha 0}}{\sqrt{|(g^{-1})^{00}|}}$,
$d \mu_g = \sqrt{|\mbox{\upshape det} g|} \, dx^1 dx^2 dx^3 dx^0$, 
and $d \mu_{\underline{g}} = \sqrt{\mbox{\upshape det} \underline{g}} \, dx^1 dx^2 dx^3
= \sqrt{|(g^{-1})^{00}|} \sqrt{|\mbox{\upshape det} g|} \, dx^1 dx^2 dx^3
$,
where the last equality is a basic linear algebraic identity.
Note that $\hat{N}$ is future-directed and timelike with respect to $g$,
and that we used the fact that $(g^{-1})^{00} < 0$ 
(which is a simple consequence of the formula \eqref{E:INVERSEACOUSTICALMETRIC} and our assumption that $0 < \speed \leq 1$).

From the above discussion, it follows that along any spacelike (with respect\footnote{A hypersurface is spacelike
with respect to $g$ if, at each point, its unit normal is timelike with respect to $g$.} to $g$)
hypersurface with future-directed unit normal $\hat{N}$,
we can construct a positive-definite energy density $g_{\alpha\beta} {^{(X)} \mkern-4mu J^{\alpha}}[\varphi] \hat{N}^{\beta}$
using any multiplier vectorfield $X$ that is future-directed and
timelike with respect to $g$. For the basic a priori estimates of interest to us, we will apply the above constructions 
along $\Sigma_t$ with $X := u$, which is future-directed  
timelike with respect to $g$.
As we described in Footnote~\ref{F:PARTIALTNOTTIMELIKE}, we cannot generally use 
$X:= \partial_t$ because $g(\partial_t,\partial_t) > 0$ can occur 
when $\sum_{a=1}^3|u^a|$ is large;
in contrast, note that by \eqref{E:ACOUSTICALMETRIC}
and the normalization condition $\upeta_{\kappa \lambda} u^{\kappa} u^{\lambda} = -1$,
we have $g_{\kappa \lambda} u^{\kappa} u^{\lambda} = -1$.
Thus, we define the following energy (where $\hat{N}^{\alpha} = - \frac{(g^{-1})^{\alpha 0}}{\sqrt{|(g^{-1})^{00}|}}$):
\begin{align} \label{E:WAVEEN}
	E_{\text{\upshape wave}}(t)
	=
	E_{\text{\upshape wave}}[\varphi](t)
	& := \int_{\Sigma_t} g_{\alpha \beta} {^{(u)} \mkern-4mu J^{\alpha}}[\varphi] \hat{N}^{\beta} \, d \mu_{\underline{g}}.
\end{align}
From \eqref{E:DIVTHMGENERAL},
definition \eqref{E:WAVEEN},
and the standard expansion\footnote{For example, 
$\nabla_{\alpha} X^{\beta} = \partial_{\alpha} X^{\beta} + \Gamma_{\alpha \ \gamma}^{\ \beta} X^{\gamma}$,
where $\Gamma_{\alpha \ \gamma}^{\ \beta}$ is defined by \eqref{E:CHRISTOFFELOFG}.} 
of covariant derivatives in terms of partial derivatives and Christoffel symbols
(which in particular can be used to derive the identity 
${^{(u)} \mkern-1mu \pi_{\alpha \beta}} 
	= u^{\kappa} \partial_{\kappa} g_{\alpha \beta}
		+
		g_{\alpha \kappa} \partial_{\beta} u^{\kappa}
		+
		g_{\beta \kappa} \partial_{\alpha} u^{\kappa}
	$),
we deduce the following energy identity relative to the standard coordinates:
\begin{align} \label{E:WAVEBASICENERGYID}
	E_{\text{\upshape wave}}[\varphi](t)
	& =
	E_{\text{\upshape wave}}[\varphi](0)
		-
		\int_{[0,t] \times \mathbb{T}^3}
			(\square_g \varphi) u^{\kappa} \partial_{\kappa} \varphi
		\, d \mu_g
			\\
		& \ \
		-
		\frac{1}{2}
		\int_{[0,t] \times \mathbb{T}^3}
				(g^{-1})^{\alpha \gamma}
				(g^{-1})^{\beta \delta}
				T_{\alpha \beta}[\varphi] 
				u^{\kappa} \partial_{\kappa} g_{\gamma \delta}
		\, d \mu_g
			\notag
				\\
		& \ \
		-
		\int_{[0,t] \times \mathbb{T}^3}
				(g^{-1})^{\beta \delta}
				T_{\alpha \beta}[\varphi] 
				\partial_{\delta} u^{\alpha}
		\, d \mu_g
			\notag
				\\
		& \ \
		+
		\int_{[0,t] \times \mathbb{T}^3}
			\left\lbrace
				(\partial_{\kappa} u^{\kappa}) \varphi^2
				+
				\Gamma_{\kappa \ \lambda}^{\ \kappa} u^{\lambda} \varphi^2
				+
				2 \varphi u^{\kappa} \partial_{\kappa} \varphi
			\right\rbrace
		\, d \mu_g.
		\notag
\end{align}
On RHS~\eqref{E:WAVEBASICENERGYID},
\begin{align} \label{E:CHRISTOFFELOFG}
	\Gamma_{\alpha \ \beta}^{\ \gamma}
	& :=
	\frac{1}{2}
	(g^{-1})^{\gamma \delta}
	\left\lbrace	
		\partial_{\alpha} g_{\delta \beta}
		+
		\partial_{\beta} g_{\alpha \delta}
		-
		\partial_{\delta} g_{\alpha \beta}
	\right\rbrace
\end{align}
are the Christoffel symbols of $g$ relative to the standard coordinates. 
Note that by \eqref{E:ACOUSTICALMETRIC}-\eqref{E:INVERSEACOUSTICALMETRIC} we have that
\begin{align} \label{E:CHRISTOFFELSCHEMATIC}
	\Gamma_{\alpha \ \beta}^{\ \gamma}
	& = 
	\mathrm{f}(\Lnenth,\Ent,u,\partial \Lnenth,\GradEnt,\partial u),
\end{align}
where $\mathrm{f}$ is a smooth function (depending on $\alpha$, $\beta$, and $\gamma$).

Next, with the help of \eqref{E:ACOUSTICALMETRIC}-\eqref{E:INVERSEACOUSTICALMETRIC}
and the normalization condition $\upeta_{\kappa \lambda} u^{\kappa} u^{\lambda} = -1$,
we compute that
\begin{align} \label{E:WAVEENERGYDENSITY}
	& g_{\alpha \beta} {^{(u)} \mkern-4mu J^{\alpha}}[\varphi] \hat{N}^{\beta}
		\\
	& =
		\left\lbrace
			\speed^2 T_{0 \beta}[\varphi] u^{\beta}
			+
			(1 - \speed^2) u^0 T_{\alpha \beta}[\varphi] u^{\alpha} u^{\beta}
			+
			u^0
			\varphi^2
		\right\rbrace
		\frac{1}{\sqrt{|(g^{-1})^{00}|}}
			\notag \\
	& = 
		\frac{
			\frac{1}{2} u^0
			\left\lbrace
			 \speed^2 (\partial_t \varphi)^2
			+
			\speed^2 \updelta^{ab} (\partial_a \varphi) \partial_b \varphi
			+
			(1 - \speed^2) (u^{\alpha} \partial_{\alpha} \varphi)^2
		\right\rbrace}{\sqrt{|(g^{-1})^{00}|}}
		\notag
			\\
	& \ \	
		+ 
		\frac{\left\lbrace
		\speed^2 (\partial_t \varphi) u^a \partial_a \varphi
		+
		u^0 \varphi^2
		\right\rbrace}{\sqrt{|(g^{-1})^{00}|}},
		\notag
\end{align}
where $\updelta^{ab}$ is the Kronecker delta.
From \eqref{E:WAVEEN} and \eqref{E:WAVEENERGYDENSITY}, it follows that
\begin{align} \label{E:ENERGYEXPLICITDENSITYCALCULATION}
	& E_{\text{\upshape wave}}[\varphi](t)
		\\
	& = 
	\frac{1}{2} 
	\int_{\Sigma_t}
		u^0
		\left\lbrace
			\speed^2 (\partial_t \varphi)^2
			+
			\speed^2 \updelta^{ab} (\partial_a \varphi) \partial_b \varphi
			+
			(1 - \speed^2) (u^{\alpha} \partial_{\alpha} \varphi)^2
		\right\rbrace
	\, \frac{d \mu_{\underline{g}}}{\sqrt{|(g^{-1})^{00}|}}
		\notag \\
	& \ \
		+
	\int_{\Sigma_t}
		\left\lbrace	
			\speed^2 (\partial_t \varphi) u^a \partial_a \varphi
			+
			u^0 \varphi^2
		\right\rbrace
		\, \frac{d \mu_{\underline{g}}}{\sqrt{|(g^{-1})^{00}|}}.
	\notag
\end{align}
The energy $E_{\text{\upshape wave}}[\varphi](t)$
will yield $L^2$ control of $\varphi$ and its first derivatives.
In Subsubsect.\,\ref{SSS:ENERGYNORMCOMPARISON}, we will establish the coerciveness 
$E_{\text{\upshape wave}}[\varphi](t)$.
To obtain $L^2$ control of the higher-order spatial derivatives of
$\varphi$, one can use energies of the form
$E_{\text{\upshape wave}}[\partial_{\vec{I}} \varphi]$,
where $\vec{I}$ is a spatial multi-index.

\subsubsection{Energy estimates for transport equations} 
\label{SSS:ENERGYESTIMATESFORTRANSPORTEQUATIONS}
One can derive energy estimates for transport equations of the form
$u^{\alpha} \partial_{\alpha} \varphi = f$
by relying on the following energy: 
\begin{align} \label{E:TRANSPORTENERGY}
	E_{\text{\upshape transport}}[\varphi](t) 
	& 
	:= \int_{\Sigma_t} \varphi^2 \, dx,
\end{align}
as in the proof of Lemma~\ref{L:CONTINUITYRESULTFORTRANSPORTEQUATIONS}.
The analog of the wave equation energy identity \eqref{E:WAVEBASICENERGYID}
is the following integral identity, whose simple proof follows from the
ideas featured in the proof of Lemma~\ref{L:CONTINUITYRESULTFORTRANSPORTEQUATIONS}:
\begin{align} \label{E:TRANSPORTEQUATIONENERGYID}
	E_{\text{\upshape transport}}[\varphi](t)
	& = 
	E_{\text{\upshape transport}}[\varphi](0)
	+
	\int_0^t
			\int_{\Sigma_{\tau}}
				\left\lbrace
					\partial_a
					\left(
						\frac{u^a}{u^0}
					\right)
				\right\rbrace
				\varphi^2
			\, dx
	\, d \tau
			\\
	& \	\
		+
		2
		\int_0^t
			\int_{\Sigma_{\tau}}
				\varphi
				\frac{u^{\alpha} \partial_{\alpha} \varphi}{u^0} 
			\, dx
		\, d \tau.
		\notag
	\end{align}
To control the higher-order derivatives of $\varphi$, one can rely on
energies of the form
$E_{\text{\upshape transport}}[\partial_{\vec{I}} \varphi]$.
We mention that the argument we have sketched here
relies on the basic fact that $u^0 > 0$, which allows us to divide by $u^0$ on RHS~\eqref{E:TRANSPORTEQUATIONENERGYID};
for the relativistic Euler equations, this fact
follows from the normalization condition $\upeta_{\kappa \lambda} u^{\kappa} u^{\lambda} = -1$
and the fact that $u$ is future-directed.

\subsubsection{Comparison of the energies with the Sobolev norm}
\label{SSS:ENERGYNORMCOMPARISON} 
The coerciveness properties of the wave equation energy 
$E_{\text{\upshape wave}}[\varphi](t)$
constructed in
Subsubsect.\,\ref{SSS:ESTIMATESFORCOVARIANTWAVEEQUATIONS}
are tied to the metric $g$; see \eqref{E:WAVEEN}.
In order to obtain our results, we need $E_{\text{\upshape wave}}[\varphi](t)$ 
to be uniformly comparable to a corresponding Sobolev norm along $\Sigma_t$.
More precisely, we need to ensure the existence of a constant $C>1$ such that
on the slab $[0,T] \times \mathbb{T}^3$ of existence guaranteed by Theorem~\ref{T:STANDLOCALWELLPOSEDNESS},
the following estimates hold:
\begin{align}
	C^{-1}
	\left\lbrace
		\norm{\varphi}^2_{H^N(\Sigma_t)} + \norm{\partial_\td \varphi}^2_{H^{N-1}(\Sigma_t)}
	\right\rbrace
	& \leq
	\sum_{0\leq |\vec{I}| \leq N-1} E_{\text{\upshape wave}}[\partial_{\vec{I}} \varphi](t)
	\label{E:COMPARISONSOBOLEVWAVE} \\
	&
	\leq 
	C
	\left\lbrace 
		\norm{\varphi}^2_{H^{N}(\Sigma_t)} + \norm{\partial_\td \varphi}^2_{H^{N-1}(\Sigma_t)}
	\right\rbrace.
	\notag
\end{align}
To see that such a constant $C$ exists, we first use Young's inequality, 
\eqref{E:UISUNITLENGTH},
and 
Cauchy--Schwarz to bound
the first product in braces on the last line of RHS~\eqref{E:ENERGYEXPLICITDENSITYCALCULATION} as
follows:
\begin{align} \label{E:SIMPLEYOUNGSINEQUALITY}
	&
	\speed^2 (\partial_t \varphi) u^a \partial_a \varphi
		\\
	& 
	\geq
	- 
	\frac{1}{2} \speed^2 \left(\sqrt{\sum_{i=1}^3 (u^i)^2} \right) (\partial_t \varphi)^2
	-
	\frac{1}{2} \speed^2 \left(\sqrt{\sum_{i=1}^3 (u^i)^2} \right) \updelta^{ab} (\partial_a \varphi) \partial_b \varphi
		\notag \\
	& 
	=
	- 
	\frac{1}{2} \speed^2 \left(\sqrt{(u^0)^2 - 1} \right)(\partial_t \varphi)^2
	-
	\frac{1}{2} \speed^2 \left(\sqrt{(u^0)^2 - 1}\right) \updelta^{ab} (\partial_a \varphi) \partial_b \varphi.
	\notag
\end{align}
Next, we recall that Theorem~\ref{T:STANDLOCALWELLPOSEDNESS} guarantees that on
$[0,T] \times \mathbb{T}^3$, the solution never escapes the compact subset $\mathfrak{K}$
featured in the statement of the theorem. In view of \eqref{E:SIMPLEYOUNGSINEQUALITY},
we see that this ensures that on $[0,T] \times \mathbb{T}^3$,
the product $\speed^2 (\partial_t \varphi) u^a \partial_a \varphi$ on the last line of RHS~\eqref{E:ENERGYEXPLICITDENSITYCALCULATION} 
can be absorbed into the sum
$\frac{1}{2} \speed^2 u^0 (\partial_t \varphi)^2
+
\frac{1}{2} \speed^2 u^0 \updelta^{ab} (\partial_a \varphi) \partial_b \varphi$ 
from the first line of RHS~\eqref{E:ENERGYEXPLICITDENSITYCALCULATION},
with room to spare. This implies that for solutions contained in $\mathfrak{K}$,
the integrands on RHS~\eqref{E:ENERGYEXPLICITDENSITYCALCULATION} are in total
uniformly comparable to $\sum_{\alpha = 0}^3 (\partial_{\alpha} \varphi)^2 + \varphi^2$.
This also ensures that on $[0,T] \times \mathbb{T}^3$,
the volume form $\displaystyle \frac{d \mu_{\underline{g}}}{\sqrt{|(g^{-1})^{00}|}}$ on $\Sigma_t$
is uniformly comparable\footnote{To see this, it is helpful to note the following identity,
which holds relative to the standard coordinates:
$\frac{d \mu_{\underline{g}}}{\sqrt{|(g^{-1})^{00}|}} = \speed^{-3} \, dx^1 dx^2 dx^3$.
This identity follows from \eqref{E:DETG} and the linear algebraic identity  
$\mbox{\upshape det} \underline{g} = (g^{-1})^{00} \mbox{\upshape det} g$.
}
to $dx := dx^1 dx^2 dx^3$. 
From these observations, it readily follows that a $C > 1$ exists such that
\eqref{E:COMPARISONSOBOLEVWAVE} holds.

\subsubsection{Proof of Prop.\,\ref{P:APRIORIESTIMATES}}
\label{SSS:PROOFOFPROPAPRIORIESTIMATES}
Recall that the assumptions of the proposition guarantee that
we have a smooth solution to the system \eqref{E:ENTHALPYEVOLUTION}-\eqref{E:ENTROPYEVOLUTION} + \eqref{E:UISUNITLENGTH}.
Consider the scalar component functions 
\begin{align}
	\label{E:LISTOFVARIABLES}
	\Lnenth, u^\alpha, \GradEnt^{\alpha}, \vort^{\alpha}, \mathcal{C}^{\alpha}, \mathcal{D},
\end{align}
introduced in Sect.\,\ref{S:SOMEBACKGROUND}. 
According to Theorem~\ref{T:NEWFORMULATIONRELEULER},
they satisfy the system of evolution equations given by equations
\eqref{E:MAINTHMCOVARIANTWAVEENTHALPY},
\eqref{E:MAINTHMCOVARIANTWAVEVELOCITY},
\eqref{E:MAINTHMENTROPYGRADIENTEVOLUTION},
\eqref{E:MAINTHMVORTICITYTRANSPORT},
\eqref{E:MAINTHMSMODIFIEDDIVERGENCEENTROPYGRADIENTTRANSPORT},
and
\eqref{E:MAINTHMTRANSPORTFORMODIFIEDVORTICITYOFVORTICITY}. 
Next, we recall that the 
hyperbolic quantities $\mathbf{H}$ and the elliptic quantities $\mathbf{E}$
were defined in Def.\,\ref{D:CLASSIFICATIONOFSOLUTIONVARIABLES}.
To prove the proposition, we claim that it suffices to show that
the following inequality holds for $t \in [0,T]$:
\begin{align} \label{E:KEYSTEPINRESTATEMENTOFAPRIORI}
	\norm{\mathbf{H}}_{H^{N-1}(\Sigma_t)}^2
	& 
	\leq 
	C \norm{\mathbf{H}}_{H^{N-1}(\Sigma_0)}^2 
	+ 
	C
	\int_0^t
		\norm{\mathbf{H}}_{H^{N-1}(\Sigma_\tau)}^2
	\, d \tau,
\end{align}
where in \eqref{E:KEYSTEPINRESTATEMENTOFAPRIORI} 
and in the rest of this proof, $C$ is as in the statement of Prop.\,\ref{P:APRIORIESTIMATES}.
For once we have shown \eqref{E:KEYSTEPINRESTATEMENTOFAPRIORI},
we can use Gronwall's inequality to deduce
(recalling that $C$ is allowed to depend on the initial data and can vary from line to line)
that the following estimate holds for $t \in [0,T]$:
\begin{align} \label{E:KEYSTEPINRESTATEMENTOFAPRIORIGRWONALLED}
	\norm{\mathbf{H}}_{H^{N-1}(\Sigma_t)}^2
	& 
	\leq 
	C \norm{\mathbf{H}}_{H^{N-1}(\Sigma_0)}^2 
	\exp(Ct)
	\leq C \exp(Ct)
	\leq C \exp(CT).
\end{align}
Then from 
\eqref{E:KEYINHOMOGENEOUSTERMSINGLEELLIPTICVARIABLESESTIMATE}
and
\eqref{E:KEYSTEPINRESTATEMENTOFAPRIORIGRWONALLED}
we conclude,
in view of Def.\,\ref{D:CLASSIFICATIONOFSOLUTIONVARIABLES},
the desired bound \eqref{E:APRIORIESTIMATE},
except for the estimates for $u^0$, $\GradEnt^0$, and $\vort^0$.
To obtain the desired estimate for these quantities, we 
first express
$u^0-1$, $\GradEnt^0$, $\vort^0$, $\partial_a u^0$, $\partial_a \GradEnt^0$, and $\partial_a \vort^0$
as $\mathrm{f}(\mathbf{H},\mathbf{E})$, with $\mathrm{f}$ smooth and satisfying $\mathrm{f}(0) = 0$
(this is possible in view of definition \eqref{E:ALLALGEBRAICVARIABLES} and \eqref{E:ALLALGEBRAICSMOOTHFUNCTIONEXPRESSION}).
We then use Lemma~\ref{L:STANDARDSOBOLEVRESULTS} 
to deduce that 
$\norm{\mathrm{f}(\mathbf{H},\mathbf{E})}_{H^{N-1}(\Sigma_t)}\leq C \norm{\mathbf{H}}_{H^{N-1}(\Sigma_t)} + C \norm{\mathbf{E}}_{H^{N-1}(\Sigma_t)}$.
Finally, we use the elliptic estimate \eqref{E:KEYINHOMOGENEOUSTERMSINGLEELLIPTICVARIABLESESTIMATE}
and \eqref{E:KEYSTEPINRESTATEMENTOFAPRIORIGRWONALLED} to 
conclude that 
$C \norm{\mathbf{H}}_{H^{N-1}(\Sigma_t)} + C \norm{\mathbf{E}}_{H^{N-1}(\Sigma_t)} \leq \mbox{RHS~\eqref{E:APRIORIESTIMATE}}$,
which yields the desired estimates.

It remains for us to prove \eqref{E:KEYSTEPINRESTATEMENTOFAPRIORI}.
We start by noting that the results described 
in Subsubsects.\,\ref{SSS:ESTIMATESFORCOVARIANTWAVEEQUATIONS}-\ref{SSS:ENERGYNORMCOMPARISON}
can be used to derive the following estimates,
where we recall that $\underline{V}$ denotes the spatial components of $V$
(i.e., the $\upeta$-orthogonal projection of $V$ onto constant-time hypersurfaces, as in Subsubsect.\,\ref{SSS:NOTATION}):
\begin{align}
	\label{E:PRELIMESTIMATEENTH}
	\norm{\Lnenth}_{H^{N}(\Sigma_t)}^2 + \norm{\partial_\td \Lnenth}_{H^{N-1}(\Sigma_t)}^2 
	& \leq  
	C \left\lbrace
		\norm{\Lnenth}_{H^{N}(\Sigma_0)}^2 + \norm{\partial_\td \Lnenth}_{H^{N-1}(\Sigma_0)}^2
	\right\rbrace
		\\
	& \ \
	+ 
	C
	\int_0^t
		\norm{\mathbf{H}}_{H^{N-1}(\Sigma_\tau)}^2
	\, d \tau,
		\notag \\
	\label{E:PRELIMESTIMATEVELOCITY}
	\norm{\underline{u}}_{H^{N}(\Sigma_t)}^2
	+
	\norm{\partial_t \underline{u}}_{H^{N}(\Sigma_t)}^2
	& \leq 
	C 
	\left\lbrace
		\norm{\underline{u}}_{H^{N}(\Sigma_0)}^2
		+
		\norm{\partial_t \underline{u}}_{H^{N}(\Sigma_t)}^2
	\right\rbrace
		\\
	& \ \
	+
	C
	\int_0^t
		\norm{\mathbf{H}}_{H^{N-1}(\Sigma_\tau)}^2
	\, d \tau,
	\notag \\
	\norm{\underline{\GradEnt}}_{H^{N-1}(\Sigma_t)}^2 
	& \leq 
	C \norm{\underline{\GradEnt}}_{H^{N-1}(\Sigma_0)}^2
		\label{E:PRELIMESTIAMTEGRADENT}
		\\
	& \ \
		+ 
	C
	\int_0^t
		\norm{\mathbf{H}}_{H^{N-1}(\Sigma_\tau)}^2
	\, d \tau,
	\notag \\
	\norm{\underline{\vort}}_{H^{N-1}(\Sigma_t)}^2 
	& \leq  
		C	\norm{\underline{\vort}}_{H^{N-1}(\Sigma_0)}^2
		\label{E:PRELIMESTIMATEVORT}
		\\
	& \ \
		+
	C
	\int_0^t
		\norm{\mathbf{H}}_{H^{N-1}(\Sigma_\tau)}^2
	\, d \tau,
	\notag
	\\
		\norm{\mathcal{D}}_{H^{N-1}(\Sigma_t)}^2
	& \leq 
	C \norm{\mathcal{D}}_{H^{N-1}(\Sigma_0)}^2
		\label{E:PRELIMESTIMATEMODIFIEDVORTVORT} 
		\\
& \ \
	+ 
	C
	\int_0^t
		\left\lbrace
			\norm{\mathbf{H}}_{H^{N-1}(\Sigma_\tau)}^2
			+
			\norm{\mathbf{E}}_{H^{N-1}(\Sigma_\tau)}^2
		\right\rbrace
	\, d \tau,
	\notag	\\
	\norm{\underline{\mathcal{C}}}_{H^{N-1}(\Sigma_t)}^2 
	& \leq
	C \norm{\underline{\mathcal{C}}}_{H^{N-1}(\Sigma_0)}^2 
		\label{E:PRELIMESTIAMTEMODIFIEDDIVGRADENT}	
		\\
	& \ \
	+
	C
	\int_0^t
		\left\lbrace
			\norm{\mathbf{H}}_{H^{N-1}(\Sigma_\tau)}^2
			+
			\norm{\mathbf{E}}_{H^{N-1}(\Sigma_\tau)}^2
		\right\rbrace
	\, d \tau.
		\notag
	\end{align}
The estimates \eqref{E:PRELIMESTIMATEENTH}-\eqref{E:PRELIMESTIAMTEMODIFIEDDIVGRADENT}
are standard and can be derived by commuting the evolution equations of
Theorem~\ref{T:NEWFORMULATIONRELEULER} 
(more precisely, only the evolution equations for the spatial components of $u$, $\vort$, $\GradEnt$, and $\mathcal{C}$)
with spatial derivative operators $\partial_{\vec{I}}$ and using the energy identities
\eqref{E:WAVEBASICENERGYID} and \eqref{E:TRANSPORTEQUATIONENERGYID} 
(and their analogs for the $\partial_{\vec{I}}-$differentiated solution variables),
the coerciveness estimate \eqref{E:COMPARISONSOBOLEVWAVE},
Lemma~\ref{L:EXPRESSIONSFORALGEBRAICQUANTITIESINTERMSOFHYPERBOLICANDELLIPTIC},
and the Sobolev--Moser-type estimate \eqref{E:WITHSMOOTHFUNCTIONFRAMENORML2PRODUCTBOUNDINERMSOFLINFINITYANDHSDOT}.
We stress that RHS~\eqref{E:WITHSMOOTHFUNCTIONFRAMENORML2PRODUCTBOUNDINERMSOFLINFINITYANDHSDOT} 
is \emph{linear} in the order $r$ derivatives of the solution; this is the reason
the integrands on RHS~\eqref{E:PRELIMESTIMATEENTH}-\eqref{E:PRELIMESTIAMTEMODIFIEDDIVGRADENT}
are quadratic in $\norm{\mathbf{H}}_{H^{N-1}(\Sigma_\tau)}$ and $\norm{\mathbf{E}}_{H^{N-1}(\Sigma_\tau)}$
(the sup-norm factors on RHS~\eqref{E:WITHSMOOTHFUNCTIONFRAMENORML2PRODUCTBOUNDINERMSOFLINFINITYANDHSDOT}
can be bounded by $\leq C$ since those factors are among the quantities that constants $C$ are allowed to depend on).
The non-standard aspect of the remaining part of the proof is the appearance of 
the term $\norm{\mathbf{E}}_{H^{N-1}(\Sigma_\tau)}^2$
on RHSs~\eqref{E:PRELIMESTIMATEMODIFIEDVORTVORT}-\eqref{E:PRELIMESTIAMTEMODIFIEDDIVGRADENT};
we clarify that these terms are generated by the terms 
$\partial_a \GradEnt_b$
and $\partial_a \vort_b$
on RHSs~\eqref{E:MAINTHMSMODIFIEDDIVERGENCEENTROPYGRADIENTTRANSPORT}
and
\eqref{E:MAINTHMTRANSPORTFORMODIFIEDVORTICITYOFVORTICITY}
(see definition~\eqref{E:ELLIPTICVARS}).
Next, adding \eqref{E:PRELIMESTIMATEENTH}-\eqref{E:PRELIMESTIAMTEMODIFIEDDIVGRADENT}
and appealing to Def.\,\ref{D:CLASSIFICATIONOFSOLUTIONVARIABLES},
we deduce that
\begin{align} \label{E:HYPERBOLICVARIABLESCONTROLLEDINTERMSOFHYPERBOLICANDELLIPTIC}
	\norm{\mathbf{H}}_{H^{N-1}(\Sigma_t)}^2
	& 
	\leq 
	C \norm{\mathbf{H}}_{H^{N-1}(\Sigma_0)}^2
		\\
& \ \
	+ 
	C
	\int_0^t
		\left\lbrace
			\norm{\mathbf{H}}_{H^{N-1}(\Sigma_\tau)}^2
			+
			\norm{\mathbf{E}}_{H^{N-1}(\Sigma_\tau)}^2
		\right\rbrace
	\, d \tau.
	\notag
\end{align}
Finally, from \eqref{E:HYPERBOLICVARIABLESCONTROLLEDINTERMSOFHYPERBOLICANDELLIPTIC}
and the elliptic estimate \eqref{E:KEYINHOMOGENEOUSTERMSINGLEELLIPTICVARIABLESESTIMATE},
we conclude the desired bound \eqref{E:KEYSTEPINRESTATEMENTOFAPRIORI}.
\hfill $\qed$

\subsection{Proof of Theorem~\ref{T:UPGRADEDLOCALWELLPOSEDNESS}}
\label{SS:PROOFOFTHEOREMUPGRADEDLOCALWELLPOSEDNESS}
We now prove Theorem~\ref{T:UPGRADEDLOCALWELLPOSEDNESS}, 
which is the main result of Sect.\,\ref{S:WELLPOSEDNESS}.
By Theorem~\ref{T:STANDLOCALWELLPOSEDNESS},
we need only to show that \textbf{i)} under the regularity assumptions on the initial data stated in Theorem~\ref{T:UPGRADEDLOCALWELLPOSEDNESS},
the standard local well-posedness results \eqref{E:FIRSTTHMSTANDARDCONTINUITYINHN}-\eqref{E:NOTYETGAINOFONEDERIVATIVE}
can be upgraded to \eqref{E:STANDARDCONTINUITYINHN}-\eqref{E:GAINOFONEDERIVATIVE}
and \textbf{ii)} that the solution depends continuously
on the initial data, where continuity is measured in the norms
corresponding to the function spaces featured in 
\eqref{E:STANDARDCONTINUITYINHN}-\eqref{E:GAINOFONEDERIVATIVE}.
Throughout this proof, 
$\mathfrak{K}$ denotes the set
featured in the statement of Theorem~\ref{T:STANDLOCALWELLPOSEDNESS}.
To proceed, we let $(\mathring{\Lnenth}_{(m)},\mathring{\Ent}_{(m)},\mathring{u}_{(m)}^i) \subset \left( C^{\infty}(\mathbb{T}^3) \right)^5$
	be a sequence of smooth initial data such that as $m \to \infty$, 
	we have
	\begin{align}
		\left\|
			\mathring{\Lnenth}_{(m)} 
			- 
			\mathring{\Lnenth}
		\right\|_{H^N(\Sigma_0)}
		& \rightarrow 0,
		&
		\left\|
			\mathring{u}_{(m)}^i
			- 
			\mathring{u}^i
		\right\|_{H^N(\Sigma_0)}
		& \rightarrow 0,
			\label{E:STANDARDDATACONVERGENCE} \\
		\left\|
			\mathring{\Ent}_{(m)} 
			- 
			\mathring{\Ent} 
		\right\|_{H^{N+1}(\Sigma_0)}
		& \rightarrow 0,
		&
		\left\|
			\mathring{\vort}_{(m)}^i 
			- 
			\mathring{\vort}^i
		\right\|_{H^N(\Sigma_0)}
		& \rightarrow 0,
		\label{E:GAINOFONEDERIVATIVEDATACONVERGENCE}
	\end{align}
	where $\mathring{\vort}_{(m)}^i$ denotes the initial vorticity of the $m^{th}$ element of the sequence and
	$\mathring{\vort}^i$ is as in the statement of the theorem.
	Let $(\Lnenth_{(m)},\Ent_{(m)},u_{(m)}^{\alpha},\GradEnt_{(m)}^{\alpha},\vort_{(m)}^{\alpha})$ 
	denote the corresponding sequence of solution variables. Theorem~\ref{T:STANDLOCALWELLPOSEDNESS} yields 
	(see, for example, \cite{jS2008a}, for additional details) that for $m$ sufficiently large, the element
	$(\Lnenth_{(m)},\Ent_{(m)},u_{(m)}^{\alpha})$ is a $C^{\infty}$
	classical solution to equations \eqref{E:ENTHALPYEVOLUTION}-\eqref{E:ENTROPYEVOLUTION} + \eqref{E:UISUNITLENGTH}
	on the fixed slab $[0,T] \times \mathbb{T}^3$ 
	with 
	\[
	(\Lnenth_{(m)}(p),\Ent_{(m)}(p),u_{(m)}^1(p),u_{(m)}^2(p),u_{(m)}^3(p)) \in \mbox{\upshape int} \mathfrak{K}
	\]
	for $p \in [0,T] \times \mathbb{T}^3$,
	and that on the same slab,
	$(\Lnenth_{(m)},\Ent_{(m)},u_{(m)}^{\alpha},\GradEnt_{(m)}^{\alpha},\vort_{(m)}^{\alpha})$
	is a $C^{\infty}$ solution to the equations of Theorem~\ref{T:NEWFORMULATIONRELEULER}
	(which are consequences of \eqref{E:ENTHALPYEVOLUTION}-\eqref{E:ENTROPYEVOLUTION} + \eqref{E:UISUNITLENGTH}).
	Moreover, Theorem~\ref{T:STANDLOCALWELLPOSEDNESS} also implies
	that the sequence converges to the solution in the following norms as $m \to \infty$:
	\begin{align}
		\left\|
			\Lnenth_{(m)} 
			- 
			\Lnenth
		\right\|_{C\big([0,T],H^N(\mathbb{T}^3)\big)}
		& \rightarrow 0,
			\label{E:1STANDARDITERATECONVERGENCE} 
			\\
		\left\|
			u_{(m)}^{\alpha}
			- 
			u^{\alpha}
		\right\|_{C\big([0,T],H^N(\mathbb{T}^3)\big)}
		& \rightarrow 0,
			\label{E:2STANDARDITERATECONVERGENCE}  \\
		\left\|
			\Ent_{(m)} 
			- 
			\Ent
		\right\|_{C\big([0,T],H^N(\mathbb{T}^3)\big)}
		& \rightarrow 0,
		\label{E:3STANDARDITERATECONVERGENCE} 
		\\
		\left\|
			\GradEnt_{(m)}^{\alpha} 
			- 
			\GradEnt^{\alpha}
		\right\|_{C\big([0,T],H^{N-1}(\mathbb{T}^3)\big)}
		& \rightarrow 0,
			\label{E:1ANOTHERSTANDARDITERATECONVERGENCE}
			\\
		\left\|
			\vort_{(m)}^{\alpha} 
			- 
			\vort^{\alpha}
		\right\|_{C\big([0,T],H^{N-1}(\mathbb{T}^3)\big)}
		& \rightarrow 0.
		\label{E:2ANOTHERSTANDARDITERATECONVERGENCE}
	\end{align}

Next, we use the convergence results \eqref{E:1STANDARDITERATECONVERGENCE}-\eqref{E:2ANOTHERSTANDARDITERATECONVERGENCE},
Theorem~\ref{T:STANDLOCALWELLPOSEDNESS},
and the a priori estimates provided by Prop.\,\ref{P:APRIORIESTIMATES}
to deduce that exist 
a constant $C > 0$, depending on $T$ and on the four types of quantities listed just above \eqref{E:APRIORIESTIMATE},
and a positive integer $m_0$
such that
\begin{align} \label{E:1ITERATESBOUNDEDINHIGHNORM}
		\sup_{m \geq m_0}
		\sup_{\tau \in [0,T]}
		\| \Ent_{(m)} \|_{H^{N+1}(\Sigma_{\tau})}
		& \leq C,
			\\
		\sup_{m \geq m_0}
		\sup_{\tau \in [0,T]}
		\| \GradEnt_{(m)}^{\alpha} \|_{H^N(\Sigma_{\tau})}
		& \leq C,
			\label{E:2ITERATESBOUNDEDINHIGHNORM} \\
		\sup_{m \geq m_0}
		\sup_{\tau \in [0,T]}
		\| \vort_{(m)}^{\alpha} \|_{H^N(\Sigma_{\tau})}
		& \leq C.
		\label{E:3ITERATESBOUNDEDINHIGHNORM}
\end{align}	
Since $H^r(\mathbb{T}^3)$ is a Hilbert space for $r \in \mathbb{R}$, 
it follows from the norm-boundedness results
\eqref{E:1ITERATESBOUNDEDINHIGHNORM}-\eqref{E:3ITERATESBOUNDEDINHIGHNORM} that for each $\tau \in [0,T]$,
there exist subsequences
$\Ent_{(m_n)}$, 
$\GradEnt_{(m_n)}^{\alpha}$,
and $\vort_{(m_n)}^{\alpha}$
that weakly converge in $H^{N+1}(\Sigma_{\tau})$, $H^N(\Sigma_{\tau})$, and $H^N(\Sigma_{\tau})$
respectively as $n \to \infty$. Moreover, since the norm is weakly lower semicontinuous
in a Hilbert space, it follows that the limits are bounded, respectively, in the norms
$\| \cdot \|_{H^{N+1}(\Sigma_{\tau})}$,
$\| \cdot \|_{H^N(\Sigma_{\tau})}$,
and  
$\| \cdot \|_{H^N(\Sigma_{\tau})}$, 
by $\leq C$, where $C$ is the same constant found on RHSs~\eqref{E:1ITERATESBOUNDEDINHIGHNORM}-\eqref{E:3ITERATESBOUNDEDINHIGHNORM}.
From \eqref{E:1ANOTHERSTANDARDITERATECONVERGENCE}-\eqref{E:2ANOTHERSTANDARDITERATECONVERGENCE}, 
it follows that the limits must be
$\Ent$, $\GradEnt^{\alpha}$, and $\vort^{\alpha}$ respectively.
We have therefore shown that
\begin{align} 
		\sup_{\tau \in [0,T]}
		\| \Ent \|_{H^{N+1}(\Sigma_{\tau})}
		& \leq C,
			\label{E:1GAININDIFFERENTIABILITY}
			\\
		\sup_{\tau \in [0,T]}
		\| \GradEnt^{\alpha} \|_{H^N(\Sigma_{\tau})}
		& \leq C,
			\label{E:2GAININDIFFERENTIABILITY}
				\\
		\sup_{\tau \in [0,T]}
		\| \vort^{\alpha} \|_{H^N(\Sigma_{\tau})}
		& \leq C.
		\label{E:3GAININDIFFERENTIABILITY}
\end{align}	

To complete the proof of \eqref{E:GAINOFONEDERIVATIVE}, 
we must show that for each spatial multi-index $\vec{I}$ with $|\vec{I}| = N$,
the map $t \rightarrow \partial_{\vec{I}} \GradEnt^{\alpha}(t,\cdot)$ is a continuous map from $[0,T]$ into $L^2(\mathbb{T}^3)$,
and similarly for $\vort^{\alpha}$
(the desired time-continuity results for $\Ent$ then follow from the relation $\partial_i \Ent = \GradEnt_i$).
To keep the presentation short, we illustrate only the right-continuity of 
these maps at $t=0$; 
the general statement can be proved by making minor modifications to the argument that we give.
That is, we will show that
\begin{subequations}
\begin{align} \label{E:DESIREDHIGHNORMCONTINUITY}
	\lim_{t \downarrow 0}
	\| \partial_{\vec{I}} \GradEnt^{\alpha}(t,\cdot) - \partial_{\vec{I}} \mathring{\GradEnt}^{\alpha}(\cdot) \|_{L^2(\mathbb{T}^3)}
	& = 0,
	&&
	|\vec{I}|=N,
		\\
	\lim_{t \downarrow 0}
	\| \partial_{\vec{I}} \vort^{\alpha}(t,\cdot) - \partial_{\vec{I}} \mathring{\vort}^{\alpha}(\cdot) \|_{L^2(\mathbb{T}^3)}
	& = 0,
	&&
	|\vec{I}|=N,
	\label{E:SECONDDESIREDHIGHNORMCONTINUITY}
\end{align}
\end{subequations}
where $\mathring{\GradEnt}^{\alpha}(\cdot) := \GradEnt^{\alpha}(0,\cdot)$.
The rest of our proof is based on Lemmas~\ref{L:CONTINUITYRESULTFORTRANSPORTEQUATIONS} and \ref{L:COMPARISONBASEDONELLIPTIC},
but to apply the lemmas, we first have to derive some preliminary results.
We will use the estimates provided by Lemma~\ref{L:STANDARDSOBOLEVRESULTS}
without giving complete details each time we use them; we will refer
to these estimates as the ``standard Sobolev calculus.''
In the rest of the proof, we will refer to the variable sets
$\mathbf{H}$,
$\mathbf{E}$,
$\mathbf{A}_{\mathbf{H}}$,
$\mathbf{A}_{\mathbf{H},\mathbf{E}}$,
and 
$\mathbf{A}$
from Def.\,\ref{D:CLASSIFICATIONOFSOLUTIONVARIABLES}.

As a first step in proving \eqref{E:DESIREDHIGHNORMCONTINUITY}-\eqref{E:SECONDDESIREDHIGHNORMCONTINUITY}, 
we will show that
\begin{align} \label{E:HYPERBOLICQUANTITIESINHNMINUSONE}
	\mathbf{H}, \, \mathbf{A}_{\mathbf{H}} \in C\big([0,T],H^{N-1}(\mathbb{T}^3)\big),
\end{align}
where $\mathbf{H}$ and $\mathbf{A}_{\mathbf{H}}$ are defined in
\eqref{E:HYPERBOLICVARS} and \eqref{E:ALGEBRAICVARSHYPERBOLICONLY}.
Note that by \eqref{E:ALGEBRAICHYPERBOLICONLYSMOOTHFUNCTIONEXPRESSION} and
the standard Sobolev calculus, 
the desired result $\mathbf{A}_{\mathbf{H}} \in C\big([0,T],H^{N-1}(\mathbb{T}^3)\big)$
would follow from
$\mathbf{H} \in C\big([0,T],H^{N-1}(\mathbb{T}^3)\big)$.
The latter statement is equivalent to showing that
$\partial_{\vec{I}} \mathbf{H} \in C\big([0,T],L^2(\mathbb{T}^3)\big)$
for $|\vec{I}| \leq N-1$. All of these results,
except in the case of the top-order (i.e., order $N-1$) derivatives of $\mathcal{C}^i$ and $\mathcal{D}$,
follow from the standard local well-posedness time-continuity results 
\eqref{E:FIRSTTHMSTANDARDCONTINUITYINHN}-\eqref{E:NOTYETGAINOFONEDERIVATIVE},
and the standard Sobolev calculus. Thus, to complete the proof of \eqref{E:HYPERBOLICQUANTITIESINHNMINUSONE},
we need only to show that for $i=1,2,3$, we have
\begin{align} \label{E:HNMINUSONECONTINUITYFORMODIFIEDDIVCURLQUANTITIES}
	\partial_{\vec{I}} \mathcal{C}^i, 
		\,
	\partial_{\vec{I}} \mathcal{D}
	\in C\big([0,T],L^2(\mathbb{T}^3)\big),
	&&
	|\vec{I}| = N-1.
\end{align}
The desired result \eqref{E:HNMINUSONECONTINUITYFORMODIFIEDDIVCURLQUANTITIES}
follows from using equations
\eqref{E:MAINTHMSMODIFIEDDIVERGENCEENTROPYGRADIENTTRANSPORT}
and
\eqref{E:MAINTHMTRANSPORTFORMODIFIEDVORTICITYOFVORTICITY}
(more precisely, we need only to consider the spatial components of \eqref{E:MAINTHMTRANSPORTFORMODIFIEDVORTICITYOFVORTICITY}),
the boundedness results \eqref{E:1GAININDIFFERENTIABILITY}-\eqref{E:3GAININDIFFERENTIABILITY},
the standard local well-posedness time-continuity results 
\eqref{E:FIRSTTHMSTANDARDCONTINUITYINHN}-\eqref{E:NOTYETGAINOFONEDERIVATIVE},
and the standard Sobolev calculus to deduce that
$\partial_{\vec{I}} \mathcal{C}^i$ and $\partial_{\vec{I}} \mathcal{D}$
solve transport equations that satisfy the hypotheses of Lemma~\ref{L:CONTINUITYRESULTFORTRANSPORTEQUATIONS};
put succinctly, we can apply
Lemma~\ref{L:CONTINUITYRESULTFORTRANSPORTEQUATIONS} with 
$f := \partial_{\vec{I}} \mathcal{C}^i$ and 
$f := \partial_{\vec{I}} \mathcal{D}$.
We have therefore proved \eqref{E:HYPERBOLICQUANTITIESINHNMINUSONE}.
In particular, it follows from \eqref{E:HYPERBOLICQUANTITIESINHNMINUSONE} and the definition of
$\mathbf{A}_{\mathbf{H}}$ that for $i=1,2,3$, we have
\begin{align} \label{E:HNMINUSONECONTINUITYFORTHREEDIVCURLQUANTITIES}
	\curlthreed^i(\underline{\vort}),
		\,
	\curlthreed^i(\underline{\GradEnt})
	\in C\big([0,T],H^{N-1}(\mathbb{T}^3)\big).
\end{align}

Next, we note that in view of Def.\,\ref{D:CLASSIFICATIONOFSOLUTIONVARIABLES},
Lemma~\ref{L:EXPRESSIONSFORALGEBRAICQUANTITIESINTERMSOFHYPERBOLICANDELLIPTIC}
(in particular the relation \eqref{E:ALGEBRAICHYPERBOLICANDELLIPTICSMOOTHFUNCTIONEXPRESSION}
for $\partial_a S^0$ and $\partial_a \vort^0$),
\eqref{E:HYPERBOLICQUANTITIESINHNMINUSONE},
and the standard Sobolev calculus,
the desired results \eqref{E:DESIREDHIGHNORMCONTINUITY}-\eqref{E:SECONDDESIREDHIGHNORMCONTINUITY}
would follow as a consequence of the following convergence result:
\begin{align} \label{E:L2CONTINUITYRESULTELLIPTICVARIABLES}
	\lim_{t \downarrow 0}
	\| \partial_{\vec{I}} \mathbf{E}(t,\cdot) - \partial_{\vec{I}} \mathbf{E}(0,\cdot) \|_{L^2(\mathbb{T}^3)}
	& = 0,
	&&
	|\vec{I}| = N-1.
\end{align}

To establish \eqref{E:L2CONTINUITYRESULTELLIPTICVARIABLES}, we first use
\eqref{E:HYPERBOLICQUANTITIESINHNMINUSONE},
\eqref{E:ALGEBRAICEXPRESSIONFORICOMMUTEDGEOMETRICDIVERGENCEOFVORTICITY},
and the standard Sobolev calculus
to deduce the following facts, where $(G^{-1})^{ij}$ is defined in Def.\,\ref{D:INVERSERIEMANNIANONSIGMAT}:
\begin{align} \label{E:HDIVOFGRADENTISCONTINUOUS}
	(G^{-1})^{ab} \partial_a \partial_{\vec{I}} \GradEnt_b,
		\,
	(G^{-1})^{ab} \partial_a \partial_{\vec{I}} \vort_b
	&
	\in C\big([0,T],L^2(\mathbb{T}^3)\big),
	&&
	|\vec{I}| = N-1.
\end{align}

In the rest of the proof,
$\upalpha_* > 0$ is as in the statement of Lemma~\ref{L:COMPARISONBASEDONELLIPTIC}
in the case $(M^{-1})^{ij}(t,x) := (G^{-1})^{ij}(t,x)$.
Next, setting
\begin{align} \label{E:INITIALRIEMANNIANMETRIC}
	(\mathring{G}^{-1})^{ij}(\cdot)
	& :=
	(G^{-1})^{ij}(0,\cdot),
\end{align}
applying Lemma~\ref{L:COMPARISONBASEDONELLIPTIC} with $(M^{-1})^{ij} := (\mathring{G}^{-1})^{ij}$,
and appealing to definition \eqref{E:ELLIPTICVARS},
we see that in order to prove \eqref{E:L2CONTINUITYRESULTELLIPTICVARIABLES}, 
it suffices to show the following convergence result:
\begin{align} \label{E:EQUIVALENTDESIREDHIGHNORMCONTINUITY}
	\lim_{t \downarrow 0}
	\mathbb{E}_{N;\mathring{G}^{-1};\upalpha_*}[(\underline{\vort},\underline{\GradEnt}) 
		- 
		(\mathring{\underline{\vort}},\mathring{\underline{\GradEnt}})](t)
	& = 0,
\end{align}
where $(\mathring{\underline{\vort}},\mathring{\underline{\GradEnt}}) := (\underline{\vort},\underline{\GradEnt})|_{\Sigma_0}$.

To initiate the proof of \eqref{E:EQUIVALENTDESIREDHIGHNORMCONTINUITY}, 
we let $\varphi \in H^{-N}(\mathbb{T}^3)$ be any element of the dual space of
$H^N(\mathbb{T}^3)$.
From the below-top-order continuity result \eqref{E:NOTYETGAINOFONEDERIVATIVE},
the top-order boundedness results \eqref{E:2GAININDIFFERENTIABILITY}-\eqref{E:3GAININDIFFERENTIABILITY},
and the density of $C^{\infty}$ functions in $H^{-N}(\mathbb{T}^3)$,
it is straightforward to deduce that the following ``weak continuity'' result holds for $i=1,2,3$:
\begin{align} \label{E:WEAKHIGHNORMCONTINUITY}
	\lim_{t \downarrow 0}
	\int_{\mathbb{T}^3}
		\GradEnt^i(t,x) \varphi
	\, dx
	& =
	\int_{\mathbb{T}^3}
		\mathring{\GradEnt}^i \varphi
	\, dx.
\end{align}
Since $\varphi$ was arbitrary, we conclude that 
$\GradEnt^i(t,\cdot)$ weakly converges to $\mathring{\GradEnt}^i$ in $H^N(\mathbb{T}^3)$
as $t \downarrow 0$. Similarly, $\vort^i(t,\cdot)$ weakly converges to $\mathring{\vort}^i$ in $H^N(\mathbb{T}^3)$
as $t \downarrow 0$. 
We now let $\langle \cdot, \cdot \rangle_{\mathring{G}^{-1};\upalpha_*}$ denote the inner product
\eqref{E:INNERPRODUCT} on the Hilbert space $\left(H^N(\Sigma_t) \right)^3 \times \left(H^N(\Sigma_t) \right)^3$,
and we let $\langle \cdot, \cdot \rangle$ denote the standard inner product on the same Hilbert space
(obtained by keeping only the two sums on the last line of RHS~\eqref{E:INNERPRODUCT}
and replacing $N-1$ with $N$ in the summation bounds).
By Lemma~\ref{L:COMPARISONBASEDONELLIPTIC}, the two corresponding norms
(i.e., the norms on the left- and right-hand sides 
of \eqref{E:ENERGYNORMCOMPARISON}-\eqref{E:SECONDENERGYNORMCOMPARISON})
are equivalent. It is a basic result of functional analysis that
given these two inner products with equivalent norms,
a sequence weakly convergences relative to 
$\langle \cdot, \cdot \rangle_{\mathring{G}^{-1};\upalpha_*}$
if and only if it weakly converges relative to
$\langle \cdot, \cdot \rangle$.
In particular, in view of the weak convergence results for $\GradEnt^i(t,\cdot)$ and $\vort^i(t,\cdot)$ proved above,
we infer that 
$(\underline{\vort}(t,\cdot),\underline{\GradEnt}(t,\cdot))$
weakly converges to $(\underline{\mathring{\vort}}(\cdot),\underline{\mathring{\GradEnt}}(\cdot))$ 
relative to the inner product $\langle \cdot, \cdot \rangle_{\mathring{G}^{-1};\upalpha_*}$
as $t \downarrow 0$.
Moreover, it is another basic result of functional analysis that
based on this weak convergence and Lemma~\ref{L:COMPARISONBASEDONELLIPTIC},
in order to prove the result \eqref{E:EQUIVALENTDESIREDHIGHNORMCONTINUITY}, 
it suffices to show that
\begin{align} \label{E:ANOTHEREQUIVALENTDESIREDHIGHNORMCONTINUITY}
		\lim \sup_{t \downarrow 0}
		\mathbb{E}_{N;\mathring{G}^{-1};\upalpha_*}[(\underline{\vort},\underline{\GradEnt})](t)
		& \leq 
		\mathbb{E}_{N;\mathring{G}^{-1};\upalpha_*}[(\mathring{\underline{\vort}},\mathring{\underline{\GradEnt}})].
\end{align}	
Moreover, 
since the standard local well-posedness time-continuity results 
\eqref{E:FIRSTTHMSTANDARDCONTINUITYINHN}
and \eqref{E:STANDARDSOBOLEVEMBEDDING} 
imply that
$
\lim_{t \downarrow 0} 
\left\| (G^{-1})^{ij}(t,\cdot) 
	- 
(\mathring{G}^{-1})^{ij} 
\right\|_{C(\mathbb{T}^3)} 
= 
0
$,
it follows from 
definitions \eqref{E:INNERPRODUCT} and \eqref{E:TOPORDERENERGY}
and the top-order boundedness results \eqref{E:1GAININDIFFERENTIABILITY}-\eqref{E:3GAININDIFFERENTIABILITY}
that in order to prove \eqref{E:ANOTHEREQUIVALENTDESIREDHIGHNORMCONTINUITY},
it suffices to show that
\begin{align} \label{E:ALMOSTFINALEQUIVALENTDESIREDHIGHNORMCONTINUITY}
		\lim \sup_{t \downarrow 0}
		\mathbb{E}_{N;G^{-1};\upalpha_*}[(\underline{\vort},\underline{\GradEnt})](t)
		& \leq 
		\mathbb{E}_{N;\mathring{G}^{-1};\upalpha_*}[(\mathring{\underline{\vort}},\mathring{\underline{\GradEnt}})],
\end{align}	
where we stress that the inverse metric $G^{-1}$ 
on LHS~\eqref{E:ALMOSTFINALEQUIVALENTDESIREDHIGHNORMCONTINUITY}
depends on $t$ (which is different compared to \eqref{E:ANOTHEREQUIVALENTDESIREDHIGHNORMCONTINUITY}).
In fact, our arguments will yield a stronger statement than \eqref{E:ALMOSTFINALEQUIVALENTDESIREDHIGHNORMCONTINUITY}.
More precisely, we will show the following time-continuity result:
\begin{align} \label{E:STRONGERALMOSTFINALEQUIVALENTDESIREDHIGHNORMCONTINUITY}
		\lim_{t \downarrow 0}
		\mathbb{E}_{N;G^{-1};\upalpha_*}[(\underline{\vort},\underline{\GradEnt})](t)
		& =
		\mathbb{E}_{N;\mathring{G}^{-1};\upalpha_*}[(\mathring{\underline{\vort}},\mathring{\underline{\GradEnt}})],
\end{align}	
To proceed, we use definitions \eqref{E:INNERPRODUCT} and \eqref{E:TOPORDERENERGY}
and the standard local well-posedness time-continuity results 
\eqref{E:FIRSTTHMSTANDARDCONTINUITYINHN}-\eqref{E:NOTYETGAINOFONEDERIVATIVE}
to deduce that all terms in the definition of $\mathbb{E}_{N;G^{-1};\upalpha_*}[(\vort,\GradEnt)](t)$
have been shown to have the desired continuous time dependence at
except for the ones depending on the order $N$ derivatives of $\vort$ or $\GradEnt$
(i.e., the ones corresponding to the terms on the first four lines of RHS~\eqref{E:INNERPRODUCT}).
The continuous time dependence of these remaining four terms follows from
\eqref{E:HNMINUSONECONTINUITYFORTHREEDIVCURLQUANTITIES},
\eqref{E:HDIVOFGRADENTISCONTINUOUS},
and the fact that
$
 (G^{-1})^{ij} \in C\big([0,T],C(\mathbb{T}^3)\big)
$
(which follows from the standard local well-posedness time-continuity results 
\eqref{E:FIRSTTHMSTANDARDCONTINUITYINHN} and \eqref{E:STANDARDSOBOLEVEMBEDDING}).
We have therefore proved \eqref{E:STRONGERALMOSTFINALEQUIVALENTDESIREDHIGHNORMCONTINUITY},
which finishes the proof of the desired result \eqref{E:GAINOFONEDERIVATIVE}.

To complete our proof of Theorem~\ref{T:UPGRADEDLOCALWELLPOSEDNESS},
we need to show continuous dependence on the initial data.
To proceed, we let $(\mathring{\Lnenth}_{(m)},\mathring{u}_{(m)}^i,\mathring{\Ent}_{(m)})$
be a sequence of initial data 
(not necessarily $C^{\infty}$ now)
such that as $m \to \infty$, 
the convergence results \eqref{E:STANDARDDATACONVERGENCE}-\eqref{E:GAINOFONEDERIVATIVEDATACONVERGENCE} hold.
We again let $(\Lnenth_{(m)},\Ent_{(m)},u_{(m)}^{\alpha},\GradEnt_{(m)}^{\alpha},\vort_{(m)}^{\alpha})$ denote the corresponding
sequence of solution variables (which are not necessarily $C^{\infty}$ now). We aim to show that the sequence 
converges to the limiting solution $(\Lnenth,u^{\alpha},\Ent,\GradEnt^{\alpha},\vort^{\alpha})$ 
in the norm $\| \cdot \|_{C\big([0,T],H^N\mathbb{T}^3)\big)}$ as $m \to \infty$.
To proceed, we first note that Theorem~\ref{T:STANDLOCALWELLPOSEDNESS}
	yields that for $m$ sufficiently large, the element
	$(\Lnenth_{(m)},\Ent_{(m)},u_{(m)}^{\alpha},\GradEnt_{(m)}^{\alpha},\vort_{(m)}^{\alpha})$ is a 
	classical solution (not necessarily $C^{\infty}$ now)
	to equations \eqref{E:ENTHALPYEVOLUTION}-\eqref{E:ENTROPYEVOLUTION} + \eqref{E:UISUNITLENGTH}
	on the fixed slab $[0,T] \times \mathbb{T}^3$ 
	with $(\Lnenth_{(m)}(p),\Ent_{(m)}(p),u_{(m)}^1(p),u_{(m)}^2(p),u_{(m)}^3(p)) 
	\in \mbox{\upshape int} \mathfrak{K}$ for $p \in [0,T] \times \mathbb{T}^3$,
	that it also is a strong solution\footnote{By ``strong solution,'' we mean in particular
	that at each fixed $t \in [0,T]$,
	the equations of Theorem~\ref{T:NEWFORMULATIONRELEULER} are satisfied for almost every $x \in \mathbb{T}^3$.} 
	to the equations of Theorem~\ref{T:NEWFORMULATIONRELEULER},
	that there exists an integer $m_0$ such that
\begin{align} \label{E:1AGAINITERATESBOUNDEDINHIGHNORM}
		\sup_{m \geq m_0}
		\| \Ent_{(m)} \|_{C\big([0,T],H^{N+1}(\mathbb{T}^3)\big)}
		& \leq C,
			\\
		\sup_{m \geq m_0}
		\| \GradEnt_{(m)}^{\alpha} \|_{C\big([0,T],H^N(\mathbb{T}^3)\big)}
		& \leq C,
			\label{E:2AGAINITERATESBOUNDEDINHIGHNORM}
			\\
		\sup_{m \geq m_0}
		\| \vort_{(m)}^{\alpha} \|_{C\big([0,T],H^N(\mathbb{T}^3)\big)}
		& \leq C,
			\label{E:3AGAINITERATESBOUNDEDINHIGHNORM}
\end{align}	
and that the following convergence results 
(which are below top-order for $\GradEnt$ and $\vort$)
hold as $m \to \infty$:
\begin{align} \label{E:1STANDARDCONTINUOUSDEPENDENCEONINITIALDATA}
	\left\| \Lnenth - \Lnenth_{(m)} \right\|_{C\big([0,T],H^N(\mathbb{T}^3)\big)}
	& \to 0,
			\\
	\left\| u^{\alpha} - u_{(m)}^{\alpha} \right\|_{C\big([0,T],H^N(\mathbb{T}^3)\big)}
	& \to 0,
		 \label{E:2STANDARDCONTINUOUSDEPENDENCEONINITIALDATA} 
		\\
	\left\| \Ent - \Ent_{(m)} \right\|_{C\big([0,T],H^N(\mathbb{T}^3)\big)}
	& \to 0,
		 \label{E:3STANDARDCONTINUOUSDEPENDENCEONINITIALDATA} 
			\\
	\left\| \GradEnt^{\alpha} - \GradEnt_{(m)}^{\alpha} \right\|_{C\big([0,T],H^{N-1}(\mathbb{T}^3)\big)}
	& \to 0,
		\label{E:1SECONDSTANDARDCONTINUOUSDEPENDENCEONINITIALDATA} \\
	\left\| \vort^{\alpha} - \vort_{(m)}^{\alpha} \right\|_{C\big([0,T],H^{N-1}(\mathbb{T}^3)\big)}
	& \to 0.
	\label{E:2SECONDSTANDARDCONTINUOUSDEPENDENCEONINITIALDATA}
\end{align}
In view of \eqref{E:1STANDARDCONTINUOUSDEPENDENCEONINITIALDATA}-\eqref{E:2SECONDSTANDARDCONTINUOUSDEPENDENCEONINITIALDATA}, 
we see that to complete our proof of Theorem~\ref{T:UPGRADEDLOCALWELLPOSEDNESS},
we need only to show continuity in the top-order norms. That is,
we must show that if $|\vec{I}| = N$, then as $m \to \infty$, 
we have
\begin{align} \label{E:1DESIREDCONTINUOUSDEPENDENCEONINITIALDATA}
	\left\| 
		\partial_{\vec{I}} \GradEnt^{\alpha} - \partial_{\vec{I}} \GradEnt_{(m)}^{\alpha} 
	\right\|_{C\big([0,T],L^2(\mathbb{T}^3)\big)}
	& \to 0,
		\\
	\left\| 
		\partial_{\vec{I}} \vort^{\alpha} - \partial_{\vec{I}} \vort_{(m)}^{\alpha} 
	\right\|_{C\big([0,T],L^2(\mathbb{T}^3)\big)}
	& \to 0.
	\label{E:2DESIREDCONTINUOUSDEPENDENCEONINITIALDATA}
\end{align}

To proceed, we first review an approach to proving the standard estimates 
\eqref{E:1STANDARDCONTINUOUSDEPENDENCEONINITIALDATA}-\eqref{E:2SECONDSTANDARDCONTINUOUSDEPENDENCEONINITIALDATA}.
These estimates can be proved by applying Kato's abstract framework
\cites{tK1970,tK1973,tK1975}, 
which is designed to handle first-order hyperbolic systems
in a rather general Banach space setting.
In particular, one can apply Kato's framework to the first-order system
\eqref{E:ENTHALPYEVOLUTION}-\eqref{E:ENTROPYEVOLUTION} + \eqref{E:UISUNITLENGTH}; 
this is described in detail, for example, in \cite{jS2008a}.
To prove \eqref{E:1DESIREDCONTINUOUSDEPENDENCEONINITIALDATA}-\eqref{E:2DESIREDCONTINUOUSDEPENDENCEONINITIALDATA}, we will modify Kato's 
framework so that it applies to the
hyperbolic variables $\mathbf{H}$ and
the elliptic variables $\mathbf{E}$ from
Def.\,\ref{D:CLASSIFICATIONOFSOLUTIONVARIABLES}.

To employ Kato's framework,
one relies on the propagators $\mathcal{U}(t,\tau) := \mathcal{U}(t,\tau;\mathbf{H})$
for the linear homogeneous hyperbolic system corresponding to 
the (nonlinear) first-order hyperbolic system that $\mathbf{H}$ satisfies.
To shorten the presentation, we will not explicitly state the form of this linear first-order hyperbolic system; 
see Remark~\ref{R:HYPERBOLICVARIABLESSOLVEFIRSTORDERTRANSPORT} for further discussion of its nature.
By definition, $\mathcal{U}(t,\tau;\mathbf{H})$ maps initial data
given at time $\tau$ to the solution of the 
linear homogeneous hyperbolic system (whose principal coefficients depend on $\mathbf{H}$)
at time $t$. Similarly, one relies
on the operators $\mathcal{U}_{(m)}(t,\tau) := \mathcal{U}(t,\tau;\mathbf{H}_{(m)})$
corresponding to the homogeneous linear system whose principal coefficients depend on $\mathbf{H}_{(m)}$.
By Duhamel's principle, we have
\begin{align} \label{E:KATODUHAMEL}
	\mathbf{H}(t)
	& = 
		\mathcal{U}(t,0) \mathring{\mathbf{H}}
		+
		\int_{\tau=0}^t
			\mathcal{U}(t,\tau) \mathrm{f}\left(\mathbf{H}(\tau),\mathbf{E}(\tau) \right)
		\, d \tau,
			\\
	\mathbf{H}_{(m)}(t)
	& = 
		\mathcal{U}_{(m)}(t,0) \mathring{\mathbf{H}}_{(m)}
		+
		\int_{\tau=0}^t
			\mathcal{U}_{(m)}(t,\tau) \mathrm{f}\left(\mathbf{H}_{(m)}(\tau),\mathbf{E}_{(m)}(\tau) \right)
		\, d \tau,
		\label{E:KATODUHAMELSEQUENCE}
\end{align}
where $\mathring{\mathbf{H}}$ and $\mathring{\mathbf{H}}_{(m)}$
respectively denote the initial data of $\mathbf{H}$ and $\mathbf{H}_{(m)}$,
and on RHSs~\eqref{E:KATODUHAMEL}-\eqref{E:KATODUHAMELSEQUENCE},
$\mathrm{f}$ denotes the inhomogeneous term in 
the first-order hyperbolic system satisfied by the 
elements of $\mathbf{H}$ and $\mathbf{H}_{(m)}$.
We have not explicitly stated the form of $\mathrm{f}$
since its precise structure is not important for our arguments here;
what matters is only the following basic facts (that can easily be checked):
$\mathrm{f}$ is a smooth function of its arguments satisfying $\mathrm{f}(0)=0$,
and the \emph{same} $\mathrm{f}$ appears on RHSs~\eqref{E:KATODUHAMEL}-\eqref{E:KATODUHAMELSEQUENCE}.

The strategy behind Kato's framework is to control the difference 
\linebreak
$
\mathbf{H}(t,\cdot) - \mathbf{H}_{(m)}(t,\cdot)
$
in the norm $\| \cdot \|_{H^{N-1}(\mathbb{T}^3)}$
by subtracting \eqref{E:KATODUHAMEL}-\eqref{E:KATODUHAMELSEQUENCE},
splitting the right-hand side of the resulting equation into various pieces, and
bounding each piece by exploiting some standard properties of
the propagators $\mathcal{U}(t,\tau)$ and $\mathcal{U}_{(m)}(t,\tau)$.
This is explained in detail in \cite{jS2008a}*{Section~7.4}, and most of the arguments given there
for controlling
$
\| \mathbf{H}(t,\cdot) - \mathbf{H}_{(m)}(t,\cdot) \|_{H^{N-1}(\mathbb{T}^3)}
$
go through without any substantial changes.
The one part of the argument that does require
substantial changes is: 
in order to obtain a closed inequality for
$
\| \mathbf{H}(t,\cdot) - \mathbf{H}_{(m)}(t,\cdot) \|_{H^{N-1}(\mathbb{T}^3)}
$,
one needs to show that the difference of the inhomogeneous terms
on RHSs~\eqref{E:KATODUHAMEL}-\eqref{E:KATODUHAMELSEQUENCE}
satisfies the following estimate for $t \in [0,T]$:
\begin{align} \label{E:DESIREDKEYINHOMOGENEOUSTERMDIFFERENCEESTIMATE}
	\left\|
		\mathrm{f}(\mathbf{H},\mathbf{E})
		-
		\mathrm{f}(\mathbf{H}_{(m)},\mathbf{E}_{(m)})
	\right\|_{H^{N-1}(\Sigma_t)}
	& \leq 
		C \|\mathbf{H} - \mathbf{H}_{(m)} \|_{H^{N-1}(\Sigma_t)},
\end{align}
where the key point is that \emph{the quantity
$\|\mathbf{E} - \mathbf{E}_{(m)} \|_{H^{N-1}(\Sigma_t)}$
does not appear on RHS~\eqref{E:DESIREDKEYINHOMOGENEOUSTERMDIFFERENCEESTIMATE}}.

The estimate \eqref{E:DESIREDKEYINHOMOGENEOUSTERMDIFFERENCEESTIMATE}
can be obtained with the help of elliptic estimates,
as we now explain. First, we note that the top-order norm-boundedness results 
\eqref{E:1AGAINITERATESBOUNDEDINHIGHNORM}-\eqref{E:3AGAINITERATESBOUNDEDINHIGHNORM}
and the convergence results
\eqref{E:1STANDARDCONTINUOUSDEPENDENCEONINITIALDATA}-\eqref{E:2SECONDSTANDARDCONTINUOUSDEPENDENCEONINITIALDATA}
imply that
\begin{align} \label{E:BASELEVELALLVARIABLEDIFFERENCEESTIMATE}
	\lim_{m \to \infty}
	\left\lbrace
	\left\|
		\mathbf{H}
		-
		\mathbf{H}_{(m)}
	\right\|_{C\big([0,T], L^2(\mathbb{T}^3)\big)}
	+
	\left\|
		\mathbf{E}
		-
		\mathbf{E}_{(m)}
	\right\|_{C\big([0,T], L^2(\mathbb{T}^3)\big)}
	\right\rbrace
	& = 0,
\end{align}
and that there exists an integer $m_0$ and a constant $C > 0$ such that
\begin{align}
	\left\|
		\mathbf{H}
	\right\|_{C\big([0,T], H^{N-1}(\mathbb{T}^3)\big)}
	+
	\left\|
		\mathbf{E}
	\right\|_{C\big([0,T], H^{N-1}(\mathbb{T}^3)\big)}
	& \leq C,
		\label{E:TOPORDERNORMBOUNDED} \\
	\sup_{m \geq m_0}
	\left\lbrace
		\left\|
		\mathbf{H}_{(m)}
	\right\|_{C\big([0,T], H^{N-1}(\mathbb{T}^3)\big)}
	+
	\left\|
		\mathbf{E}_{(m)}
	\right\|_{C\big([0,T], H^{N-1}(\mathbb{T}^3)\big)}
	\right\rbrace
	& \leq C.
	\label{E:SEQUENCETOPORDERNORMBOUNDED}
\end{align}
From 
\eqref{E:BASELEVELALLVARIABLEDIFFERENCEESTIMATE},
\eqref{E:TOPORDERNORMBOUNDED}-\eqref{E:SEQUENCETOPORDERNORMBOUNDED},
and the Sobolev interpolation result \eqref{E:STANDARDSOBOLEVINTERPOLATION},
we deduce that if $N' < N-1$, then
\begin{align} \label{E:ANYTHINGBELOWTOPALLVARIABLEDIFFERENCEESTIMATE}
	\lim_{m \to \infty}
	\left\lbrace
		\left\|
			\mathbf{H}
			-
			\mathbf{H}_{(m)}
		\right\|_{C\big([0,T],H^{N'}(\mathbb{T}^3)\big)}
		+
		\left\|
			\mathbf{E}
			-
			\mathbf{E}_{(m)}
		\right\|_{C\big([0,T],H^{N'}(\mathbb{T}^3)\big)}
	\right\rbrace
	& = 0.
\end{align}
Fixing a real number $N'$ satisfying $3/2 < N' < 2$
and using \eqref{E:ANYTHINGBELOWTOPALLVARIABLEDIFFERENCEESTIMATE} and the Sobolev embedding result \eqref{E:STANDARDSOBOLEVEMBEDDING}, 
we deduce that
\begin{align} \label{E:SUPNORMHYPVARIABLEDIFFERENCEESTIMATE}
	\lim_{m \to \infty}
	\left\lbrace
	\left\|
		\mathbf{H}
		-
		\mathbf{H}_{(m)}
	\right\|_{C([0,T] \times \mathbb{T}^3)}
	+
	\left\|
		\mathbf{E}
		-
		\mathbf{E}_{(m)}
	\right\|_{C([0,T] \times \mathbb{T}^3)}
	\right\rbrace
	& = 0.
\end{align}
Next, we use 
\eqref{E:TOPORDERNORMBOUNDED},
\eqref{E:SEQUENCETOPORDERNORMBOUNDED},
\eqref{E:SUPNORMHYPVARIABLEDIFFERENCEESTIMATE},
\eqref{E:STANDARDSOBOLEVEMBEDDING},
and \eqref{E:CONTINUITYESTIMATEWITHSMOOTHFUNCTIONFRAMENORML2PRODUCTBOUNDINERMSOFLINFINITYANDHMDOT}
to deduce that there is a constant $C>0$ such that if $m$ is sufficiently large, 
then for $t \in [0,T]$, the following estimate holds for
the function $\mathrm{f}$ appearing on RHSs~\eqref{E:KATODUHAMEL}-\eqref{E:KATODUHAMELSEQUENCE}:
\begin{align} \label{E:KEYINHOMOGENEOUSTERMDIFFERENCEESTIMATE}
	\left\|
		\mathrm{f}(\mathbf{H},\mathbf{E})
		-
		\mathrm{f}(\mathbf{H}_{(m)},\mathbf{E}_{(m)})
	\right\|_{H^{N-1}(\Sigma_t)}
	& \leq 
		C \|\mathbf{H} - \mathbf{H}_{(m)} \|_{H^{N-1}(\Sigma_t)}
			\\
	& \ \
		+
		C \|\mathbf{E} - \mathbf{E}_{(m)} \|_{H^{N-1}(\Sigma_t)}.
		\notag
\end{align}
Next, we use 
\eqref{E:TOPORDERNORMBOUNDED},
\eqref{E:SEQUENCETOPORDERNORMBOUNDED},
\eqref{E:SUPNORMHYPVARIABLEDIFFERENCEESTIMATE},
and
\eqref{E:KEYINHOMOGENEOUSTERMELLIPTICVARIABLESESTIMATE}
to deduce that if $m$ is sufficiently large, then for $t \in [0,T]$,
the last term on RHS~\eqref{E:KEYINHOMOGENEOUSTERMDIFFERENCEESTIMATE} 
obeys the following bound:
\begin{align} \label{E:KEYINHOMOGENEOUSTERMELLIPTICVARIABLESDIFFERENCESEQUENCEESTIMATE}
	\|\mathbf{E} - \mathbf{E}_{(m)} \|_{H^{N-1}(\Sigma_{t})}
	& \leq
	C \|\mathbf{H} - \mathbf{H}_{(m)} \|_{H^{N-1}(\Sigma_{t})}.
\end{align}
The desired bound \eqref{E:DESIREDKEYINHOMOGENEOUSTERMDIFFERENCEESTIMATE}
follows from
\eqref{E:KEYINHOMOGENEOUSTERMDIFFERENCEESTIMATE}
and
\eqref{E:KEYINHOMOGENEOUSTERMELLIPTICVARIABLESDIFFERENCESEQUENCEESTIMATE}.
Kato's framework 
(see \cite{jS2008a}*{Section~7.4})
then allows one to conclude that
\begin{align} \label{E:DESIREDHYPERBOLICVARIABLEESTIMATE}
	\lim_{m \to \infty}
	\left\|
		\mathbf{H}
		-
		\mathbf{H}_{(m)}
	\right\|_{C\big([0,T], H^{N-1}(\mathbb{T}^3)\big)}
	& = 0.
\end{align}
Moreover,
\eqref{E:KEYINHOMOGENEOUSTERMELLIPTICVARIABLESDIFFERENCESEQUENCEESTIMATE}
and
\eqref{E:DESIREDHYPERBOLICVARIABLEESTIMATE}
imply that
\begin{align} \label{E:DESIREDELLIPTICVARIABLEESTIMATE}
	\lim_{m \to \infty}
	\left\|
		\mathbf{E}
		-
		\mathbf{E}_{(m)}
	\right\|_{C\big([0,T], H^{N-1}(\mathbb{T}^3)\big)}
	& = 0.
\end{align}
Finally, in view of Def.\,\ref{D:CLASSIFICATIONOFSOLUTIONVARIABLES}
and the relation \eqref{E:ALLALGEBRAICSMOOTHFUNCTIONEXPRESSION},
we note that the desired convergence results \eqref{E:1DESIREDCONTINUOUSDEPENDENCEONINITIALDATA}-\eqref{E:2DESIREDCONTINUOUSDEPENDENCEONINITIALDATA}
follow from \eqref{E:DESIREDHYPERBOLICVARIABLEESTIMATE}-\eqref{E:DESIREDELLIPTICVARIABLEESTIMATE}
and the standard Sobolev calculus 
(which is needed to handle the components $\alpha = 0$ in \eqref{E:1DESIREDCONTINUOUSDEPENDENCEONINITIALDATA}-\eqref{E:2DESIREDCONTINUOUSDEPENDENCEONINITIALDATA}).

\hfill $\qed$

\bibliographystyle{amsalpha}
\bibliography{JBib}

\end{document}